\documentclass[11pt,reqno]{amsart}

\pdfoutput=1

\usepackage{L25_macros}

\begin{document}

\title[Asymptotic stability of solitary waves for the 1D focusing cubic NLS]{asymptotic stability of solitary waves for the 1D focusing cubic Schr\"odinger equation}

\begin{abstract}
We establish the full asymptotic stability of solitary wave solutions for the 1D focusing cubic Schr\"odinger equation on the line under small perturbations in weighted Sobolev spaces, building upon our results in \cite{LL24}. The proof integrates the space-time resonances approach, based on the distorted Fourier transform, with modulation techniques to show modified scattering for the radiation term and convergence for the modulation parameters. A key challenge throughout the nonlinear analysis is the slow local decay of the radiation term, caused by  threshold resonances in the linearized operator. The presence of favorable null structures in the quadratic nonlinearities mitigates this problem through the use of normal form transformations. Another essential step in the proof involves developing a variant of the local smoothing estimate that incorporates a moving center.
\end{abstract}

\author[Y. Li]{Yongming Li}
\address{Department of Mathematics \\ Texas A\&M University \\ College Station, TX 77843, USA}
\email{liyo0008@tamu.edu}

\thanks{The author was partially supported by NSF CAREER grant DMS-2235233.}

\maketitle 

\tableofcontents

\section{Introduction} \label{sec:intro}
\subsection{Main result}
We consider the one-dimensional focusing cubic Schr\"odinger equation
\begin{equation} \label{equ:cubic_NLS}
\left\{ \begin{aligned}
i\pt \psi + \px^2 \psi + |\psi|^2 \psi &= 0, \quad (t,x) \in \bbR \times \bbR, \\
\psi(0) &= \psi_0.
\end{aligned} \right.
\end{equation}
The Cauchy problem for \eqref{equ:cubic_NLS} is globally well-posed both in the space $L^2(\bbR)$ as well as in the energy space $H^1(\bbR)$; see \cite{Tsutsumi87,GV78}. Moreover, any $H^1$-solution $\psi(t,x)$ to  \eqref{equ:cubic_NLS} conserves mass, momentum, and energy 
\begin{equation}\label{equ:conservation_laws}
M[\psi] := \frac12 \int_{\bbR} |\psi|^2 \, \ud x, \quad P[\psi] := \frac12 \Im \int_\bbR \partial_x \psi \overline{\psi} \,\ud x, \quad E[\psi] := \int_{\bbR} \biggl( \frac12 |\px \psi|^2 - \frac14 |\psi|^4 \biggr) \, \ud x.
\end{equation}
The conservation laws are generated by the invariances of \eqref{equ:cubic_NLS} under phase shifts, Galilei boosts, and translations: namely if $\psi(t,x)$ is a solution to \eqref{equ:cubic_NLS}, then for any $p, \gamma, \sigma \in \bbR$ the function
\begin{equation*}
e^{i p x} e^{-i p^2 t} e^{i \gamma} \psi(t, x - 2pt -\sigma)
\end{equation*}
is also a solution to \eqref{equ:cubic_NLS}. Moreover, there is the scaling invariance
\begin{equation*}
\psi(t,x) \mapsto \lambda \psi(\lambda^2 t, \lambda x), \quad \lambda > 0.
\end{equation*}
It is also well-known that the equation \eqref{equ:cubic_NLS} admits solitary wave solutions of the form $\psi(t,x) = e^{it\omega}\phi_\omega(x)$, where $\omega>0$ and $\phi_\omega(x)$ is the ground state (positive and decaying) of  the equation 
\begin{equation} \label{equ:intro_phi_omega_equation}
-\px^2 \phi_\omega + \omega \phi_\omega = \phi_\omega^3, \quad x \in \bbR.
\end{equation}
The explicit solutions to \eqref{equ:intro_phi_omega_equation}, modulo translations, are in fact given by 
\begin{equation*}
\phi_\omega(x) := \sqrt{\omega} \phi(\sqrt{\omega} x) = \sqrt{2\omega} \sech(\sqrt{\omega}x).
\end{equation*}
The invariances of \eqref{equ:cubic_NLS} then generate a 4-parameter family of solitary wave solutions
\begin{equation} \label{equ:intro_family_4parameter}
e^{i p x} e^{i (\omega - p^2) t} e^{i\gamma} \phi_\omega(x-2pt-\sigma), \quad (\omega, p, \gamma, \sigma) \in (0,\infty) \times \bbR^3.
\end{equation}

Orbital stability of the solitary waves in the energy space follows from the classical works of Cazenave-Lions \cite{82CazenaveLions}, Shatah-Strauss \cite{85ShatahStrauss}, and Weinstein \cite{Weinstein85,Weinstein86}. Using the complete integrability of \eqref{equ:cubic_NLS}, Mizumachi-Pelinovsky \cite{MizumachiPelinovsky12} proved, via the B\"acklund transformation, the orbital stability of solitary waves for \eqref{equ:cubic_NLS} in the weaker topology $L^2(\bbR)$. 

As pointed out by Martel \cite{Martel22}, the asymptotic stability of the solitary waves for \eqref{equ:cubic_NLS} cannot hold true in the energy space due to the existence of 2-solitons, as constructed by Olmedilla \cite{Olmedilla87} (see also \cite{ZakhShab72}), which are arbitrary close in the energy norm $H^1$ (but not in weighted $L^2$ spaces) to the soliton \eqref{equ:intro_family_4parameter}. Nevertheless, the asymptotic stability of the solitary waves for small perturbations in weighted $L^2$ spaces were established by Cuccagna-Pelinovsky \cite{CuccPeli14}, using the inverse scattering transform. See also \cite{Saalmann17, BorghJenMcL18} for related results using completely integrable techniques. 
 
The aim of this paper is to establish the asymptotic stability of the solitary waves \eqref{equ:intro_family_4parameter} for small perturbations in weighted Sobolev spaces by perturbative techniques. The main result is the following theorem, which extends our results in \cite{LL24} and removes the (even) symmetry restriction imposed in \cite{LL24} on the initial data.

\begin{theorem} \label{thm:main_theorem}
For any $\omega_0 \in (0,\infty)$ there exist constants $0 < \varepsilon_0 \ll 1$, $0 < \delta \ll 1$, and $C_0 \geq 1$ with the following property:
Let $(\gamma_0,p_0,\sigma_0) \in \bbR^3$ and let $u_0 \in H^1_x(\bbR) \cap L^{2,1}_x(\bbR)$  with
\begin{equation} \label{equ:theorem_statement_smallness_initial_condition}
\varepsilon := \| u_0 \|_{H^1_x(\bbR)} + \|\jx u_0\|_{L^2_x(\bbR)} \leq \varepsilon_0.
\end{equation}
Then the $H^1_x \cap L^{2,1}_x$--solution $\psi(t,x)$ to \eqref{equ:cubic_NLS} with initial condition
\begin{equation} \label{equ:theorem_statement_initial_condition}
\psi_0(x) = e^{ip_0(x-\sigma_0)} e^{i \gamma_0} \bigl( \phi_{\omega_0}(x-\sigma_0) + u_0(x-\sigma_0) \bigr)
\end{equation}
exists globally in time and there exist continuously differentiable paths $(\omega, \gamma,p,\sigma) \colon [0,\infty) \to (0,\infty) \times \bbR^3$ with
\begin{equation*}
	|\omega(0)-\omega_0| + |\gamma(0)-\gamma_0| + |p(0)-p_0| + |\sigma(0)-\sigma_0| \leq C_0 \eps
\end{equation*}
such that the following holds for all $t \geq 0$:
\begin{enumerate}[leftmargin=1.8em]
	\item Decomposition of the solution into a modulated solitary wave and a radiation term:
\begin{equation*}
	\psi(t,x) = e^{ip(t)(x-\sigma(t))} e^{i \gamma(t)} \bigl( \phi_{\omega(t)}(x-\sigma(t)) + u(t,x-\sigma(t)) \bigr).
\end{equation*}
\item Dispersive decay of the radiation term:
\begin{equation} \label{equ:theorem_statement_decay_radiation}
	\|u(t)\|_{L^{\infty}_x(\bbR)} \leq C_0\varepsilon \jt^{-\frac12}.
\end{equation}
\item Asymptotic behavior of the modulation parameters: there exist  $\omega_\infty, p_\infty \in (0,\infty) \times \bbR$ such that 
\begin{align} 
	|\omega(t)-\omega_\infty| + |p(t)-p_\infty| \leq C_0 \varepsilon \jt^{-1+\delta},\label{equ:theorem_statement_decay_modulation_parameters}\\
|\dot{\sigma}(t)-2p_\infty| +  |\dot{\gamma}(t)-p_\infty^2 -\omega_\infty| \leq C_0 \varepsilon \jt^{-1+\delta}.\label{equ:theorem_statement_decay_modulation_parameters2}
\end{align}
\item Asymptotics for the radiation term: there exist asymptotic profiles $W_+,W_-\in L_\xi^\infty(\bbR)$ such that for all $t \geq 1$,
\begin{equation} \label{eqn: theorem-u-asymptotics}
\| u(t,x-\sigma(t)) - u_\infty(t,x) \|_{L_x^\infty(\bbR)} \leq C_0 \varepsilon t^{-\frac35+\delta},
\end{equation}
with
\begin{equation}\label{equ:u_infty_formula}
\begin{split}
u_\infty(t,x) &:= \frac{1}{\sqrt{2t}} e^{i\Omega(t,x,\xi_*)} e^{-i\log(t)|W_+(\xi_*)|^2}W_+(\xi_*)m_{1,\infty}(t,x,\xi_*)\\
&\quad  - \frac{1}{\sqrt{2t}}e^{-i\Omega(t,x,\xi_*)} e^{i\log(t)|W_-(\xi_*)|^2} W_-(-\xi_*)m_{2,\infty}(t,x,-\xi_*)
\end{split}
\end{equation}
where
\begin{equation}\label{equ:theorem_asymptotics_notation}
\begin{split}
\xi_* &:= \frac{x-\sigma(t)}{2t},\\%\qquad
\theta_{1,\infty}(t) &:= \int_0^t \big(\dot{\gamma}(s)-\omega_\infty + p_\infty^2-\dot{\sigma}(s)p_\infty \big)\,\ud s,\\
\theta_{2,\infty}(t) &:= \int_0^t \big(\dot{\sigma}(s) - 2p_\infty \big)\,\ud s, \\%\qquad
m_{1,\infty}(t,x,\xi_*) &:= \frac{\big(\xi_*+i\sqrt{\omega_\infty}\tanh(\sqrt{\omega_\infty}(x-\sigma(t))\big)^2}{\big(|\xi_*|-i\sqrt{\omega_\infty}\big)^2},\\
m_{2,\infty}(t,x,\xi_*) &:= \frac{\big(\sqrt{\omega_\infty}\sech(\sqrt{\omega_\infty}(x-\sigma(t))\big)^2}{\big(|\xi_*|-i\sqrt{\omega_\infty}\big)^2},\\
\Omega(t,x,\xi_*)&:= - \frac{\pi}{4} + 2t(p_\infty-p(t))\xi_* -t \omega_\infty + t\xi_*^2 - \theta_{1,\infty}(t) + \theta_{2,\infty}(t)\xi_*.
\end{split}
\end{equation}

\end{enumerate}
Analogous statements hold for negative times $t \leq 0$.
\end{theorem}
We proceed with a few remarks on Theorem~\ref{thm:main_theorem}.
\begin{remark}\label{remark1}A significant challenge is to improve the convergence rates of the modulation parameters in \eqref{equ:theorem_statement_decay_modulation_parameters} and \eqref{equ:theorem_statement_decay_modulation_parameters2}. The main difficulty is due to the strong coupling  in the modulation equations, where the modulation parameters are forced by the slow local decay caused from threshold resonances in the quadratic nonlinearities. This interaction prevents the scaling $\dot{\omega}(t)$ and momentum $\dot{p}(t)$ parameters from being twice-integrable in time, even after performing normal form transformations on the leading quadratic nonlinearities contributions as we do in Section~\ref{sec:modulation_parameters}. The twice integrable decay rates on $\dot{\omega}(t)$ and $\dot{p}(t)$ are a necessary condition for improving the convergence rates for the phase shift $\gamma(t)$ and translation $\sigma(t)$ parameters in \eqref{equ:theorem_statement_decay_modulation_parameters2}.  Consequently, we only obtain the final convergence for the time derivatives of the phase shift and translation parameters, as stated in \eqref{equ:theorem_statement_decay_modulation_parameters2}.
\end{remark}
\begin{remark}\label{remark2}
Related to the previous remark, the bounds on the function $\theta_{2,\infty}(t)$ defined in \eqref{equ:theorem_asymptotics_notation} measure the precision of the center of the modulated solitary wave. The decay rates from \eqref{equ:theorem_statement_decay_modulation_parameters2} imply the slowly growing bound on the phase $|\theta_{2,\infty}(t)| \lesssim \eps \jt^\delta$. This bound reveals a small uncertainty in the localization of the soliton's center which may cause further losses in the weighted energy estimates for the remainder terms in the profile equation. Inspired by \cite[Section~6.3]{CL24}, we overcome this difficulty by developing smoothing estimates with moving centers in Section~\ref{sec:linear_decay}.
\end{remark}
\begin{remark}\label{remark3}
As a final remark, we note that the asymptotic stability of solitary waves for the focusing cubic Schr\"odinger equation has been successfully proven using inverse scattering techniques, with notable results from Cuccagna-Pelinovsky \cite{CuccPeli14}, Saalman \cite{Saalmann17}, and Borghese et al. \cite{BorghJenMcL18}. While these methods are powerful for completely integrable equations, our motivation is different. We aim to develop robust techniques for a broader class of nonlinear Schr\"odinger models with low-power nonlinearities where the analysis is further complicated by the slow local decay from threshold resonances, and a complex interplay between the modulation dynamics and the modified scattering of the radiation. 
\end{remark}

\subsection{Main challenges and outline of the proof}
The proof of Theorem~\ref{thm:main_theorem} builds upon and extends our work in \cite{LL24}, which is largely based on a combination of techniques related to the space-time resonances approach with \textit{distorted Fourier transform}, \textit{modulation}, and \textit{modified scattering}. To provide a roadmap for the reader, we will now summarize the main ideas of the proof and highlight its key challenges.

\textit{Linearized operator and distorted Fourier transform.}  The linearization of \eqref{equ:cubic_NLS} around a solitary wave $e^{it\omega}\phi_\omega(x)$, for a fixed scale $\omega >0$, leads to the following non-self-adjoint matrix Schr\"odinger operator 
\begin{equation}\label{eqn:calH}
\calH(\omega) := \begin{bmatrix}- \partial_x^2 + \omega & 0 \\ 0 & \partial_x^2 - \omega \end{bmatrix}  + \begin{bmatrix}
-2\phi_\omega^2(x) &- \phi_\omega^2(x)\\
\phi_\omega^2(x) & 2\phi_\omega^2(x)
\end{bmatrix}.
\end{equation}
As summarized in Proposition~\ref{prop:spectrum_of_calH}, the spectrum of this operator consists of the essential spectrum $(-\infty,-\omega]\cup[\omega,\infty)$ and a generalized kernel spanned by the eigenfunctions $\{Y_{j,\omega}(x)\}_{j=1}^4$ (see \eqref{eqn:nullspace of calH-omega}) which correspond to the invariances of \eqref{equ:cubic_NLS} under phase shifts, scaling, translations and Galilei boosts. Specifically for \eqref{equ:cubic_NLS}, the main features of the linearized operator \eqref{eqn:calH} are that $\lambda = 0$ is the only eigenvalue but there are threshold resonances at the edges $\lambda = \pm \omega$ of the essential spectrum. The functions associated to the threshold resonances are given explicitly by 
\begin{equation}\label{intro_equ:resonance_eq}
\Phi_{+,\omega}(x) := \frac{1}{\sqrt{2\pi}}\begin{bmatrix} \tanh^2(\sqrt{\omega}x) \\ -\sech^2(\sqrt{\omega}x) \end{bmatrix}, \quad \Phi_{-,\omega}(x) := \frac{1}{\sqrt{2\pi}}\begin{bmatrix} -\sech^2(\sqrt{\omega}x) \\	\tanh^2(\sqrt{\omega}x) \end{bmatrix}
\end{equation}
which belong to $L^\infty(\bbR) \times L^\infty(\bbR) \setminus L^2(\bbR) \times L^2(\bbR)$ and satisfy $\calH(\omega) \Phi_{\pm,\omega} = \pm\omega \Phi_{\pm,\omega}$. As in \cite{LL24}, we conduct our analysis in the frequency space, and the main tool is the \textit{distorted Fourier transform} which diagonalizes  $\calH(\omega)$ in the following sense: for each $\xi \in \bbR$, there exists a pair of bounded (vector-valued) basis functions $\Psi_{+,\omega}(x,\xi)$ and $\Psi_{-,\omega}(x,\xi)$ that satisfy the relations
\begin{equation}\label{intro_equ:dft_elements}
\calH(\omega)\Psi_{+,\omega}(x,\xi) = (\omega+\xi^2)\Psi_{+,\omega}(x,\xi), \quad \calH(\omega)\Psi_{-,\omega}(x,\xi) = -(\omega+\xi^2)\Psi_{-,\omega}(x,\xi).
\end{equation}
Denoting by $P_\mathrm{e}$ the Riesz projection to the essential spectrum of $\calH(\omega)$, the free  dispersive evolution can be recovered from the representation formula 
\begin{equation}\label{intro_equ:rep_formula}
e^{it\calH(\omega)}P_\mathrm{e}F(x) = \int_\bbR e^{it(\xi^2+\omega)}\wtilcalF_{+, \omega}[F](\xi) \Psi_{+,\omega}(x,\xi)\,\ud \xi - \int_\bbR e^{-it(\xi^2+\omega)}\wtilcalF_{-, \omega}[F](\xi) \Psi_{-,\omega}(x,\xi)\,\ud \xi 
\end{equation}
where the distorted Fourier transforms are the linear maps 
\begin{equation*}
\wtilcalF_{+, \omega}[F](\xi) := \langle F,\sigma_3 \Psi_{+,\omega}(\cdot,\xi)\rangle,\quad \wtilcalF_{-, \omega}[F](\xi) := \langle F,\sigma_3 \Psi_{-,\omega}(\cdot,\xi)\rangle,
\end{equation*}
for sufficiently smooth and decaying functions $F\colon\bbR \rightarrow \bbC^2$. Note that the slow local decay of the free Schr\"odinger waves \eqref{intro_equ:rep_formula}  is a consequence of the relation $\Psi_{\pm,\omega}(x,0) = \Phi_{\pm,\omega}(x)$, which connects the distorted Fourier theory to the threshold resonances associated to the functions defined in  \eqref{intro_equ:resonance_eq}. The slow local decay caused from these threshold resonances is the main obstacle for obtaining decay in many parts of the nonlinear analysis, as we will elaborate below.

\textit{Modulation and evolution equations.} The starting point of the nonlinear analysis is to use standard modulation and orbital stability techniques to decompose the perturbed solution into a modulated solitary wave and a radiation term that is orthogonal to the generalized eigenfunctions $\{\sigma_2Y_{j,\omega(t)}\}_{j=1}^4$, which spans the generalized kernel of $\calH(\omega(t))^*$. Namely, we choose the ansatz
\begin{equation*}
\psi(t,x) =  e^{ip(t)(x-\sigma(t))} e^{i \gamma(t)} \bigl( \phi_{\omega(t)}(x-\sigma(t)) + u(t,x-\sigma(t)) \bigr)
\end{equation*}
and allow the time-dependent modulation parameters $\big(\omega(t),p(t),\gamma(t),\sigma(t)\big)$ to be determined by the following  orthogonality conditions for the vectorial perturbation $U(t) :=\big(u(t),\baru(t)\big)^\top$ where 
\begin{equation}\label{intro_equ:orthogonal_conditions}
\big\langle U(t),\sigma_2 Y_{j,\omega(t)} \big\rangle = 0, \quad \text{for } j=1,\ldots,4,\quad t \geq 0.
\end{equation}
The resulting equation for the radiation term $U(t,y)$ in the moving frame coordinate $y := x-\sigma(t)$ becomes 
\begin{equation}\label{intro_equ:radiation_eq}
i\pt U - \calH(\omega)U = \dot{p} y\sigma_3U + i(\dot{\sigma}-2p)\py U + (\dot{\gamma}+p^2-\omega-p\dot{\sigma})\sigma_3 U + \calM_2 + \calQ_\omega(U)+\calC(U),
\end{equation}
and the radiation term is coupled to the modulation parameters through the following ODE system
\begin{equation}\label{intro_eq:modulation_eq}
\bbM \begin{bmatrix}
\dot{\gamma}+p^2-\omega-p\dot{\sigma}\\
\dot{\omega}\\
\dot{\sigma}-2p\\
\dot{p}
\end{bmatrix} =\begin{bmatrix}
\big\langle i\big(\calQ_\omega(U) + \calC(U)\big),\sigma_2 Y_{1,\omega}\big\rangle \\
\big\langle i\big(\calQ_\omega(U) + \calC(U)\big),\sigma_2 Y_{2,\omega}\big\rangle \\
\big\langle i\big(\calQ_\omega(U) + \calC(U)\big),\sigma_2 Y_{3,\omega}\big\rangle \\
\big\langle i\big(\calQ_\omega(U) + \calC(U)\big),\sigma_2 Y_{4,\omega}\big\rangle 
\end{bmatrix} 
\end{equation}
which are derived from differentiating the orthogonality conditions \eqref{intro_equ:orthogonal_conditions} in time. In the preceeding equations, $\sigma_2,\sigma_3$ are the Pauli matrices (see \eqref{eqn:pauli_matrices}), $\calM_2$ consist of spatially localized time-dependent terms involving the modulation parameters and the eigenfunctions $Y_{j,\omega(t)}$ (see \eqref{equ:setup_definition_calM2}), and the quadratic and cubic nonlinear terms are given by 
\begin{equation}\label{intro_equ:nonlinearity}
\calQ_\omega(U) := \begin{bmatrix}
-\phi_\omega(u^2+2u \baru) \\ \phi_\omega(\baru^2 + 2u \baru)
\end{bmatrix}, \quad \calC(U) := \begin{bmatrix}
-u \baru u\\ \baru u \baru
\end{bmatrix}.
\end{equation}
The time-dependent matrix $\bbM$ given on the left-hand side of \eqref{intro_eq:modulation_eq} is explicit in \eqref{equ:setup_matrix_modulation_equation}, and is invertible for small perturbations $U(t)$. By expanding the conservation laws \eqref{equ:conservation_laws} for the solution around the modulated solitary wave and using the coercivity properties of the Schr\"odinger operators $L_{\pm,\omega}$ (see \eqref{equ:operators_Lplusminus} and \cite{Weinstein85,Weinstein86}),  we obtain the \textit{orbital stability} estimate 
\begin{equation}\label{intro_equ:orbital_estimate}
\sup_{t \geq 0 } \big(\|U(t)\|_{H^1} + |\omega(t)-\omega_0| + |p(t)-p_0|\big) \lesssim \eps.
\end{equation}

\textit{Bootstrap setup.}\footnote{For the sake of exposition, we have oversimplified the lengthy technical details of the bootstrap strategy. The proof of the main theorem is an elaborate argument based on two bootstrap Propositions~\ref{prop:modulation_parameters} and \ref{prop:profile_bounds}.} In order to establish asymptotics and decay for the modulation parameters and for the radiation term $U(t)$, it is natural to pass to fixed time-independent parameters $\ulomega:=\omega(T)$, $\ulp:=p(T)$ in the evolution equation \eqref{intro_equ:radiation_eq} on a given bootstrap time interval $[0,T]$. We adopt the renormalization techniques in Collot-Germain \cite{CollotGermain23} by working on the transformed variable 
\begin{equation}\label{intro_equ:renormalized_rad_def}
\Vp(t,y) := e^{i(p(t)-\ulp)y\sigma_3}U(t,y) = \begin{bmatrix}
e^{i(p(t)-\ulp)y }u(t,y) \\e^{-i(p(t)-\ulp)y }\baru(t,y) 
\end{bmatrix}.
\end{equation}
Keeping only the main order terms, the new variable $\Vp(t)$ now satisfies the nonlinear evolution equation
\begin{equation}\label{intro_equ:renormalized_rad_eq}
i\pt \Vp - \calH(\ulomega)\Vp = \dot{\theta}_1(t)\sigma_3 \Vp + i \dot{\theta}_2(t) \py \Vp + \calQ_{\ulomega}(\Vp) + \calC(\Vp) +  \calR(t), 
\end{equation}
where $\dot{\theta}_1(t),\dot{\theta}_2(t)$ are real-valued scalar coefficients depending on the modulation parameters, and $\calR(t)$ is a remainder term. The advantage of the transformation \eqref{intro_equ:renormalized_rad_def} is that we filter out the problematic term $\dot{p}y\sigma_3 U$ in \eqref{intro_equ:radiation_eq} while we ensure that the norms for $\Vp(t)$ and $U(t)$ remain comparable thanks to the stability estimate \eqref{intro_equ:orbital_estimate}. 

Now write the spectral decomposition $\Vp(t) = \ulPd \Vp(t) + \ulPe \Vp(t)$ relative to the time-independent operator $\calH(\ulomega)$. Our objective is to establish decay for the dispersive component $\ulPe \Vp(t)$, as the favorable decay of the discrete component $\ulPd \Vp(t)$ is essentially guaranteed by the orthogonality conditions \eqref{intro_equ:orthogonal_conditions}. Like in \cite{LL24,CollotGermain23}, we proceed in the spirit of space-time resonances approach and introduce the profile of the radiation term by filtering out the linear evolution 
\begin{equation}
F_{\ulomega,\ulp}(t,y) := e^{it\calH(\ulomega)}\ulPe \Vp(t,y),
\end{equation}
and we analyze the evolution equations for the distorted Fourier transform of the profile $\tilf_{\pm, \ulomega,\ulp}(t,\xi) := \wtilcalF_{\pm, \ulomega}[F_{\ulomega,\ulp}(t)](\xi)$ which reads
\begin{equation}\label{intro_equ:profile_eq}
i\pt\big(e^{i\theta_1(t)}e^{-i\theta_2(t)\xi}\tilf_{+, \ulomega,\ulp}(t,\xi)\big)= e^{i\theta_1(t)}e^{-i\theta_2(t)\xi} e^{it(\xi^2+\ulomega)}\Big(\wtilcalF_{+, \ulomega}[\calQ_{\ulomega}(\ulPe\Vp)+\calC(\ulPe\Vp)] + \wtilcalR(t)\Big),
\end{equation}
where $\wtilcalR(t)$ is a remainder term. While the evolution equation for $\tilf_{-, \ulomega,\ulp}(t,\xi)$ can be written down, the simple relation in Lemma~\ref{lem:distFT_components_relation} makes it sufficient to analyze only the profile equation \eqref{intro_equ:profile_eq}. 

In view of the representation formula \eqref{intro_equ:rep_formula} and the linear dispersive estimate from Lemma~\ref{lem:linear_dispersive_decay}, in order to establish the linear decay rate $t^{-\frac12}$ on the radiation term, it suffices to bootstrap  uniform-in-time pointwise bounds
\begin{equation}\label{intro_equ:pointwise_estimate}
\big\| \big(\tilf_{+, \ulomega,\ulp}(t,\xi),\tilf_{-, \ulomega,\ulp}(t,\xi)\big) \big\|_{L_\xi^\infty \times L_\xi^\infty} \lesssim \eps,\quad t \in [0,T], 
\end{equation}
and slowly growing weighted energy estimates 
\begin{equation}\label{intro_equ:weighted_energy_estimate}
\big\| \big( \pxi\tilf_{+, \ulomega,\ulp}(t,\xi),\pxi\tilf_{-, \ulomega,\ulp}(t,\xi)\big) \big\|_{L_\xi^2 \times L_\xi^2} \lesssim \eps\jt^\delta, \quad t \in [0,T],
\end{equation}
where $0<\delta \ll 1$ is a small absolute constant. This functional setting is standard for performing the modified scattering analysis, as it accounts for the critical nature of the nonlocalized cubic nonlinearity in \eqref{intro_equ:nonlinearity}. However, because the evolution equation \eqref{intro_equ:profile_eq} is also strongly coupled to the modulation parameters, it is also necessary to bootstrap the following decay estimates 
\begin{equation}\label{intro_equ:bootstrap_mod}
|\omega(t)-\ulomega|\lesssim \eps \jt^{-1+\delta}, \quad |p(t)-\ulp|\lesssim \eps \jt^{-1+\delta},\quad t \in [0,T].
\end{equation}
As we will explain below, obtaining \eqref{intro_equ:bootstrap_mod} is significantly difficult due to the weak decay arising from the quadratic and cubic nonlinearities forced by the radiation term on the right hand side of the modulation equations \eqref{intro_eq:modulation_eq}.

\textit{Main challenges arising from slow local decay.} The decay of the spatially localized quadratic nonlinearity present in both the profile equation \eqref{intro_equ:profile_eq} and the modulation equations \eqref{intro_eq:modulation_eq} depends entirely on the \textit{local decay} properties of the radiation term.  A formal stationary phase analysis for the free Schr\"odinger wave representation formula  \eqref{intro_equ:rep_formula} reveals that the leading decay contribution to the free Schr\"odinger waves is determined by the functions \eqref{intro_equ:resonance_eq} associated to the threshold resonances. In fact, linear dispersive and local decay estimates for (general) matrix Schr\"odinger operators with threshold resonances were proven by the author in \cite[Theorem~1.1]{Li23}. Since we are working under limited smoothness in the frequency space, we follow the approach in \cite{LL24} by decomposing the radiation term \ref{intro_equ:renormalized_rad_def} as
\begin{equation}\label{intro_equ:decomp_rad}
\Vp(t,y) = h_+(t)\Phi_{+,\ulomega}(y) + h_-(t)\Phi_{-,\ulomega}(y) + R(t,y)
\end{equation}
where the coefficients behave as $h_\pm(t) \sim \eps t^{-\frac12} e^{\mp it\ulomega}$ and the remainder $R(t,y)$ satisfies the improved local decay bounds 
\begin{equation}\label{intro_equ:decomp_rad_rem_decay}
\| \jy^{-2}R(t,y)\|_{L_y^\infty} +\| \jy^{-3} \py R(t,y)\|_{L_y^2}\lesssim \eps \jt^{-1+\delta}, \quad t \in [0,T].
\end{equation}
Inserting the decomposition \eqref{intro_equ:decomp_rad} into the spatially localized quadratic nonlinearity on the distorted Fourier space \eqref{intro_equ:profile_eq} and tracking the space-time resonant terms in the expansion, we find the leading source term
\begin{equation}\label{intro_equ:resonant_quadratic}
e^{it(\xi^2+\ulomega)}\wtilcalF_{+, \ulomega}[\calQ_{\ulomega}(\ulPe\Vp(t))](\xi) = e^{it(\xi^2-\ulomega)}\big(e^{it\ulomega}h_+(t)\big)^2 \wtilcalF_{+, \ulomega}[\calQ_{1,\ulomega}](\xi) + \cdots,
\end{equation}
where the term $\calQ_{1,\ulomega}(y)$ can be computed from inserting $\Phi_{+,\ulomega}(y)$ into the quadratic nonlinearity. The phase of this source term vanishes at the resonant frequencies $\xi = \pm \sqrt{\ulomega}$, which can cause difficulty in extracting further decay. However, a key cancellation occurs. By an explicit computation, which was first observed by the author in \cite[Lemma~1.6]{Li23}, the term $\wtilcalF_{+, \ulomega}[\calQ_{1,\ulomega}](\xi)$ is divisible by the phase $(\xi^2-\ulomega)$, meaning it also vanishes at the resonant frequencies $\xi = \pm \sqrt{\ulomega}$. This crucial null structure allows us to remove the worst effects of the quadratic nonlinearities in \eqref{intro_equ:profile_eq} via a variable coefficient normal form transformation, which effectively turns the profile equation \eqref{intro_equ:profile_eq} into cubic order  where we can apply modified scattering techniques to infer dispserive decay for the radiation term \eqref{intro_equ:renormalized_rad_def}. 

As for the modulation parameters, we use the modulation equations \eqref{intro_eq:modulation_eq} and the fundamental theorem of calculus to write down the the leading quadratic contributions for $\omega(t)$ and $p(t)$:
\begin{align}
\omega(t)-\ulomega&= \omega(t)-\omega(T) = -\frac{i\sqrt{\ulomega}}{2} \int_t^T \big\langle i\calQ_{\ulomega}\big(\ulPe\Vp(s)\big),\sigma_2 Y_{1,\ulomega}\big\rangle \,\ud s + \cdots,\label{intro_equ:omega_quadratic}\\
p(t)-\ulp &=p(t)-p(T) = \frac{1}{4\sqrt{\ulomega}} \int_t^T \big\langle i\calQ_{\ulomega}\big(\ulPe\Vp(s)\big),\sigma_2 Y_{3,\ulomega}\big\rangle \,\ud s + \cdots. \label{intro_equ:momentum_quadratic}
\end{align}
A direct but naive approach to infer decay would be to insert the decomposition of the radiation term \eqref{intro_equ:decomp_rad} into the quadratic nonlinearities in \eqref{intro_equ:omega_quadratic} and \eqref{intro_equ:momentum_quadratic}. However, such an approach is clearly insufficient for establishing the almost twice integrable rates required by the bootstrap \eqref{intro_equ:bootstrap_mod}. Hence, we adopt the strategy developed in our previous work \cite{LL24} by inserting the representation formula \eqref{intro_equ:rep_formula} into \eqref{intro_equ:omega_quadratic} and \eqref{intro_equ:momentum_quadratic}, which yields several multilinear expressions. The most resonant quadratic contribution for the momentum parameter, for example, takes the form
\begin{equation}\label{intro_equ:resonant_quadratic_momentum}
\int_t^T \iint_{\bbR^2} e^{-is(\xi_1^2-\xi_2^2)}\tilf_{+, \ulomega,\ulp}(s,\xi_1)\tilf_{-, \ulomega,\ulp}(s,\xi_2)\lambda_{+-\ulomega}(\xi_1,\xi_2)\,\ud \xi_1\,\ud \xi_2 \,\ud s,
\end{equation}
where the quadratic spectral distribution $\lambda_{+-\ulomega}(\xi_1,\xi_2)$ results from the interaction of the distorted Fourier basis elements \eqref{intro_equ:dft_elements} and the generalized kernel eigenfunction $Y_{3,\ulomega}(y)$. The phase in \eqref{intro_equ:resonant_quadratic_momentum} clearly exhibits a large set of time resonances on $\{\xi_1=\xi_2\}\cup\{\xi_1=-\xi_2\}$. Nevertheless, from an explicit computation shown in Lemma~\ref{lem:null_structure_modulation}, we discover that the quadratic spectral distribution $\lambda_{+-\ulomega}(\xi_1,\xi_2)$ is divisible by the phase $(\xi_1^2-\xi_2^2)$. This new null structure (which was not recorded in \cite{LL24}) allow us to apply a normal form on the contribution \eqref{intro_equ:resonant_quadratic_momentum} and convert it into cubic order. We then observe that the resulting cubic interactions are non-resonant, whence we can apply another normal form, and turn them into quartic order terms and establish an almost twice integrable time decay $\lesssim \eps^4\js^{-2+\delta}$ for the integrand of \eqref{intro_equ:resonant_quadratic_momentum}, which suffices to conclude the desired decay rate \eqref{intro_equ:bootstrap_mod}. 

\textit{Smoothing estimates.} Lastly, we briefly explain the difficulty mentioned in Remark~\ref{remark2} concerning establishing weighted energy estimates for the profile. A key challenge arises from the $\xi$ derivative acting on the phase $e^{it(\xi^2+\ulomega)}$, which produces a strong contribution term in the Duhamel expansion of \eqref{intro_equ:profile_eq} as shown below 
\begin{align}
\pxi\big(e^{i\theta_1(t)}e^{-i\theta_2(t)\xi}\tilf_{+, \ulomega,\ulp}(t,\xi)\big) = \cdots + &\underbrace{\int_0^t e^{i\theta_1(s)}e^{-i\theta_2(s)\xi} e^{is(\xi^2+\ulomega)}(2is \xi) \wtilcalR(s,\xi)\,\ud s} + \cdots \label{intro_equ:remainder_term1}\\
&\approx \int_0^t \int_\bbR e^{i\theta_1(s)}e^{-i(y+\theta_2(s))\xi} e^{is(\xi^2+\ulomega)}(2is \xi) \calR(s,y)\,\ud y\,\ud s. \label{intro_equ:remainder_term2}
\end{align}
The term \eqref{intro_equ:remainder_term2} is obtained from $\wtilR(s,\xi)$ by schematically inverting the distorted Fourier transform, where $\calR(s,y)$ is spatially localized and decays at a near cubic rate $\lesssim \eps^2 \js^{-\frac{3}{2}+\delta}$. Since we propagate the slowly growing rate on the phase $|\theta_2(s)|\lesssim \eps \js^\delta$ throughout this paper, the spatial mismatch between the phase $e^{-i(y+\theta_2(s))\xi}$ and the localized remainder term $\calR(s,y)$ in \eqref{intro_equ:remainder_term2} renders standard local smoothing estimates (e.g. \cite[Lemma~4.4]{LL24}) ineffective for obtaining the desired bootstrap estimate \eqref{intro_equ:weighted_energy_estimate}. To overcome this challenge, we have developed smoothing estimates with a moving center in Proposition~\ref{prop:local-smoothing} to carefully handle this issue. The reader is referred to \eqref{equ:proof_weighted_estimate_3_expanded} for the application of such linear estimates.

\subsection{Related literature}
We conclude our introduction with a brief discussion of related works on the asymptotic stability of solitary waves of nonlinear Schr\"odinger equations (NLS). For a general overview on the topic of asymptotic stability of solitons, we refer the reader to the recent reviews \cite{Martel25_survey,Germain24Review} for the NLS equation and the earlier surveys in \cite{CuccMaeda_SurveyII, KMM17_Survey, Schl07, Tao_Survey}.  

The pioneering work by \cite{BusPerel92} established the asymptotic stability of solitary waves for the NLS equation under the assumption that the linearized operator does not have any non-zero eigenvalues or threshold resonances, and that the nonlinearity vanishes at a sufficiently high order. In \cite{BusPerel92}, the spectral and distorted Fourier theory for the matrix linearized Schr\"odinger operator, and its linear dispersive estimates were introduced but briefly sketched. 
The seminal work in \cite{KS06} culminated to a proof of the co-dimensional asymptotic stability for solitary wave solutions to the (mass) super-critical focusing NLS equation, building a rigorous justification and development of the linear theory introduced in \cite{BusPerel92}. More recently, \cite{CollotGermain23} established asymptotic stability of solitary waves for NLS equations allowing for the scattering-critical, cubic-order nonlinearity under the spectral assumption that there are no non-zero eigenvalues or threshold resonances in the spectrum of the linearized operator. The analysis in \cite{CollotGermain23} is a combination of modulation techniques and the space-time resonances approach based on the distorted Fourier transform for the linearized operator. Our analysis in this article and in \cite{LL24} further develops the framework of \cite{CollotGermain23} by allowing the presence of threshold resonances in the linearized operator. We also mention the articles \cite{BusPerel92, BusSul03, KS06,KS09_quintic_NLS,LL24, CollotGermain23, MasakiMurphySegata, Mizumachi08, Chen21} where asymptotic stability results for the NLS equation have been obtained.

For small perturbations in the energy space, the \textit{virial} method was recently applied to prove local asymptotic stability of solitary waves for NLS equation  with a cubic-quintic, focusing-defocusing nonlinearity in \cite{Martel22}. This method has seen further successes in \cite{Martel22, Martel23, Rialland23, Rialland24, CuccMaeda2404, CuccMaeda2405}, even allowing internal modes in the dynamics. An interesting observation from these articles is that threshold resonances in the cubic NLS model \eqref{equ:cubic_NLS} either bifurcate into internal modes or vanish entirely under small modifications to the cubic nonlinearity (see earlier spectral studies in \cite{PKA98,15ColesGustafson,CGNT08}).  We remark that the virial method was effectively applied in the earlier articles \cite{KMM17, KMM19, KMMV20} and in \cite{KM22, LL1, CuccMaeda_kink_23, CuccMaedaMurgScrob23, PalaciosPusateri24} for local asymptotic stability results in nonlinear Klein-Gordon (NLKG) settings. See also \cite{CuccMaeda19, CuccMaeda22} for the virial method applied in nonlinear Schr\"odinger equations with a trapping potential.

\textit{Other related works.} References on dispersive estimates and other spectral aspects for Schr\"odinger operators on the line can be found in \cite{00Weder,  04GoldbergSchlag, 07Goldberg, CGNT08, KS06, Li23, BusPerel92, 25ChenMou, 25ChenMou3, ColGerEli25_embedded}. We refer to \cite{KoSchlag25} for an introduction to the distorted Fourier transform, and we mention recent interesting developments in the distorted Fourier theory for linearized operators around Ginzburg-Landau vortices in  \cite{LanSha25,ColGerEli25,PalaciosPusateri24GL,LuhSchSha25}. We gained insights from modified scattering analysis of small solutions to cubic Schr\"odinger equations on the line from \cite{KatPus11,GermPusRou18,ChenPus19,ChenPus22} in the framework of space-time resonances method, and from other techniques in \cite{HN98, LS06, IT15, Del16, Naum16, Naum18, MasMurphSeg19, NaumWed22, 25KriSchWid}. 
Modified scattering results in NLKG equations in \cite{LS05_1, HN08, HN10, HN12, Del16_KG, Stingo18, Sterb16, LS15, LLS1, LLS2, LLSS, GP20} are also helpful to our analysis. 
We benefited largely from the investigation of long-time behavior of solutions to NLKG equations under the presence of a threshold resonance in the articles \cite{LLS1, LLS2, LLSS, LS1, LS2, CL24}. For asymptotic stability of (topological) solitons in scalar field theories, we mention \cite{KK11_1, KK11_2, GP20, LS1, GermPusZhang22, KairzhanPusateri22, LS2, CL24}. As pointed out in the beginning of the introduction, \cite{CuccPeli14, Saalmann17, BorghJenMcL18} proved asymptotic stability results for \eqref{equ:cubic_NLS} using completely integrable techniques, while classical orbital stability results were due to   \cite{Weinstein85,Weinstein86,85ShatahStrauss,87GrillakisShatahStrauss,82CazenaveLions}. Lastly, we mention that \cite{25ChenMou2} recently proved the co-dimensional asymptotic stability of multi-solitons for the (mass) super-critical focusing NLS.

\medspace

\noindent \textbf{Acknowledgements}. The author would like to thank his Ph.D. advisor Jonas L\"uhrmann for valuable comments on the manuscript and for our collaboration in \cite{LL24}.

\section{Notation} \label{sec:preliminaries}

\noindent {\it Notation and conventions.} 
All implicit constants in estimates of this paper may depend on the initial scaling parameter $\omega_0 \in (0,\infty)$ fixed in the statement of Theorem~\ref{thm:main_theorem}. We denote by $C > 0$ an absolute constant whose value may change from line to line. 
For non-negative $X, Y$ we write $X \lesssim Y$ if $X \leq C Y$ and the notation $X \ll Y$ indicates that the implicit constant should be regarded as small. Throughout, we use the Japanese bracket notation
\begin{equation*}
\jap{t} := (1+t^2)^{\frac12}, \quad \jap{x} := (1+x^2)^{\frac12}, \quad \jap{\xi} := (1+\xi^2)^{\frac12}.
\end{equation*}
We use the standard notation for the Lebesgue spaces $L^p$ and for the Sobolev spaces $H^k$, $W^{k,p}$. The space $L^{2,1}$ is defined by the weighted norm $\|f\|_{L^{2,1}} := \|\jx f\|_{L^2}$.

\medskip 

\noindent {\it Pauli matrices.}
We recall the definitions of the Pauli matrices
\begin{align} \label{eqn:pauli_matrices}
\sigma_1 = \begin{bmatrix} 0 & 1 \\ 1 & 0 \end{bmatrix}, \quad \sigma_2 = \begin{bmatrix} 0 & -i \\ i & 0 \end{bmatrix}, \quad \sigma_3 = \begin{bmatrix} 1 & 0 \\ 0 & -1 \end{bmatrix}.
\end{align}

\medskip 

\noindent {\it Inner products.}
We use the following $L^2$-inner product of two scalar-valued functions $f, g \colon \bbR \to \bbC$ by
\begin{equation} \label{equ:inner_product_scalar}
\langle f, g \rangle := \int_\bbR f(x) \barg(x) \, \ud x,
\end{equation}
and the $L^2$-inner product of two vector-valued functions $U, V \colon \bbR \to \bbC^2$ by
\begin{equation} \label{equ:inner_product_vector}
\langle U, V \rangle := \int_{\bbR} \bigl( u_1(x) \bar{v}_1(x) + u_2(x) \barv_2(x) \bigr) \, \ud x, \quad U := \begin{bmatrix} u_1 \\ u_2 \end{bmatrix}, \quad V := \begin{bmatrix} v_1 \\ v_2 \end{bmatrix}.
\end{equation}
\medskip 
\noindent {\it $\calJ$-invariance.} 
Following the terminology in \cite{KS06}, we say that a vector $U \in \bbC^2$ is $\calJ$-invariant if
\begin{equation*}
\overline{\sigma_1 U} = U.
\end{equation*}
The inner product \eqref{equ:inner_product_vector} of two $\calJ$-invariant vectors is real-valued.
The $\calJ$-invariance of the evolution equation \eqref{equ:setup_perturbation_equ} for the vectorial perturbation will allow us to go back from the system to the scalar cubic Schr\"odinger equation \eqref{equ:cubic_NLS}.

\medskip 
\noindent {\it Fourier transform.} 
Our conventions for the (flat) Fourier transform of a Schwartz function are
\begin{equation*}
\begin{aligned}
\widehat{\calF}[f](\xi) = \hatf(\xi) &:= \frac{1}{\sqrt{2\pi}} \int_\bbR e^{-ix\xi} f(x) \, \ud x, \\
\widehat{\calF}^{-1}[f](\xi) = \check{f}(\xi) &:= \frac{1}{\sqrt{2\pi}} \int_\bbR e^{ix\xi} f(x) \, \ud x.
\end{aligned}
\end{equation*}
The convolution laws for $g,h \in \calS(\bbR)$ then read 
\begin{equation*}
\widehat{\calF}\bigl[g \ast h\bigr] = \sqrt{2\pi} \hatg \hath, \quad \widehat{\calF}\bigl[g h\bigr] = \frac{1}{\sqrt{2\pi}} \hatg \ast \hath.
\end{equation*}
We recall that in the sense of tempered distributions
\begin{align}
\widehat{\calF}[1](\xi) &= \sqrt{2\pi} \delta_0(\xi), \label{equ:preliminaries_FT_one} \\
\widehat{\calF}[\tanh(\cdot)](\xi) &= -i \sqrt{\frac{\pi}{2}} \pvdots \cosech \Bigl(\frac{\pi}{2} \xi\Bigr). \label{equ:preliminaries_FT_tanh}
\end{align}

\section{Spectral Theory and Distorted Fourier Transform} \label{sec:spectral_and_dist_FT}
In this section, we recollect the spectral and distorted Fourier theory for the linearized operator \eqref{eqn:calH} developed in \cite{LL24}. 

\subsection{Spectrum of linearized operator}

The spectrum of the matrix Schr\"odinger operator $\calH(\omega)$ is  provided in the following proposition.

\begin{proposition} \label{prop:spectrum_of_calH}
Fix $\omega \in (0,\infty)$.
Then the following holds for the operator $\calH(\omega)$ defined in \eqref{eqn:calH} on $L^2(\bbR) \times L^2(\bbR)$ with domain $H^2(\bbR) \times H^2(\bbR)$:
\begin{enumerate}[leftmargin=1.8em]

\item Discrete spectrum: $0$ is the only eigenvalue with algebraic multiplicity equal to 4 and geometric multiplicity equal to 2.

\item Generalized nullspace: $\calN_\mathrm{g}(\calH(\omega)) = \ker(\calH(\omega)^2)$ is spanned by the generalized eigenfunctions
\begin{equation} \label{eqn:nullspace of calH-omega}
Y_{1,\omega} := \begin{bmatrix}
i\phi_\omega \\ -i\phi_\omega
\end{bmatrix}, \quad Y_{2,\omega} := \begin{bmatrix}
\partial_\omega \phi_\omega \\ \partial_\omega \phi_\omega
\end{bmatrix}, \quad Y_{3,\omega} := \begin{bmatrix}
\partial_x \phi_\omega \\ \partial_x \phi_\omega
\end{bmatrix}, \quad Y_{4,\omega} := \begin{bmatrix}
ix\phi_\omega \\ -ix\phi_\omega
\end{bmatrix},
\end{equation}
which satisfy
\begin{equation} \label{eqn:eig-eqn for nullspace of calH-omega}
\calH(\omega) Y_{1,\omega} = 0,\quad \calH(\omega) Y_{2,\omega} = i Y_{1,\omega}, \quad \calH(\omega) Y_{2,\omega} = 0, \quad \calH(\omega) Y_{4,\omega} = -2i Y_{3,\omega}.
\end{equation}
Moreover, we have $\calN_\mathrm{g}(\calH(\omega)^*) = \ker((\calH(\omega)^*)^2)= \mathrm{span}(\sigma_{2}Y_{k,\omega}:1\leq k\leq4)$, and the vectors $Y_{k,\omega}$ and $\sigma_2 Y_{k,\omega}$, $1 \leq k \leq 4$, are $\calJ$-invariant.
\item The essential spectrum of $\calH(\omega)$ is $(-\infty,-\omega] \cup [\omega,\infty)$ and there are no embedded eigenvalues.
\item Threshold resonances: The vector-valued functions 
\begin{equation}
\Phi_{+,\omega}(x) := \frac{1}{\sqrt{2\pi}}\begin{bmatrix} \tanh^2(\sqrt{\omega}x) \\ -\sech^2(\sqrt{\omega}x) \end{bmatrix}, \quad \Phi_{-,\omega}(x) := \sigma_1 \Phi_{+,\omega}(x) = \frac{1}{\sqrt{2\pi}}\begin{bmatrix} -\sech^2(\sqrt{\omega}x) \\	\tanh^2(\sqrt{\omega}x) \end{bmatrix}
\end{equation}
belong to $L^\infty(\bbR) \times L^\infty(\bbR) \setminus L^2(\bbR) \times L^2(\bbR)$ and satisfy
\begin{equation}
\calH(\omega) \Phi_{+,\omega} = \omega \Phi_{+,\omega}, \quad \calH(\omega) \Phi_{-,\omega} = -\omega \Phi_{-,\omega}.
\end{equation}
\end{enumerate}
\end{proposition}
\begin{proof}
See \cite[Proposition~3.1]{LL24}.
\end{proof}

The Riesz projection onto the discrete spectrum of $\calH(\omega)$ is given by 
\begin{equation} \label{equ:definition_Pd}
P_\mathrm{d} := \frac{1}{2\pi i} \oint_\gamma \bigl(\calH(\omega) - zI\bigr)^{-1} \, \ud z,
\end{equation}
where $\gamma$ is a simple closed curve around the origin that lies within the resolvent set, while the projection onto the essential spectrum is defined by
\begin{equation*}
P_{\mathrm{e}} := I - P_{\mathrm{d}}.
\end{equation*}
We recall from \cite[Remark~9.5]{KS06} that both projections $P_\mathrm{d}$ and $P_\mathrm{e}$ preserve the space of $\calJ$-invariant functions on $L^2(\bbR) \times L^2(\bbR)$. The next lemma provides an explicit decomposition of $L^2(\bbR) \times L^2(\bbR)$ into the spectral subspaces associated  with the essential spectrum and with the discrete spectrum of $\calH(\omega)$.

\begin{lemma} \label{lemma: L2 decomposition}
There is the decomposition
\begin{equation} \label{eqn: L2 decomposition}
L^2(\bbR) \times L^2(\bbR) = \big(\calN_\mathrm{g}(\calH(\omega)^*) \big)^\perp + \calN_\mathrm{g}(\calH(\omega)),
\end{equation}
where the individual summands are linearly independent, but not necessarily orthogonal. The decomposition \eqref{eqn: L2 decomposition} is invariant under the flow of $\calH(\omega)$, and the projection $P_\mathrm{e}$ is the projection onto the orthogonal complement in \eqref{eqn: L2 decomposition}.
Explicitly, for $U \in L^2(\bbR) \times L^2(\bbR)$ it holds that
\begin{equation} \label{equ:spectral_decomposition_L2L2}
U = P_{\mathrm{e}} U + 	P_{\mathrm{d}}U
\end{equation}
where 
\begin{equation}\label{equ:project_Pd}
P_{\mathrm{d}}U \equiv  \sum_{j=1}^4 d_{j,\omega} Y_{j,\omega}
\end{equation}
with 
\begin{equation*}
d_{1,\omega} = \frac{\langle U, \sigma_2 Y_{2, \omega}\rangle}{\langle Y_{1,\omega}, \sigma_2 Y_{2,\omega}\rangle}, \quad d_{2,\omega} = \frac{\langle U, \sigma_2 Y_{1,\omega} \rangle}{\langle Y_{2,\omega}, \sigma_2 Y_{1,\omega}\rangle}, \quad d_{3,\omega} = \frac{\langle U, \sigma_2 Y_{4,\omega}\rangle}{\langle Y_{3, \omega}, \sigma_2 Y_{4,\omega}\rangle}, \quad d_{4,\omega} = \frac{\langle U, \sigma_2 Y_{3,\omega}\rangle}{\langle Y_{4, \omega}, \sigma_2 Y_{3,\omega}\rangle}.
\end{equation*}
We also have the following relations
\begin{align}\label{equ:innerproductrelations1}
\langle Y_{1,\omega},\sigma_2 Y_{j,\omega} \rangle &= \begin{cases}
-c_\omega \quad \text{for }j = 2,\\
0 \quad \text{otherwise},
\end{cases}\qquad \langle Y_{2,\omega},\sigma_2 Y_{j,\omega} \rangle = \begin{cases}
-c_\omega \quad \text{for }j = 1,\\
0 \quad \text{otherwise},
\end{cases}\\
\langle Y_{3,\omega},\sigma_2 Y_{j,\omega} \rangle &= \begin{cases}
\|\phi_\omega\|_{L^2}^2\quad \text{for }j = 4,\\
0 \quad \text{otherwise},
\end{cases}\qquad \langle Y_{4,\omega},\sigma_2 Y_{j,\omega} \rangle = \begin{cases}
\|\phi_{\omega}\|_{L^2}^2\quad \text{for }j = 3,\\
0 \quad \text{otherwise},\label{equ:innerproductrelations2}
\end{cases}
\end{align}
as well as 
\begin{equation}\label{equ:orthogonality_PeU}
\langle P_{\mathrm{e}}U,\sigma_2 Y_{j,\omega}\rangle = 0, \quad j \in \{1,\ldots,4\}.
\end{equation}
Moreover, the projection $P_{\mathrm{e}}$ is bounded as an operator from $\jx^\alpha W^{k,p} \rightarrow \jx^\alpha W^{k,p}$ for any $p \in [1,\infty]$, $\alpha \in \bbR$, and $k \geq 0$.
\end{lemma}
\begin{proof}
See \cite[Lemma~3.8]{LL24} for the proof. Note that the last statement about the boundedness of $P_{\mathrm{e}}$ is stronger than \cite[Lemma~3.8]{LL24} but the asserted operator bound is immediate from the definition $P_{\mathrm{e}} = I - P_\mathrm{d}$, the formula for $P_\mathrm{d}$ in \eqref{equ:project_Pd}, and the fact that $Y_{1,\omega},\ldots,Y_{4,\omega}$ are Schwarz functions.
\end{proof}

\subsection{Distorted Fourier Transform}
The following proposition expresses the evolution in terms of its distorted Fourier transform. 
\begin{proposition} \label{prop: representation formula}
Fix $\omega \in (0,\infty)$. For $x, \xi \in \bbR$ set 
\begin{equation} \label{eqn: Psi_pm_omega}
\begin{aligned}
\Psi_{+,\omega}(x,\xi) &:=
\begin{bmatrix}
\Psi_{1,\omega}(x,\xi)\\
\Psi_{2,\omega}(x,\xi)
\end{bmatrix}, \\
\Psi_{-,\omega}(x,\xi) &:= \sigma_1 \Psi_{+,\omega}(x,\xi) = \begin{bmatrix}
\Psi_{2,\omega}(x,\xi) \\
\Psi_{1,\omega}(x,\xi)
\end{bmatrix},
\end{aligned}
\end{equation}
where 
\begin{align}
\Psi_{1,\omega}(x,\xi) &:= \frac{1}{\sqrt{2\pi}}  \frac{\bigl(\xi + i \sqrt{\omega}\tanh(\sqrt{\omega}x)\bigr)^2}{(\vert \xi \vert - i \sqrt{\omega})^2} e^{i x \xi} \equiv \frac{1}{\sqrt{2\pi}} m_{1,\omega}(x,\xi) e^{ix\xi}, \label{eqn:m-1,omega} \\
\Psi_{2,\omega}(x,\xi) &:= \frac{1}{\sqrt{2\pi}}  \frac{\big(\sqrt{\omega}\sech(\sqrt{\omega}x))^2}{(\vert \xi \vert - i \sqrt{\omega}\big)^2} e^{i x \xi} \equiv \frac{1}{\sqrt{2\pi}} m_{2,\omega}(x,\xi) e^{i x \xi} \label{eqn:m-2,omega}.
\end{align}
Then we have for every $F, G \in \calS(\bbR) \times \calS(\bbR)$ that 
\begin{equation}
\begin{split} \label{eqn: e-itH Pe pairing formula}
\langle e^{it\calH(\omega)}\Pe F,G\rangle &= \int_\bbR e^{it(\xi^2+\omega)}\langle F,\sigma_3 \Psi_{+,\omega}(\cdot,\xi)\rangle \overline{\langle G, \Psi_{+,\omega}(\cdot,\xi)\rangle} \,\ud \xi\\
&\quad - \int_\bbR e^{-it(\xi^2+\omega)}\langle F,\sigma_3 \Psi_{-,\omega}(\cdot,\xi)\rangle \overline{\langle G, \Psi_{-,\omega}(\cdot,\xi)\rangle} \,\ud \xi,
\end{split}
\end{equation}
and
\begin{equation} \label{eqn: Pe pairing formula}
\langle \Pe F,G \rangle = \int_\bbR \langle F,\sigma_3 \Psi_{+,\omega}(\cdot,\xi)\rangle \overline{\langle G, \Psi_{+,\omega}(\cdot,\xi)\rangle} \,\ud \xi - \int_\bbR \langle F,\sigma_3 \Psi_{-,\omega}(\cdot,\xi)\rangle \overline{\langle G, \Psi_{-,\omega}(\cdot,\xi)\rangle} \,\ud \xi.
\end{equation}
These integrals are absolutely convergent because the integrands are rapidly decaying.
\end{proposition}
\begin{proof}
See \cite[Proposition~3.10]{LL24}.
\end{proof}
The functions $\Psi_{\pm, \omega}(x,\xi)$ form the distorted Fourier basis elements for the linearized operator $\calH(\omega)$. In view of the representation formulas \eqref{eqn: e-itH Pe pairing formula} and \eqref{eqn: Pe pairing formula}, we define the distorted Fourier transform relative to $\calH(\omega)$ of a vector-valued function $F \in \calS(\bbR) \times \calS(\bbR)$ by
\begin{equation} \label{eqn:def-dFT}
\wtilcalF_\omega[F](\xi) := \begin{bmatrix}\wtilcalF_{+, \omega}[F](\xi)\\ \wtilcalF_{-, \omega}[F](\xi)\end{bmatrix} := \begin{bmatrix}
\langle F,\sigma_3 \Psi_{+,\omega}(\cdot,\xi)\rangle \\ \langle F,\sigma_3 \Psi_{-,\omega}(\cdot,\xi)\rangle
\end{bmatrix}.
\end{equation}
By duality, we infer from Proposition~\ref{prop: representation formula} that for any $F \in \calS(\bbR)\times\calS(\bbR)$
\begin{equation} \label{eqn: representation e-itH}
\begin{aligned}
(e^{it\calH(\omega)}\Pe F)(x) &= \int_\bbR e^{it(\xi^2+\omega)}\wtilcalF_{+,\omega}[F](\xi)\Psi_{+,\omega}(x,\xi)\,\ud\xi \\
&\quad - \int_\bbR e^{-it(\xi^2+\omega)}\wtilcalF_{-,\omega}[F](\xi)\Psi_{-,\omega}(x,\xi)\,\ud\xi,
\end{aligned}
\end{equation}
and
\begin{equation} \label{eqn: representation Pe}
(\Pe F)(x) =   \int_\bbR \wtilcalF_{+,\omega}[F](\xi)\Psi_{+,\omega}(x,\xi)\,\ud\xi -\int_\bbR \wtilcalF_{-,\omega}[F](\xi)\Psi_{-,\omega}(x,\xi)\,\ud\xi.
\end{equation}
In view of the formulas \eqref{eqn: representation e-itH} and \eqref{eqn: representation Pe}, we define the inverse of the distorted Fourier transform relative to $\calH(\omega)$ of a vector-valued function $G = (G_1, G_2)^{\top} \in \calS(\bbR) \times \calS(\bbR)$ by
\begin{equation} \label{equ:def-inverse-dFT}
\wtilcalF_\omega^{-1}\bigl[G\bigr](x) := \int_\bbR G_1(\xi) \Psi_{+,\omega}(x,\xi) \, \ud \xi - \int_\bbR G_2(\xi) \Psi_{-,\omega}(x,\xi) \, \ud \xi.
\end{equation}
Then we can write \eqref{eqn: representation Pe} as $P_{\mathrm{e}} = \wtilcalF_\omega^{-1} \wtilcalF_\omega$. Moveover, we have the following representation formulas for the evolution $e^{it\calH(\omega)}$ on the distorted Fourier side.

\begin{corollary} \label{cor:distFT_of_propagator}
Fix $\omega \in (0,\infty)$. We have 
\begin{equation} \label{equ:wtilcalF_applied_to_P}
\wtilcalF_{\pm, \omega} P_{\mathrm{e}} = \wtilcalF_{\pm, \omega}, \quad \wtilcalF_{\pm, \omega} P_{\mathrm{d}} = 0,
\end{equation}
and the evolution $e^{it\calH(\omega)}$ satisfies
\begin{align}
\wtilcalF_{+, \omega}[e^{it\calH(\omega)}\Pe F](\xi) &= e^{ it(\xi^2+\omega)}\wtilcalF_{+, \omega}[F](\xi),\label{eqn: diag+}\\
\wtilcalF_{-, \omega}[e^{it\calH(\omega)}\Pe F](\xi) &= e^{- it(\xi^2+\omega)}\wtilcalF_{-, \omega}[F](\xi).	\label{eqn: diag-}
\end{align}
\end{corollary}
\begin{proof}
See \cite[Corollary~3.12]{LL24}.
\end{proof}

Next, we recall the expressions for the action of $\sigma_3$ and of $\px$ on the distorted Fourier side.

\begin{lemma} \label{lem:distFT_applied_to_sigmathree_F}
Fix $\omega \in (0,\infty)$. 
Let $F= (f_1,f_2)^\top \in \calS(\bbR) \times \calS(\bbR)$.	We have
\begin{align*}
\wtilcalF_{+, \omega}[\sigma_3 F] &=\wtilcalF_{+, \omega}[ F] + \calL_{+,\omega}[F],	\\
\wtilcalF_{-, \omega}[\sigma_3 F] &= -\wtilcalF_{-, \omega}[ F] + \calL_{-,\omega}[F],
\end{align*}
where
\begin{equation*}
\begin{split}
\calL_{+,\omega}[F](\xi) &:= 2\langle f_2,\Psi_{2,\omega}(\cdot,\xi)\rangle,\\
\calL_{-,\omega}[F](\xi) &:= 2\langle f_1,\Psi_{2,\omega}(\cdot,\xi)\rangle.
\end{split}
\end{equation*}
Moreover, we have 
\begin{align*}
\wtilcalF_{+,\omega}[\px F] &= i \xi \wtilcalF_{+,\omega}[F] + \calK_{+,\omega}[F],  \\
\wtilcalF_{-,\omega}[\px F] &= i \xi \wtilcalF_{-,\omega}[F] + \calK_{-,\omega}[F],
\end{align*}
where 
\begin{align*}
\calK_{+,\omega}[F](\xi) &:= \frac{1}{\sqrt{2\pi}}\langle f_1,e^{ix\xi}\px m_{1,\omega}(x,\xi)\rangle -  \frac{1}{\sqrt{2\pi}}\langle f_2,e^{ix\xi}\px m_{2,\omega}(x,\xi)\rangle,\\
\calK_{-,\omega}[F](\xi) &:= \frac{1}{\sqrt{2\pi}}\langle f_1,e^{ix\xi}\px m_{2,\omega}(x,\xi)\rangle -  \frac{1}{\sqrt{2\pi}}\langle f_2,e^{ix\xi}\px m_{1,\omega}(x,\xi)\rangle.
\end{align*}
\end{lemma}
\begin{proof}
See \cite[Lemma~3.13]{LL24}.
\end{proof}

The relation between the distorted Fourier transforms $\wtilcalF_{+,\omega}$ and $\wtilcalF_{-,\omega}$ under complex conjugation is given below.

\begin{lemma} \label{lem:distFT_components_relation}
Fix $\omega \in (0,\infty)$. 
Let $F= (f,\barf)^\top$ for $f \in \calS(\bbR)$. Then we have
\begin{equation} \label{equ:distFT_components_relation}
\wtilcalF_{+,\omega}[F](\xi) = - \frac{(\vert \xi \vert - i \sqrt{\omega})^2}{(\vert \xi \vert + i \sqrt{\omega})^2} \overline{\wtilcalF_{-, \omega}[F](-\xi)}.
\end{equation}
\end{lemma}
\begin{proof}
See \cite[Lemma~3.14]{LL24}.
\end{proof}

\subsection{Mapping properties}
In the next lemma, we recall pseudo-differential bounds for \eqref{eqn:m-1,omega} and \eqref{eqn:m-2,omega}.

\begin{lemma} \label{lemma: PDO on m12}
For a fixed $\omega \in (0,\infty)$, we have for any $j \in \bbN_0$ and any $k\in \{0,1\}$ that
\begin{equation}\label{equ:PDO_on_m12}
\sup_{x,\xi \in \bbR} \, \bigl|\partial_x^j \partial_\xi^k m_{1,\omega}(x,\xi)\bigr| + \sup_{x,\xi \in \bbR} \, \bigl|\partial_x^j \partial_\xi^k m_{2,\omega}(x,\xi)\bigr| \lesssim_{\omega,j,k} 1.
\end{equation}
Moreover, both $m_{1,\omega}(x,\xi)$ and $m_{2,\omega}(x,\xi)$ can be written as sums of tensorized terms of the form $\fraka(x)\frakb(\xi)$, where $\fraka(x)$ is a smooth bounded function and $\frakb(\xi)$ is a bounded Lipschitz function.
\end{lemma}
\begin{proof}
See \cite[Lemma~3.16]{LL24}.
\end{proof}
The following proposition records the mapping properties of the distorted Fourier transform.

\begin{proposition} \label{prop:mapping_properties_dist_FT}
The distorted Fourier transform \eqref{eqn:def-dFT} and its inverse \eqref{equ:def-inverse-dFT} are bounded from $L^p$ to $L^{p'}$ for $1 \leq p \leq 2$, and from $H^1$ to $L^{2,1}$ and from $L^{2,1}$ to $H^1$.
\end{proposition}
\begin{proof}
See \cite[Proposition~3.17]{LL24}.
\end{proof}
\section{Linear Decay Estimates} \label{sec:linear_decay}

In this section we revisit several decay estimates for the linear Schr\"odinger evolution from \cite{LL24}. 

\subsection{Dispersive decay}
The following lemma is a standard dispersive estimate.
\begin{lemma} \label{lem:linear_dispersive_decay}
Suppose $m \colon \bbR^2 \to \bbC$ satisfies
\begin{equation}\label{eqn: lem4.1-assumption-a}
\sup_{x,\xi \in \bbR} \, \bigl( |m(x,\xi)| + |\pxi m(x,\xi)| \bigr) \lesssim 1.
\end{equation}
Then there exists a constant $C \geq 1$ such that uniformly for all $t \geq 1$,
\begin{equation}\label{eqn:linear_dispersive_decay}
\biggl\| \int_\bbR e^{\pm i t \xi^2} e^{ix\xi} m(x,\xi) g(\xi) \, \ud \xi - \frac{\sqrt{\pi}}{t^{\frac12}} e^{\mp i \frac{x^2}{4t}} e^{\pm i \frac{\pi}{4}} m\Bigl(x, \mp \frac{x}{2t} \Bigr) g\Bigl(\mp\frac{x}{2t}\Bigr) \biggr\|_{L^\infty_x} \leq \frac{C}{t^{\frac34}} \| g\|_{H^1_\xi}.
\end{equation}
In particular, it follows that for all $t \geq 1$,
\begin{equation}\label{eqn:linear_dispersive_decay2}
\biggl\| \int_\bbR e^{\pm i t \xi^2} e^{ix\xi} m(x,\xi) g(\xi) \, \ud \xi \biggr\|_{L^\infty_x} \lesssim \frac{1}{t^{\frac12}} \|g\|_{L^\infty_\xi} + \frac{1}{t^{\frac34}} \| g(\xi)\|_{H^1_\xi}.
\end{equation}
\end{lemma}
\begin{proof}
See \cite[Lemma~4.1]{LL24}.
\end{proof}

\subsection{Improved local decay}
The lemma below captures the leading local decay due to the threshold resonances $\Psi_{j,\omega}(x,0)$ of the linearized operator. 
\begin{lemma} \label{lem:improved_local_decay_difference_Psis}
Fix $\omega \in (0, \infty)$. Let $\Psi_{1,\omega}(x,\xi)$, $\Psi_{2,\omega}(x,\xi)$ be given by \eqref{eqn:m-1,omega}, \eqref{eqn:m-2,omega} respectively. Then there exists $C \geq 1$ such that for $j = 1,2$ we have uniformly for all $t \geq 0$ that
\begin{equation} \label{equ:improved_local_decay_difference_Psis_no_px}
\begin{aligned}
\biggl\| \jx^{-2} \int_\bbR e^{\pm i t \xi^2} \bigl( \Psi_{j,\omega}(x,\xi) - \Psi_{j,\omega}(x,0) \bigr) g(\xi) \, \ud \xi \biggr\|_{L^\infty_x} \leq \frac{C}{\jt} \Bigl( \|g\|_{L^\infty_\xi} + \|\pxi g\|_{L^2_\xi} + \|\jxi g\|_{L^2_\xi} \Bigr)
\end{aligned}
\end{equation}
and
\begin{equation} \label{equ:improved_local_decay_difference_Psis_with_px}
\begin{aligned}
\biggl\| \jx^{-3} \int_\bbR e^{\pm i t \xi^2} \px \bigl( \Psi_{j,\omega}(x,\xi) - \Psi_{j,\omega}(x,0) \bigr) g(\xi) \, \ud \xi \biggr\|_{L^2_x} \leq \frac{C}{\jt} \Bigl( \|g\|_{L^\infty_\xi} + \|\pxi g\|_{L^2_\xi} + \|\jxi g\|_{L^2_\xi} \Bigr).
\end{aligned}
\end{equation}
\end{lemma}
\begin{proof}
See \cite[Lemma~4.2]{LL24}.
\end{proof}
The next lemma shows that high frequencies enjoy improved local decay.
\begin{lemma} \label{lem:improved_local_decay_simple} 
Let $\fraka \in W^{1,\infty}(\bbR)$.
Denote by $\chi_0(\xi)$ a smooth even non-negative cut-off function with $\chi_0(\xi) = 1$ for $|\xi| \leq 1$ and $\chi_0(\xi) = 0$ for $|\xi| \geq 2$. There exists $C \geq 1$ such that uniformly for all $t \geq 0$,
\begin{align}
\biggl| \int_\bbR e^{\pm i t \xi^2} \fraka(\xi) \bigl( 1 - \chi_0(\xi) \bigr) g(\xi) \, \ud \xi \biggr| &\leq \frac{C}{\jt} \Bigl( \|\pxi g(\xi)\|_{L^2_\xi} + \|\jxi g(\xi)\|_{L^2_\xi} \Bigr), \label{equ:improved_local_decay_simple_bound1} \\
\biggl\| \jx^{-2} \int_\bbR \bigl( e^{ix\xi} - 1 \bigr) e^{\pm i t \xi^2} \fraka(\xi) g(\xi) \, \ud \xi \biggr\|_{L^\infty_x} &\leq \frac{C}{\jt} \Bigl( \|\pxi g(\xi)\|_{L^2_\xi} + \|\jxi g(\xi)\|_{L^2_\xi} \Bigr). \label{equ:improved_local_decay_simple_bound2}
\end{align}
\end{lemma}
\begin{proof}
	See \cite[Lemma~4.3]{LL24}.
\end{proof}

\subsection{Smoothing estimates with a moving center}
The derivation of the weighted energy estimates for the spatialliy localized terms with cubic-type time decay happens in Section~\ref{sec:energy_estimates}. When a $\xi$ derivative hits the phase $e^{is\xi^2}$, a divergent factor of $s\xi$ appears, and we handle such terms by the use of local smoothing estimates. However, unlike \cite{LL24}, we cannot employ the standard local smoothing estimate (e.g. \cite[Lemma~4.4]{LL24}) to treat these terms in our setting due to the presence of a moving center with a small uncertainty. Instead, we have to prove the following proposition, where the smoothing estimates \eqref{equ:estimate_smoothing1} and \eqref{equ:estimate_smoothing2} incorporate a moving center. The reader is referred to \eqref{equ:proof_weighted_estimate_3_expanded} where such estimates are used. 
\begin{proposition}\label{prop:local-smoothing}
Let $\eta(\xi)$ be a smooth cutoff function supported on $[-4,4]$ with $0 \leq \eta \leq 1$ and $\eta(\xi) = 1$ for $\xi \in [-2,2]$.	Let $\fraka \in C^\infty(\bbR)$ with $\| \px^k \fraka \|_{L_x^\infty} \lesssim 1$ for all integers $k \geq 0$, and let $ \frakb \in W_\xi^{1,\infty}(\bbR)$. Suppose that $\theta:[0,\infty) \rightarrow \bbR$ satisfies $\theta(0)=0$ and $|\theta'(s)|\lesssim \eps \js^{-1+\delta}$ for all $s \geq 0$ and for some $0 \leq \eps,\delta \ll 1$.  Then for all $t\geq 0$, we have
\begin{equation}\label{equ:estimate_smoothing1}
\begin{split}
\left \| \int_0^t e^{is \xi^2}\big(\xi \eta(\xi)\big) e^{i\theta(s)\xi}\left[\int_\bbR e^{-ix\xi} \fraka(x) \frakb(\xi) F(s,x) \,\ud x\right]\,\ud s   \right \|_{L_\xi^2} \lesssim \big\|\jxm F(s,x) \big\|_{L_s^2([0,t];L_x^2)},
\end{split}	
\end{equation}
and
\begin{equation}\label{equ:estimate_smoothing2}
	\begin{split}
\left \| \int_0^t e^{is \xi^2}\big(1-\eta(\xi)\big) e^{i\theta(s)\xi}\left[\int_\bbR e^{-ix\xi} \fraka(x) \frakb(\xi) F(s,x) \,\ud x\right]\,\ud s   \right \|_{L_\xi^2} \lesssim \big\|\jxm F(s,x) \big\|_{L_s^2([0,t];L_x^2)}.
	\end{split}	
\end{equation}
\end{proposition}

Proposition~\ref{prop:local-smoothing} is inspired by \cite[Proposition~6.5]{CL24} (which was based on \cite{NS12}). It requires some preparations. Let $\eta(\xi)$ and $\theta(s)$ be given as in  Proposition~\ref{prop:local-smoothing} throughout this section. With the notation $D := -i \partial_x$ we define the following Fourier multipliers
\begin{align}
[\frakm_{\mathrm{l}}(D)f](x) &:= \widehat{\calF}^{-1}\big[\xi \eta(\xi)\hatf(\xi)\big](x),\\
[\frakm_{\mathrm{h}}(D)f](x) &:= \widehat{\calF}^{-1}\big[\big(1 - \eta(\xi)\big) \hatf(\xi)\big](x),\\
[\mathrm{T}_\theta(s,t)f](x) &:= \widehat{\calF}^{-1}\big[e^{i(\theta(t)-\theta(s))\xi} \hatf(\xi)\big](x) \equiv f(x+\theta(t)-\theta(s)).
\end{align}
The proof of Proposition~\ref{prop:local-smoothing} is implied by the following lemma.
\begin{lemma}\label{lemma:inhom_smoothing} Under the same assumptions as Proposition~\ref{prop:local-smoothing}, we have  
\begin{align}
\left\| \int_0^t e^{-is\px^2}\frakm_{\mathrm{l}}(D)[\mathrm{T}_{\theta}(0,s)F(s)] \,\ud s \right\|_{L_x^2} &\lesssim \left \| \jxm F \right\|_{L_s^2([0,t];L_x^2)},\label{equ:dual_hom_low}\\
\left\| \int_0^t e^{-is\px^2}\frakm_{\mathrm{h}}(D)[\mathrm{T}_{\theta}(0,s)F(s)] \,\ud s \right\|_{L_x^2} &\lesssim \left \| \jxm F \right\|_{L_s^2([0,t];L_x^2)}. \label{equ:dual_hom_high}
\end{align}
\end{lemma}

We first present a short proof of Proposition~\ref{prop:local-smoothing} using Lemma~\ref{lemma:inhom_smoothing}.
\begin{proof}[Proof of Proposition~\ref{prop:local-smoothing}] By the Plancherel's identity and the estimate \eqref{equ:dual_hom_low}, we obtain 
\begin{equation*}
\begin{split}
&\left \| \int_0^t e^{is \xi^2}\big(\xi \eta(\xi)\big) e^{i\theta(s)\xi}\left[\int_\bbR e^{-ix\xi} \fraka(x) \frakb(\xi) F(s,x) \,\ud x\right]\,\ud s   \right \|_{L_\xi^2}\\
&\lesssim \| \frakb \|_{L_\xi^\infty} \left \| \int_0^t e^{is \xi^2}\big(\xi \eta(\xi)\big) e^{i\theta(s)\xi} \widehat{\calF}\left[\fraka(\cdot) F(s,\cdot)\right](\xi)\,\ud s   \right \|_{L_\xi^2}\\
&\lesssim \left\| \int_0^t e^{-is\px^2}\frakm_{\mathrm{l}}(D)[\mathrm{T}_{\theta}(0,s)a(\cdot)F(s,\cdot)] \,\ud s \right\|_{L_x^2} \\
&\lesssim \big\| \jxm \fraka(x) F(s,x) \big\|_{L_s^2([0,t];L_x^2)} \lesssim \big\| \jxm  F(s,x) \big\|_{L_s^2([0,t];L_x^2)}.
\end{split}
\end{equation*}
This proves \eqref{equ:estimate_smoothing1}. The proof for \eqref{equ:estimate_smoothing2} follows from \eqref{equ:dual_hom_high} similarly.
\end{proof}

The proof of Lemma~\ref{lemma:inhom_smoothing} can in turn be inferred from the following Lemma~\ref{lemma:dual_low_smoothing} and Lemma~\ref{lemma:dual_high_smoothing}.

\begin{lemma}\label{lemma:dual_low_smoothing} Let $w(x) := \jxmn$. Under the same assumptions as Proposition~\ref{prop:local-smoothing}, for any $F,G \in L_t^2([0,\infty);L_x^2(\bbR))$ it holds that
\begin{align}
\left| \int_0^\infty \int_0^t \left \langle e^{i(t-s)\px^2} \frakm_{\mathrm{l}}(D)^2 [\mathrm{T}_{\theta}(s,t)wF(s)], wG(t) \right \rangle \,\ud s \,\ud t \right| &\lesssim \| F \|_{L_t^2([0,\infty);L_x^2)} \| G \|_{L_t^2([0,\infty);L_x^2)}, \label{equ:dual_low_smoothing_0t}\\
\left| \int_0^\infty \int_t^\infty \left \langle e^{i(t-s)\px^2} \frakm_{\mathrm{l}}(D)^2 [\mathrm{T}_{\theta}(s,t)wF(s)], wG(t) \right \rangle \,\ud s \,\ud t \right| &\lesssim \| F \|_{L_t^2([0,\infty);L_x^2)} \| G \|_{L_t^2([0,\infty);L_x^2)}.\label{equ:dual_low_smoothing_ti}
\end{align}
\end{lemma}

\begin{lemma}\label{lemma:dual_high_smoothing} Let $w(x) := \jxmn$. Under the same assumptions as Proposition~\ref{prop:local-smoothing}, for any $F,G \in L_t^2([0,\infty);L_x^2(\bbR))$ it holds that
\begin{align}
\left| \int_0^\infty \int_0^t \left \langle e^{i(t-s)\px^2} \frakm_{\mathrm{h}}(D)^2 [\mathrm{T}_{\theta}(s,t)wF(s)], wG(t) \right \rangle \,\ud s \,\ud t \right| &\lesssim \| F \|_{L_t^2([0,\infty);L_x^2)} \| G \|_{L_t^2([0,\infty);L_x^2)}, \label{equ:dual_high_smoothing_0t}\\
\left| \int_0^\infty \int_t^\infty \left \langle e^{i(t-s)\px^2} \frakm_{\mathrm{h}}(D)^2 [\mathrm{T}_{\theta}(s,t)wF(s)], wG(t) \right \rangle \,\ud s \,\ud t \right| &\lesssim \| F \|_{L_t^2([0,\infty);L_x^2)} \| G \|_{L_t^2([0,\infty);L_x^2)}. \label{equ:dual_high_smoothing_ti}
\end{align}
\end{lemma}
The proof of Lemma~\ref{lemma:dual_low_smoothing} and of Lemma~\ref{lemma:dual_high_smoothing} require some technical preparations. In Subsection~\ref{subsec:low_energy} and Subsection~\ref{subsec:high_energy}, we provide proofs for \eqref{equ:dual_low_smoothing_0t} and \eqref{equ:dual_high_smoothing_0t} respectively, and we leave the analogous proofs for \eqref{equ:dual_low_smoothing_ti} and \eqref{equ:dual_high_smoothing_ti} to the reader. We first finish the proof of Lemma~\ref{lemma:inhom_smoothing} which is based on Lemma~\ref{lemma:dual_low_smoothing} and Lemma~\ref{lemma:dual_high_smoothing}.

\begin{proof}[Proof of Lemma~\ref{lemma:inhom_smoothing}]
For the proof of the estimate \eqref{equ:dual_hom_low} we introduce the operator $\calK: L_x^2 \rightarrow L_t^2([0,\infty);L_x^2)$ where
\begin{equation*}
(\calK f)(t,x) := w(x) e^{it\px^2} \frakm_{\mathrm{l}}(D) [\mathrm{T}_{\theta}(t,0)f](x),
\end{equation*}
with $w(x) := \jxmn$. We note that the adjoint of $\calK$ as an operator $\calK^*:L_t^2([0,\infty);L_x^2) \rightarrow L_x^2$ is given by 
\begin{equation*}
(\calK^*F)(x) = \left(\int_0^\infty e^{-is\px^2} \frakm_{\mathrm{l}}(D) [\mathrm{T}_{\theta}(0,s)w(\cdot)F(s,\cdot)]\,\ud s \right)(x).
\end{equation*}
Hence, the estimate \eqref{equ:dual_hom_low} follows from the boundedness of $\calK^*$ by restricting the time interval to $[0,t]$. By standard duality arguments, proving the boundedness of $\calK$ (and $\calK^*$) is equivalent to proving the boundedness of $\calK\calK^* : L_t^2([0,\infty);L_x^2) \rightarrow L_t^2([0,\infty);L_x^2)$ where
\begin{equation}\label{equ:proof_KKstar}
(\calK\calK^* F)(t,x) = w(x) \left(\int_0^\infty e^{i(t-s)\px^2}\frakm_{\mathrm{l}}(D)^2 [\mathrm{T}_{\theta}(t,s) w(\cdot)F(s,\cdot)]\right)(x).
\end{equation}
Then, we distinguish the time integration between $[0,t]$ and $[t,\infty)$, and infer the boundedness of \eqref{equ:proof_KKstar} from the following estimates 
\begin{equation*}
\begin{split}
\left\| w(x)\int_0^t e^{i(t-s)\px^2} \frakm_{\mathrm{l}}(D)^2 [\mathrm{T}_{\theta}(t,s)w(\cdot)F(\cdot)]\,\ud s \right\|_{L_t^2([0,\infty);L_x^2)} &\lesssim \| F \|_{L_t^2([0,\infty);L_x^2)},\\
\left\| w(x)\int_t^\infty e^{i(t-s)\px^2} \frakm_{\mathrm{l}}(D)^2 [\mathrm{T}_{\theta}(t,s)w(\cdot)F(\cdot)]\,\ud s \right\|_{L_t^2([0,\infty);L_x^2)} &\lesssim \| F \|_{L_t^2([0,\infty);L_x^2)},
\end{split}
\end{equation*}
where $F \in L_t^2([0,\infty);L_x^2)$. By duality, these two estimates  are equivalent to \eqref{equ:dual_low_smoothing_0t} and \eqref{equ:dual_low_smoothing_ti} respectively. The same argument given above can be applied to conclude \eqref{equ:dual_hom_high} using \eqref{equ:dual_high_smoothing_0t} and \eqref{equ:dual_high_smoothing_ti}.

\end{proof}
We now prove Lemma~\ref{lemma:dual_low_smoothing} and of Lemma~\ref{lemma:dual_high_smoothing}. First, we recall the standard  Littlewood-Paley decomposition.  Let $\psi \in C_c^\infty(\bbR)$ be a smooth non-negative even bump function with $\psi(\xi) = 1$ for $|\xi|\leq 1$ and $\psi(\xi) = 0$ for $|\xi| \geq 2$, and let $\varphi(\xi) := \psi(\xi) -  \psi(2\xi)$. We use the standard dyadic Littlewood-Paley decomposition $I = \sum_{j=0}^\infty P_j$ where 
\begin{align}
P_j g &:= \widehat{\calF}^{-1}[\varphi(2^{-j}\xi)\hatg(\xi)], \quad j \geq 1,\\
P_0 g &:= \widehat{\calF}^{-1}[\psi(\xi)\hatg(\xi)].
\end{align}

\subsubsection{Low energy case. Proof of Lemma~\ref{lemma:dual_low_smoothing} }\label{subsec:low_energy}

In this setting, we observe that $\widehat{\mathcal{F}}[\frakm_{\mathrm{l}}(D)^2f](\xi) = \big(\xi \eta(\xi)\big)^2 \hatf(\xi)$ has a compact support inside the interval $[-4,4]$ for any $f \in L^2$. Hence, we have 
\begin{equation}\label{equ:small_support_observation}
\frakm_{\mathrm{l}}(D)^2 P_j f \equiv 0, \quad \text{for $j \geq 5$.}
\end{equation}
In the next lemma, we establish improved local decay estimates for the integral kernels of frequency localized Schr\"odinger evolutions.
\begin{lemma}
Define the kernels
\begin{equation}\label{equ:low_kernel_def}
\begin{split}
K_{\mathrm{l},j}(t,x) &:= \int_\bbR e^{ix\xi} e^{-it\xi^2} \big(\xi \eta(\xi)\big)^2 \varphi(2^{-j}\xi)\,\ud \xi, \quad 1 \leq j \leq 4,\\
K_{\mathrm{l},0}(t,x) &:= \int_\bbR e^{ix\xi} e^{-it\xi^2} \big(\xi \eta(\xi)\big)^2 \psi(\xi)\,\ud \xi.
\end{split}
\end{equation}
Then for all $0 \leq j \leq 4$ we have 
\begin{equation}\label{equ:low_kernel_estimate}
\big|\jx^{-1}K_{\mathrm{l},j}(t,x)\big| \lesssim \jt^{-\frac32}.
\end{equation}
\end{lemma}
\begin{proof}
For short times $t \leq 1$, we have 
\begin{equation*}
\begin{split}
|K_{\mathrm{l},j}(t,x)| &\lesssim \int_\bbR| \xi^2\eta(\xi)^2 \varphi(2^{-j}\xi) |\,\ud \xi \leq 2^{3j} \int_{[-4 \times 2^{-j},4 \times 2^{-j}]} \gamma^2 \,\ud \gamma \lesssim 1, \quad 1 \leq j \leq 4,\\
|K_{\mathrm{l},0}(t,x)| &\lesssim 1.
\end{split}
\end{equation*}
For times $t \geq 1$, we integrate by parts to obtain 
\begin{equation*}
\begin{split}
K_{\mathrm{l},j}(t,x) &= \frac{x}{2t} \int_\bbR e^{ix\xi} e^{-it\xi^2} \xi \eta(\xi)^2 \varphi(2^{-j}\xi)\,\ud \xi + \frac{1}{2it} \int_\bbR e^{ix\xi} e^{-it\xi^2}\big(\xi \eta(\xi)^2\big)' \varphi(2^{-j}\xi)\,\ud \xi\\
&\quad +\frac{2^{-j}}{2it} \int_\bbR e^{ix\xi} e^{-it\xi^2} \xi \eta(\xi)^2 \varphi'(2^{-j}\xi)\,\ud \xi.
\end{split}
\end{equation*}
Using the standard estimate for the flat Schr\"odinger evolution $\| e^{it\px^2} \|_{L^1 \rightarrow L^\infty} \lesssim t^{-\frac12}$, we conclude 
\begin{equation*}
\begin{split}
&|\jx^{-1} K_{\mathrm{l},j}(t,x) | \\
&\lesssim t^{-\frac32} \left(\Big\| \widehat{\calF}^{-1}[\xi \eta(\xi)^2 \varphi(2^{-j}\xi)] \Big\|_{L_x^1} + \Big\| \widehat{\calF}^{-1}[(\xi \eta(\xi)^2)' \varphi(2^{-j}\xi)] \Big\|_{L_x^1} + \Big\| \widehat{\calF}^{-1}[\xi \eta(\xi)^2 \varphi'(2^{-j}\xi)] \Big\|_{L_x^1}\right)\\
&\lesssim_j t^{-\frac32},
\end{split}
\end{equation*}
as desired. The bound \eqref{equ:low_kernel_estimate} for $j=0$ also follows analogously.

\end{proof}

Next, we derive bounds for the following bilinear form.
\begin{lemma}\label{lemma:bilinear_low_j}
For $j \geq 0$, $M >0$, and for $F,G \in L_t^2([0,\infty);L_x^2)$ let 
\begin{equation}\label{equ:bilinear_form_low}
I_{\mathrm{l},j}^M(F,G) := \int_0^\infty \int_0^{\max\{t-M,0\}} \left \langle e^{i(t-s)\px^2}\frakm_\mathrm{l}(D)^2 [\mathrm{T}_{\theta}(t,s)P_j(wF(s))],wG(t) \right \rangle \,\ud s \,\ud t.
\end{equation}
Then, for $j \geq 0$, 
\begin{equation}\label{equ_bilinear_low_estimate}
\left|I_{\mathrm{l},j}^M(F,G)\right| \lesssim M^{-\frac12 + \delta} \|F \|_{L_t^2([0,\infty);L_x^2)} \|G \|_{L_t^2([0,\infty);L_x^2)}.
\end{equation}
\end{lemma}
\begin{proof}
Let us set 
\begin{equation}
b = b(t,s) := \theta(t) - \theta(s).
\end{equation}
Then, we express the bilinear form \eqref{equ:bilinear_form_low} using the kernel \eqref{equ:low_kernel_def} where
\begin{equation*}
\begin{split}
I_{\mathrm{l},j}^M(F,G) &= \int_0^\infty \int_0^{\max\{t-M,0\}} \int_\bbR \int_\bbR  K_{\mathrm{l},j}(t-s,x-y-b)w(y)F(s,y) \cdot \overline{w(x)G(t,x)} \,\ud y \,\ud  x \,\ud s \,\ud t\\
&=   \int_\bbR \int_\bbR \int_{s=0}^\infty \int_{t = s+M}^\infty   K_{\mathrm{l},j}(t-s,x-y-b)w(y)F(s,y) \cdot \overline{w(x)G(t,x)} \,\ud t \,\ud s\,\ud y \,\ud  x .
\end{split}
\end{equation*}
Since the integral kernel \eqref{equ:low_kernel_def} is anti-symmetric, we may write 
\begin{equation*}
I_{\mathrm{l},j}^M(F,G) = J_{\mathrm{l},j}^M(F,G) + \overline{J_{\mathrm{l},j}^M(G,F)},
\end{equation*}
where
\begin{equation*}
J_{\mathrm{l},j}^M(F,G) :=  \iint_{\{|y| > |x|\}} \int_{s=0}^\infty \int_{t = s+M}^\infty   K_{\mathrm{l},j}(t-s,x-y-b)w(y)F(s,y) w(x) \overline{G(t,x)} \,\ud t \,\ud s\,\ud y \,\ud  x. 
\end{equation*}
Like in \cite[Lemma~6.10]{CL24}, we switch to the new variables $(t,s,y,x) \leftrightarrow (u,s,z,h)$, where 
\begin{equation}\label{equ:proof_new_variables}
u := t-s, \quad z := x- y -b, \quad h := y + b.
\end{equation}
In these new variables, we have 
\begin{equation}\label{equ:proof_J_bilinear_low}
\begin{split}
J_{\mathrm{l},j}^M(F,G) =  \iint_{\{|h-b| > |h+z|\}} \int_{s=0}^\infty \int_{u = M}^\infty  & K_{\mathrm{l},j}(u,z) w(h-b)\\
&\times F(s,h-b) w(h+z) \overline{G(s+u,h+z)} \,\ud u \,\ud s\,\ud h \,\ud  z. 
\end{split}
\end{equation}
From the local decay estimate \eqref{equ:low_kernel_estimate} we bound \eqref{equ:proof_J_bilinear_low} by 
\begin{equation*}
\begin{split}
|J_{\mathrm{l},j}^M(F,G)| \lesssim \iint_{\{|h-b| > |h+z|\}} \int_{s=0}^\infty \int_{u = M}^\infty  & \jap{u}^{-\frac32} |\jap{z} w(h-b) F(s,h-b)\\
&\times  w(h+z) \overline{G(s+u,h+z)}| \,\ud u \,\ud s\,\ud h \,\ud  z\\
\lesssim \iint_{\{|h-b| > |h+z|\}} \int_{s=0}^\infty \int_{u = M}^\infty  & \jap{u}^{-\frac32} |\jap{h-b} w(h-b) F(s,h-b)\\
&\times  w(h+z) \overline{G(s+u,h+z)}| \,\ud u \,\ud s\,\ud h \,\ud  z\\
\quad + \iint_{\{|h-b| > |h+z|\}} \int_{s=0}^\infty \int_{u = M}^\infty  & \jap{u}^{-\frac32} |w(h-b) F(s,h-b)\\
&\times  \jap{h+z}w(h+z) \overline{G(s+u,h+z)}| \,\ud u \,\ud s\,\ud h \,\ud  z\\
\quad + \iint_{\{|h-b| > |h+z|\}} \int_{s=0}^\infty \int_{u = M}^\infty  & \jap{u}^{-\frac32} \jap{b} |w(h-b) F(s,h-b)\\
&\times  w(h+z) \overline{G(s+u,h+z)}| \,\ud u \,\ud s\,\ud h \,\ud  z,
\end{split}
\end{equation*}
where we have also used $\jap{z} \lesssim \jap{h-b} + \jap{h+z} + \jap{b}$ in the last inequality. By \cite[Lemma~6.8]{CL24}, we have 
\begin{equation*}
|b| = |\theta(t)-\theta(s)| \lesssim \eps |t-s|^\delta\lesssim |u|^\delta.
\end{equation*}
Hence, by applying the Cauchy-Schwarz inequality and the above bound, we obtain that 
\begin{equation*}
|J_{\mathrm{l},j}^M(F,G)| \lesssim \int_{u=M}^\infty \jap{u}^{-\frac32+ \delta} \|F \|_{L_t^2 L_x^2} \|G \|_{L_t^2 L_x^2} \|\jap{x} w(x)\|_{L_x^2}^2 \,\ud u \lesssim M^{-\frac12 + \delta} \|F \|_{L_t^2 L_x^2} \|G \|_{L_t^2 L_x^2}.
\end{equation*}
The preceeding bound then implies the asserted estimate \eqref{equ_bilinear_low_estimate} as desired. 
\end{proof}

We now provide the proof of the estimate \eqref{equ:dual_low_smoothing_0t} for Lemma~\ref{lemma:dual_low_smoothing}.
\begin{proof}[Proof of Lemma~\ref{lemma:dual_low_smoothing}]
As noted before, we leave the second asserted estimate \eqref{equ:dual_low_smoothing_ti} for the reader and we provide the proof for the  estimate \eqref{equ:dual_low_smoothing_0t}. By taking \eqref{equ:small_support_observation} into account and inserting the Littlewood-Paley decomposition for $wF(s)$ in \eqref{equ:dual_low_smoothing_0t}, we find that 
\begin{equation*}
\begin{split}
\left| \int_0^\infty \int_0^t \left \langle e^{i(t-s)\px^2} \frakm_{\mathrm{l}}(D)^2 [\mathrm{T}_{\theta}(s,t)wF(s)], wG(t) \right \rangle \,\ud s \,\ud t \right| \leq \sum_{j=0}^4 | I_{\mathrm{l},j}^0(F,G)| \\
\leq  \sum_{j=0}^4 | I_{\mathrm{l},j}^{M=1}(F,G)| + \sum_{j=0}^4 | I_{\mathrm{l},j}^{M=1}(F,G) - I_{\mathrm{l},j}^0(F,G)|,
\end{split}
\end{equation*}
where we have also added the bilinear form \eqref{equ:bilinear_form_low} with the choice $M=1$, using the  triangle inequality.  The interaction of the second term is stronger since $|t-s|\leq 1$. In this case, the $L^2$ conservation for the Schr\"odinger propagator and the $L^2$ boundedness of $\frakm_{\mathrm{l}}(D)$ imply that 
\begin{equation}\label{equ:proof_low_estimate_dual}
\sum_{j=0}^4 | I_{\mathrm{l},j}^{M=1}(F,G) - I_{\mathrm{l},j}^0(F,G)| \lesssim \| \frakm_{\mathrm{l}}(D)^2 wF \|_{L_t^2L_x^2} \| w G \|_{L_t^2L_x^2} \lesssim \| F \|_{L_t^2L_x^2} \| G \|_{L_t^2L_x^2}.
\end{equation}
Hence, by combining \eqref{equ_bilinear_low_estimate} and \eqref{equ:proof_low_estimate_dual}, we obtain 
\begin{equation*}
\left| \int_0^\infty \int_0^t \left \langle e^{i(t-s)\px^2} \frakm_{\mathrm{l}}(D)^2 [\mathrm{T}_{\theta}(s,t)wF(s)], wG(t) \right \rangle \,\ud s \,\ud t \right| \lesssim \| F \|_{L_t^2L_x^2} \| G \|_{L_t^2L_x^2},
\end{equation*}
as desired.
\end{proof}

\subsubsection{High energy case. Proof of Lemma~\ref{lemma:dual_high_smoothing}}\label{subsec:high_energy}

In this setting, we note that $\frakm_{\mathrm{h}}(D)P_0 \equiv 0$ since the support of $(1-\eta(\xi))$ does not contain the interval $[-2,2]$. In the next lemma, we establish pointwise decay estimates for the frequency localized integral kernels for the Schr\"odinger evolution. 

\begin{lemma}
Define the kernels
\begin{equation}
\begin{split}
K_{\mathrm{h},j}(t,x) &:= \int_\bbR e^{ix\xi} e^{-it\xi^2} \big(1- \eta(\xi)\big)^2 \varphi(2^{-j}\xi)\,\ud \xi, \quad j \geq 1.
\end{split}
\end{equation}
Then for any $j \geq 1$ and for any $t\geq 1$ we have 
\begin{equation}\label{equ:kernel_high_disp}
|K_{\mathrm{h},j}(t,x)| \lesssim t^{-\frac12}.
\end{equation}
Furthermore, if $|x|>2^{j+4}|t|$ or $|x|<2^{j-4}|t|$, then for any $N\geq 1$,
\begin{equation}\label{equ:kernel_high_nonsp}
|K_{\mathrm{h},j}(t,x)| \lesssim_N 2^j (2^{2j}t)^{-N}.
\end{equation}
\end{lemma}
\begin{proof}
We first establish \eqref{equ:kernel_high_disp} by distinguishing the cases for $1 \leq j \leq j_0$ and $j >j_0$ where we fix $j_0 := 4$. For $1 \leq j \leq j_0$, using the dispersive estimate $\| e^{it\px^2} \|_{L^1 \rightarrow L^\infty} \lesssim t^{-\frac12}$, we have
\begin{equation*}
|K_{\mathrm{h},j}(t,x)| \lesssim t^{-\frac12} \big\| \widehat{\calF}^{-1}[ (1-\eta(\xi))^2 \varphi(2^{-j}\xi)]\big\|_{L^1} \lesssim_{j_0} t^{-\frac12}.
\end{equation*}
On the other hand, for $j > j_0$, we observe that $\eta(\xi) \equiv 0$ for all $\xi \in \supp(\varphi(2^{-j}\cdot))$. Hence, for $j >j_0$, we have 
\begin{equation}\label{equ:proof_rescaling}
K_{\mathrm{h},j}(t,x) = 2^j \tilK(2^{2j}t,2^jx), \quad \tilK(s,y) := \int_\bbR e^{iy\xi} e^{-is\xi^2}\varphi(\xi)\,\ud \xi.
\end{equation}
Using again the estimate $\| e^{it\px^2} \|_{L^1 \rightarrow L^\infty} \lesssim t^{-\frac12}$, we find that $\sup_y |\tilK(s,y)| \lesssim s^{-\frac12}$. By rescaling according to \eqref{equ:proof_rescaling}, we obtain
\begin{equation*}
|K_{\mathrm{h},j}(t,x)| \lesssim 2^j |\tilK(2^{2j}t,2^j x)| \lesssim t^{-\frac12},
\end{equation*}
for all $j > j_0$. Thus, we have established \eqref{equ:kernel_high_disp} for all $j \geq 1$. Next, we show \eqref{equ:kernel_high_nonsp}. For $j\geq1$ we write
\begin{equation*}
K_{\mathrm{h},j}(t,x) = \int_\bbR e^{it\phi(\xi;x,t)}a_j(\xi)\,\ud \xi
\end{equation*}
with $a_j(\xi) := \big(1-\eta(\xi)\big)^2 \varphi(2^{-j}\xi)$ and $\phi(\xi;x,t) := -\xi^2 + \xi \frac{x}{t}$. Note that $\supp(a_j) \subset [-2^{j+1},-2^{j-1}]\cup[2^{j-1},2^{j+1}]$. For $|x|>2^{j+4}|t|$, we find by direct computation that for $j \geq 1$ on the support of $a_j(\xi)$, we have 
\begin{equation}\label{equ:proof_smoothing_phi'1}
|\pxi \phi(\xi;x,t)| =\left|-2\xi + \frac{x}{t}\right| \geq \left|\frac{x}{t}\right| -2 |\xi| \gtrsim 2^j.
\end{equation}
Similarly, for $|x|<2^{j-4}|t|$, $j \geq 1$, and for $\xi \in \supp(a_j)$,
\begin{equation}\label{equ:proof_smoothing_phi'2}
|\pxi \phi(\xi;x,t)| \geq 2 |\xi| - \left|\frac{x}{t}\right|  \gtrsim 2^j.
\end{equation}
Moreover, we note that $\pxi^2 \phi \equiv -2$. Integrating by parts in $\xi$ repeatedly, we obtain for any integer $N\geq1$ that 
\begin{equation*}
K_{\mathrm{h},j}(t,x) = \int_\bbR e^{it\phi(\xi;x,t)} L^N \big(a_j(\xi)\big)\,\ud \xi
\end{equation*}
where 
\begin{equation*}
L := - \pxi \left(\frac{1}{it \pxi \phi} \cdot\right).
\end{equation*}
Noting that $\pxi\big((\pxi \phi)^{-1}\big)= 2(\pxi \phi)^{-2}$, we conclude from \eqref{equ:proof_smoothing_phi'1} and \eqref{equ:proof_smoothing_phi'2} that 
\begin{equation*}
|K_{\mathrm{h},j}(t,x)| \lesssim_N 2^j (2^{2j} t)^{-N}
\end{equation*}
for all $|x|>2^{j+4}|t|$ or $|x|<2^{j-4}|t|$ and for all $j \geq 1$.
\end{proof}

Next, we derive bounds for the following bilinear form.
\begin{lemma}
For $j \geq 1$, $M \geq 1$, and for $F,G \in L_t^2([0,\infty);L_x^2)$ let 
\begin{equation}\label{equ:bilinear_form_high}
I_{\mathrm{h},j}^M(F,G) := \int_0^\infty \int_0^{\max\{t-M,0\}} \left \langle e^{i(t-s)\px^2}\frakm_\mathrm{h}(D)^2 [\mathrm{T}_{\theta}(t,s)P_j(wF(s))],wG(t) \right \rangle \,\ud s \,\ud t.
\end{equation}
Then we have  
\begin{equation}\label{equ_bilinear_high_estimate}
\left|I_{\mathrm{h},j}^M(F,G)\right| \lesssim M^{-1} \|F \|_{L_t^2([0,\infty);L_x^2)} \|G \|_{L_t^2([0,\infty);L_x^2)}.
\end{equation}
\end{lemma}
\begin{proof}
Like in the proof of Lemma~\ref{lemma:bilinear_low_j}, we may write 
\begin{equation}\label{equ:proof_identity_I_hj}
I_{\mathrm{h},j}^M(F,G) = J_{\mathrm{h},j}^M(F,G) + \overline{J_{\mathrm{h},j}^M(G,F)},
\end{equation}
where
\begin{equation*}
J_{\mathrm{h},j}^M(F,G) :=  \iint_{\{|y| > |x|\}} \int_{s=0}^\infty \int_{t = s+M}^\infty   K_{\mathrm{h},j}(t-s,x-y-b)w(y)F(s,y) w(x) \overline{G(t,x)} \,\ud t \,\ud s\,\ud y \,\ud  x.
\end{equation*}
Switching to the new variables \eqref{equ:proof_new_variables}, we have 
\begin{equation*}
\begin{split}
J_{\mathrm{h},j}^M(F,G) = \iint_{\{|h-b| > |h+z|\}} \int_{s=0}^\infty \int_{u = M}^\infty  & K_{\mathrm{h},j}(u,z) w(h-b)\\
&\times F(s,h-b) w(h+z) \overline{G(s+u,h+z)} \,\ud u \,\ud s\,\ud h \,\ud  z. 
\end{split}
\end{equation*}
Then we subdivide the $(u,z)$-plane into two disjoint regions
\begin{equation*}
\begin{split}
A &:= \{(u,z): |z|< 2^{j-4}|u| \quad \text{or} \quad |z| > 2^{j+4}|u|\},\\
A^{\mathrm{c}}&:=\{(u,z): 2^{j-4}|u| \leq |z| \leq 2^{j+4}|u|\}.
\end{split}
\end{equation*}
To distinguish the two regions we insert sharp cut-offs functions $1 = \mathbbm{1}_{A} + \mathbbm{1}_{A^\mathrm{c}}$ into the bilinear form $J_{\mathrm{h},j}^M(F,G)$. In region $A$, we apply the non-stationary phase estimate \eqref{equ:kernel_high_nonsp} to obtain 
\begin{equation}\label{equ:proof_region_A_bound}
\begin{split}
&\iint_{\{|h-b| > |h+z|\}} \int_{s=0}^\infty \int_{u = M}^\infty   \mathbbm{1}_{A}(u,z) |K_{\mathrm{h},j}(u,z) w(h-b) F(s,h-b)\\
&\hspace{16em}\times w(h+z) \overline{G(s+u,h+z)}| \,\ud u \,\ud s\,\ud h \,\ud  z\\
&\lesssim_N \iint_{\{|h-b| > |h+z|\}} \int_{s=0}^\infty \int_{u = M}^\infty   2^j(2^{2j}u)^{-N-1}|w(h-b) F(s,h-b)\\
&\hspace{16em}\times w(h+z) \overline{G(s+u,h+z)}| \,\ud u \,\ud s\,\ud h \,\ud  z\\
&\lesssim 2^{-jN} M^{-N} \|F\|_{L_t^2 L_x^2} \|G\|_{L_t^2 L_x^2}.
\end{split}
\end{equation}
In the complement region $A^{\mathrm{c}}$, we note that 
\begin{equation*}
|b| \ll |u| \leq  2^{j}|u| \simeq |z| \leq |h-b| + |h+z| + |b|,
\end{equation*}
for all $j \geq 1$. Since $|h+z|<|h-b|$ in the domain of integration of $J_{\mathrm{h},j}^M(F,G)$ and since $|b|<|u|$ we find that 
\begin{equation}\label{equ:proof_h-b_bound}
w(h-b) = \jap{h-b}^{-2} \lesssim 2^{-j} \jap{u}^{-2}.
\end{equation}
Hence, using the pointwise estimate \eqref{equ:kernel_high_disp}, \eqref{equ:proof_h-b_bound}, and the Cauchy-Schwarz inequality repeatedly, we obtain that 
\begin{equation}\label{equ:proof_region_Ac_bound}
\begin{split}
&\iint_{\{|h-b| > |h+z|\}} \int_{s=0}^\infty \int_{u = M}^\infty   \mathbbm{1}_{A^{\mathrm{c}}}(u,z) |K_{\mathrm{h},j}(u,z) w(h-b) F(s,h-b)\\
&\hspace{16em}\times w(h+z) \overline{G(s+u,h+z)}| \,\ud u \,\ud s\,\ud h \,\ud  z\\
&\lesssim \iint_{\{|h-b| > |h+z|\}} \int_{s=0}^\infty \int_{u = M}^\infty   {u}^{-\frac12} \mathbbm{1}_{A^{\mathrm{c}}}(u,z) \jap{h-b}^{-2}|F(s,h-b)|\\
&\hspace{16em}\times \jap{h+z}^{-2} |G(s+u,h+z)| \,\ud u \,\ud s\,\ud h \,\ud  z\\
&\lesssim \int_{s=0}^\infty \int_{u = M}^\infty {u}^{-\frac12} 2^{-j}  \jap{u}^{-2} \| \mathbbm{1}_{A^{\mathrm{c}}}(u,z) \jap{h+z}^{-2}\|_{L_h^2 L_z^2}\\
&\hspace{16em} \times  \| F(s,h-b)G(s+u,h+z)\|_{L_h^2L_z^2}  \,\ud u \,\ud  s\\
&\lesssim \int_{s=0}^\infty \int_{u = M}^\infty {u}^{-\frac12} 2^{-j}  \jap{u}^{-2} u^{\frac12} 2^{\frac{j}{2}} \| F(s,h-b)G(s+u,h+z)\|_{L_h^2L_z^2}  \,\ud u \,\ud  s\\
&\lesssim M^{-1} \|F\|_{L_t^2 L_x^2} \|G\|_{L_t^2 L_x^2}.
\end{split}
\end{equation}
The identity \eqref{equ:proof_identity_I_hj} combined with the estimates \eqref{equ:proof_region_A_bound} and \eqref{equ:proof_region_Ac_bound} concludes the asserted estimate \eqref{equ_bilinear_high_estimate}.

\end{proof}

In the proof of Lemma~\ref{lemma:dual_high_smoothing} we will use the following almost orthogonality property 
\begin{equation}\label{equ:proof_almost_orthogonality}
I_{\mathrm{h},j}^M(F,G) = I_{\mathrm{h},j}^M(Q_j F,Q_j G) 
\end{equation}
where 
\begin{equation}
Q_j g := \sum_{|k-j|\leq 5} \left(P_k g + w^{-1}[P_k,w]g\right),
\end{equation}
with $w(x) = \jxmn$. The commutator term is small in the following sense:
\begin{equation}\label{equ:proof_commutator}
\| w^{-1} [P_k,w] g \|_{L_x^2} \lesssim 2^{-k} \| g \|_{L_x^2}.
\end{equation}

Finally, we are now ready to establish the proof of Lemma~\ref{lemma:dual_high_smoothing}.
\begin{proof}[Proof of Lemma~\ref{lemma:dual_high_smoothing}]
We provide the proof of \eqref{equ:dual_high_smoothing_0t} and leave the analagous details for the proof of \eqref{equ:dual_high_smoothing_ti} to the reader. We insert the Littlewood-Paley decomposition on the left hand side of \eqref{equ:dual_high_smoothing_0t} and use the almost orthogonality property \eqref{equ:proof_almost_orthogonality} to obtain 
\begin{equation}\label{equ:proof_47_1}
\begin{split}
&\left| \int_0^\infty \int_0^t \left \langle e^{i(t-s)\px^2} \frakm_{\mathrm{h}}(D)^2 [\mathrm{T}_{\theta}(s,t)wF(s)], wG(t) \right \rangle \,\ud s \,\ud t \right|\\
&\lesssim \sum_{j=1}^\infty | I_{\mathrm{h},j}^0(F,G)| =\sum_{j=1}^\infty | I_{\mathrm{h},j}^0(Q_j F, Q_j G)|,
\end{split}
\end{equation}
where we note that the summation starts from $j=1$ due to the support of $\frakm_{\mathrm{h}}(\xi)$. For each $j \geq 1$, we insert the term \eqref{equ:bilinear_form_high} with $M \equiv 1$ and write 
\begin{equation}\label{equ:proof_47_2}
I_{\mathrm{h},j}^0(Q_j F, Q_j G) = \left(I_{\mathrm{h},j}^0(Q_j F, Q_j G) - I_{\mathrm{h},j}^{1}(Q_j F, Q_j G)\right) + I_{\mathrm{h},j}^{1}(Q_j F, Q_j G).
\end{equation}
Using the $L^2$ conservation for the Schr\"odinger propagator and the $L^2$ boundedness of $\frakm_{\mathrm{h}}(D)$, we have 
\begin{equation}\label{equ:proof_47_3}
\left| I_{\mathrm{h},j}^0(Q_j F, Q_j G) - I_{\mathrm{h},j}^{1}(Q_j F, Q_j G)\right| \lesssim \|  Q_j F \|_{L_t^2L_x^2} \|  Q_j G \|_{L_t^2L_x^2}.
\end{equation}
We also obtain from \eqref{equ_bilinear_high_estimate} that 
\begin{equation}\label{equ:proof_47_4}
\left| I_{\mathrm{h},j}^{1}(Q_j F, Q_j G)\right| \lesssim\|  Q_j F \|_{L_t^2L_x^2} \|  Q_j G \|_{L_t^2L_x^2}.
\end{equation}
Hence, from \eqref{equ:proof_47_2}, the estimates \eqref{equ:proof_commutator}, \eqref{equ:proof_47_3}, and \eqref{equ:proof_47_4}, we conclude that the right hand side of \eqref{equ:proof_47_1} is bounded by 
\begin{equation*}
\begin{split}
\sum_{j=1}^\infty | I_{\mathrm{h},j}^0(Q_j F, Q_j G)| &\lesssim \sum_{j=0}^\infty 2^{-j} \| F \|_{L_t^2 L_x^2} \| G \|_{L_t^2 L_x^2} + \sum_{j=0}^\infty \sum_{|k_1 - j | \leq 5} \sum_{|k_2 - j | \leq 5} \| P_{k_1} F \|_{L_t^2 L_x^2} \| P_{k_2} G \|_{L_t^2 L_x^2}\\
&\lesssim  \| F \|_{L_t^2 L_x^2} \| G \|_{L_t^2 L_x^2},
\end{split}
\end{equation*}
as desired. 
\end{proof}

\section{Setting up the Analysis} \label{sec:setting_up}
In this section we begin setting up the nonlinear analysis for the asymptotic stability of the solitary wave solutions \eqref{equ:intro_family_4parameter} for the focusing cubic Schr\"odinger equation \eqref{equ:cubic_NLS}. The local well-posedness result for  \eqref{equ:cubic_NLS} in Lemma~\ref{lem:setup_local_existence} suffices for our bootstrap purposes. In Subsection~\ref{subsec:modulation}, we apply standard modulation techniques to decompose the solution to \eqref{equ:cubic_NLS} into a modulated solitary wave and a radiation term. In Subsections \ref{subsec:profile_equations}--\ref{subsec:normalform} we prepare the evolution equation for the profile using the distorted Fourier theory established earlier, and we apply  a normal form transformation on the profile equation to renormalize the resonant quadratic terms. In the last Subsection~\ref{subsec:cubic_spectral_distributions} we provide a full description for the nonlinear spectral distribution of the cubic nonlinear terms.

\subsection{Local existence theory} We recall the following local well-posedness result for \eqref{equ:cubic_NLS} in the space $H_x^1 \cap L_x^{2,1}$ by Ginibre and Velo \cite{GV78}.
\begin{lemma}[Local existence] \label{lem:setup_local_existence}
For any $\psi_0 \in H^1_x(\bbR) \cap L^{2,1}_x(\bbR)$ there exists a unique solution $\psi \in \calC([0,T_\ast); H^1_x \cap L^{2,1}_x)$ to \eqref{equ:cubic_NLS} with initial condition $\psi(0) = \psi_0$ defined on a maximal interval of existence $[0,T_\ast)$ for some $T_\ast > 0$. Moreover, the following continuation criterion holds
\begin{equation} \label{equ:setup_continuation_criterion}
T_\ast < \infty \quad \Longrightarrow \quad \limsup_{t \nearrow T_\ast} \| \psi(t) \|_{H^1_x \cap L^{2,1}_x} = \infty.
\end{equation}
\end{lemma}

\subsection{Modulation and orbital stability}\label{subsec:modulation}
In this subsection, we perform standard modulation and orbital stability techniques to set up the perturbation equations. We consider small perturbations of the solitary waves	\eqref{equ:intro_family_4parameter} in the energy space, and we establish a decomposition of the corresponding solution to \eqref{equ:cubic_NLS} into a modulated solitary wave and a radiation term that is orthogonal to directions related to the invariances of \eqref{equ:cubic_NLS}. Moreover, we derive the modulation equations for the soliton parameters and the evolution equation for the radiation term.

\begin{proposition} \label{prop:modulation_and_orbital}
Let $\omega_0 \in (0, \infty)$. There exist constants $C_1 \geq 1$ and $0 < \varepsilon_1 \ll 1$ such that the following holds:
For any $(\gamma_0,p_0,\sigma_0) \in \bbR^3$ and any $u_0 \in H^1_x(\bbR) \cap L^{2,1}_x(\bbR)$ with $\|u_0\|_{H^1_x} \leq \varepsilon_1$, denote by $\psi(t,x)$ the $H^1_x \cap L^{2,1}_x$--solution to \eqref{equ:cubic_NLS} with initial condition 
\begin{equation} \label{equ:setup_modulation_orbital_initial_datum}
\psi_0(x) = e^{ip_0 (x-\sigma_0)} e^{i \gamma_0} \bigl( \phi_{\omega_0}(x-\sigma_0) + u_0(x-\sigma_0) \bigr)
\end{equation}
on its maximal interval of existence $[0, T_\ast)$ furnished by Lemma~\ref{lem:setup_local_existence}.
Then there exist unique continuously differentiable paths $(\omega, \gamma,p,\sigma) \colon [0,T_\ast) \to (0,\infty) \times \bbR^3$ so that the solution $\psi(t,x)$ to \eqref{equ:cubic_NLS} can be decomposed as
\begin{equation} \label{equ:setup_modulation_prop_decomposition}
\psi(t,x) = e^{ip(t)(x-\sigma(t))}e^{i \gamma(t)} \bigl( \phi_{\omega(t)}(x-\sigma(t)) + u(t,x-\sigma(t)) \bigr), \quad 0 \leq t < T_\ast,
\end{equation}
and the following properties hold:
\begin{itemize}[leftmargin=1.8em]
\item[(1)] Evolution equation\footnote{By direct inspection we find that the equation \eqref{equ:setup_perturbation_equ} is $\calJ$-invariant. Since $U(0) := \big(u(0),\baru(0)\big)^\top$ is $\calJ$-invariant, this ensures that  $U(t)$ is $\calJ$-invariant for all $t \geq 0$.} for $u(t,x-\sigma(t))$: 
\begin{equation} \label{equ:setup_perturbation_equ}
i\partial_t U - \calH(\omega)U = \calM_1 + \calM_2 + \calN(U), \quad U := \begin{bmatrix} u \\ \overline{u} \end{bmatrix},
\end{equation}
with
\begin{align}
\calH(\omega) &:= \calH_0(\omega) + \calV(\omega) :=  \begin{bmatrix}
	-\partial_x^2 + \omega & 0 \\ 0 & \partial_x^2 - \omega
\end{bmatrix}
+
\begin{bmatrix}
	-2\phi_\omega^2	 & - \phi_\omega^2 \\ \phi_\omega^2 & 2\phi_\omega^2
\end{bmatrix}, \\
\calM_1 &:= \dot{p}(x-\sigma) \sigma_3 U + i(\dot{\sigma}-2p)\partial_x U + (\dot{\gamma}+p^2-\omega - p\dot{\sigma})\sigma_3 U  \label{equ:setup_definition_calM1}\\
\calM_2  &:= -i(\dot{\gamma}+p^2-\omega - p\dot{\sigma}) Y_{1, \omega} - i \dot{\omega} Y_{2, \omega} + i(\dot{\sigma}-2p)Y_{3,\omega}-i\dot{p}Y_{4,\omega}, \label{equ:setup_definition_calM2} \\
\calN(U) &:= \calQ_\omega(U) + \calC(U) := \begin{bmatrix}
	-\phi_\omega(u^2 + 2u \bar{u}) \\ \phi_\omega(\baru^2 + 2u \baru)
\end{bmatrix}
+
\begin{bmatrix}
	-u \baru u \\ \baru u \baru
\end{bmatrix}. \label{equ:setup_definition_calN}
\end{align}

\item[(2)] Modulation equations%\footnote{The matrix $\mathbb{M}$ and the components of the vector on the right-hand side of \eqref{equ:setup_modulation_equation} are real-valued, because $c_\omega \in \bbR$ and because all involved inner products are between two $\calJ$-invariant vectors, and are therefore real-valued.}:
\begin{equation} \label{equ:setup_modulation_equation}
\mathbb{M} \begin{bmatrix}
	\dot{\gamma}+p^2-\omega - p\dot{\sigma} \\
	\dot \omega\\
	\dot{\sigma}-2p\\
	\dot{p}
\end{bmatrix}
= \begin{bmatrix}
	\langle i \calN(U), \sigma_2 Y_{1,\omega} \rangle \\
	\langle i \calN(U), \sigma_2 Y_{2, \omega} \rangle\\
	\langle i \calN(U), \sigma_2 Y_{3,\omega} \rangle \\
	\langle i \calN(U), \sigma_2 Y_{4,\omega} \rangle 					
\end{bmatrix}
\end{equation}
with
\begin{equation}\label{equ:setup_matrix_modulation_equation}
\bbM := \bbM_1(\omega) + \bbM_2(U,\omega),
\end{equation}
where
\begin{equation}\label{equ:setup_matrix_M1}
\bbM_1(\omega) := \begin{bmatrix}
	0 & c_\omega & 0 & 0  \\
	c_\omega & 0 &0 & 0 \\
	0 & 0 & 0 & - \|\phi_\omega\|_{L^2}^2\\
	0 & 0 & \|\phi_\omega\|_{L^2}^2 & 0					
\end{bmatrix} = \begin{bmatrix}
	0 & 2\omega^{-\frac12} & 0 & 0  \\
	2\omega^{-\frac12} & 0 &0 & 0 \\
	0 & 0 & 0 & - 4\omega^{\frac12}\\
	0 & 0 & 4\omega^{\frac12} & 0				
\end{bmatrix},
\end{equation}
and\footnote{The functions $Y_{1,\omega},\ldots,Y_{4,\omega}$ in the terms  $\calM_2$ and $\bbM_2(U,\omega)$ are evaluated at $y := x-\sigma(t)$.}	
\begin{equation} \label{equ:setup_matrix_M2}
\bbM_2(U,\omega):=
\begin{bmatrix}
	\langle U, \sigma_1Y_{1, \omega} \rangle &  \langle U,  \sigma_2 \partial_\omega Y_{1, \omega} \rangle &  \langle U,  -\sigma_2\py  Y_{1, \omega} \rangle &  \langle U,  y\sigma_1 Y_{1, \omega} \rangle \\
	\langle U, \sigma_1 Y_{2, \omega} \rangle &  \langle U, \sigma_2 \partial_\omega Y_{2, \omega} \rangle &  \langle U,  -\sigma_2\py  Y_{2, \omega} \rangle &  \langle U,  y\sigma_1 Y_{2, \omega} \rangle\\
	\langle U, \sigma_1 Y_{3, \omega} \rangle &  \langle U, \sigma_2 \partial_\omega Y_{3, \omega} \rangle&  \langle U,  -\sigma_2\py  Y_{3, \omega} \rangle &  \langle U,  y\sigma_1 Y_{3, \omega} \rangle\\
	\langle U, \sigma_1 Y_{4, \omega} \rangle &  \langle U, \sigma_2 \partial_\omega Y_{4, \omega} \rangle&  \langle U,  -\sigma_2\py  Y_{4, \omega} \rangle &  \langle U,  y\sigma_1 Y_{4, \omega} \rangle\\										
\end{bmatrix}.
\end{equation}

\item[(3)] Orthogonality: for $j\in \{1,\ldots,4\}$
\begin{equation} \label{equ:setup_orthogonality_radiation}
\langle U(t), \sigma_2 Y_{j, \omega(t)} \rangle = 0, \quad 0 \leq t < T_\ast.
\end{equation}

\item[(4)] Stability:
\begin{equation} \label{equ:setup_smallness_orbital}
\sup_{0 \leq t < T_\ast} \, \bigl( \|u(t)\|_{H^1_x} + |\omega(t) - \omega_0| + |p(t)-p_0| \bigr) \leq C_1 \|u_0\|_{H^1_x}.
\end{equation}

\item[(5)] Comparison estimate:
\begin{equation} \label{equ:setup_comparison_estimate}
|p(t)-p_0|\leq \frac{1}{2}\omega_0, \quad\frac12 \omega_0 \leq \omega(t) \leq 2 \omega_0, \quad 0 \leq t < T_{\ast}.
\end{equation}
\end{itemize}
\end{proposition}

\begin{remark}
As a consistency check, the case when $U\equiv 0$ should correspond to a transformed solitary wave. Indeed, from the modulation equations \eqref{equ:setup_modulation_equation} we obtain the Hamiltonian system
\begin{equation*}
	\begin{split}
\dot{\omega}=0, \quad \dot{p}=0, \quad 	\dot{\sigma}=2p,\quad  \dot{\gamma} = p^2+\omega, \\
\big(\omega(0),\gamma(0),p(0),\sigma(0)\big) = \big(\omega_0,\gamma_0,p_0,\sigma_0\big).
	\end{split}
\end{equation*}	
The system can be solved exactly and it gives a transformed solitary wave \eqref{equ:intro_family_4parameter}.
	\end{remark}
Although the proof of Proposition~\ref{prop:modulation_and_orbital} is standard, we provide the details below. We first record some identities related to the conservation laws \eqref{equ:conservation_laws}. 
\begin{lemma}
For any $f\in H^1(\bbR)$ and any $p \in \bbR$, we have
\begin{align}
M[e^{ipx}f] &= M[f],	\label{equ:setup_lemma_mass_shift}\\
P[e^{ipx}f] &=P[f] + pM[f], \label{equ:setup_lemma_momentum_shift}\\
E[e^{ipx}f] &=E[f]+2p P[f] + p^2M[f]. \label{equ:setup_lemma_energy_shift}
\end{align}
\end{lemma}
\begin{proof}
These identities follow by direct computation.
\end{proof}
In the next lemma, we use the following expansions of the conservation laws around a solitary wave.

\begin{lemma} \label{lem:setup_expansions_conservation_laws}
Let $\omega \in (0, \infty)$ and $u \in H^1_x(\bbR)$. Set
\begin{equation*}
u_1 := \Re(u), \quad u_2 := \Im(u), \quad U := \begin{bmatrix} u \\ \baru \end{bmatrix}.
\end{equation*}
We have the following identities.
\begin{itemize}[leftmargin=1.8em]
\item[(a)] Expansion of the mass around a solitary wave:
\begin{equation} \label{equ:setup_expansion_mass}
M[\phi_\omega + u] - M[\phi_\omega] = M[u] - \frac12 \langle U, \sigma_2 Y_{1,\omega} \rangle.
\end{equation}
\item[(b)] Expansion of the momentum around a solitary wave:
\begin{equation}\label{equ:setup_expansion_momentum}
P[\phi_\omega+u] = \frac12\langle U,\sigma_2 Y_{3,\omega}\rangle + P[u]
\end{equation}
\item[(c)] Expansion of the energy around a solitary wave:
\begin{equation} \label{equ:setup_expansion_energy1}
\begin{aligned}
	E[\phi_\omega + u] - E[\phi_\omega] &= \frac12 \langle (-\px^2 - 3\phi_\omega^2) u_1, u_1 \rangle + \frac12 \langle (-\px^2 - \phi_\omega^2) u_2, u_2 \rangle + \frac{\omega}{2} \langle U, \sigma_2 Y_{1,\omega} \rangle \\ 
	&\quad - \frac12 \int_\bbR \phi_\omega \bigl( u^2 \baru + u \baru^2 \bigr) \, \ud x - \frac14 \int_\bbR |u|^4 \, \ud x
\end{aligned}
\end{equation}
as well as 
\begin{equation} \label{equ:setup_expansion_energy2}
\begin{aligned}
	E[\phi_\omega + u] - E[\phi_\omega] &= E[u] + \frac{\omega}{2} \langle U, \sigma_2 Y_{1,\omega} \rangle 
	- \frac14 \int_\bbR \phi_\omega^2 \bigl( u^2 + 4 u \baru + \baru^2 \bigr) \, \ud x \\
	&\quad - \frac12 \int_\bbR \phi_\omega \bigl( u^2 \baru + u \baru^2 \bigr) \, \ud x.
\end{aligned}
\end{equation}        
\end{itemize}
\end{lemma}
\begin{proof}
The identities follow by direct computation. For \eqref{equ:setup_expansion_momentum}, note that $P[\phi_\omega]=0$  by evenness.
\end{proof}

\begin{proof}[Proof of Proposition~\ref{prop:modulation_and_orbital}]
We denote by $\psi(t,x)$ the (local) $H^1_x \cap L^{2,1}_x$-solution to \eqref{equ:cubic_NLS} with initial condition \eqref{equ:setup_modulation_orbital_initial_datum} on its maximal interval of existence $[0,T_\ast)$ furnished by Lemma~\ref{lem:setup_local_existence}.

\medskip 

\noindent {\bf Step 1.} (Choice of parameters at $t=0$)
We show that there exist constants $C_1 \equiv C_1(\omega_0) \geq 1$ and $\kappa_1 \equiv \kappa_1(\omega_0)$ with $0 < \kappa_1 \ll 1$ such that if $\|u_0\|_{H^1_x} \leq \kappa_1$, then there exist $\big(\omega(0), \gamma(0), p(0), \sigma(0)\big) \in (0,\infty) \times \bbR^3$ satisfying 
\begin{equation} \label{equ:setup_modulation_proposition_proof_step1}
|\omega(0)-\omega_0| + |\gamma(0)-\gamma_0| + |p(0)-p_0| + |\sigma(0)-\sigma_0| \leq C_1 \|u_0\|_{H^1_x}
\end{equation}
as well as
\begin{equation} \label{equ:setup_modulation_proposition_proof_step1.b}
\langle U(0), \sigma_2 Y_{j,\omega(0)} \rangle= 0, 
\quad j \in \{1,\ldots,4\},
\end{equation}
with 
\begin{equation*}
U(0,x) 
= \begin{bmatrix} u(0,x) \\ \overline{u(0,x)} \end{bmatrix}, \quad u(0,x) := e^{-ip(0)(x)}e^{-i\gamma(0)} \psi_0(x+\sigma(0)) - \phi_{\omega(0)}(x).
\end{equation*}
From this construction we have $\|u(0,\cdot)\|_{H^1_x} \lesssim \|u_0\|_{H^1_x}$. Next,  consider the functional
\begin{equation}
\begin{split}
&\calK: H_x^1(\bbR) \times (0,\infty) \times \bbR^3 \longrightarrow \bbR^4\\
&\calK[\psi(\cdot),\omega,\gamma,p,\sigma] = \begin{bmatrix}
	\left \langle \begin{pmatrix}
		e^{-ip(\cdot-\sigma)}e^{-i\gamma}\psi(\cdot) - \phi_\omega(\cdot-\sigma)\\
		\overline{e^{-ip(\cdot-\sigma)}e^{-i\gamma}\psi(\cdot)} - \phi_\omega(\cdot-\sigma)
	\end{pmatrix},\sigma_2 Y_{j,\omega}(\cdot-\sigma)  \right \rangle
\end{bmatrix}_{j=1,\ldots,4}.
\end{split}
\end{equation}
By direct computation we find that 
\begin{equation*}
\calK[e^{ip_0(\cdot-\sigma_0)}e^{i\gamma_0}\phi_{\omega_0}(\cdot-\sigma_0),\omega_0,\gamma_0,p_0,\sigma_0]= \bf0,
\end{equation*}
and that the Jacobian 
\begin{equation*}
\frac{\partial \calK}{\partial(\omega, \gamma,p,\sigma)}\bigl[e^{ip_0(\cdot-\sigma_0)}e^{i\gamma_0}\phi_{\omega_0}(\cdot-\sigma_0),\omega_0,\gamma_0,p_0,\sigma_0\bigr] = \begin{bmatrix} c_{\omega_0} & 0 & 0 & 0 \\ 0 & c_{\omega_0} & 0 & p_0c_{\omega_0} \\ 0 & 0 & -\|\phi_{\omega_0}\|_{L^2}^2 & 0 \\ 0 & 0 & 0 & \|\phi_{\omega_0}\|_{L^2}^2\end{bmatrix}
\end{equation*}
is invertible. Moreover, an analogous computation shows that the Jacobian with respect to the modulation parameters is uniformly non-degenerate in a neighborhood of $(e^{ip_0(\cdot-\sigma_0)}e^{i\gamma_0}\phi_{\omega_0}(\cdot-\sigma_0),\omega_0,\gamma_0,p_0,\sigma_0)$. The claims \eqref{equ:setup_modulation_proposition_proof_step1}--\eqref{equ:setup_modulation_proposition_proof_step1.b} now follow from a quantitative version of the implicit function theorem, see for instance \cite[Remark 3.2]{Jendrej15}.

\medskip 

\noindent {\bf Step 2.} (Derivation of perturbation equation and modulation equations)
Next, given continuously differentiable paths $(\omega, \gamma,p,\sigma) \colon [0,T_\ast) \to (0,\infty) \times \bbR^3$ and a decomposition of $\psi(t,x)$ into
\begin{equation} \label{equ:setup_modulation_proposition_proof_decomposition}
\psi(t,x) = e^{ip(t)(x-\sigma(t))} e^{i\gamma(t)} \bigl( \phi_{\omega(t)}(x-\sigma(t)) + u(t,x-\sigma(t)) \bigr), \quad 0 \leq t < T_\ast,
\end{equation}
we formally compute, using the equations \eqref{equ:cubic_NLS} and \eqref{equ:intro_phi_omega_equation}, that the perturbation $u(t,x-\sigma(t))$ satisfies
\begin{equation}\label{equ: scalar_equation_u}
\begin{split}
&i\partial_t u+ \partial_x^2 u - \omega u +\phi_\omega^2(2u + \baru) \\
&= \dot{p}(x-\sigma) (\phi_\omega+u) + i(\dot{\sigma}-2p)(\partial_x \phi_\omega + \partial_x u)+ (\dot{\gamma}+p^2-\omega-p\dot{\sigma})(\phi_\omega + u) - i \dot{\omega}\partial_\omega \phi_\omega \\
&\quad - \phi_\omega(u^2 + 2u \baru) - \vert u \vert^2 u.	
\end{split}
\end{equation}
Hence, we obtain the evolution equation \eqref{equ:setup_perturbation_equ} for the vectorial perturbation $U(t) = \big(u(t),\baru(t)\big)^\top$.
By formally differentiating the orthogonality conditions 
\begin{equation} \label{equ:setup_modulation_proposition_proof_orthogonality}
\langle U(t), \sigma_2 Y_{j,\omega(t)} \rangle = 0, \qquad j \in \{1,\ldots,4\}
\end{equation}
and inserting the evolution equation \eqref{equ:setup_perturbation_equ}, we find  for $0 \leq t < T_\ast$ and for each $j \in \{1,\ldots,4\}$ that
\begin{equation}\label{equ:proof_modulation_setup1}
0=-i \big\langle \big(\calH(\omega)U + \calM_1 + \calM_2  + \calN(U)\big), \sigma_2 Y_{j,\omega}\big \rangle + \dot{\omega} \langle U,\sigma_2 \partial_\omega Y_{j,\omega}\rangle.
\end{equation}
We proceed to simplify each term in \eqref{equ:proof_modulation_setup1} to obtain the modulation equations \eqref{equ:setup_modulation_equation}. By $\sigma_2 \calH(\omega) = - \calH(\omega)^*\sigma_2$, the relations \eqref{eqn:eig-eqn for nullspace of calH-omega}, and \eqref{equ:setup_modulation_proposition_proof_orthogonality}, we have
\begin{equation*}
\langle \calH(\omega)U,\sigma_2 Y_{j,\omega}\rangle  = 0, \quad j \in \{1,\ldots,4\}.
\end{equation*}
By using the inner product relations \eqref{equ:innerproductrelations1}--\eqref{equ:innerproductrelations2}, we may further simplify the terms $-i\langle \calM_1 + \calM_2,\sigma_2 Y_{j,\omega}\rangle + \dot{\omega} \langle U,\sigma_2 \partial_\omega Y_{j,\omega}\rangle$ to obtain the $j$-th row of the left-hand side of \eqref{equ:setup_modulation_equation}. Thus, we obtain the system of modulation equations\eqref{equ:setup_modulation_equation}.

\medskip 

\noindent {\bf Step 3.} (ODE system for the modulation parameters) 
Given the $H^1_x \cap L^{2,1}_x$-solution $\psi(t,x)$ to \eqref{equ:cubic_NLS} with initial condition \eqref{equ:setup_modulation_orbital_initial_datum} on its maximal interval of existence $[0,T_\ast)$, we write 
$$u(t,x-\sigma(t)) = e^{-ip(t)(x-\sigma(t))}e^{-i\gamma(t)}\psi(t,x) - \phi_{\omega(t)}(x-\sigma(t))$$
and view the modulation equations \eqref{equ:setup_modulation_equation} as a system of first-order nonlinear ODEs for the modulation parameters $(\omega(t), \gamma(t),p(t),\sigma(t))$ on $[0,T_\ast)$ with initial conditions $(\omega(0),\gamma(0),p(0),\sigma(0))$  determined in Step~1. If $|\omega(t) - \omega(0)| + \|u(t)\|_{H^1_x} \ll 1$ are sufficiently small (depending on the size of $\omega_0$), the matrix $\bbM$ in \eqref{equ:setup_matrix_modulation_equation} is invertible and we obtain the following non-autonomous systems of differential equations for the modulation parameters,
\begin{equation}\label{equ:proof_mod_equations1}
 \begin{bmatrix}
	\dot{\gamma}+p^2-\omega - p\dot{\sigma} \\
	\dot \omega\\
	\dot{\sigma}-2p\\
	\dot{p}
\end{bmatrix}
= \mathbb{M}^{-1}\begin{bmatrix}
	\langle i \calN(U), \sigma_2 Y_{1,\omega} \rangle \\
	\langle i \calN(U), \sigma_2 Y_{2, \omega} \rangle\\
	\langle i \calN(U), \sigma_2 Y_{3,\omega} \rangle \\
	\langle i \calN(U), \sigma_2 Y_{4,\omega} \rangle 					
\end{bmatrix}.
\end{equation}
Note that the right-hand side of the ODE system \eqref{equ:proof_mod_equations1} depends Lipschitz continuously on the modulation parameters, and continuously in time. Thus, we obtain by the Picard-Lindel\"of theorem a local-in-time continuously differentiable solution $(\omega(t), \gamma(t),p(t),\sigma(t))$ to this system defined on some time interval $[0,T_0]$ with $0 < T_0 < T_\ast$. As long as $|\omega(t)-\omega(0)| + \|u(t)\|_{H^1_x} \ll 1$, this solution can be continued to the entire time interval $[0,T_\ast)$. In the next step, we achieve this stability control by expanding the conservation laws around the solitary wave from Lemma~\ref{lem:setup_expansions_conservation_laws} and by invoking the orthogonality conditions \eqref{equ:setup_modulation_proposition_proof_orthogonality}. 
\medskip 

\noindent {\bf Step 4.} (Orbital stability and conclusion) 
Following the general framework of  \cite{CuccagnaMaeda25} and \cite{MartelMerleTsai06} we consider the functional 
\begin{equation}\label{eqn:modulation_functional}
\bfH[\psi] := E[\psi] + \big(\omega(t) +p(0)^2\big)M[\psi] - 2p(0)P[\psi].
\end{equation}
For a fixed $\omega \in (0,\infty)$ we will also need  the function
\begin{equation}
g(\omega)  := E[\phi_\omega] + \omega M[\phi_\omega] \equiv \frac43 \omega^{\frac32},
\end{equation}
which is strictly increasing and convex on $(0,\infty)$ due to
\begin{equation*}
g'(\omega)=M[\phi_{\omega}] = 2\sqrt{\omega}>0,\quad g''(\omega) = c_{\omega}>0.		
\end{equation*}

We insert the decomposition \eqref{equ:setup_modulation_proposition_proof_decomposition} for the solution $\psi(t)$ into the functional \eqref{eqn:modulation_functional} and we use the invariance under phase rotation and translation with the identities \eqref{equ:setup_lemma_mass_shift}, \eqref{equ:setup_lemma_momentum_shift}, \eqref{equ:setup_lemma_energy_shift}  to compute that 
\begin{equation*}
\begin{split}
&\bfH[\psi(t)] =E[\phi_{\omega(t)}+u(t)] + \big(\omega(t) + \big(p(t)-p(0)\big)^2\big) M[\phi_{\omega(t)}+u(t)] + 2 \big(p(t)-p(0)\big)P[\phi_{\omega(t)}+u(t)]	.	
\end{split}
\end{equation*}
On the one hand, using the expansion identities \eqref{equ:setup_expansion_mass}, \eqref{equ:setup_expansion_momentum}, \eqref{equ:setup_expansion_energy1} and the orthogonality conditions \eqref{equ:setup_modulation_proposition_proof_orthogonality}, we obtain
\begin{equation*}
\begin{split}
\bfH[\psi(t)]&= g(\omega(t)) + \frac12 \langle L_{+,\omega(t)} u_1(t),u_1(t)\rangle + \frac12 \langle L_{-,\omega(t)}u_2(t),u_2(t)\rangle +\big(p(t)-p(0)\big)^2 g'(\omega(0))	\\
&\quad + 2 \big(p(t)-p(0)\big)P[u(t)]+ \big(p(t)-p(0)\big)^2\big(M[\phi_{\omega(t)}]-M[\phi_{\omega(0)}]+M[u(t)]\big)\\
&\quad - \frac12 \int_\bbR \phi_{\omega(t)}\big(u(t)^2\baru(t)+u(t)\baru(t)^2\big)\,\ud x - \frac12 \int_\bbR |u(t)|^4 \,\ud x
\end{split}
\end{equation*}
where $(u_1,u_2) := (\Re(u),\Im(u))$ and 
\begin{equation}\label{equ:operators_Lplusminus}
L_{+, \omega} := -\px^2 + \omega - 3 \phi_\omega^2, \quad L_{-,\omega} := -\px^2 + \omega - \phi_\omega^2, \quad \omega \in (0,\infty).
\end{equation}
On the other hand, by  \eqref{equ:setup_expansion_energy2} and conservation of mass and energy for solutions to \eqref{equ:cubic_NLS} at $t=0$, we infer from  the orthogonality conditions \eqref{equ:setup_modulation_proposition_proof_orthogonality} that
\begin{equation*}
\begin{split}
\bfH[\psi(t)]&= g(\omega(0))+g'(\omega(0)) \big(\omega(t)-\omega(0)\big) + \omega(t)M[u(0)] + E[u(0)] \\
&\quad - \frac14 \int_\bbR \phi_\omega^2 \bigl( u^2 + 4 u \baru + \baru^2 \bigr)\big|_{t=0} \, \ud x - \frac12 \int_\bbR \phi_\omega \bigl( u^2 \baru + u \baru^2 \bigr)\big|_{t=0} \, \ud x.
\end{split}
\end{equation*}
From the orthogonality conditions \eqref{equ:setup_modulation_proposition_proof_orthogonality} we have by \cite[Theorem 2.5]{Weinstein85} the coercivity estimate
\begin{equation*}
\langle L_{+, \omega(t)} u_1(t), u_1(t) \rangle + \langle L_{-,\omega(t)} u_2(t), u_2(t) \rangle \gtrsim \|u(t)\|_{H^1_x}^2,
\end{equation*}
and from the strict convexity of $g(\omega)$ we obtain the inequality 
\begin{equation*}
g(\omega(t))-g(\omega(0))-g'(\omega(0))\big(\omega(t)-\omega(0)\big) = \frac{c_{\omega(0)}}{2}\big(\omega(t)-\omega(0)\big)^2 + \calO\big(\big(\omega(t)-\omega(0)\big)^3\big).
\end{equation*}
By combining the identities for $\bfH[\psi(t)]$ and  using the preceeding estimates, we  deduce that there is a constant $C$ depending only on $\omega(0)$ such that 

\begin{equation*}
\begin{aligned}
&\|u(t)\|_{H^1_x}^2 + \bigl( \omega(t) - \omega(0) \bigr)^2 +\bigl( p(t) - p(0) \bigr)^2\\
&\leq C \Bigl(|p(t)-p(0)| \|u(t)\|_{H_x^1}^2 + |p(t)-p(0)|^2 \big(|\omega(t)-\omega(0)|+\|u(t)\|_{L_x^2}^2\big) + \|u(t)\|_{H^1_x}^3 + \|u(t)\|_{H^1_x}^4 \\ 
& \qquad+ |\omega(t)-\omega(0)| \|u(0)\|_{L^2_x}^2 + |\omega(t)-\omega(0)|^3  + \|u(0)\|_{H^1_x}^2 + \|u(0)\|_{H^1_x}^4 + |\omega(0)| \|u(0)\|_{L^2_x}^2 \Bigr).
\end{aligned}
\end{equation*}
From \eqref{equ:setup_modulation_proposition_proof_step1} we can bound the constant $C$ by another larger absolute constant $C'$ which depends only on $\omega_0$ and $\|u_0\|_{H_x^1}$. Hence, for sufficiently small $\|u(0)\|_{H^1_x} \lesssim \|u_0\|_{H^1_x} \ll 1$ depending only on the size of $\omega_0$  a standard continuity argument then gives the desired bound on the time interval $[0,T_0]$,
\begin{equation} \label{equ:setup_modulation_proposition_proof_orbital_stability_bound}
\|u(t)\|_{H^1_x} + |\omega(t)-\omega(0)| + |p(t)-p(0)| \lesssim \|u(0)\|_{H^1_x} \lesssim \|u_0\|_{H^1_x}, \quad 0 \leq t \leq T_0.
\end{equation}
By \eqref{equ:setup_modulation_proposition_proof_orbital_stability_bound} we can then continue the modulation parameters $\omega(t), \gamma(t), p(t),\sigma(t)$ to the entire interval $[0,T_\ast)$.
Finally, we obtain the comparison estimate \eqref{equ:setup_comparison_estimate} in the statement of Proposition~\ref{prop:modulation_and_orbital} from \eqref{equ:setup_modulation_proposition_proof_step1} and from \eqref{equ:setup_modulation_proposition_proof_orbital_stability_bound} by choosing $\eps_1=\eps_1(\omega_0)$ sufficiently small so that $\|u_0\|_{H^1_x}\leq \eps_1 ~\ll~1$. This finishes the proof of the proposition.
\end{proof}	

\subsection{Evolution equation for the profile}\label{subsec:profile_equations}
From now we fix the moving frame coordinate
\begin{equation}
y := x - \sigma(t),
\end{equation}
and we treat the radiation variable $U$ as a function of $t$ and $y$. From \eqref{equ:setup_perturbation_equ} the evolution equation for the radiation term in the moving frame coordinate reads
\begin{equation}\label{equ: evol1}
i\partial_t U - \calH(\omega)U = \dot{p}y \sigma_3 U + i(\dot{\sigma}-2p)\py U + (\dot{\gamma}+p^2-\omega-p \dot{\sigma})\sigma_3 U + \calM_2 + \calN(U).
\end{equation}

In order to analyze the dispersive behavior of the \eqref{equ: evol1}, we need to freeze  suitably chosen time-independent parameters $(\ulomega,\ulp)\in (0,\infty) \times \bbR$. Following the renormalization idea in \cite{CollotGermain23}, we work with the renormalized radiation variable 
\begin{equation}\label{equ: V-definition}
\Vp(t,y) := e^{i(p(t)-\ulp)y \sigma_3}U(t,y) \equiv \begin{bmatrix}
e^{i(p(t)-\ulp)y}u(t,y)\\
e^{-i(p(t)-\ulp)y}\baru(t,y)
\end{bmatrix} =: \begin{bmatrix}
\vp(t,y) \\ \barvp(t,y)
\end{bmatrix}.
\end{equation}
This transformation helps to gauge away the problematic term $\dot{p}y\sigma_3 U$ in \eqref{equ: evol1} while ensuring that the norms for $\Vp(t)$ and $U(t)$ are comparable. Moreover, the new variable $\Vp(t)$ remains $\calJ$-invariant. We then fix the reference operator $\calH(\ulomega)$ and use \eqref{equ: evol1} to compute that the renormalized variable $\Vp(t,y)$ satisfies the evolution equation
\begin{equation}\label{equ: V-equation}
i\pt \Vp - \calH(\ulomega)\Vp = \dot{\theta}_1(t) \sigma_3 \Vp + i\dot{\theta}_2(t) \py \Vp + \calQ_{\ulomega}(\Vp) + \calC(\Vp) + \calMod + \calE_1 + \calE_2 , \quad \Vp := \begin{bmatrix}
\vp \\ \barvp
\end{bmatrix},
\end{equation}
with the nonlinear and modulation terms 
\begin{align}
\calQ_{\ulomega}(\Vp) &:= \begin{bmatrix}
-\phi_{\ulomega}(\vp^2 + 2 \vp \barvp )\\
\phi_{\ulomega}(\barvp^2 + 2\vp  \barvp)
\end{bmatrix}, \quad \calC(\Vp) := \begin{bmatrix}
- \vp \barvp \vp\\
\barvp \vp \barvp
\end{bmatrix},\label{equ:setup_definition_nonl}\\
\calMod &:= e^{i(p-\ulp)y \sigma_3} \calM_2,\label{equ:setup_definition_Mod}
\end{align}
and time-dependent coefficients
\begin{align}
\dot{\theta}_1(t) &:= \big(\omega - \ulomega + (\dot{\sigma}-2\ulp)(p-\ulp) -(p-\ulp)^2 + (\dot{\gamma}+p^2-\omega-p\dot{\sigma})\big),\label{equ:def-dot-theta1}\\
\dot{\theta}_2(t) &:= \dot{\sigma}-2p + 2(p-\ulp).\label{equ:def-dot-theta2}
\end{align}
When passing to the fixed parameters $(\ulomega,\ulp)$, we produced the error terms in \eqref{equ: V-equation} which are defined by 
\begin{align}
\calE_1 &:= \calV(\omega) \Vp - \calV(\ulomega) \Vp = \begin{bmatrix}
-(\phi_{\omega}^2-\phi_{\ulomega}^2)(2 \vp + \barvp)\\
(\phi_{\omega}^2-\phi_{\ulomega}^2)(2 \barvp + \vp)
\end{bmatrix}+ \begin{bmatrix}
\phi_{\omega}^2(1-e^{2i(p-\ulp)y}) \barvp \\
-\phi_{\omega}^2(1-e^{-2i(p-\ulp)y}) \vp \\
\end{bmatrix}, \label{equ:def-calE1} \\
\calE_2 &:= \calQ_{\omega}(\Vp)-\calQ_{\ulomega}(\Vp) =  \begin{bmatrix}
-(\phi_\omega -	\phi_{\ulomega})( \vp^2 + 2 \vp \barvp)\\
(\phi_\omega -	\phi_{\ulomega})(\barvp^2 + 2 \vp \barvp)
\end{bmatrix} \label{equ:def-calE2}\\
&\hspace{10em}+ \begin{bmatrix}
\phi_{\omega}(1-e^{-i(p-\ulp)y}) \vp^2 + 2\phi_{\omega}(1-e^{i(p-\ulp)y}) \vp \barvp\\
-\phi_{\omega}(1-e^{i(p-\ulp)y})\barvp^2 - 2\phi_{\omega}(1-e^{-i(p-\ulp)y}) \vp \barvp
\end{bmatrix} \nonumber.
\end{align}
We also record the following identities that will be needed for the proof of the theorem
\begin{equation}\label{equ:theta_identities} 
	\dot{\theta}_1 = \dot{\gamma}-\ulp^2 - \ulomega + \ulp \dot{\theta}_2, \quad \dot{\theta}_2 = \dot{\sigma}-2\ulp.
\end{equation}

Next we perform a spectral decomposition of $\Vp(t,y)$ relative to the spectrum of the operator $\calH(\ulomega)$:
\begin{equation}
\Vp(t) = \ulPe \Vp(t) + \ulPd \Vp(t) 
\end{equation}
where the discrete component of $\Vp(t)$ is given by 
\begin{equation}
\ulPd \Vp(t,y) = \sum_{j=1}^4 d_{j,\ulomega,\ulp}(t) Y_{j,\ulomega}(y),	
\end{equation}
with
\begin{equation}
\begin{aligned}
d_{1,\ulomega,\ulp}(t) := \frac{\langle \Vp(t),\sigma_2 Y_{2,\ulomega}\rangle }{\langle Y_{1,\ulomega},\sigma_2 Y_{2,\ulomega} \rangle},	 \quad d_{2,\ulomega,\ulp}(t) := \frac{\langle \Vp(t),\sigma_2 Y_{1,\ulomega}\rangle }{\langle Y_{2,\ulomega},\sigma_2 Y_{1,\ulomega} \rangle}, \\
d_{3,\ulomega,\ulp}(t) :=\frac{\langle \Vp(t),\sigma_2 Y_{4,\ulomega}\rangle }{\langle Y_{3,\ulomega},\sigma_2 Y_{4,\ulomega} \rangle}, \quad d_{4,\ulomega,\ulp}(t) := \frac{\langle \Vp(t),\sigma_2 Y_{3,\ulomega}\rangle }{\langle Y_{4,\ulomega},\sigma_2 Y_{3,\ulomega} \rangle}.	
\end{aligned}
\end{equation}

We denote the components of the vector $\ulPe \Vp$ by $\ulPe \Vp(t) = \begin{bmatrix}
\ve & \barve
\end{bmatrix}^\top$
and we recall from Lemma~\ref{lemma: L2 decomposition} that the projection $\ulPe$ preserves the $\calJ$-invariant subspace: $\sigma_1\overline{(\ulPe \Vp)} = \ulPe \Vp$. In the spirit of the space-time resonances method, we define the profile $F_{\ulomega,\ulp}(t,y)$ for the radiation term relative to the fixed parameters $(\ulomega,\ulp)$ by 
\begin{equation}\label{equ:profile}
F_{\ulomega,\ulp}(t,y) := e^{it\calH(\ulomega)}(\ulPe \Vp)(t,y),
\end{equation}
and we note that $F_{\ulomega,\ulp}(t,y)$ is also a $\calJ$-invariant vector. Furthermore,  it follows from \eqref{equ: V-equation} that the evolution equation for the profile $F_{\ulomega,\ulp}(t)$ is given by 
\begin{equation}\label{equ:evol_equa}
i \pt 	F_{\ulomega,\ulp}(t)= e^{it\calH(\ulomega)}\ulPe\Big(\dot{\theta}_1(t) \sigma_3 \Vp + \dot{\theta}_2(t)i\py \Vp + \calQ_{\ulomega}(\Vp) + \calC(\Vp)+\calMod + \calE_1 + \calE_2 \Big)
\end{equation}
We then take the distorted Fourier transform of the profile
\begin{equation}\label{equ:profile_dFT}
\wtilcalF_{\ulomega}[F_{\ulomega,\ulp}(t,\cdot)](\xi) = \begin{bmatrix}
\wtilcalF_{+,\ulomega}[F_{\ulomega,\ulp}(t,\cdot)](\xi)\\
\wtilcalF_{-,\ulomega}[F_{\ulomega,\ulp}(t,\cdot)](\xi)
\end{bmatrix}
\end{equation}
and use the following notations for its components
\begin{equation}
\tilf_{+,\ulomega,\ulp}(t,\xi) := \wtilcalF_{+,\ulomega}[F_{\ulomega,\ulp}(t,\cdot)](\xi),  \qquad 	\tilf_{-,\ulomega,\ulp}(t,\xi) := \wtilcalF_{-,\ulomega}[F_{\ulomega,\ulp}(t,\cdot)](\xi).
\end{equation}
By Proposition~\ref{prop: representation formula} and \eqref{eqn: representation e-itH} we have the representation formula for the renormalized radiation variable in terms of the distorted Fourier transform of the profile
\begin{equation}\label{equ:setup_ulPe_V_repformula}
\begin{split}
(\ulPe \Vp)(t,y) &= \big( e^{-it\calH(\ulomega)}F_{\ulomega,\ulp}(t)\big)(y) \\
&= \int_\bbR e^{-it(\xi^2+\ulomega)}	\tilf_{+,\ulomega,\ulp}(t,\xi) \Psi_{+,\ulomega}(y,\xi)\,\ud \xi 	-\int_\bbR e^{it(\xi^2+\ulomega)}	\tilf_{-,\ulomega,\ulp}(t,\xi) \Psi_{-,\ulomega}(y,\xi)\,\ud \xi ,
\end{split}
\end{equation}
or in components
\begin{align}
\ve(t,y) = \int_\bbR e^{-it(\xi^2+\ulomega)}	\tilf_{+,\ulomega,\ulp}(t,\xi) \Psi_{1,\ulomega}(y,\xi)\,\ud \xi 	-\int_\bbR e^{it(\xi^2+\ulomega)}	\tilf_{-,\ulomega,\ulp}(t,\xi) \Psi_{2,\ulomega}(y,\xi)\,\ud \xi,\label{equ:setup_v_e_repformula}\\
\barve(t,y) = \int_\bbR e^{-it(\xi^2+\ulomega)}	\tilf_{+,\ulomega,\ulp}(t,\xi) \Psi_{2,\ulomega}(y,\xi)\,\ud \xi 	-\int_\bbR e^{it(\xi^2+\ulomega)}	\tilf_{-,\ulomega,\ulp}(t,\xi) \Psi_{1,\ulomega}(y,\xi)\,\ud \xi.\label{equ:setup_barv_e_repformula}
\end{align} 
Using Corollary~\ref{cor:distFT_of_propagator} and \eqref{equ:evol_equa} we find that the evolution equations for the components of the distorted Fourier transform of the profile are given by
\begin{align}
i\pt \tilf_{+,\ulomega,\ulp}(t,\xi) &= e^{it (\xi^2+\ulomega)}\wtilcalF_{+,\ulomega}\Big[\dot{\theta}_1 \sigma_3 \Vp + i\dot{\theta}_2\py \Vp + \calQ_{\ulomega}(\Vp) + \calC(\Vp)+\calMod + \calE_1 + \calE_2\Big],\label{equ:DFT-profile+1}\\
i\pt \tilf_{-,\ulomega,\ulp}(t,\xi) &= e^{-it (\xi^2+\ulomega)}\wtilcalF_{-,\ulomega}\Big[\dot{\theta}_1 \sigma_3 \Vp + i\dot{\theta}_2\py \Vp + \calQ_{\ulomega}(\Vp) + \calC(\Vp)+\calMod + \calE_1 + \calE_2\Big].\label{equ:DFT-profile-1}
\end{align}

Let us now recast the right-hand side of \eqref{equ:DFT-profile+1} and \eqref{equ:DFT-profile-1} into a form that is more suitable for the nonlinear analysis. By Lemma~\ref{lem:distFT_applied_to_sigmathree_F}, we may rewrite the distorted Fourier transform of $\sigma_3 \Vp$ and $ \partial_y \Vp$ by
\begin{align*}
\wtilcalF_{+,\ulomega}[\sigma_3 \Vp(t)](\xi) &= \wtilcalF_{+,\ulomega}[\Vp(t)](\xi) + \calL_{+,\ulomega}[\Vp(t)](\xi),\\
\wtilcalF_{-,\ulomega}[\sigma_3 \Vp(t)](\xi) &= -\wtilcalF_{-,\ulomega}[\Vp(t)](\xi) + \calL_{-,\ulomega}[\Vp(t)](\xi),
\end{align*} 
and
\begin{align*}
\wtilcalF_{+,\ulomega}[\partial_y \Vp(t)](\xi) = i \xi \wtilcalF_{+,\ulomega}[\Vp(t)](\xi) + \calK_{+,\ulomega}[\Vp(t)](\xi),\\
\wtilcalF_{-,\ulomega}[\partial_y \Vp(t)](\xi) = i \xi \wtilcalF_{-,\ulomega}[\Vp(t)](\xi) + \calK_{-,\ulomega}[\Vp(t)](\xi),
\end{align*}
where
\begin{align*}
\calL_{+,\ulomega}[\Vp(t)](\xi) &:= 2 \langle \barvp(t,\cdot),\Psi_{2,\ulomega}(\cdot,\xi)\rangle , \qquad \calL_{-,\ulomega}[\Vp(t)](\xi) := 2 \langle \vp(t,\cdot),\Psi_{2,\ulomega}(\cdot,\xi)\rangle ,\\
\calK_{+,\ulomega}[\Vp(t)](\xi) &:= \frac{1}{\sqrt{2\pi}} \langle \vp(t,y), e^{iy\xi}\py m_{1,\ulomega}(y,\xi)\rangle -\frac{1}{\sqrt{2\pi}} \langle  \barvp(t,y), e^{iy\xi}\py m_{2,\ulomega}(y,\xi)\rangle ,\\
\calK_{-,\ulomega}[\Vp(t)](\xi) &:= \frac{1}{\sqrt{2\pi}} \langle \vp(t,y), e^{iy\xi}\py m_{2,\ulomega}(y,\xi)\rangle -\frac{1}{\sqrt{2\pi}} \langle \barvp(t,y), e^{iy\xi}\py m_{1,\ulomega}(y,\xi)\rangle .
\end{align*}
Moreover, by Corollary~\ref{cor:distFT_of_propagator}, the following relation between the distorted Fourier transform of $\Vp(t)$ and of $F_{\ulomega,\ulp}(t)$ holds 
\begin{align*}
\wtilcalF_{+,\ulomega}[\Vp(t)](\xi)=\wtilcalF_{+,\ulomega}[\ulPe \Vp(t)](\xi) &= e^{-it(\xi^2+\ulomega)}\tilf_{+,\ulomega,\ulp}(t,\xi),\\
\wtilcalF_{-,\ulomega}[\Vp(t)](\xi)=\wtilcalF_{-,\ulomega}[\ulPe \Vp(t)](\xi) &= e^{it(\xi^2+\ulomega)}\tilf_{-,\ulomega,\ulp}(t,\xi).
\end{align*}

As for the quadratic and cubic terms, the leading contributions stem from the terms with all inputs given by $\ulPe \Vp(t)$ since we expect the discrete components $\ulPd \Vp(t)$ to decay faster. Correspondingly, we write 
\begin{equation*}
\calQ_{\ulomega}(\Vp) + \calC(\Vp) = \calQ_{\ulomega}(\ulPe \Vp) + \calC(\ulPe \Vp) + \calE_3 	
\end{equation*}
with a remainder term 
\begin{equation}\label{equ:def_calE_3}
\calE_3 = \calQ_{\ulomega}(\Vp) - \calQ_{\ulomega}(\ulPe \Vp) + \calC(\Vp) - \calC(\ulPe \Vp).
\end{equation}
Lastly, note that  $\wtilcalF_{+,\ulomega}\big[\calMod(t)] = \wtilcalF_{+,\ulomega}\big[\ulPe \calMod(t)]$ by \eqref{equ:wtilcalF_applied_to_P}. By inserting the preceeding identities for \eqref{equ:DFT-profile+1} and \eqref{equ:DFT-profile-1}, we obtain 
\begin{equation}\label{equ:profile-dft-+}
\begin{split}
\pt \tilf_{+,\ulomega,\ulp}(t,\xi)&= -i \dot{\theta}_1(t) \tilf_{+,\ulomega,\ulp}(t,\xi) + i \dot{\theta}_2(t)\xi \tilf_{+,\ulomega,\ulp}(t,\xi) - i \dot{\theta}_1(t)e^{it(\xi^2+\ulomega)}\calL_{+,\ulomega}[\Vp(t)](\xi) \\
&\quad + \dot{\theta}_2(t) e^{it(\xi^2+\ulomega)} \calK_{+,\ulomega}[\Vp(t)](\xi) -i e^{it(\xi^2+\ulomega)}\wtilcalF_{+,\ulomega}\big[\calQ_{\ulomega}(\ulPe \Vp(t)) + \calC(\ulPe \Vp(t))\big](\xi)\\
&\quad -i e^{it(\xi^2+\ulomega)}\wtilcalF_{+,\ulomega}\big[\ulPe \calMod(t) + \calE_1(t) + \calE_2(t) + \calE_3(t)\big](\xi),
\end{split}
\end{equation}
and 
\begin{equation}\label{equ:profile-dft-minus}
\begin{split}
\pt \tilf_{-,\ulomega,\ulp}(t,\xi) &= i \dot{\theta}_1(t) \tilf_{-,\ulomega,\ulp}(t,\xi) + i \dot{\theta}_2(t)\xi \tilf_{-,\ulomega,\ulp}(t,\xi) - i \dot{\theta}_1(t)e^{it(\xi^2+\ulomega)}\calL_{-,\ulomega}[\Vp(t)](\xi) \\
&\quad + \dot{\theta}_2(t)e^{-it(\xi^2+\ulomega)}\calK_{-,\ulomega}[\Vp(t)](\xi) -i e^{-it(\xi^2+\ulomega)}\wtilcalF_{-,\ulomega}\big[\calQ_{\ulomega}(\ulPe \Vp(t)) + \calC(\ulPe \Vp(t))\big](\xi)\\
&\quad -i e^{-it(\xi^2+\ulomega)}\wtilcalF_{-,\ulomega}\big[\ulPe \calMod(t) + \calE_1(t) + \calE_2(t) + \calE_3(t)\big](\xi).
\end{split}
\end{equation}
Finally, inspired by \cite[Proposition~9.5]{CollotGermain23}, we absorb the linear terms on the right-hand side of \eqref{equ:profile-dft-+} using integrating factors. By defining the phase functions
\begin{equation}\label{equ:def_theta(t)}
\theta_1(t) :=\int_0^t \dot{\theta}_1(s)\,\ud s, \quad 	\theta_2(t) :=\int_0^t \dot{\theta}_2(s)\,\ud s,
\end{equation}
we then rewrite \eqref{equ:profile-dft-+} as 
\begin{equation}\label{equ:DFT-profile+2}
\begin{split}
&\pt\big( e^{i\theta_1(t)} e^{-i\theta_2(t)\xi}\tilf_{+,\ulomega,\ulp}(t,\xi)\big)	\\
&\quad  = -i e^{i\theta_1(t)} e^{-i \theta_2(t) \xi} e^{it(\xi^2+\ulomega)}  \times\\
&\qquad \times \bigg(\wtilcalF_{+,\ulomega}\big[\calQ_{\ulomega}(\ulPe \Vp(t)) + \calC(\ulPe \Vp(t))\big](\xi) +\wtilcalF_{+,\ulomega}\big[\ulPe \calMod(t) + \calE_1(t) + \calE_2(t) + \calE_3(t)\big](\xi)\\
&\qquad \qquad + \dot{\theta}_1(t) \calL_{+,\ulomega}[\Vp(t)](\xi)  + i\dot{\theta}_2(\xi) \calK_{+,\ulomega}[\Vp(t)](\xi)\bigg).
\end{split}
\end{equation}
The evolution equation \eqref{equ:profile-dft-minus} for $\tilf_{-,\ulomega,\ulp}(t,\xi)$ can be rewitten as well but we will mostly work with the equation \eqref{equ:DFT-profile+2} for $\tilf_{+,\ulomega,\ulp}(t,\xi)$ to derive weighted energy estimates and pointwise estimates for both $\tilf_{+,\ulomega,\ulp}(t,\xi)$ and $\tilf_{-,\ulomega,\ulp}(t,\xi)$. The reason is that thanks to Lemma~\ref{lem:distFT_components_relation} there is the following scattering relation 
\begin{equation}\label{equ:setup_components_relation}
\tilf_{-,\ulomega,\ulp}(t,\xi) = - \frac{\big(|\xi|-i\sqrt{\ulomega}\big)^2}{\big(|\xi|+i\sqrt{\ulomega}\big)^2}\overline{\tilf_{+,\ulomega,\ulp}(t,-\xi)},
\end{equation}
so that bounds on $\tilf_{+,\ulomega,\ulp}(t,\xi)$ can be transferred to $\tilf_{-,\ulomega,\ulp}(t,\xi)$.

\subsection{Normal form transformation on the quadratic nonlinearities}\label{subsec:normalform}
Like in \cite{LL24} we implement a variable coefficient normal form to recast the quadratic nonlinearity in \eqref{equ:DFT-profile+2} into a better form. The starting point is to isolate the slow decaying part of the radiation term $\ulPe \Vp(t)$ due to the two threshold resonances on the edges of the essential spectrum of $\calH(\ulomega)$. We implement a low-pass frequency filter in \eqref{equ:setup_v_e_repformula} and \eqref{equ:setup_barv_e_repformula} by rewriting the components of $\ulPe \Vp(t)$ as
\begin{equation} \label{equ:setup_decompositions_usube_local_decay}
\begin{aligned}
\ve(t,y) &= h_{1,\ulomega,\ulp}(t) \Phi_{1,\ulomega}(y) - h_{2,\ulomega,\ulp}(t) \Phi_{2,\ulomega}(y) + R_{\vp,\ulomega}(t,y), \\
\barve(t,y) &= h_{1,\ulomega,\ulp}(t) \Phi_{2,\ulomega}(y) - h_{2,\ulomega,\ulp}(t) \Phi_{1,\ulomega}(y) + R_{\barvp,\ulomega}(t,y),
\end{aligned}
\end{equation}
where we recall the functions associated to the threshold resonances are defined by 
\begin{equation} \label{equ:setup_definitions_Phijulomega}
\begin{aligned}
\Phi_{1,\ulomega}(y) &:= \Psi_{1,\ulomega}(y,0) = \frac{1}{\sqrt{2\pi}} \tanh^2(\sqrt{\ulomega}y), \\
\Phi_{2,\ulomega}(y) &:= \Psi_{2,\ulomega}(y,0) = - \frac{1}{\sqrt{2\pi}} \sech^2(\sqrt{\omega}y). 
\end{aligned}
\end{equation}
The low-frequency amplitudes are given by 
\begin{equation} \label{equ:setup_definitions_h12ulomega}
\begin{aligned}
h_{1,\ulomega,\ulp}(t) &:= e^{-it\ulomega} \int_\bbR e^{-it\xi^2} \chi_0(\xi) \tilf_{+,\ulomega,\ulp}(t,\xi) \, \ud \xi, \\
h_{2,\ulomega,\ulp}(t) &:= e^{it\ulomega} \int_\bbR e^{it\xi^2} \chi_0(\xi) \tilf_{-,\ulomega,\ulp}(t,\xi) \, \ud \xi,
\end{aligned}
\end{equation}
where $\chi_0(\xi)$ is a smooth even non-negative bump function  with $\chi_0(\xi) = 1$ for $|\xi| \leq 1$ and $\chi_0(\xi) = 0$ for $|\xi| \geq 2$. 
Inserting the decompositions \eqref{equ:setup_decompositions_usube_local_decay} into the quadratic nonlinearity, we obtain the expansion of the quadratic nonlinearity
\begin{equation} \label{equ:setup_expansion_quadratic_nonlinearity}
\begin{aligned}
&\calQ_\ulomega\bigl( \ulPe \Vp(t)\bigr)(y) \\
&\quad = h_{1,\ulomega,\ulp}(t)^2 \calQ_{1,\ulomega}(y) + h_{1,\ulomega,\ulp}(t) h_{2,\ulomega,\ulp}(t) \calQ_{2,\ulomega}(y) + h_{2,\ulomega,\ulp}(t)^2 \calQ_{3,\ulomega}(y) + \calQsubrom\bigl((\ulPe \Vp)(t)\bigr)(y)
\end{aligned}
\end{equation}
with the following source terms
\begin{align*}
\calQ_{1,\ulomega}(y) &:= \begin{bmatrix}
-\phi_{\ulomega}(y) \bigl( \Phi_{1,\ulomega}(y)^2 + 2 \Phi_{1,\ulomega}(y) \Phi_{2,\ulomega}(y) \bigr) \\
\phi_{\ulomega}(y) \bigl( \Phi_{2,\ulomega}(y)^2 + 2 \Phi_{1,\ulomega}(y) \Phi_{2,\ulomega}(y) \bigr)
\end{bmatrix}, \\
\calQ_{2,\ulomega}(y) &:= \begin{bmatrix}
2 \phiulomega(y) \bigl( \Phioneulomega(y)^2 + \Phitwoulomega(y)^2 + \Phioneulomega(y) \Phitwoulomega(y) \bigr) \\
-2 \phiulomega(y) \bigl( \Phioneulomega(y)^2 + \Phitwoulomega(y)^2 + \Phioneulomega(y) \Phitwoulomega(y) \bigr)
\end{bmatrix}, \\
\calQ_{3,\ulomega}(y) &:= \begin{bmatrix}
- \phiulomega(y) \bigl( \Phitwoulomega(y)^2 + 2 \Phioneulomega(y) \Phitwoulomega(y) \bigr) \\ \phiulomega(y) \bigl( \Phioneulomega(y)^2 + 2 \Phioneulomega(y) \Phitwoulomega(y) \bigr)
\end{bmatrix},
\end{align*}
and a remainder term given by 
\begin{align}
\calQsubrom\bigl((\ulPe \Vp)(t)\bigr)(y) &:= \calQ_{\mathrm{r},1,\ulomega}\bigl((\ulPe \Vp)(t)\bigr)(y) + \calQ_{\mathrm{r},2,\ulomega}\bigl((\ulPe \Vp)(t)\bigr)(y) + \calQ_{\mathrm{r},3,\ulomega}\bigl((\ulPe \Vp)(t)\bigr)(y) \label{eqn:def_renormalized_quadratic}
\end{align}
consisting of the terms 
\begin{align*}
\calQ_{\mathrm{r},1,\ulomega}\bigl((\ulPe \Vp)(t)\bigr)(y) &:= \begin{bmatrix} - 2\phiulomega(y) \bigl( h_{1,\ulomega,\ulp}(t) \Phi_{1,\ulomega}(y) - h_{2,\ulomega,\ulp}(t) \Phi_{2,\ulomega}(y)\bigr) \bigl( R_{\vp,\ulomega}(t,y) + R_{\barvp,\ulomega}(t,y) \bigr)\\ 
2 \phiulomega(y) \bigl( h_{1,\ulomega,\ulp}(t) \Phi_{1,\ulomega}(y) - h_{2,\ulomega,\ulp}(t) \Phi_{2,\ulomega}(y) \bigr) R_{\barvp,\ulomega}(t,y) \end{bmatrix},  \\
\calQ_{\mathrm{r},2,\ulomega}\bigl((\ulPe \Vp)(t)\bigr)(y) &:= \begin{bmatrix} - 2 \phiulomega(y) \bigl( h_{1,\ulomega,\ulp}(t) \Phi_{2,\ulomega}(y) - h_{2,\ulomega,\ulp}(t) \Phi_{1,\ulomega}(y)\bigr) R_{\barvp,\ulomega}(t,y) \\ 
2 \phiulomega(y) \bigl( h_{1,\ulomega,\ulp}(t) \Phi_{2,\ulomega}(y) - h_{2,\ulomega,\ulp}(t) \Phi_{1,\ulomega}(y)\bigr) \bigl( R_{\vp,\ulomega}(t,y) + R_{\barvp,\ulomega}(t,y) \bigr) \end{bmatrix}, \\
\calQ_{\mathrm{r},3,\ulomega}\bigl((\ulPe \Vp)(t)\bigr)(y) &:= \begin{bmatrix} -\phiulomega(y) \bigl( R_{\vp,\ulomega}(t,y)^2 + 2 R_{\vp,\ulomega}(t,y) R_{\barvp,\ulomega}(t,y) \bigr) \\ \phiulomega(y) \bigl( R_{\barvp,\ulomega}(t,y)^2 + 2 R_{\vp,\ulomega}(t,y) R_{\barvp,\ulomega}(t,y) \bigr) \end{bmatrix}.
\end{align*}

By a formal stationary phase analysis, we expect the amplitudes behave like $h_{1,\ulomega,\ulp}(t) \sim t^{-\frac12} e^{-it\ulomega}$ and $h_{2,\ulomega,\ulp}(t) \sim t^{-\frac12} e^{it\ulomega}$ as $t \to \infty$. Correspondingly, we expect that the filtered amplitudes $\pt ( e^{it\ulomega} h_{1,\ulomega,\ulp}(t) )$ and $\pt ( e^{-it\ulomega} h_{2,\ulomega,\ulp}(t) )$ have stronger time decay. We filter out the phases from $h_{1,\ulomega,\ulp}(t)$ and $h_{2,\ulomega,\ulp}(t)$ to obtain the following decomposition for the quadratic term
\begin{equation} \label{equ:wtilcalF_of_Q1}
\begin{aligned}
&e^{i t(\xi^2+\ulomega)} \wtilcalF_{+, \ulomega}\bigl[ \calQ_\ulomega\bigl((\ulPe \Vp)(t)\bigr) \bigr](\xi) \\
&= e^{i t(\xi^2-\ulomega)} \bigl( e^{it\ulomega} h_{1,\ulomega,\ulp}(t) \bigr)^2 \wtilcalF_{+,\ulomega}\bigl[\calQ_{1,\ulomega}\bigr](\xi)+ e^{i t(\xi^2+\ulomega)} \bigl( e^{it\ulomega} h_{1,\ulomega,\ulp}(t) \bigr) \bigl( e^{-it\ulomega} h_{2,\ulomega,\ulp}(t) \bigr) \wtilcalF_{+,\ulomega}\bigl[\calQ_{2,\ulomega}\bigr](\xi) \\
&\quad + e^{i t(\xi^2+3\ulomega)} \bigl( e^{-it\ulomega} h_{2,\ulomega,\ulp}(t) \bigr)^2 \wtilcalF_{+,\ulomega}\bigl[\calQ_{3,\ulomega}\bigr](\xi) + e^{i t(\xi^2+\ulomega)} \wtilcalF\bigl[ \calQ_{\mathrm{r},\ulomega}\bigl((\ulPe \Vp)(t)\bigr) \bigr](\xi).
\end{aligned}
\end{equation}

While the phases $e^{i t(\xi^2+\ulomega)}$ and $e^{i t(\xi^2+3\ulomega)}$ in \eqref{equ:wtilcalF_of_Q1} have no time resonances, the phase $e^{i t(\xi^2-\ulomega)}$ does have time resonances at the bad frequencies $\xi = \pm \sqrt{\ulomega}$. However, the source term $\wtilcalF_{+,\ulomega}\bigl[\calQ_{1,\ulomega}\bigr](\xi)$ has a remarkable null structure, which was first observed by the author in \cite[Lemma 1.6]{Li23}.
\begin{lemma} \label{lem:null_structure_radiation}
For any $\ulomega \in (0,\infty)$ we have
\begin{equation*}
\wtilcalF_{+,\ulomega}\bigl[\calQ_{1,\ulomega}\bigr](\xi) = (\ulomega-\xi^2) \frac{1}{24\sqrt{\pi}} \frac{\xi^2 (\ulomega+\xi^2 )}{\ulomega^2(\vert \xi \vert+i\sqrt{\ulomega})^2} \sech\left( \frac{\pi}{2} \frac{\xi}{\sqrt{\ulomega}}\right).
\end{equation*}
\end{lemma}
\begin{proof}
See \cite[Lemma~5.4]{LL24}.
\end{proof}
To lighten the notation, we introduce the following functions
\begin{align}
\frakq_{1,\ulomega}(\xi) &:= (\xi^2-\ulomega)^{-1}\wtilcalF_{+,\ulomega}\bigl[\calQ_{1,\ulomega}\bigr](\xi),\label{equ:def_frakq_1}\\
\frakq_{2,\ulomega}(\xi) &:= (\xi^2+\ulomega)^{-1}\wtilcalF_{+,\ulomega}\bigl[\calQ_{2,\ulomega}\bigr](\xi),\label{equ:def_frakq_2}\\
\frakq_{3,\ulomega}(\xi) &:= (\xi^2+3\ulomega)^{-1}\wtilcalF_{+,\ulomega}\bigl[\calQ_{3,\ulomega}\bigr](\xi)\label{equ:def_frakq_3}.
\end{align}
Using differentiation by parts in time,  we perform a normal form on the quadratic term \eqref{equ:wtilcalF_of_Q1} by rewritting
\begin{equation} \label{equ:setup_quadratic_rewritten}
\begin{aligned}
&-i e^{i t(\xi^2+\ulomega)} \wtilcalF_{+, \ulomega}\bigl[ \calQ_\ulomega\bigl((\ulPe \Vp)(t)\bigr) \bigr](\xi) \\
&\qquad = - \pt \bigl( \wtilB_{\ulomega,\ulp}(t,\xi) \bigr) -i e^{it(\xi^2+\ulomega)} \Bigl( \calR_{\frakq,\ulomega,\ulp}(t,\xi) + \wtilcalF\bigl[ \calQ_{\mathrm{r},\ulomega}\bigl((\ulPe \Vp)(t)\bigr) \bigr](\xi) \Bigr),
\end{aligned}
\end{equation}
with the bilinear form defined by
\begin{equation}\label{equ:setup_definition_wtilBulomega}
\begin{aligned}
\wtilB_{\ulomega,\ulp}(t,\xi) &:= e^{it(\xi^2-\ulomega)} \bigl( e^{it\ulomega} h_{1,\ulomega,\ulp}(t) \bigr)^2 \frakq_{1,\ulomega}(\xi)  + e^{it(\xi^2+\ulomega)} \bigl( e^{it\ulomega} h_{1,\ulomega,\ulp}(t) \bigr) \bigl( e^{-it\ulomega} h_{2,\ulomega,\ulp}(t) \bigr) \frakq_{2,\ulomega}(\xi)\\
&\quad + e^{it(\xi^2+3\ulomega)} \bigl( e^{-it\ulomega} h_{2,\ulomega,\ulp}(t) \bigr)^2 \frakq_{3,\ulomega}(\xi),
\end{aligned}
\end{equation}
and the remainder term given by 
\begin{equation}\label{equ:setup_definition_wtilcalR_qulomega}
\begin{aligned}
\widetilde{\calR}_{\frakq,\ulomega,\ulp}(t,\xi) &:= 2i e^{-2it\ulomega} \bigl( e^{it\ulomega} h_{1,\ulomega,\ulp}(t) \bigr) \pt \bigl( e^{it\ulomega} h_{1,\ulomega,\ulp}(t) \bigr) \frakq_{1,\ulomega}(\xi) \\
&\quad + i \pt \bigl( e^{it\ulomega} h_{1,\ulomega,\ulp}(t) \bigr) \bigl( e^{-it\ulomega} h_{2,\ulomega,\ulp}(t) \bigr) \frakq_{2,\ulomega}(\xi) \\
&\quad + i \bigl( e^{it\ulomega} h_{1,\ulomega,\ulp}(t) \bigr) \pt \bigl( e^{-it\ulomega} h_{2,\ulomega,\ulp}(t) \bigr) \frakq_{2,\ulomega}(\xi) \\
&\quad + 2i e^{2it\ulomega} \bigl( e^{-it\ulomega} h_{2,\ulomega,\ulp}(t) \bigr) \pt \bigl( e^{-it\ulomega} h_{2,\ulomega,\ulp}(t) \bigr) \frakq_{3,\ulomega}(\xi).
\end{aligned}
\end{equation}
The regularity of $\wtilB_{\ulomega,\ulp}(t,\xi)$ and $\wtilcalR_{\frakq,\ulomega,\ulp}(t,\xi)$ is addressed in the remark below.
\begin{remark}[On the smoothness of $\frakq_{1,\ulomega},\frakq_{2,\ulomega},\frakq_{3,\ulomega}$] \label{remark:smoothness_of_q}
By observation, the components $\calQ_{1,\ulomega}(y)$, $\calQ_{2,\ulomega}(y)$,  $\calQ_{3,\ulomega}(y)$ in the leading part of the quadratic term $\calQ_\ulomega\bigl((\ulPe \Vp)(t)$ are real-valued Schwartz functions. Hence, by Lemma~\ref{lem:null_structure_radiation}, the functions  $\frakq_{1,\ulomega}(\xi),\frakq_{2,\ulomega}(\xi),\frakq_{3,\ulomega}(\xi)$ are rapidly decaying and smooth up to a factor of $(|\xi|+i\sqrt{\ulomega})^{-2}$ steming from the definition of the distorted Fourier basis elements \eqref{eqn:m-1,omega}--\eqref{eqn:m-2,omega}.
\end{remark}

After inserting \eqref{equ:setup_quadratic_rewritten} back into the equation  \eqref{equ:DFT-profile+2}, we arrive at the renormalized evolution equation 
\begin{equation}\label{equ:setup_evol_equ_renormalized_tilfplus}
\begin{aligned}
&\pt \Bigl( e^{i\theta_1(t)} e^{-i\theta_2(t)\xi} \bigl( \tilf_{+,\ulomega,\ulp}(t,\xi) + \wtilB_{\ulomega,\ulp}(t,\xi) \bigr) \Bigr)\\
&= -i e^{i\theta_1(t)} e^{-i\theta_2(t)\xi} e^{i t(\xi^2+\ulomega)} \Big(\wtilcalF_{+, \ulomega}\bigl[\calC\bigl((\ulPe \Vp)(t)\bigr)\bigr](\xi) +\widetilde{\calR}_{\ulomega,\ulp}(t,\xi) \Big) \\
&\quad \, + i e^{i\theta_1(t)} e^{-i\theta_2(t)\xi} \big(\dot{\theta}_1(t)-\dot{\theta}_2(t)\xi\big) \wtilB_{\ulomega,\ulp}(t,\xi)\\
\end{aligned}
\end{equation}
with the remainder term 
\begin{equation}\label{equ:setup_definition_wtilcalRulomega}
\begin{split}
\widetilde{\calR}_{\ulomega,\ulp}(t,\xi) &:= \wtilcalF_{+,\ulomega}\bigl[ \calQ_{\mathrm{r},\ulomega}\bigl((\ulPe \Vp)(t)\bigr) +\ulPe \calMod(t)+\calE_1(t)+\calE_2(t)+\calE_3(t)\bigr](\xi)\\
&\quad + \dot{\theta}_1(t) \calL_{+,\ulomega}[\Vp(t)](\xi)  + i \dot{\theta}_2(t) \calK_{+,\ulomega}[\Vp(t)](\xi)+\widetilde{\calR}_{\frakq,\ulomega,\ulp}(t,\xi).
\end{split}
\end{equation}
The renormalized evolution equation \eqref{equ:setup_evol_equ_renormalized_tilfplus} will be the starting point for the analysis of weighted energy estimates and of pointwise estimates in Sections~\ref{sec:energy_estimates} and \ref{sec:pointwise_profile} respectively. We show in Corollary~\ref{corollary:prep_Q_r-ulomega} that $\widetilde{\calR}_{\ulomega,\ulp}(t,\xi)$ is spatially localized with at least cubic-type $\jt^{-\frac32+\delta}$ time decay.

\subsection{Structure of the cubic nonlinearities} \label{subsec:cubic_spectral_distributions}

For the final preparations, we uncover the structure of the cubic nonlinear term in the equation \eqref{equ:profile-dft-+}. We recall the following nonlinear cubic spectral distributions in \cite[Sect.~6.1]{LL24}. The multilinear expression for the cubic nonlinearity is given by 
\begin{equation*}
\begin{split}
&\wtilcalF_{+, \ulomega}\bigl[\calC\big((\ulPe \Vp)(t)\big)\bigr](\xi) \\
&=\sum_{S}\fraks \iiint e^{it\Phi_{j_1 j_2 j_3}(\xi_1,\xi_2,\xi_3)}	\tilf_{j_1}(t,\xi_1)\overline{\tilf_{j_2}}(t,\xi_2)\tilf_{j_3}(t,\xi_3)\mu_{1,k_1,k_2,k_3,\ulomega}(\xi,\xi_1,\xi_2,\xi_3)\,\ud\xi_1 \,\ud\xi_2 \,\ud\xi_3 \\
&\quad +\sum_{S} \fraks  \iiint e^{-it\Phi_{j_1 j_2 j_3}(\xi_1,\xi_2,\xi_3)}	\overline{\tilf_{j_1}}(t,\xi_1)\tilf_{j_2}(t,\xi_2)\overline{\tilf_{j_3}}(t,\xi_3)\mu_{2,k_1,k_2,k_3,\ulomega}(\xi,\xi_1,\xi_2,\xi_3)\,\ud\xi_1 \,\ud\xi_2 \,\ud\xi_3.
\end{split}
\end{equation*}
We use the shortened notation $\tilf_{j} := \tilf_{j,\ulomega,\ulp}$ for $j \in\{+,-\}$, and the following set of notations: The set $S$ is given by
\begin{equation*}
S := \big\{ (j_1,k_1),(j_2,k_2),(j_3,k_3) \mid (j_\ell,k_\ell) \in  \{(+,1),(-,2)\}, \quad 1\leq \ell \leq 3 \big\},
\end{equation*}
and $\fraks = \fraks(j_1,j_2,j_3) \in \{-1,+1\}$ are signs defined by
\begin{equation*}
\fraks(j_1,j_2,j_3) := -(j_1 1)(j_2 1)(j_3 1).
\end{equation*}
The phase takes the form
\begin{equation*}
\Phi_{j_1 j_2 j_3}(\xi_1,\xi_2,\xi_3) := -j_1(\xi_1^2+\ulomega)+j_2(\xi_2^2+\ulomega)-j_3(\xi_3^2+\ulomega),
\end{equation*}
where $j_1,j_2,j_3 \in \{+,-\}$. The $16 = 8 + 8$ cubic spectral distributions are formally given by
\begin{align}
\mu_{1,k_1,k_2,k_3,\ulomega}(\xi,\xi_1,\xi_2,\xi_3) &:= \int_\bbR \overline{\Psi_{1,\ulomega}}(y,\xi) \Psi_{k_1,\ulomega}(y,\xi_1)\overline{\Psi_{k_2,\ulomega}}(y,\xi_2)\Psi_{k_3,\ulomega}(y,\xi_3) \,\ud y,\label{eqn: mu_1k1k2k3}\\
\mu_{2,k_1,k_2,k_3,\ulomega}(\xi,\xi_1,\xi_2,\xi_3) &:= \int_\bbR \overline{\Psi_{2,\ulomega}}(y,\xi) \overline{\Psi_{k_1,\ulomega}}(y,\xi_1)\Psi_{k_2,\ulomega}(y,\xi_2)\overline{\Psi_{k_3,\ulomega}}(y,\xi_3) \,\ud y,\label{eqn: mu_2k1k2k3}
\end{align}
where $k_1,k_2,k_3 \in \{1,2\}$. Note that only $\mu_{1,1,1,1,\ulomega}(\xi,\xi_1,\xi_2,\xi_3)$ is singular, while the other 15 cubic spectral distributions are regular in the sense of the following definition.

\begin{definition} \label{def:reg-cubic}
A cubic spectral distribution $\mu(\xi,\xi_1,\xi_2,\xi_3)$ is regular if it can be written as a linear combination of terms of the form
\begin{equation*}
\overline{\frakb_0}(\xi)	\frakb_1(\xi_1)	\overline{\frakb_2}(\xi_2)	\frakb_3(\xi_3) \kappa(\xi-\xi_1+\xi_2-\xi_3) \quad \text{or}\quad \overline{\frakb_0}(\xi)	\overline{\frakb_1}(\xi_1)	{\frakb_2}(\xi_2)	\overline{\frakb_3}(\xi_3) \kappa(\xi+\xi_1-\xi_2+\xi_3),
\end{equation*}
where the multipliers $\frakb_0,\frakb_1,\frakb_2,\frakb_3 \in W^{1,\infty}(\bbR)$ are of the form
\begin{equation}\label{eqn: symbol_frakb}
\frac{1}{(\vert \xi \vert - i\sqrt{\ulomega})^2},\quad 	\frac{\xi}{(\vert \xi \vert - i\sqrt{\ulomega})^2},\quad \text{or} \quad 	\frac{\xi^2}{(\vert \xi \vert - i\sqrt{\ulomega})^2},
\end{equation}
and where $\kappa(\xi) = \widehat{\calF}[\varphi](\xi)$ for  some Schwartz function $\varphi(y)$.
\end{definition}

The precise structure of the singular cubic spectral distribution is provided in the lemma below.

\begin{lemma} \label{lemma:cubic_NSD}
We have
\begin{equation}
\begin{split}
&	\mu_{1,1,1,1,\ulomega}(\xi,\xi_1,\xi_2,\xi_3) = \mu_{\delta_0,\ulomega}(\xi,\xi_1,\xi_2,\xi_3) + \mu_{\pvdots,\ulomega}(\xi,\xi_1,\xi_2,\xi_3) + \mu_{\mathrm{reg},\ulomega}(\xi,\xi_1,\xi_2,\xi_3),
\end{split}
\end{equation}
where $\mu_{\mathrm{reg},\ulomega}$ is a regular cubic spectral distribution in the sense of Definition~\ref{def:reg-cubic},  and as tempered distributions,
\begin{equation}
\mu_{\delta_0,\ulomega}(\xi,\xi_1,\xi_2,\xi_3) :=	\frac{1}{2\pi\sqrt{\ulomega}} \frac{\frakp_1\Big(\tfrac{\xi}{\sqrt{\ulomega}},\tfrac{\xi_1}{\sqrt{\ulomega}},\tfrac{\xi_2}{\sqrt{\ulomega}},\tfrac{\xi_3}{\sqrt{\ulomega}}\Big)}{\frakp\Big(\tfrac{\xi}{\sqrt{\ulomega}},\tfrac{\xi_1}{\sqrt{\ulomega}},\tfrac{\xi_2}{\sqrt{\ulomega}},\tfrac{\xi_3}{\sqrt{\ulomega}}\Big)}\delta_0\Big(\tfrac{\xi}{\sqrt{\ulomega}}-\tfrac{\xi_1}{\sqrt{\ulomega}}+\tfrac{\xi_2}{\sqrt{\ulomega}}-\tfrac{\xi_3}{\sqrt{\ulomega}}\Big),
\end{equation}
and
\begin{equation}
\mu_{\pvdots,\ulomega}(\xi,\xi_1,\xi_2,\xi_3) :=	\frac{1}{(2\pi)^{\frac32}}	\frac{\frakp_2\Big(\tfrac{\xi}{\sqrt{\ulomega}},\tfrac{\xi_1}{\sqrt{\ulomega}},\tfrac{\xi_2}{\sqrt{\ulomega}},\tfrac{\xi_3}{\sqrt{\ulomega}}\Big)}{\frakp\Big(\tfrac{\xi}{\sqrt{\ulomega}},\tfrac{\xi_1}{\sqrt{\ulomega}},\tfrac{\xi_2}{\sqrt{\ulomega}},\tfrac{\xi_3}{\sqrt{\ulomega}}\Big)} \sqrt{\frac{\pi}{2\ulomega}}\pvdots \cosech\left(\frac{\pi(\xi-\xi_1+\xi_2-\xi_3)}{2\sqrt{\ulomega}}\right),
\end{equation}
with $\frakp,\frakp_1,\frakp_2$ given by 
\begin{align}
\frakp(\xi,\xi_1,\xi_2,\xi_3) &:= (\vert \xi \vert+i)^2(\vert \xi_1 \vert-i)^2 (\vert \xi_2 \vert+i)^2(\vert \xi_3 \vert - i)^3,	\label{eqn: cubic-frakp}\\
\frakp_1(\xi,\xi_1,\xi_2,\xi_3)
&:= (\xi^2-1)(\xi_1^2-1)(\xi_2^2-1)(\xi_3^2-1)+ 16\xi\xi_1\xi_2\xi_3 \label{eqn: cubic-frakp1}\\
&\quad + 4 \Big(\xi\xi_1(\xi_2^2-1)(\xi_3^2-1) - \xi(\xi_1^2-1)\xi_2(\xi_3^2-1) +\xi(\xi_1^2-1)(\xi_2^2-1)\xi_3 \nonumber \\
&\quad \qquad +(\xi^2-1)\xi_1\xi_2(\xi_3^2-1) -(\xi^2-1)\xi_1(\xi_2^2-1)\xi_3+(\xi^2-1)(\xi_1^2-1)\xi_2\xi_3\Big),\nonumber \\
\frakp_2(\xi,\xi_1,\xi_2,\xi_3) &:= 2\Big((\xi^2-1)\xi_1(\xi_2^2-1)(\xi_3^2-1) -(\xi^2-1)(\xi_1^2-1)\xi_2(\xi_3^2-1)\label{eqn: cubic-frakp2}\\
&\quad \qquad +(\xi^2-1)(\xi_1^2-1)(\xi_2^2-1)\xi_3-\xi(\xi_1^2-1)(\xi_2^2-1)(\xi_3^2-1)\Big)\nonumber	\\
&\quad + 8\Big( (\xi^2-1)\xi_1\xi_2\xi_3 - \xi(\xi_1^2-1)\xi_2\xi_3 +\xi\xi_1(\xi_2^2-1)\xi_3 -\xi\xi_1\xi_2(\xi_3^2-1)\Big).\nonumber
\end{align}

Furthermore, we have the following properties:
\begin{enumerate}
\item Tensorized structure: both symbols $(\frakp_1/\frakp)(\xi,\xi_1,\xi_2,\xi_3)$ and $(\frakp_2/\frakp)(\xi,\xi_1,\xi_2,\xi_3)$ are a linear combination of the form $	\overline{\frakb_0}(\xi)	\frakb_1(\xi_1)	\overline{\frakb_2}(\xi_2)	\frakb_3(\xi_3)$ with  multipliers $\frakb_0,\frakb_1,\frakb_2,\frakb_3 \in W^{1,\infty}(\bbR)$,
\item Diagonal property: for any $\xi \in \bbR$, we have
\begin{equation}\label{eqn:cubic-diagonal-property}
\frac{\frakp_1(\xi,\xi,\xi,\xi)}{\frakp(\xi,\xi,\xi,\xi)} = 1 \quad \text{and} \quad \frac{\frakp_2(\xi,\xi,\xi,\xi)}{\frakp(\xi,\xi,\xi,\xi)} = 0,
\end{equation}
\item Vanishing property for $\mu_{\pvdots,\ulomega}$: in each tensorized term there is at least one frequency variable that vanishes at zero, that is, for every product $\overline{\frakb_0}(\xi)	\frakb_1(\xi_1) \overline{\frakb_2}(\xi_2) \frakb_3(\xi_3)$ in $(\frakp_2/\frakp)(\xi,\xi_1,\xi_2,\xi_3)$, there exists some $j \in \{0,1,2,3\}$ for which $\frakb_j(0) = 0$.
\end{enumerate}
\end{lemma}
\begin{proof}
See \cite[Lemma~6.2]{LL24}.
\end{proof}
\section{Bootstrap Setup and Proof of Theorem~\ref{thm:main_theorem}} \label{sec:bootstrap_setup}

In this section we formulate the two main bootstrap propositions, and we establish the proof of Theorem~\ref{thm:main_theorem} as a consequence. 
The proofs of these two propositions will then occupy the remainder of the paper.

\subsection{The main bootstrap propositions}
Denoting by $0 < \delta \ll 1$ a small absolute constant, we define for $0 < T < T_\ast$,
\begin{equation} \label{equ:definition_bootstrap_norm_XT}
	\begin{aligned}
		&\bigl\| \bigl( \tilf_{+, \ulomega,\ulp}(t), \tilf_{-, \ulomega,\ulp}(t) \bigr) \bigr\|_{X(T)} \\
		&\quad := \sup_{0 \leq t \leq T} \, \Bigl( \bigl\| \bigl( \tilf_{+, \ulomega,\ulp}(t,\xi), \tilf_{-, \ulomega,\ulp}(t,\xi) \bigr) \bigr\|_{L^\infty_\xi} + \jt^{-\delta} \bigl\|  \bigl( \pxi \tilf_{+, \ulomega,\ulp}(t,\xi), \pxi \tilf_{-, \ulomega,\ulp}(t,\xi) \bigr) \bigr\|_{L^2_\xi} \Bigr).
	\end{aligned}
\end{equation}

In the first bootstrap proposition we establish decay of the modulation parameters $\omega(t)$ and $p(t)$ to their respective final values $\omega(T)$ and $p(T)$ on a given time interval $[0,T]$.

\begin{proposition}[Control of modulation parameters] \label{prop:modulation_parameters}
	Let $\omega_0 \in (0,\infty)$ and let $0 < \varepsilon_1 \ll 1$ be the small constant from the statement of Proposition~\ref{prop:modulation_and_orbital}.
	There exist constants $0 < \varepsilon_0 \ll \varepsilon_1 \ll 1$ and $C_0 \geq 1$ with the following properties:
	Let $(\gamma_0,p_0,\sigma_0) \in \bbR^3$ and let $u_0 \in H^1_x(\bbR) \cap L^{2,1}_x(\bbR)$  with $\varepsilon := \|u_0\|_{H^1_x \cap L^{2,1}_x} \leq \varepsilon_0$.
	Denote by $\psi(t,x)$ the solution to \eqref{equ:cubic_NLS} with initial condition
	\begin{equation}
		\psi_0(x) = e^{i p_0(x-\sigma_0)} e^{i\gamma_0} \bigl( \phi_{\omega_0}(x-\sigma_0) + u_0(x-\sigma_0) \bigr)
	\end{equation}
	on its maximal interval of existence $[0,T_\ast)$ furnished by Lemma~\ref{lem:setup_local_existence}.
	Let $(\omega, \gamma,p,\sigma) \colon [0,T_\ast) \to (0,\infty) \times \bbR^3$ be the unique continuously differentiable paths so that the decomposition
	\begin{equation}
		\psi(t,x) = e^{ip(t)(x-\sigma(t))} e^{i \gamma(t)} \bigl( \phi_{\omega(t)}(x-\sigma(t)) + u(t,x-\sigma(t)) \bigr), \quad 0 \leq t < T_\ast,
	\end{equation}
	satisfies (1)--(5) in the statement of Proposition~\ref{prop:modulation_and_orbital}.
	Fix $T \in (0,T_\ast)$, and fix $\ulomega \in (0,\infty)$, $\ulp \in \bbR$ so that $|\ulp-p_0|\leq \frac{1}{2}\omega_0$ and $\frac12 \omega_0 \leq \ulomega \leq 2 \omega_0$.
	Denote by $\bigl(\tilf_{+,\ulomega,\ulp}(t,\xi), \tilf_{-,\ulomega,\ulp}(t,\xi)\bigr)$ the components of the distorted Fourier transform \eqref{equ:profile_dFT} of the profile defined in \eqref{equ:profile}.
	Suppose
	\begin{align}
		\sup_{0 \leq t \leq T} \, \jt^{1-\delta}\Big( |\omega(t) - \ulomega| + |p(t) - \ulp|\Big) &\leq 2 C_0 \varepsilon, \label{equ:prop_modulation_parameters_assumption1} \\
		\bigl\| \bigl( \tilf_{+, \ulomega,\ulp}(t), \tilf_{-, \ulomega,\ulp}(t) \bigr) \bigr\|_{X(T)} &\leq 2 C_0 \varepsilon. \label{equ:prop_modulation_parameters_assumption2}
	\end{align}
	Then it follows that
	\begin{equation} \label{equ:prop_modulation_parameters_conclusion}
		\sup_{0 \leq t \leq T} \, \jt^{1-\delta}\Big( |\omega(t) - \omega(T)| + |p(t) - p(T)|\Big)\leq C_0 \varepsilon.
	\end{equation}
\end{proposition}

We defer the proof of Proposition~\ref{prop:modulation_parameters} to Section~\ref{sec:modulation_parameters}.
In the second bootstrap proposition we obtain control of the norm \eqref{equ:definition_bootstrap_norm_XT} of the components of the distorted Fourier transform of the profile of the renormalized radiation term.

\begin{proposition}[Profile bounds] \label{prop:profile_bounds}
	Let $\omega_0 \in (0,\infty)$ and let $0 < \varepsilon_1 \ll 1$ be the small constant from the statement of Proposition~\ref{prop:modulation_and_orbital}.
	There exist constants $0 < \varepsilon_0 \ll \varepsilon_1 \ll 1$ and $C_0 \geq 1$ with the following properties:
	Let $(\gamma_0,p_0,\sigma_0) \in \bbR^3$ and let $u_0 \in H^1_x(\bbR) \cap L^{2,1}_x(\bbR)$  with $\varepsilon := \|u_0\|_{H^1_x \cap L^{2,1}_x} \leq \varepsilon_0$.
	Denote by $\psi(t,x)$ the solution to \eqref{equ:cubic_NLS} with initial condition
	\begin{equation}
		\psi_0(x) = e^{i p_0(x-\sigma_0)} e^{i\gamma_0} \bigl( \phi_{\omega_0}(x-\sigma_0) + u_0(x-\sigma_0) \bigr)
	\end{equation}
	on its maximal interval of existence $[0,T_\ast)$ furnished by Lemma~\ref{lem:setup_local_existence}.
	Let $(\omega, \gamma,p,\sigma) \colon [0,T_\ast) \to (0,\infty) \times \bbR^3$ be the unique continuously differentiable paths so that the decomposition
	\begin{equation}
		\psi(t,x) = e^{ip(t)(x-\sigma(t))} e^{i \gamma(t)} \bigl( \phi_{\omega(t)}(x-\sigma(t)) + u(t,x-\sigma(t)) \bigr), \quad 0 \leq t < T_\ast,
	\end{equation}
	satisfies (1)--(5) in the statement of Proposition~\ref{prop:modulation_and_orbital}.
	Fix $T \in (0,T_\ast)$, and fix $\ulomega \in (0,\infty)$, $\ulp \in \bbR$ so that $|\ulp - p_0|\leq \frac{1}{2}\omega_0$ and $\frac12 \omega_0 \leq \ulomega \leq 2 \omega_0$. 
Denote by $\bigl(\tilf_{+,\ulomega,\ulp}(t,\xi), \tilf_{-,\ulomega,\ulp}(t,\xi)\bigr)$ the components of the distorted Fourier transform \eqref{equ:profile_dFT} of the profile defined in \eqref{equ:profile}.
	Suppose
	\begin{align}
		\sup_{0 \leq t \leq T} \, \jt^{1-\delta}\Big( |\omega(t) - \ulomega| + |p(t) - \ulp|\Big) &\leq 2 C_0 \varepsilon, \label{equ:prop_profile_bounds_assumption1} \\
		\bigl\| \bigl( \tilf_{+, \ulomega,\ulp}(t), \tilf_{-,\ulomega,\ulp}(t) \bigr) \bigr\|_{X(T)} &\leq 2 C_0 \varepsilon. \label{equ:prop_profile_bounds_assumption2}
	\end{align}
	Then it follows that
	\begin{equation} \label{equ:prop_profile_bounds_conclusion}
		%\bigl\| \bigl( \tilf_{+, \ulomega}(t), \tilf_{-,\ulomega}(t) \bigr) \bigr\|_{X(T)} \leq 2 C_0 \varepsilon \quad \Rightarrow \quad 
		\bigl\| \bigl( \tilf_{+,\ulomega,\ulp}(t), \tilf_{-,\ulomega,\ulp}(t) \bigr) \bigr\|_{X(T)} \leq C_0 \varepsilon.
	\end{equation}
\end{proposition}

The proof of Proposition~\ref{prop:profile_bounds} follows by a standard continuity argument from Proposition~\ref{prop: weighted-energy-estimate}, Proposition~\ref{prop:pointwise_estimate}, and the local existence theory.
From the bootstrap assumptions in the statements of Proposition~\ref{prop:modulation_parameters} and of Proposition~\ref{prop:profile_bounds}, we now deduce several decay estimates and some auxiliary bounds that will be used throughout the remainder of the paper.

\begin{corollary} \label{cor:consequences}
	Let $\omega_0 \in (0,\infty)$ and let $0 < \varepsilon_1 \ll 1$ be the small constant from the statement of Proposition~\ref{prop:modulation_and_orbital}.
	There exist constants $0 < \varepsilon_0 \ll \varepsilon_1 \ll 1$ and $C_0 \geq 1$ with the following properties:
	Let $(\gamma_0,p_0,\sigma_0) \in \bbR^3$ and let $u_0 \in H^1_x(\bbR) \cap L^{2,1}_x(\bbR)$  with $\varepsilon := \|u_0\|_{H^1_x \cap L^{2,1}_x} \leq \varepsilon_0$.
	Denote by $\psi(t,x)$ the solution to \eqref{equ:cubic_NLS} with initial condition
	\begin{equation}
		\psi_0(x) = e^{i p_0(x-\sigma_0)} e^{i\gamma_0} \bigl( \phi_{\omega_0}(x-\sigma_0) + u_0(x-\sigma_0) \bigr)
	\end{equation}
	on its maximal interval of existence $[0,T_\ast)$ furnished by Lemma~\ref{lem:setup_local_existence}.
	Let $(\omega, \gamma,p,\sigma) \colon [0,T_\ast) \to (0,\infty) \times \bbR^3$ be the unique continuously differentiable paths so that the decomposition
	\begin{equation}
		\psi(t,x) = e^{ip(t)(x-\sigma(t))} e^{i \gamma(t)} \bigl( \phi_{\omega(t)}(x-\sigma(t)) + u(t,x-\sigma(t)) \bigr), \quad 0 \leq t < T_\ast,
	\end{equation}
	satisfies (1)--(5) in the statement of Proposition~\ref{prop:modulation_and_orbital}.
Fix $T \in (0,T_\ast)$, and fix $\ulomega \in (0,\infty)$, $\ulp \in \bbR$ so that $|\ulp- p_0|\leq \frac{1}{2}\omega_0$ and $\frac12 \omega_0 \leq \ulomega \leq 2 \omega_0$. 
Denote by $\bigl(\tilf_{+,\ulomega,\ulp}(t,\xi), \tilf_{-,\ulomega,\ulp}(t,\xi)\bigr)$ the components of the distorted Fourier transform \eqref{equ:profile_dFT} of the profile defined in \eqref{equ:profile}.
	Suppose
	\begin{align}
		\sup_{0 \leq t \leq T} \, \jt^{1-\delta}\Big( |\omega(t) - \ulomega| + |p(t) - \ulp|\Big) &\leq 2 C_0 \varepsilon, \label{equ:consequences_assumption1} \\
		\bigl\| \bigl( \tilf_{+, \ulomega,\ulp}(t), \tilf_{-, \ulomega,\ulp}(t) \bigr) \bigr\|_{X(T)} &\leq 2 C_0 \varepsilon. \label{equ:consequences_assumption2}
	\end{align}
	Then the following estimates hold:
	
	\begin{itemize}[leftmargin=1.8em]
		\item[(1)] Sobolev bound for the renormalized radiation variable and for the profile:
		\begin{equation}\label{equ:consequences_sobolev_bound_V}
			\sup_{0\leq t \leq T} \| \Vp(t)\|_{H_y^1} \lesssim \eps,
		\end{equation}		
		\begin{equation} \label{equ:consequences_sobolev_bound_profile}
			\sup_{0 \leq t \leq T} \, \Bigl( \bigl\|\jxi \tilf_{+, \ulomega,\ulp}(t)\bigr\|_{L^2_\xi} + \bigl\|\jxi \tilf_{-, \ulomega,\ulp}(t)\bigr\|_{L^2_\xi} \Bigr) \lesssim \varepsilon.
		\end{equation}
		
		\item[(2)] Decomposition of the renormalized radiation term:
		\begin{equation} \label{equ:consequences_decomposition_radiation}
			\Vp(t,y) = (\ulPe \Vp)(t,y) + \sum_{k=1}^4 d_{k,\ulomega,\ulp}(t) Y_{k,\ulomega}(y), \quad 0 \leq t \leq T,
		\end{equation}
		with
		\begin{align}
			\sup_{0 \leq t \leq T} \, \jt^{\frac12} \| (\ulPe \Vp)(t) \|_{L^\infty_y} &\lesssim \varepsilon, \label{equ:consequences_ulPe_U_disp_decay} \\
\sup_{0 \leq t \leq T} \, \jt^{\frac12}\|\jy^{-3} \py \ulPe\Vp(t)\|_{L^2_y} &\lesssim \varepsilon,\label{equ:consequences_pyulPeU_disp_decay}\\
\sup_{0 \leq t \leq T}\, \jt^{\frac32-\delta} \big( |d_{1,\ulomega,\ulp}(t)|+|d_{2,\ulomega,\ulp}(t)|+|d_{3,\ulomega,\ulp}(t)|+|d_{4,\ulomega,\ulp}(t)|\big)  &\lesssim  \varepsilon. \label{equ:consequences_discrete_components_decay}
		\end{align}
		
		\item[(3)] Dispersive decay for the radiation terms:
\begin{align}
\sup_{0 \leq t \leq T} \, \jt^{\frac12}\big( \|U(t)\|_{L^\infty_y}+\|\Vp(t)\|_{L^\infty_y}\big) &\lesssim \varepsilon, \label{equ:consequences_U_disp_decay} \\
\sup_{0 \leq t \leq T} \, \jt^{\frac12}\big( \|\jy^{-3}\py U(t)\|_{L^2_y}+\|\jy^{-3} \py \Vp(t)\|_{L^2_y}\big) &\lesssim \varepsilon.\label{equ:consequences_pyU_disp_decay}
\end{align}

		\item[(4)] Auxiliary decay estimates for the modulation parameters:
		\begin{align}
			\sup_{0 \leq t \leq T} \, \jt \bigl( |\dot{\omega}(t)| + |\dot{p}|  + |\dot{\gamma}(t) + p^2 - \omega(t) - p \dot{\sigma}| + |\dot{\sigma} - 2p| \bigr) &\lesssim \varepsilon, \label{equ:consequences_aux_bound_modulation1} \\
			\sup_{0 \leq t \leq T} \, \jt^{1-\delta} \big(|\dot{\theta}_1(t)| + | \dot{\theta}_2(t)|\big)  &\lesssim \varepsilon. \label{equ:consequences_aux_bound_modulation2}
		\end{align}
		
		\item[(5)] Growth bound for the phase:
		\begin{equation} \label{equ:consequences_growth_bound_theta}
			\sup_{0 \leq t \leq T} \, \jt^{-\delta} \big(|\theta_1(t)| + |\theta_2(t)|\big) \lesssim \varepsilon.
		\end{equation}
		
		\item[(6)] Auxiliary bounds for remainder terms in the evolution equations for the profiles:
		\begin{align}
			\sup_{0 \leq t \leq T} \, \jt^{\frac12} \Bigl( \bigl\| \calL_{+, \ulomega}[\Vp(t)](\xi) \bigr\|_{H_\xi^1} + \bigl\| \calL_{-,\ulomega}[\Vp(t)](\xi) \bigr\|_{H_\xi^1} \Bigr) &\lesssim \varepsilon, \label{equ:consequences_calL_bounds} \\
			\sup_{0 \leq t \leq T} \, \jt^{\frac12} \Bigl( \bigl\| \calK_{+, \ulomega}[\Vp(t)](\xi) \bigr\|_{H_\xi^1} + \bigl\| \calK_{-,\ulomega}[\Vp(t)](\xi) \bigr\|_{H_\xi^1} \Bigr) &\lesssim \varepsilon, \label{equ:consequences_calK_bounds} \\
			\sup_{0 \leq t \leq T} \, \jt^{2-\delta} \Big( \bigl\| \jym \ulPe \calMod(t) \bigr\|_{L^{2}_y}+\bigl\| \jym \py \ulPe \calMod(t) \bigr\|_{L^{2}_y} \Big)&\lesssim \varepsilon^2, \label{equ:consequences_ulPe_Mod_bounds} \\
			\sup_{0 \leq t \leq T} \, \jt^{\frac32 - \delta} \Big( \bigl\| \jym  \calE_1(t) \bigr\|_{L^{2}_y}+\bigl\| \jym \py \calE_1(t) \bigr\|_{L^{2}_y} \Big) &\lesssim \varepsilon^2, \label{equ:consequences_calE1_bounds} \\
			\sup_{0 \leq t \leq T} \, \jt^{2 - \delta} \Big( \bigl\| \jym  \calE_2(t) \bigr\|_{L^{2}_y}+\bigl\| \jym \py \calE_2(t) \bigr\|_{L^{2}_y} \Big) &\lesssim \varepsilon^3, \label{equ:consequences_calE2_bounds} \\
			\sup_{0 \leq t \leq T} \, \jt^{2 - \delta} \Big( \bigl\| \jym  \calE_3(t) \bigr\|_{L^{2}_y}+\bigl\| \jym \py \calE_3(t) \bigr\|_{L^{2}_y} \Big) &\lesssim \varepsilon^3. \label{equ:consequences_calE3_bounds} 
		\end{align}
		
		\item[(7)] Leading order local decay:
		\begin{align}
			\ve(t,y) &= h_{1,\ulomega,\ulp}(t) \Phi_{1,\ulomega}(y) - h_{2,\ulomega,\ulp}(t) \Phi_{2,\ulomega}(y) + R_{\vp,\ulomega}(t,y), \quad 0 \leq t \leq T, \label{equ:consequences_usube_leading_order_local_decay_decomp} \\
			\barve(t,y) &= h_{1,\ulomega,\ulp}(t) \Phi_{2,\ulomega}(y) - h_{2,\ulomega,\ulp}(t) \Phi_{1,\ulomega}(y) + R_{\barvp,\ulomega}(t,y), \quad 0 \leq t \leq T, \label{equ:consequences_barusube_leading_order_local_decay_decomp}
		\end{align}
		with
		\begin{align}
			\sup_{0 \leq t \leq T} \, \jt^{\frac12} \bigl( |h_{1,\ulomega,\ulp}(t)| + |h_{2,\ulomega,\ulp}(t)| \bigr) &\lesssim \varepsilon, \label{equ:consequences_h12_decay} \\
			\sup_{0 \leq t \leq T} \, \jt^{1-\delta} \Bigl( \bigl|\pt \bigl( e^{i t \ulomega} h_{1,\ulomega,\ulp}(t) \bigr) \bigr| + \bigl|\pt \bigl( e^{-i t \ulomega} h_{2,\ulomega,\ulp}(t) \bigr) \bigr| \Bigr) &\lesssim \varepsilon, \label{equ:consequences_h12_phase_filtered_decay}  \\
			\sup_{0 \leq t \leq T} \, \jt^{1-\delta} \Bigl( \bigl\|\jy^{-2} R_{\vp,\ulomega}(t,y)\bigr\|_{L^\infty_y} + \bigl\| \jy^{-2} R_{\barvp,\ulomega}(t,y) \bigr\|_{L^\infty_y} \Bigr) &\lesssim \varepsilon, \label{equ:consequences_Ru_local_decay} \\
			\sup_{0 \leq t \leq T} \, \jt^{1-\delta} \Bigl( \bigl\|\jy^{-3} \py R_{\vp,\ulomega}(t,y)\bigr\|_{L^2_y} + \bigl\| \jy^{-3} \py R_{\barvp,\ulomega}(t,y) \bigr\|_{L^2_y} \Bigr) &\lesssim \varepsilon. \label{equ:consequences_px_Ru_local_decay}
		\end{align}
		
		\item[(8)] Bounds for auxiliary free Schr\"odinger evolutions:
		
		\noindent Denote by $\chi_0(\xi)$ a smooth even non-negative cut-off function with $\chi_0(\xi) = 1$ for $|\xi| \leq 1$ and $\chi_0(\xi) = 0$ for $|\xi| \geq 2$. Given symbols $\fraka, \frakb \in W^{1,\infty}(\bbR)$, we define
\begin{equation*}
\begin{aligned}
v_{+,\ulomega,\ulp}(t,y) &:= e^{-i t \ulomega} \int_\bbR e^{i y \xi_1} e^{-it\xi_1^2} \fraka(\xi_1) \tilfplusulo(t,\xi_1) \, \ud \xi_1, \\
v_{-,\ulomega,\ulp}(t,y) &:= e^{i t \ulomega} \int_\bbR e^{i y \xi_2} e^{it\xi_2^2} \frakb(\xi_2) \tilfminusulo(t,\xi_2) \, \ud \xi_2.
\end{aligned}
\end{equation*}
Morover, we introduce the decompositions
\begin{align*}
v_{+,\ulomega,\ulp}(t,y) &= \tilde{h}_{1,\ulomega,\ulp}(t) + R_{v_+,\ulomega,\ulp}(t,y), \\
v_{-,\ulomega,\ulp}(t,y) &= \tilde{h}_{2,\ulomega,\ulp}(t) + R_{v_-,\ulomega,\ulp}(t,y),
\end{align*}
with
\begin{align*}
\tilde{h}_{1,\ulomega,\ulp}(t) &= e^{-i t \ulomega} \int_\bbR e^{-i t \xi_1^2} \chi_0(\xi_1) \fraka(\xi_1) \tilfplusulo(t,\xi_1) \, \ud \xi_1, \\
\tilde{h}_{2,\ulomega,\ulp}(t) &= e^{i t \ulomega} \int_\bbR e^{i t \xi_2^2} \chi_0(\xi_2) \frakb(\xi_2) \tilfminusulo(t,\xi_2) \, \ud \xi_2.
\end{align*}
Then the following estimates hold:
\begin{align}
\sup_{0 \leq t \leq T} \, \jt^{\frac12} \bigl( \|v_{+,\ulomega,\ulp}(t)\|_{L^\infty_y} + \|v_{-,\ulomega,\ulp}(t)\|_{L^\infty_y} \bigr) &\lesssim \varepsilon, \label{equ:preparation_flat_Schrodinger_wave_bound1} \\
\sup_{0 \leq t \leq T} \, \bigl( \|v_{+,\ulomega,\ulp}(t)\|_{L^2_y} + \|v_{-,\ulomega,\ulp}(t)\|_{L^2_y} \bigr) &\lesssim \varepsilon, \label{equ:preparation_flat_Schrodinger_wave_bound_L2} \\
\sup_{0 \leq t \leq T} \, \jt^{1-\delta} \bigl( \|\jy^{-1} \py v_{+,\ulomega,\ulp}(t)\|_{L^2_y} + \|\jy^{-1} \py v_{-,\ulomega,\ulp}(t)\|_{L^2_y} \bigr) &\lesssim \varepsilon, \label{equ:preparation_flat_Schrodinger_wave_bound2} \\
\sup_{0 \leq t \leq T} \, \jt^{\frac12} \bigl( |\tilde{h}_{1,\ulomega,\ulp}(t)| + |\tilde{h}_{2,\ulomega,\ulp}(t)| \bigr) &\lesssim \varepsilon, \label{equ:preparation_flat_Schrodinger_wave_bound3} \\
\sup_{0 \leq t \leq T} \, \jt^{1-\delta} \Bigl( \bigl| \pt \bigl( e^{it\ulomega} \tilde{h}_{1,\ulomega,\ulp}(t) \bigr) \bigr| + \bigl| \pt \bigl( e^{-it\ulomega} \tilde{h}_{2,\ulomega,\ulp}(t) \bigr) \bigr| \Bigr) &\lesssim \varepsilon, \label{equ:preparation_flat_Schrodinger_wave_bound4} \\
\sup_{0 \leq t \leq T} \, \jt^{1-\delta} \Bigl( \bigl\| \jy^{-2} R_{v_+,\ulomega,\ulp}(t,y) \bigr\|_{L^\infty_y} + \bigl\| \jy^{-2} R_{v_-,\ulomega,\ulp}(t,x) \bigr\|_{L^\infty_y} \Bigr) &\lesssim \varepsilon. \label{equ:preparation_flat_Schrodinger_wave_bound5}
		\end{align}
In the special case 
\begin{equation*}
\fraka(\xi) = \frakb(\xi) = \frac{\xi}{(|\xi|+i\sqrt{\ulomega})^2},
\end{equation*}		
the following improved local decay estimate holds
\begin{equation}\label{equ:prep_flat_local_decay}
\sup_{0\leq t \leq T} \jt^{1-\delta}\big( \| \jy^{-1} v_{+,\ulomega,\ulp}(t)\|_{L_y^2}+ \| \jy^{-1} v_{-,\ulomega,\ulp}(t)\|_{L_y^2}\big) \lesssim \eps.
\end{equation}

	\end{itemize}
\end{corollary}

\begin{proof}[Proof of Corollary~\ref{cor:consequences}] We prove these asserted estimates item by item.
	
	\noindent \underline{Proof of (1).} Recalling the definition \eqref{equ: V-definition} for $\Vp(t) = (\vp(t),\barvp(t))^\top$, we find that $|\vp(t,y)| = |u(t,y)|$ and $|\py \vp(t,y)| \leq |p(t)-\ulp| |u(t,y)| + | \py u(t,y)|$ for all $y \in \bbR$. Hence, the orbital stability bound \eqref{equ:setup_smallness_orbital} and the bootstrap assumption \eqref{equ:consequences_assumption1} imply \eqref{equ:consequences_sobolev_bound_V}. Using the $\wtilcalF_{\pm,\ulomega}:H^1 \rightarrow L^{2,1}$ boundedness of the distorted Fourier transform (Proposition~\ref{prop:mapping_properties_dist_FT}) and using the Sobolev bound \eqref{equ:consequences_sobolev_bound_V}, we obtain for all $t \in [0,T]$ that 
	\begin{equation*}
		\|\jxi \tilf_{\pm,\ulomega,\ulp}(t,\xi)\|_{L_\xi^2} = \| \jxi \wtilcalF_{\pm,\ulomega}[e^{it\calH(\ulomega)}\ulPe \Vp(t)](\xi)\|_{L_\xi^2}= \| \jxi \wtilcalF_{\pm,\ulomega}[\Vp(t)](\xi)\|_{L_\xi^2} \lesssim \|\Vp(t)\|_{H_y^1}\lesssim \eps.
	\end{equation*}
	Thus, \eqref{equ:consequences_sobolev_bound_profile} holds.
	
\medskip 
	\noindent \underline{Proof of (2).} The spectral decomposition \eqref{equ:consequences_decomposition_radiation} for $\Vp(t)$ follows from \eqref{equ:spectral_decomposition_L2L2} in Lemma~\ref{lemma: L2 decomposition}. Let us prove \eqref{equ:consequences_ulPe_U_disp_decay}. For short times $0 \leq t \leq 1$, the bound $\|\ulPe \Vp(t)\|_{L_y^\infty} \lesssim \eps$ follows from the Sobolev embedding $H_y^1 \hookrightarrow L_y^\infty$, the $H^1$-boundedness of the projection $\ulPe$ by Lemma~\ref{lemma: L2 decomposition}, and the stability bound \eqref{equ:consequences_sobolev_bound_V}. For times $1 \leq t \leq T$, we infer from \eqref{eqn:linear_dispersive_decay} in Lemma~\ref{lem:linear_dispersive_decay} the dispersive estimate $\|\ulPe \Vp(t)\|_{L_y^\infty} \lesssim \eps t^{-\frac12}$ using the representation formula \eqref{equ:setup_ulPe_V_repformula} and the bounds \eqref{equ:consequences_assumption2}, \eqref{equ:consequences_sobolev_bound_profile}. The proof for the dispersive decay \eqref{equ:consequences_pyulPeU_disp_decay} is a consequence of the leading order decomposition \eqref{equ:consequences_usube_leading_order_local_decay_decomp}-\eqref{equ:consequences_barusube_leading_order_local_decay_decomp}, the fact that $\py\Phi_{1,\ulomega}(y),\py\Phi_{2,\ulomega}(y)$ are bounded functions, and the decay estimates \eqref{equ:consequences_discrete_components_decay}, \eqref{equ:consequences_h12_decay}, \eqref{equ:consequences_px_Ru_local_decay} (which will be proved below).

	Next, we prove the decay estimates for the discrete component $\ulPd \Vp(t) =  \sum_{k=1}^4 d_{k,\ulomega,\ulp}(t) Y_{k,\ulomega}(y)$.  By the orthogonality conditions \eqref{equ:setup_orthogonality_radiation}, we find that 
	\begin{equation}\label{equ:proof_orthogonality_condtions1}
		\langle \Vp(t),e^{i(p(t)-\ulp)\sigma_3 y}\sigma_2 Y_{j,\omega(t)}\rangle = \langle U(t),\sigma_2 Y_{j,\omega(t)}\rangle = 0, \qquad j\in \{1,\ldots,4\}.
	\end{equation}
	Hence we obtain from \eqref{equ:consequences_decomposition_radiation} the following system of linear equations for the coefficients 
	\begin{equation}\label{equ:proof_system_for_d}
		\bbA(t)	 \begin{bmatrix}
			d_{1,\ulomega,\ulp}(t)\\
			d_{2,\ulomega,\ulp}(t)\\
			d_{3,\ulomega,\ulp}(t)\\
			d_{4,\ulomega,\ulp}(t)
		\end{bmatrix} = \begin{bmatrix}
			- \langle \ulPe \Vp(t), e^{i(p(t)-\ulp)\sigma_3 y}\sigma_2 Y_{1,\omega(t)}\rangle \\
			- \langle \ulPe \Vp(t), e^{i(p(t)-\ulp)\sigma_3 y}\sigma_2 Y_{2,\omega(t)}\rangle \\
			- \langle \ulPe \Vp(t), e^{i(p(t)-\ulp)\sigma_3 y}\sigma_2 Y_{3,\omega(t)}\rangle \\
			- \langle \ulPe \Vp(t), e^{i(p(t)-\ulp)\sigma_3 y}\sigma_2 Y_{4,\omega(t)}\rangle 
		\end{bmatrix},
	\end{equation}
	where $\bbA(t)$ is a $4$ by $4$ matrix with entries 
	\begin{equation*}
		\big(\bbA(t)\big)_{j,k}=
		\langle Y_{j,\ulomega},e^{i(p(t)-\ulp)\sigma_3 y}\sigma_2 Y_{k,\omega(t)} \rangle 
		,\qquad j,k \in \{1,\ldots,4\}.
	\end{equation*}
	In order to infer decay on the coefficients, we proceed by expanding the leading order terms in the matrix $\bbA(t)$ and the right hand side of \eqref{equ:proof_system_for_d}. For each $j,k \in \{1,\ldots,4\}$ we write 
	\begin{equation*}
		\big(\bbA(t)\big)_{j,k}= \langle Y_{j,\ulomega},\sigma_2 Y_{k,\ulomega}\rangle + \langle Y_{j,\ulomega},\sigma_2 (Y_{k,\omega(t)} - Y_{k,\ulomega})\rangle +  \langle Y_{j,\ulomega},(e^{i(p(t)-\ulp)\sigma_3 y}-\bbI)\sigma_2 Y_{k,\omega(t)} \rangle 
	\end{equation*}
	where $\bbI$ is the 2 by 2 identity matrix. From \eqref{equ:innerproductrelations1}, \eqref{equ:innerproductrelations2} we find that  
	\begin{equation*}
		\bbA(t) = \bbM_1(\ulomega) + \calO(|\omega(t)-\ulomega|) + \calO(|p(t)-\ulp|)
	\end{equation*}
	where $\bbM_1(\ulomega)$ is exactly the same matrix as \eqref{equ:setup_matrix_M1} with $\ulomega$ as input. In view of the bootstrap assumptions \eqref{equ:consequences_assumption1} and the comparison estimate $\ulomega \in [\frac12 \omega_0,2\omega_0]$, we infer that the matrix $\bbA(t)$ is invertible with a uniform-in-time upper bound on its operator norm whose size depends only on $\omega_0$. By the orthogonality relations \eqref{equ:orthogonality_PeU}, we  may expand the right hand side of \eqref{equ:proof_system_for_d} by writing
	\begin{equation*}
		\begin{split}
			&\langle \ulPe \Vp(t), e^{i(p(t)-\ulp)\sigma_3 y}\sigma_2 Y_{k,\omega(t)}\rangle	= \langle \ulPe \Vp(t), e^{i(p(t)-\ulp)\sigma_3 y}\sigma_2 Y_{k,\omega(t)} - \sigma_{2} Y_{k,\ulomega}\rangle\\
			&=	\langle \ulPe \Vp(t), (e^{i(p(t)-\ulp)\sigma_3 y} - \bbI)\sigma_2 Y_{k,\omega(t)}\rangle + \langle \ulPe \Vp(t), \sigma_2 (Y_{k,\omega(t)} - Y_{k,\ulomega})\rangle, \qquad k \in \{1,\ldots,4\}.
		\end{split}
	\end{equation*}
	Thus by \eqref{equ:consequences_assumption1} and \eqref{equ:consequences_ulPe_U_disp_decay} we conclude the desired bound for all $t \in [0,T]$,
\begin{equation*}
	\begin{split}
\sum_{k=1}^4 |d_{k,\ulomega,\ulp}(t)| &\lesssim \|\bbA(t)^{-1}\|_{\mathrm{op}} \sum_{k=1}^4|\langle \ulPe \Vp(t), e^{i(p(t)-\ulp)\sigma_3 y}\sigma_2 Y_{k,\omega(t)}\rangle|  \\
&\lesssim_{\omega_0} \| \ulPe \Vp(t)\|_{L_y^\infty}	\big(|\omega(t)-\ulomega| + |p(t) - \ulp|\big) \lesssim \eps^2 \jt^{-\frac32 + \delta}.		
	\end{split}
\end{equation*}

\medskip 	\noindent \underline{Proof of (3).} The dispersive estimate \eqref{equ:consequences_U_disp_decay} for $U(t)$ and $\Vp(t)$ follows from the fact that $|u(t,y)|=|\vp(t,y)|$ for all $y \in \bbR$ and the bound \eqref{equ:consequences_ulPe_U_disp_decay}, \eqref{equ:consequences_discrete_components_decay}. Simlarly, the dispersive estimate for its derivative \eqref{equ:consequences_pyU_disp_decay}  follows from the pointwise estimate $|\py U(t,y)| \leq |p(t)-\ulp| |\Vp(t,y)| + |\py \Vp(t,y)|$ and the bounds \eqref{equ:consequences_ulPe_U_disp_decay}, \eqref{equ:consequences_pyulPeU_disp_decay}, \eqref{equ:consequences_discrete_components_decay}.

\medskip \noindent \underline{Proof of (4).} For the decay estimates \eqref{equ:consequences_aux_bound_modulation1}--\eqref{equ:consequences_aux_bound_modulation2} we return to the modulation equations \eqref{equ:setup_modulation_equation}. Recall that the time-dependent matrix \eqref{equ:setup_matrix_modulation_equation} on the left-hand side of the modulation equations \eqref{equ:setup_modulation_equation} is given by  $\bbM(t) = \bbM_1(\omega(t)) + \bbM_2(U(t),\omega(t))$. By inspecting \eqref{equ:setup_matrix_M1}, \eqref{equ:setup_matrix_M2} we find that $\bbM_1(\omega(t)) = \bbM_1(\ulomega) + \calO(|\omega(t)-\ulomega|)$ and  $\bbM_2(U(t),\omega(t)) = \calO(\eps)$ thanks to \eqref{equ:setup_smallness_orbital} and \eqref{equ:setup_comparison_estimate}. Hence, the matrix $\bbM(t)$ is invertible and the operator norm of its inverse has a uniform-in-time upper bound depending only on $\omega_0$. The asserted bound \eqref{equ:consequences_aux_bound_modulation1} then follows from the modulation equations using the dispersive estimate \eqref{equ:consequences_ulPe_U_disp_decay} and H\"older's inequality
	\begin{equation*}
		\begin{split}
			|\dot{\omega}(t)| + |\dot{p}|  + |\dot{\gamma}(t) + p^2 - \omega(t) - p \dot{\sigma}| + |\dot{\sigma} - 2p|  \lesssim \|\bbM(t)^{-1}\|_{\mathrm{op}} \|\calN(U)\|_{L_y^\infty}\Big(\sum_{k=1}^4 \| Y_{k,\omega(t)}\|_{L_y^1}\Big) \lesssim \eps^2 \jt^{-1}.
		\end{split}
	\end{equation*}
	The second bound \eqref{equ:consequences_aux_bound_modulation2} then follows from \eqref{equ:consequences_assumption1}, \eqref{equ:consequences_aux_bound_modulation1} using the formulas \eqref{equ:def-dot-theta1}, \eqref{equ:def-dot-theta2}.
	
	\medskip \noindent \underline{Proof of (5).} The growth bounds \eqref{equ:consequences_growth_bound_theta} for $\theta_1(t)$ and $\theta_2(t)$ follows from its definition \eqref{equ:def_theta(t)} and the decay estimate \eqref{equ:consequences_aux_bound_modulation2}.
	
\medskip \noindent \underline{Proof of (6).} We start with the proof of  \eqref{equ:consequences_calL_bounds} and  \eqref{equ:consequences_calK_bounds}. From Lemma~\ref{lem:distFT_applied_to_sigmathree_F} we have 
\begin{equation*}
\calL_{+, \ulomega}[\Vp(t)](\xi) = \frac{2}{\sqrt{2\pi}} \int_\bbR e^{-iy \xi} \barvp(t,y) \overline{m_{2,\ulomega}(y,\xi)} \,\ud y,
\end{equation*}
and
\begin{equation*}
\calK_{+,\ulomega}[\Vp(s)](\xi) = \frac{1}{\sqrt{2\pi}} \int_\bbR e^{-iy \xi}\Big( \vp(s,y)\overline{\py m_{1,\ulomega}(y,\xi)}+ \barvp(s,y)\overline{\py m_{2,\ulomega}(y,\xi)} \Big)\,\ud y.
\end{equation*}
By direct computation using \eqref{eqn:m-1,omega} and \eqref{eqn:m-2,omega} we find that 
\begin{align}
\overline{m_{2,\ulomega}(y,\xi)}&=\frac{\ulomega }{(|\xi|+i\sqrt{\ulomega})^2}\sech^2(\sqrt{\ulomega}y) =: \frakb_{\ulomega}(\xi)\fraka_{\ulomega}(y),\\
\overline{\py m_{2,\ulomega}(y,\xi)} &= - \frac{2\ulomega^{\frac32}}{(|\xi|+i\sqrt{\ulomega})^2}\tanh(\sqrt{\ulomega}y)\sech^2(\sqrt{\ulomega} y) \equiv \frakb_{\ulomega}(\xi) \fraka_{\ulomega}'(y),\\
\overline{\py m_{1,\ulomega}(y,\xi)} &= \frac{-2i\xi \sqrt{\ulomega}}{(|\xi|+i\sqrt{\ulomega})^2}\sech^2(\sqrt{\ulomega} y) - \frac{2\ulomega^{\frac32}}{(|\xi|+i\sqrt{\ulomega})^2}\tanh(\sqrt{\ulomega}y)\sech^2(\sqrt{\ulomega} y)\\
&\quad =: \frakc_{\ulomega}(\xi)\fraka_{\ulomega}(y) + \frakb_{\ulomega}(\xi) \fraka_{\ulomega}'(y),\nonumber
\end{align}
where we set
\begin{equation}\label{equ:def_fraka-b-c}
\fraka_{\ulomega}(y) := \sech^2(\sqrt{\ulomega}y), \quad \frakb_{\ulomega}(\xi) := \frac{\ulomega }{(|\xi|+i\sqrt{\ulomega})^2}, \quad \frakc_{\ulomega}(\xi) := \frac{-2i\xi \sqrt{\ulomega}}{(|\xi|+i\sqrt{\ulomega})^2},
\end{equation}
in order to lighten some notation. Since $\sech^2(\cdot)$ is a Schwartz function, we observe that $\fraka_{\ulomega}(y)$ is also Schwartz. Moreover, it is clear that $\frakb_{\ulomega}(\xi), \frakc_{\ulomega}(\xi)\in W^{1,\infty}(\bbR)$. By Plancherel's identity and the bound \eqref{equ:consequences_U_disp_decay}, we obtain that
\begin{equation*}
\begin{split}
\|\calL_{+, \ulomega}[\Vp(t)]\|_{H_\xi^1} &\lesssim \|\calL_{+, \ulomega}[\Vp(t)]\|_{L_\xi^2}+\|\pxi \calL_{+, \ulomega}[\Vp(t)]\|_{L_\xi^2} \\
&\lesssim \big(\|\frakb_{\ulomega} \|_{L_\xi^\infty} +\|\pxi\frakb_{\ulomega} \|_{L_\xi^\infty}\big)\| \fraka_{\ulomega}(y)\barvp(t,y) \|_{L_y^2}+ \|\frakb_{\ulomega} \|_{L_\xi^\infty} \|y \ \fraka_{\ulomega}(y)\barvp(t,y)\|_{L_y^2}\\
&\lesssim \|\Vp\|_{L_y^\infty} \| \frakb_{\ulomega} \|_{W^{1,\infty}} \| \jy \fraka_{\ulomega}(y)\|_{L_y^2} \lesssim_{\omega_0} \eps \jt^{-\frac12}.
\end{split}
\end{equation*}
Similarly, we have 
\begin{equation*}
\begin{split}
\|\calK_{+, \ulomega}[\Vp(t)]\|_{H_\xi^1} &\lesssim \|\calK_{+, \ulomega}[\Vp(t)]\|_{L_\xi^2}+\|\pxi \calK_{+, \ulomega}[\Vp(t)]\|_{L_\xi^2} \\
&\lesssim \|\Vp\|_{L_y^\infty} \big(\| \frakb_{\ulomega} \|_{W^{1,\infty}}+\| \frakc_{\ulomega} \|_{W^{1,\infty}}\big)\big( \| \jy \fraka_{\ulomega}(y)\|_{L_y^2}+\| \jy \fraka_{\ulomega}'(y)\|_{L_y^2}\big) \lesssim \eps_{\omega_0} \jt^{-\frac12}.
\end{split}
\end{equation*}
The bounds for $\|\calL_{-, \ulomega}[\Vp(t)]\|_{H_\xi^1}$ and $\|\calK_{-, \ulomega}[\Vp(t)]\|_{H_\xi^1}$ proceeds in the same manner, and we conclude  \eqref{equ:consequences_calL_bounds} and \eqref{equ:consequences_calK_bounds}. Next, we prove the estimates \eqref{equ:consequences_ulPe_Mod_bounds}-\eqref{equ:consequences_calE3_bounds} item by item. By \eqref{equ:setup_definition_calM2}, \eqref{equ:setup_definition_Mod}, and the orthogonality condition \eqref{equ:orthogonality_PeU}, we write 
\begin{equation}
\begin{split}
\ulPe\calMod(t) &= -i(\dot{\gamma}+p^2-\omega-p\dot{\sigma}) \ulPe\big(e^{i(p-\ulp)y \sigma_3}Y_{1,\omega}- Y_{1,\ulomega}\big) -i\dot{\omega} \ulPe\big(e^{i(p-\ulp)y \sigma_3}Y_{2,\omega}- Y_{2,\ulomega}\big)\\
&\quad +i(\dot{\sigma}-2p)\ulPe\big(e^{i(p-\ulp)y \sigma_3}Y_{3,\omega}- Y_{3,\ulomega}\big)-i \dot{p}\ulPe\big(e^{i(p-\ulp)y \sigma_3}Y_{4,\omega}- Y_{4,\ulomega}\big).
\end{split}
\end{equation}
Since the eigenfunctions are Schwartz functions, we infer from the boundedness of $\ulPe$ given by Lemma~\ref{lemma: L2 decomposition} and the bounds \eqref{equ:consequences_assumption1},  \eqref{equ:consequences_aux_bound_modulation1} that 
\begin{equation*}
\begin{split}
&\| \jym \ulPe\calMod(t)\|_{L_y^2}+\| \jym \py\ulPe\calMod(t)\|_{L_y^2} \\
&\lesssim \eps \jt^{-1} \sum_{k=1}^4 \left(\left \| \jym\big(e^{i(p-\ulp)y \sigma_3}Y_{k,\omega}- Y_{k,\ulomega}\big)\right\|_{L_y^2}+\left \| \jym \py\big(e^{i(p-\ulp)y \sigma_3}Y_{k,\omega}- Y_{k,\ulomega}\big)\right\|_{L_y^2}\right)\\
&\lesssim \eps \jt^{-1} \big(|\omega(t)-\ulomega| + |p(t)-\ulp|\big)\lesssim \eps^2 \jt^{-2+\delta}.
\end{split}
\end{equation*}
Similarly, using the fact that $\phi_\omega(y),\phi_\ulomega(y)$ are Schwartz functions, we obtain from \eqref{equ:def-calE1}, \eqref{equ:def-calE2}, \eqref{equ:consequences_assumption1}, \eqref{equ:consequences_U_disp_decay}, \eqref{equ:consequences_pyU_disp_decay} that 
\begin{equation*}
\begin{split}
&\| \jym \calE_1(t)\|_{L_y^2}+\| \jym \py \calE_1(t)\|_{L_y^2} \\
&= \big\| \jym \big(\calV(\omega)\Vp(t) - \calV(\ulomega)\Vp(t)\big)\big\|_{L_y^2}+ \big\| \jym \py \big(\calV(\omega)\Vp(t) - \calV(\ulomega)\Vp(t)\big) \big\|_{L_y^2}\\
&\lesssim \big(|\omega(t)-\ulomega| + |p(t)-\ulp|\big)\big(\|\Vp\|_{L_y^\infty} + \|\jy^{-3}\py \Vp\|_{L_y^2}\big)\lesssim \eps^2 \jt^{-\frac32 + \delta},
\end{split}
\end{equation*}
and
\begin{equation*}
\begin{split}
&\| \jym \calE_2(t)\|_{L_y^2}+\| \jym \py \calE_2(t)\|_{L_y^2} \\
&= \| \jym \big(\calQ_{\omega}(\Vp(t)) - \calQ_{\ulomega}(\Vp(t))\big)\|_{L_y^2}+ \| \jym \py \big(\calQ_{\omega}(\Vp(t)) - \calQ_{\ulomega}(\Vp(t))\big)\|_{L_y^2}\\
&\lesssim \big(|\omega(t)-\ulomega| + |p(t)-\ulp|\big)\big(\|\Vp\|_{L_y^\infty} + \|\jy^{-3}\py \Vp\|_{L_y^2}\big)^2 \lesssim \eps^3 \jt^{-2 + \delta}.
\end{split}
\end{equation*}
Hence, we have proved \eqref{equ:consequences_ulPe_Mod_bounds}-\eqref{equ:consequences_calE2_bounds}.

For the final remainder term, we recall from \eqref{equ:def_calE_3} that 
\begin{equation*}
\calE_3 = \calQ_{\ulomega}(\Vp)-\calQ_{\ulomega}(\ulPe\Vp)+\calC(\Vp)-\calC(\ulPe\Vp).
\end{equation*}
Upon inserting the spectral decomposition \eqref{equ:consequences_decomposition_radiation} for $\Vp(t,y)$ into the nonlinear terms, it becomes clear that $\calE_3(t,y)$ consists of quadratic or cubic terms with at least one input given by $d_{k,\ulomega,\ulp}(t)Y_{k,\ulomega}(y)$ for some $k \in \{1,\ldots,4\}$ and with either one or two copies of $\ulPe \Vp(t,y)$. Since the eigenfunctions and its derivative are spatially localized, we infer from \eqref{equ:consequences_ulPe_U_disp_decay}, \eqref{equ:consequences_pyulPeU_disp_decay}, \eqref{equ:consequences_discrete_components_decay} that 
\begin{equation*}
\begin{split}
&\| \jym \big(\calQ_{\ulomega}(\Vp(t)) - \calQ_{\ulomega}(\ulPe\Vp(t))\big)\|_{L_y^2}+ \| \jym \py \big(\calQ_{\ulomega}(\Vp(t)) - \calQ_{\ulomega}(\ulPe\Vp(t))\big)\|_{L_y^2}\\
&\lesssim \big(\|\ulPe\Vp\|_{L_y^\infty} + \|\jy^{-3} \py \ulPe\Vp\|_{L_y^2} + |d_{1,\ulomega,\ulp}(t)| + \cdots +|d_{4,\ulomega,\ulp}(t)|\big)\big(|d_{1,\ulomega,\ulp}(t)| + \cdots +|d_{4,\ulomega,\ulp}(t)|\big)\\
&\lesssim \eps \jt^{-\frac12}\cdot \eps \jt^{-\frac32 + \delta} \lesssim \eps^2 \jt^{-2+\delta},
\end{split}
\end{equation*}
as well as
\begin{equation*}
\begin{split}
&\| \jym \big(\calC(\Vp(t)) - \calC(\ulPe\Vp(t))\big)\|_{L_y^2}+ \| \jym \py \big(\calC(\Vp(t)) - \calC(\ulPe\Vp(t))\big)\|_{L_y^2}\\
&\lesssim \big(\|\ulPe\Vp\|_{L_y^\infty} + \|\jy^{-3} \py \ulPe\Vp\|_{L_y^2} + |d_{1,\ulomega,\ulp}(t)| + \cdots +|d_{4,\ulomega,\ulp}(t)|\big)^2\big(|d_{1,\ulomega,\ulp}(t)| + \cdots +|d_{4,\ulomega,\ulp}(t)|\big)\\
&\lesssim \eps^2 \jt^{-1}\cdot \eps \jt^{-\frac32 + \delta} \lesssim \eps^3 \jt^{-\frac52+\delta}.
\end{split}
\end{equation*}
The two preceeding estimates imply the asserted bound \eqref{equ:consequences_calE3_bounds}.

\medskip \noindent \underline{Proof of (7) and (8).} The main idea to prove the bounds \eqref{equ:consequences_h12_decay}--\eqref{equ:preparation_flat_Schrodinger_wave_bound5} is to use the dispersive and improved local decay estimates in Section~\ref{sec:linear_decay}, and to use the $L^2$ boundedness of the distorted Fourier transform. Since the proof of  \eqref{equ:consequences_h12_decay}--\eqref{equ:preparation_flat_Schrodinger_wave_bound5} is essentially identical to the proof of (7.30)--(7.39) in \cite[Corollary~7.3]{LL24} from our previous work, and we will omit repeating the lengthy proof here.  It remains to establish \eqref{equ:prep_flat_local_decay} assuming the special case $\fraka(\xi) = \frakb(\xi)= \xi(|\xi|+i\sqrt{\ulomega})^{-2}$. For short times $t \leq 1$, the bound is implied by \eqref{equ:preparation_flat_Schrodinger_wave_bound_L2} so we assume that $t \geq 1$. By  integration by parts, we have 
\begin{equation*}
\begin{split}
v_{+,\ulomega,\ulp}(t,y) &= \frac{i}{2t} e^{-it\ulomega}\int_\bbR e^{-it\xi^2} \pxi\left( e^{iy\xi} \frac{ \tilf_{+, \ulomega,\ulp}(t,\xi)}{(|\xi|+i\sqrt{\ulomega})^2} \right)\,\ud \xi\\
&=-\frac{y}{2t} e^{-it\ulomega}\int_\bbR e^{-it\xi^2}  e^{iy\xi} \frac{ \tilf_{+, \ulomega,\ulp}(t,\xi)}{(|\xi|+i\sqrt{\ulomega})^2} \,\ud \xi + \frac{i}{2t} e^{-it\ulomega}\int_\bbR e^{-it\xi^2} e^{iy\xi} \pxi\left(  \frac{ \tilf_{+, \ulomega,\ulp}(t,\xi)}{(|\xi|+i\sqrt{\ulomega})^2} \right)\,\ud \xi.
\end{split}
\end{equation*}
Using Plancherel's identity, \eqref{equ:consequences_assumption2}, and \eqref{equ:consequences_sobolev_bound_profile}, we obtain the asserted bound 
\begin{equation*}
\big\| \jy^{-1} v_{+,\ulomega,\ulp}(t,y)\big\|_{L_y^2} \lesssim t^{-1} \|\tilf_{+, \ulomega,\ulp}(t,\xi)\|_{H_\xi^1} \lesssim \eps \jt^{-1+\delta}.
\end{equation*}
The proof of \eqref{equ:prep_flat_local_decay} for $v_{-,\ulomega,\ulp}(t,y)$ is analogous. 
\end{proof}

\subsection{Proof of Theorem~\ref{thm:main_theorem}}

 Now we are in the position to prove Theorem~\ref{thm:main_theorem} based on the conclusions of Proposition~\ref{prop:modulation_parameters}, Proposition~\ref{prop:profile_bounds}, and Corollary~\ref{cor:consequences}.

 \begin{proof}[Proof of Theorem~\ref{thm:main_theorem}]
We fix $\omega_0 \in (0,\infty)$, and let $0 < \varepsilon_0 \ll 1$ and $C_0 \geq 1$ be the constants from the statements of Proposition~\ref{prop:modulation_parameters} and Proposition~\ref{prop:profile_bounds}. Let $\gamma_0,p_0,\sigma_0 \in \bbR$ be given and let $u_0 \in H^1_x(\bbR) \cap L^{2,1}_x(\bbR)$ satisfies the smallness initial condition  \eqref{equ:theorem_statement_smallness_initial_condition} in the statement of Theorem~\ref{thm:main_theorem}.

From Lemma~\ref{lem:setup_local_existence}, we have a local-in-time $H^1_x \cap L^{2,1}_x$--solution $\psi(t,x)$ to \eqref{equ:cubic_NLS} for the initial condition \eqref{equ:theorem_statement_initial_condition} given by 
 \begin{equation*}
     \psi(0,x) = e^{i p_0(x-\sigma_0)} e^{i \gamma_0} \bigl( \phi_{\omega_0}(x-\sigma_0) + u_0(x-\sigma_0) \bigr).
 \end{equation*}
The solution $\psi(t,x)$ exists on a maximal interval of existence $[0,T_\ast)$ for some $0 < T_\ast \leq \infty$, and the continuation criterion~\eqref{equ:setup_continuation_criterion} holds, and Proposition~\ref{prop:modulation_and_orbital} provides unique continuously differentiable paths $(\omega, \gamma,p,\sigma) \colon [0, T_\ast) \to (0,\infty) \times \bbR^3$ so that the decomposition of the solution
 \begin{equation*}
     \psi(t,x) = e^{i p(t)(x-\sigma(t))} e^{i\gamma(t)} \bigl( \phi_{\omega(t)}(x-\sigma(t)) + u(t,x-\sigma(t)) \bigr), \quad 0 \leq t < T_\ast,
 \end{equation*}
 satisfies $(1)$--$(5)$ in the statement of Proposition~\ref{prop:modulation_and_orbital}. Let $T_0 \in [0,T_\ast]$ be the exit time given by 
 \begin{equation*}
     \begin{aligned}
         T_0 := \sup \, \biggl\{ 0 \leq T < T_\ast \, \bigg| \quad &\bigl\| \bigl( \tilf_{+, \omega(T),p(T)}(t), \tilf_{-, \omega(T),p(T)}(t) \bigr) \bigr\|_{X(T)} \leq C_0 \varepsilon,\\
 &\quad \sup_{0 \leq t \leq T} \, \jt^{1-\delta} \Big(\bigl|\omega(t) - \omega(T)\bigr|+\bigl|p(t) - p(T)\bigr| \Big)\leq C_0 \varepsilon \biggr\},
     \end{aligned}
 \end{equation*}
 where $\bigl( \tilf_{+, \omega(T), p(T)}(t), \tilf_{-, \omega(T), p(T)}(t) \bigr)$ denote the components of the distorted Fourier transform~\eqref{equ:profile_dFT} of the profile of the radiation term \eqref{equ: V-definition} relative to the linearized operator $\calH\bigl(\omega(T)\big)$ with parameters $\omega(T)$ and $p(T)$. By the local existence theory one has $T_0 > 0$. 

We claim that $T_0 = T_\ast$. We argue by contradiction and suppose that $T_0 < T_\ast$. Let $\{T_n\}_{n \in \bbN} \subset [0,T_0)$ be a strictly monotone increasing sequence such that $T_n \nearrow T_0$ as $n \to \infty$. Since the paths $(\omega, \gamma,p,\sigma) \colon [0,T_\ast) \to (0,\infty) \times \bbR$ are continuous, it follows that
 \begin{equation*}
     \sup_{0 \leq t \leq T_0} \, \jt^{1-\delta} \Big(\bigl|\omega(t) - \omega(T_0)\bigr| + \big|p(t)-p(T_0)\big|\Big) \leq C_0 \varepsilon.
 \end{equation*}

By the mapping properties for the distorted Fourier transform from Proposition~\ref{prop:mapping_properties_dist_FT} we have 
 \begin{equation*} 
     \begin{aligned}
         \bigl\| \bigl( \tilf_{+, \omega(T_0), p(T_0)}(0,\xi), \tilf_{-, \omega(T_0),p(T_0)}(0,\xi) \bigr) \bigr\|_{L^\infty_\xi} + \bigl\|  \bigl( \pxi \tilf_{+, \omega(T_0), p(T_0)}(0,\xi), \pxi \tilf_{-, \omega(T_0), p(T_0)}(0,\xi) \bigr) \bigr\|_{L^2_\xi} \\
\lesssim_{\omega_0} \|u_0\|_{H^1_x \cap L^{2,1}_x} \lesssim \varepsilon.
     \end{aligned}
 \end{equation*}
Hence, we can use Proposition~\ref{prop:profile_bounds} with $\ulomega = \omega(T_0)$ and  $\ulp=p(T_0)$ fixed to conclude via a continuity argument that
 \begin{equation*}
     \bigl\| \bigl( \tilf_{+, \omega(T_0), p(T_0)}(t), \tilf_{-, \omega(T_0), p(T_0)}(t) \bigr) \bigr\|_{X(T_0)} \leq C_0 \varepsilon.
 \end{equation*}
Keeping $\omega(T_0)$ and $p(T_0)$ fixed, it follows by continuity that there exists $T_0 < \widetilde{T} < T_\ast$ such that
 \begin{align} 
     \bigl\| \bigl( \tilf_{+, \omega(T_0), p(T_0)}(t), \tilf_{-, \omega(T_0), p(T_0)}(t) \bigr) \bigr\|_{X(\wtilT)} &\leq 2C_0\varepsilon, \label{equ:proof_theorem_aux1} \\
     \sup_{0 \leq t \leq \wtilT} \, \jt^{1-\delta} \Big( \bigl|\omega(t) - \omega(T_0)\bigr| + \big|p(t)-p(T_0)\big|\Big) &\leq 2C_0\varepsilon. \label{equ:proof_theorem_aux2}
 \end{align}
Thanks to \eqref{equ:proof_theorem_aux1} and \eqref{equ:proof_theorem_aux2}, the assumptions of  Proposition~\ref{prop:modulation_parameters} are satisfied on  $[0, \wtilT]$ with $\ulomega = \omega(T_0)$ and $\ulp = p(T_0)$ fixed, and we infer from \eqref{equ:prop_modulation_parameters_conclusion} that
 \begin{equation} \label{equ:theorem_proof_wtilT_bound1}
     \begin{aligned}
         \sup_{0 \leq t \leq \wtilT} \, \jt^{1-\delta} \Big(\bigl|\omega(t) - \omega(\wtilT)\bigr| + \big|p(t)-p(\wtilT)\big| \Big) &\leq C_0 \varepsilon.
     \end{aligned}
 \end{equation}
Having established \eqref{equ:theorem_proof_wtilT_bound1}, we can now run a continuity argument using Proposition~\ref{prop:profile_bounds} on the time interval $[0, \wtilT]$ with $\ulomega = \omega(\wtilT)$ and $\ulp = p(\wtilT)$ to deduce that
 \begin{equation} \label{equ:theorem_proof_wtilT_bound2}
     \bigl\| \bigl( \tilf_{+, \omega(\wtilT), p(\wtilT)}(t), \tilf_{-, \omega(\wtilT), p(\wtilT)}(t) \bigr) \bigr\|_{X(\wtilT)} \leq C_0 \varepsilon.
 \end{equation}
But then \eqref{equ:theorem_proof_wtilT_bound1} and \eqref{equ:theorem_proof_wtilT_bound2} are a contradiction to $\wtilT > T_0$. Hence we conclude that $T_0 = T_\ast$.

Next, we claim that $T_\ast = \infty$. By the definition of the exit time we therefore have uniformly for all $0 < T < T_\ast$ that $ \bigl\| \bigl( \tilf_{+, \omega(T), p(T)}(t), \tilf_{-, \omega(T), p(T)}(t) \bigr) \bigr\|_{X(T)} \leq C_0\varepsilon$. In addition to  \eqref{equ:consequences_sobolev_bound_profile}, \eqref{equ:consequences_decomposition_radiation}, and \eqref{equ:consequences_discrete_components_decay}, the mapping properties of the distorted Fourier transform from Proposition~\ref{prop:mapping_properties_dist_FT}, we conclude that $\|u(t)\|_{H_y^1 \cap L_y^{2,1}} \lesssim \eps \jap{T_\ast}^\delta < \infty$ for all $0 \leq t <T_*$. Thus, the continuation criterion \eqref{equ:setup_continuation_criterion} implies $T_\ast = \infty$. 

Finally, we determine the scaling and momentum parameters of the final solition. Let $\{t_n\}_{n\in\bbN} \subset (0, \infty)$ be a strictly increasing sequence such that $t_n \nearrow \infty$ as $n \to \infty$. Using that for each $n \in \bbN$ we have
 \begin{equation}  \label{equ:proof_theorem_aux3}
     \sup_{0 \leq t \leq t_n} \, \jt^{1-\delta} \Big( \big|\omega(t) - \omega(t_n)\big| + \big|p(t)-p(t_n)\big|\Big) \leq C_0 \varepsilon,
 \end{equation}
we conclude from \eqref{equ:setup_comparison_estimate} that $\{ \omega(t_n) \}_n \subset (0,\infty)$ and $\{p(t_n)\} \subset \bbR$ are Cauchy sequences that satisfy
\begin{equation*}
|p(t_n)-p_0| \leq \frac12 \omega_0, \quad |\omega(t_n)-\omega_0| \leq \frac12 \omega_0, \quad \forall n \in \bbN.
\end{equation*}
Hence, there exists $\omega_\infty \in (0, \infty)$ with $|\omega_\infty - \omega_0| \leq \frac12 \omega_0$ such that $\omega(t_n) \to \omega_\infty$ as $n \to \infty$, and there exists $p_\infty \in \bbR$ with $|p_\infty - p_0|\leq \frac12 \omega_0$ such that $p(t_n) \to p_\infty$ as $n \to \infty$. The final values $\omega_\infty, p_\infty$ are independent of the chosen sequence $\{t_n\}_{n\in\bbN}$ thanks to the continuity of the paths $\omega(t),p(t)$. Moreover,  \eqref{equ:proof_theorem_aux3} implies the decay rate
 \begin{equation} \label{equ:proof_theorem_aux4}
     \sup_{0 \leq t < \infty} \, \jt^{1-\delta} \Big( \big|\omega(t) - \omega_\infty\big| + \big| p(t) - p_\infty \big| \Big)\leq C_0 \varepsilon.
 \end{equation}
Then from \eqref{equ:proof_theorem_aux4} and from Proposition~\ref{prop:profile_bounds} with $\ulomega = \omega_\infty$ and $\ulp = p_\infty$ fixed, we conclude again by continuity  that 
 \begin{equation} \label{equ:proof_theorem_uniform_profile_bounds}
  \sup_{0 \leq t < \infty} \, \Bigl( \bigl\| \bigl( \tilf_{+, \omega_\infty,p_\infty}(t), \tilf_{-, \omega_\infty,p_\infty}(t) \bigr) \bigr\|_{L^\infty_\xi} + \jt^{-\delta} \bigl\|  \bigl( \pxi \tilf_{+, \omega_\infty,p_\infty}(t), \pxi \tilf_{-, \omega_\infty,p_\infty}(t) \bigr) \bigr\|_{L^2_\xi} \Bigr) \leq C_0 \varepsilon.
 \end{equation}

Thanks to \eqref{equ:proof_theorem_aux4} and \eqref{equ:proof_theorem_uniform_profile_bounds}, the assumptions of Corollary~\ref{cor:consequences} are satisfied with $\ulomega = \omega_\infty$ and $\ulp = p_\infty$. We obtain the desired dispersive decay rate \eqref{equ:theorem_statement_decay_radiation} in Theorem~\ref{thm:main_theorem} as a consequence of \eqref{equ:consequences_U_disp_decay} from Corollary~\ref{cor:consequences}. The asserted asymptotics of the modulation paramteters \eqref{equ:theorem_statement_decay_modulation_parameters} and \eqref{equ:theorem_statement_decay_modulation_parameters2} in Theorem~\ref{thm:main_theorem} follows from \eqref{equ:proof_theorem_aux4} and \eqref{equ:consequences_aux_bound_modulation2} in combination with the formula \eqref{equ:theta_identities} respectively.

Lastly, it remains to prove the asymptotic formula \eqref{eqn: theorem-u-asymptotics} for the radiation term for Theorem~\ref{thm:main_theorem}. Using the Cauchy-in-time bounds \eqref{eqn: Cauchy-in-time-estimate}--\eqref{eqn: Cauchy-in-time-estimate2} in Proposition~\ref{prop:pointwise_estimate} with $\ulomega  = \omega_\infty$ and $\ulp = p_\infty$,  we infer that there exists asymptotic effective profiles $V_+,V_- \in L_\xi^\infty(\bbR)$ such that 
 \begin{align}
 \Big\Vert e^{-i\Lambda_{+,\infty}(t,\xi)} e^{i\theta_{1,\infty}(t)}e^{-i\theta_{2,\infty}(t)\xi}\tilf_{+,\omega_\infty,p_\infty}(t,\xi) -V_+(\xi)\Big \|_{L_\xi^\infty} &\lesssim \varepsilon^2 t^{-\frac{1}{10}+\delta}, \quad t \geq 1,\label{equ:proof_theorem_cauchy1}\\
  \Big\Vert e^{i\Lambda_{-,\infty}(t,\xi)} e^{-i\theta_{1,\infty}(t)}e^{-i\theta_{2,\infty}(t)\xi}\tilf_{-,\omega_\infty,p_\infty}(t,\xi) -V_-(\xi) \Big \|_{L_\xi^\infty} &\lesssim \varepsilon^2 t^{-\frac{1}{10}+\delta},  \quad t \geq 1,\label{equ:proof_theorem_cauchy2}
 \end{align}
where
\begin{equation*}
\Lambda_{\pm,\infty}(t,\xi) := \frac{1}{2}\int_1^t \frac{1}{s} |\tilf_{\pm, \omega_\infty,p_\infty}(s,\xi)|^2\,\ud s
\end{equation*}
and $\theta_{1,\infty}(t)$, $\theta_{2,\infty}(t)$ defined in \eqref{equ:theorem_asymptotics_notation}. We then multiply the differential equation \eqref{eqn: effective-ODE-profile} with the integrating factor $e^{-i|V_+(\xi)|^2 \log(t)}$ and repeat the arguments of the proof of Proposition~\ref{prop:pointwise_estimate} using \eqref{equ:proof_theorem_cauchy1} to infer that there exists an asymptotic profile $W_+ \in L_\xi^\infty(\bbR)$ with $|V_+(\xi)|= |W_+(\xi)|$ such that 
\begin{equation*}
\Big\Vert e^{-i|W_+(\xi)|^2\log(t)} e^{i\theta_{1,\infty}(t)}e^{-i\theta_{2,\infty}(t)\xi}\tilf_{+,\omega_\infty,p_\infty}(t,\xi) -W_+(\xi)\Big \|_{L_\xi^\infty} \lesssim \varepsilon^2 t^{-\frac{1}{10}+\delta}, \quad t \geq 1.
\end{equation*}
We then repeat these same arguments to infer the existence of $W_- \in L_\xi^\infty(\bbR)$ that satisfies $|V_-(\xi)|=|W_-(\xi)|$ and 
\begin{equation*}
\Big\Vert e^{i|W_-(\xi)|^2\log(t)} e^{-i\theta_{1,\infty}(t)}e^{-i\theta_{2,\infty}(t)\xi}\tilf_{-,\omega_\infty,p_\infty}(t,\xi) -W_-(\xi)\Big \|_{L_\xi^\infty} \lesssim \varepsilon^2 t^{-\frac{1}{10}+\delta}, \quad t \geq 1.
\end{equation*}
In the moving frame coordinate $y= x - \sigma(t)$, we recall that
\begin{equation*}
U(t,y) = e^{i(p_\infty - p(t))y\sigma_3}U_{p_{\infty}}(t,y) = e^{i(p_\infty - p(t))y\sigma_3} \Big(\ulPe U_{p_\infty}(t,y) + \ulPd U_{p_{\infty}}(t,y) \Big) , 
\end{equation*}
where the discrete component $\ulPd U_{p_{\infty}}(t,y) := \sum_{k=1}^4 d_{k,\omega_\infty,p_\infty}(t)Y_{k,\omega_\infty}(y)$ enjoy the better decay rate $|d_{k,\omega_\infty,p_\infty}(t)|\lesssim \eps \jt^{-\frac32 + \delta}$ thanks to \eqref{equ:consequences_discrete_components_decay}. The asserted asymptotics for the radiation term \eqref{eqn: theorem-u-asymptotics} now follows from applying the asymptotic formula \eqref{eqn:linear_dispersive_decay} to the representation of $\ulPe U_{p_\infty}(t,y)$ given in \eqref{equ:setup_v_e_repformula}. This finishes the proof of Theorem~\ref{thm:main_theorem}.
\end{proof}

\section{Control of the Modulation Parameters}\label{sec:modulation_parameters}

Our goal of this section is to show that  uniformly for all $0 \leq t \leq T$ it holds that 
\begin{equation}\label{equ:proof_mod_bootstrap}
|\omega(t) - \omega(T)| + |p(t) - p(T)| \lesssim \eps^2 \jt^{-1+\delta}.
\end{equation}
The asserted improved bootstrap bound \eqref{equ:prop_modulation_parameters_conclusion} in Proposition~\ref{prop:modulation_parameters} then follows from \eqref{equ:proof_mod_bootstrap}. Our strategy to prove \eqref{equ:proof_mod_bootstrap} essentially follows the proof in  \cite[Sect.~8]{LL24} with a few more technical details. The starting point is to use the fundamental theorem of calculus to write 
\begin{equation}\label{equ:proof_mod_FTC}
\omega(t) - \omega(T) = - \int_t^T \dot{\omega}(s) \,\ud s, \quad  p(t) - p(T)= - \int_t^T \dot{p}(s) \,\ud s,
\end{equation}
and insert the equations for $\dot{\omega}(s)$ and $\dot{p}(s)$ from the modulation equations \eqref{equ:setup_modulation_equation} to infer decay. One then has to prove an almost twice integrable rate $\eps^2\js^{-2+\delta}$ for the contributions to the integrand of \eqref{equ:proof_mod_FTC}. In the following preparatory lemma, we extract the leading quadratic and cubic terms for the right-hand of the modulation equations \eqref{equ:setup_modulation_equation}.

\begin{lemma}\label{lem:prep_modulation}
Suppose the assumptions in the statement of Proposition~\ref{prop:modulation_parameters} are in place. Then for each $j \in \{1,\ldots,4\}$ and for all $0 \leq s \leq T$, we have
\begin{equation}
\langle i \calN(U),\sigma_2 Y_{j,\omega(s)}\rangle = \langle i \calQ_{\ulomega}(\ulPe \Vp),\sigma_2 Y_{j,\ulomega}\rangle + \langle i \calC(\ulPe \Vp),\sigma_2 Y_{j,\ulomega}\rangle + \wtilcalE_j(s)
\end{equation}
with a remainder $\wtilcalE_j(s)$ that satisfies the estimate 
\begin{equation}\label{equ:BS_modulation_error}
\sup_{0 \leq s \leq T}	\js^{2-\delta}\big|\wtilcalE_j(s) \big| \lesssim \eps^2 .
\end{equation}

\end{lemma}
\begin{proof}
For each $j\in \{1,\ldots,4\}$ we begin by decomposing the remainder term as follows
\begin{equation*}
	\begin{split}
\wtilcalE_j(s) &= 	\langle i \calN(U),\sigma_2 Y_{j,\omega(s)}\rangle - \langle i \calQ_{\ulomega}(\ulPe V),\sigma_2 Y_{j,\ulomega}\rangle - \langle i \calC(\ulPe V),\sigma_2 Y_{j,\ulomega}\rangle\\
&= \langle i\calQ_{\omega(s)}(U),\sigma_{2} (Y_{j,\omega(s)}-Y_{j,\ulomega})\rangle +  \langle i \big(\calQ_{\omega(s)}(U)-\calQ_{\ulomega}(U) \big),\sigma_2 Y_{j,\ulomega}\rangle +\langle i\calC(U),\sigma_{2} (Y_{j,\omega(s)}-Y_{j,\ulomega})\rangle  \\
&\quad + \langle i \big(\calQ_{\ulomega}(U)-\calQ_{\ulomega}(\Vp) \big),\sigma_2 Y_{j,\ulomega}\rangle + \langle i \big(\calC(U)-\calC(\Vp) \big),\sigma_2 Y_{j,\ulomega}\rangle + \langle i \calE_3,\sigma_2 Y_{j,\ulomega}\rangle
	\end{split}
\end{equation*}
where we recall $\calE_3$ from \eqref{equ:def_calE_3}. Then, using the bootstrap assumptions \eqref{equ:prop_modulation_parameters_assumption1} and the dispersive decay \eqref{equ:consequences_U_disp_decay}, we bound the terms 
\begin{equation*}
	\begin{split}
&|\langle i\calQ_{\omega(s)}(U),\sigma_{2} (Y_{j,\omega(s)}-Y_{j,\ulomega})\rangle| +  |\langle i \big(\calQ_{\omega(s)}(U)-\calQ_{\ulomega}(U) \big),\sigma_2 Y_{j,\ulomega}\rangle| + |\langle i\calC(U),\sigma_{2} (Y_{j,\omega(s)}-Y_{j,\ulomega})\rangle|		\\
&\quad + |\langle i \big(\calQ_{\ulomega}(U)-\calQ_{\ulomega}(\Vp) \big),\sigma_2 Y_{j,\ulomega}\rangle| + |\langle i \big(\calC(U)-\calC(\Vp) \big),\sigma_2 Y_{j,\ulomega}\rangle |\\
&\lesssim |\omega(s)-\ulomega| \big( \|U(s) \|_{L_y^\infty}^2 + \|U(s) \|_{L_y^\infty}^3 \big) + |p(s)-\ulp| \big( \|\Vp(s) \|_{L_y^\infty}^2 + \|\Vp(s) \|_{L_y^\infty}^3 \big) \lesssim \eps^3 \js^{-2+\delta},
	\end{split}
\end{equation*}
From \eqref{equ:consequences_calE3_bounds} we have $|\langle i \calE_3,\sigma_2 Y_{j,\ulomega}\rangle| \lesssim \eps^2 \js^{-2+\delta}$. Hence, we conclude \eqref{equ:BS_modulation_error} by combining the preceeding bounds.
\end{proof}

In the next lemma, we compute the spectral distributions which arises from the leading quadratic nonlinearities in the right-hand side of \eqref{equ:proof_mod_FTC}. 
\begin{lemma} \label{lem:null_structure_modulation}
Let 
\begin{equation}
\kappa_{\ulomega}(\xi_1+\xi_2) := \frac{\xi_1+\xi_2}{\sqrt{\ulomega}} \cosech\left(\frac{\pi}{2} \frac{\xi_1+\xi_2}{\sqrt{\ulomega}}\right) = \frac{\sqrt{\ulomega}}{\pi} \int_\bbR e^{iy(\xi_1+\xi_2)}\sech^2(\sqrt{\ulomega} y) \,\ud y.
\end{equation}
We have
\begin{align}
\nu_{++,\ulomega}(\xi_1, \xi_2) &= \bigl(\xi_1^2+\xi_2^2+2\ulomega\bigr) \frac{\sqrt{\ulomega}}{12}\frac{(\xi_1^2-4\xi_1\xi_2 + \xi_2^2-2\ulomega)}{(\vert \xi_1\vert - i\sqrt{\ulomega})^2( \vert \xi_2\vert-i\sqrt{\ulomega} )^2} \kappa_{\ulomega}(\xi_1+\xi_2) \label{eqn: nu++},	\\
\nu_{+-,\ulomega}(\xi_1, \xi_2) &= \bigl( \xi_1^2-\xi_2^2 \bigr) \frac{\sqrt{\ulomega}}{12}\frac{(\xi_1^2+2\xi_1\xi_2+\xi_2^2+4\ulomega)}{(\vert \xi_1\vert - i\sqrt{\ulomega})^2( \vert \xi_2\vert-i\sqrt{\ulomega} )^2} \kappa_{\ulomega}(\xi_1+\xi_2)	\label{eqn: nu+-},\\
\nu_{--,\ulomega}(\xi_1, \xi_2) &= - \nu_{++,\ulomega}(\xi_1, \xi_2),\label{eqn: nu--}
\end{align}
as well as 
\begin{align}
\lambda_{++,\ulomega}(\xi_1, \xi_2) &= \bigl(\xi_1^2+\xi_2^2+2\ulomega\bigr) \frac{i\sqrt{\ulomega}}{6}\frac{(\xi_1^2-\xi_1\xi_2+\xi_2^2+\ulomega)}{(\vert \xi_1\vert - i\sqrt{\ulomega})^2( \vert \xi_2\vert-i\sqrt{\ulomega} )^2}  (\xi_1+\xi_2)\kappa_{\ulomega}(\xi_1+\xi_2) \label{eqn: nu_p++},	\\
\lambda_{+-,\ulomega}(\xi_1, \xi_2) &= \bigl( \xi_1^2-\xi_2^2 \bigr) \frac{i\sqrt{\ulomega}}{6}\frac{(\xi_1^2-\xi_1\xi_2+\xi_2^2+\ulomega)}{(\vert \xi_1\vert - i\sqrt{\ulomega})^2( \vert \xi_2\vert-i\sqrt{\ulomega} )^2} (\xi_1-\xi_2)\kappa_{\ulomega}(\xi_1+\xi_2),	\label{eqn: nu_p+-}\\
\lambda_{--,\ulomega}(\xi_1,\xi_2) &= \lambda_{++,\ulomega}(\xi_1,\xi_2).\label{eqn: nu_p--}
\end{align}
See \eqref{equ:nu-NSD} and \eqref{equ:lambda-NSD} for the above definitions.
\end{lemma}
\begin{remark}\label{remark: generic_null_expression}
From the above explicit expressions, we observe that the distributions $\nu_{+-,\ulomega}(\xi_1,\xi_2)$ and $\lambda_{+-,\ulomega}(\xi_1,\xi_2)$ are a linear combination of terms of the form
\begin{equation}\label{equ:proof_generic_exp}
(\xi_1^2-\xi_2^2)\fraka(\xi_1)\frakb(\xi_2)\kappa(\xi_1+\xi_2)
\end{equation}
where the symbols $\fraka(\xi_1)$, $\frakb(\xi_2)$ are of the form \eqref{eqn: symbol_frakb}, and where
\begin{equation*}
\kappa(\xi_1+\xi_2) = \int_\bbR e^{iy(\xi_1+\xi_2)}q(y)\,\ud y
\end{equation*}
for some Schwartz function $q(y)$. In order to see this property for $\lambda_{+-,\ulomega}(\xi_1,\xi_2)$ more clearly, we may rewrite the factor $(\xi_1-\xi_2)\kappa_{\ulomega}(\xi_1+\xi_2)$  in $\lambda_{+-,\ulomega}(\xi_1,\xi_2)$ as 
\begin{equation}\label{equ:rewrite_kappa}
(\xi_1-\xi_2)\kappa_{\ulomega}(\xi_1+\xi_2) = \tilde{\kappa}_\ulomega(\xi_1+\xi_2) - 2 \xi_2 \kappa_{\ulomega}(\xi_1+\xi_2) = -\tilde{\kappa}_\ulomega(\xi_1+\xi_2) + 2 \xi_2 \kappa_{\ulomega}(\xi_1+\xi_2)
\end{equation}
where 
\begin{equation*}
\tilde{\kappa}_\ulomega(\xi_1+\xi_2) := (\xi_1+\xi_2)\kappa_{\ulomega}(\xi_1+\xi_2) = i\int_\bbR e^{iy(\xi_1+\xi_2)}q_{\ulomega}'(y)\,\ud y, \quad q_{\ulomega}(y) := \frac{\sqrt{\ulomega}}{\pi}  \sech^2(\sqrt{\ulomega} y).
\end{equation*}
Thus, by expanding the terms in $\lambda_{+-,\ulomega}(\xi_1,\xi_2)$ using \eqref{equ:rewrite_kappa} repeatedly, we obtain the generic expression \eqref{equ:proof_generic_exp}.
\end{remark}
\begin{proof}[Proof of Lemma~\ref{lem:null_structure_modulation}]
The formulas for \eqref{eqn: nu++}--\eqref{eqn: nu--} were already proven in \cite[Lemma~6.3]{LL24}. The identity \eqref{eqn: nu_p--} follows from its definition in \eqref{equ:lambda-NSD}, and hence it remains to prove \eqref{eqn: nu_p++} and \eqref{eqn: nu_p+-}. 

Let $Q(y) := \sqrt{2}\sech(y)$ so that $\phi_{\ulomega}(y)\phi_\ulomega'(y) =  \ulomega^{\frac32}Q(\sqrt{\ulomega}y)Q'(\sqrt{\ulomega}y)$. We recall from \eqref{eqn:m-1,omega}--\eqref{eqn:m-2,omega} that $\Psi_{j,\ulomega}(y,\xi) = \Psi_{j,1}(\sqrt{\ulomega}y,\xi/\sqrt{\ulomega})$ for $j\in\{1,2\}$. By rescaling, we write 
\begin{equation}\label{equ:proof_scaling_momentum}
\lambda_{++,\ulomega}(\xi_1, \xi_2) = \ulomega \lambda_{++}\Big(\frac{\xi_1}{\sqrt{\ulomega}},\frac{\xi_2}{\sqrt{\ulomega}}\Big), \qquad \lambda_{+-,\ulomega}(\xi_1, \xi_2) = \ulomega \lambda_{+-}\Big(\frac{\xi_1}{\sqrt{\ulomega}},\frac{\xi_2}{\sqrt{\ulomega}}\Big)
\end{equation}
where
\begin{equation*}
\begin{split}
\lambda_{++}(\xi_1,\xi_2) &:= \int_\bbR \Bigl( \Psi_{1,1}(y,\xi_1) \Psi_{1,1}(y,\xi_2) + 2\Psi_{1,1}(y,\xi_1) \Psi_{2,1}(y,\xi_2) + 2\Psi_{2,1}(y,\xi_1) \Psi_{1,1}(y,\xi_2) \\
&\qquad \qquad +  \Psi_{2,1}(y,\xi_1) \Psi_{2,1}(y,\xi_2) \Bigr) Q(y)Q'(y) \, \ud y, \\
\lambda_{+-}(\xi_1,\xi_2) &:= \int_\bbR \Bigl( \Psi_{1,1}(y,\xi_1) \Psi_{2,1}(y,\xi_2) + 2\Psi_{1,1}(y,\xi_1) \Psi_{1,1}(y,\xi_2) + 2\Psi_{2,1}(y,\xi_1) \Psi_{2,1}(y,\xi_2) \\
&\qquad \qquad +  \Psi_{2,1}(y,\xi_1) \Psi_{1,1}(y,\xi_2) \Bigr) Q(y)Q'(y) \, \ud y.
\end{split}
\end{equation*}
We insert the expressions \eqref{eqn:m-1,omega}--\eqref{eqn:m-2,omega} for $\Psi_{1,1}(y,\cdot)$, $\Psi_{2,1}(y,\cdot)$ and repeatedly invoke the hyperbolic trigonometric identity $\tanh^2(y)=1 - \sech^2(y)$ to compute that
\begin{equation*}
\begin{split}
&\lambda_{++}(\xi_1,\xi_2)\\
&= \frac{1}{2\pi} \frac{1}{(|\xi_1|-i)^2} \frac{1}{(|\xi_2|-i)^2} \times \\
&\qquad \times \int_\bbR e^{iy(\xi_1+\xi_2)}\Big( \big(\xi_1 + i \tanh(y)\big)^2\big(\xi_2 + i \tanh(y)\big)^2 + 2\big(\xi_1 + i \tanh(y)\big)^2\sech^2(y) +\\
&\qquad \qquad \qquad \qquad \qquad + 2\sech^2(y) \big(\xi_2 + i \tanh(y)\big)^2+ \sech^4(y)\Big) \big(2\sech^2(y)\tanh(y)\big) \,\ud  y\\
&= \frac{1}{\pi} \frac{1}{(|\xi_1|-i)^2} \frac{1}{(|\xi_2|-i)^2} \times\\
&\qquad \times \Big( \big(1-\xi_1^2-\xi_2^2 - 4\xi_1\xi_2+\xi_1^2\xi_2^2\big) \kappa_{2,1} + \big(-6+4\xi_1\xi_2 + 3\xi_1^2 + 3\xi_2^2\big) \kappa_{4,1} + 6 \kappa_{6,1} \\
&\qquad \qquad +2i(\xi_1+\xi_2)(\xi_1\xi_2-1)\kappa_{2,0}+ 2i(\xi_1+\xi_2)(4-\xi_1\xi_2) \kappa_{4,0} - 6i(\xi_1+\xi_2)\kappa_{6,0}\Big)
\end{split}
\end{equation*}
and
\begin{equation*}
\begin{split}
&\lambda_{+-}(\xi_1,\xi_2)\\
&= \frac{1}{2\pi} \frac{1}{(|\xi_1|-i)^2} \frac{1}{(|\xi_2|-i)^2} \times \\
&\qquad \times \int_\bbR e^{iy(\xi_1+\xi_2)}\Big(2\big(\xi_1 + i \tanh(y)\big)^2 \big(\xi_2 + i \tanh(y)\big)^2 + \big(\xi_1 + i \tanh(y)\big)^2\sech^2(y) +\\
&\qquad \qquad \qquad \qquad \qquad + \sech^2(y) \big(\xi_2 + i \tanh(y)\big)^2+ 2\sech^4(y)\Big) \big(2\sech^2(y)\tanh(y)\big) \,\ud  y\\
&= \frac{1}{\pi} \frac{1}{(|\xi_1|-i)^2} \frac{1}{(|\xi_2|-i)^2} \times \\
&\qquad \times \Big( 2\big(1-\xi_1^2-\xi_2^2 - 4\xi_1\xi_2+\xi_1^2\xi_2^2\big) \kappa_{2,1} + \big(-6+8\xi_1\xi_2 + 3\xi_1^2 + 3\xi_2^2\big) \kappa_{4,1} + 6 \kappa_{6,1} \\
&\qquad \qquad +4i(\xi_1+\xi_2)(\xi_1\xi_2-1)\kappa_{2,0}+ 2i(\xi_1+\xi_2)(5-2\xi_1\xi_2) \kappa_{4,0} - 6i(\xi_1+\xi_2)\kappa_{6,0}\Big)
\end{split}
\end{equation*}
with the shortened notation 
\begin{equation}
\kappa_{m,n} \equiv \kappa_{m,n}(\xi_1+\xi_2) := \int_\bbR e^{i y(\xi_1+\xi_2)} \sech^m(y)\tanh^n(y) \,\ud y, \qquad m,n \in \bbN_0.
\end{equation}
The explicit formulas for $\kappa_{m,n}$ were derived in \cite[Appendix~A]{LL24}, and \cite[Lemma~A.1]{LL24} implies the following identities
\begin{equation*}
\begin{split}
\kappa_{2,1} &= \frac{1}{2}i(\xi_1+\xi_2) \kappa_{2,0},\quad \kappa_{4,1} = \frac{1}{24}i(\xi_1+\xi_2)\big(4 + (\xi_1+\xi_2)^2\big) \kappa_{2,0},\\
\kappa_{6,1} &= \frac{1}{720}i(\xi_1+\xi_2)\big(4 + (\xi_1+\xi_2)^2\big)\big(16 + (\xi_1+\xi_2)^2\big) \kappa_{2,0},\\
\kappa_{4,0} &= \frac{1}{6}\big(4 + (\xi_1+\xi_2)^2\big) \kappa_{2,0}, \quad \kappa_{6,0} = \frac{1}{120}\big(4 + (\xi_1+\xi_2)^2\big)\big(16 + (\xi_1+\xi_2)^2\big)  \kappa_{2,0}.
\end{split}
\end{equation*}

%\begin{equation*}
%\begin{split}
%\kappa_{2,1} &= \frac{1}{2}i(\xi_1+\xi_2) \kappa_{2,0},\\
%\kappa_{4,1} &= \frac{1}{24}i(\xi_1+\xi_2)\big(4 + (\xi_1+\xi_2)^2\big) \kappa_{2,0},\\
%\kappa_{6,1} &= \frac{1}{720}i(\xi_1+\xi_2)\big(4 + (\xi_1+\xi_2)^2\big)\big(16 + (\xi_1+\xi_2)^2\big) \kappa_{2,0},\\
%\kappa_{4,0} &= \frac{1}{6}\big(4 + (\xi_1+\xi_2)^2\big) \kappa_{2,0},\\
%\kappa_{6,0} &= \frac{1}{120}\big(4 + (\xi_1+\xi_2)^2\big)\big(16 + (\xi_1+\xi_2)^2\big)  \kappa_{2,0}.
%\end{split}
%\end{equation*}
We insert these identities back into $\lambda_{++}$ and $\lambda_{+-}$ to obtain
\begin{equation*}
\begin{split}
&\lambda_{++}(\xi_1,\xi_2) = \frac{i}{2\pi} \frac{1}{(|\xi_1|-i)^2} \frac{1}{(|\xi_2|-i)^2} \fraka_{++}(\xi_1,\xi_2) \cdot (\xi_1+\xi_2) \kappa_{2,0}(\xi_1+\xi_2),\\
&\lambda_{+-}(\xi_1,\xi_2) = \frac{i}{\pi} \frac{1}{(|\xi_1|-i)^2} \frac{1}{(|\xi_2|-i)^2} \fraka_{+-}(\xi_1,\xi_2) \cdot (\xi_1+\xi_2) \kappa_{2,0}(\xi_1+\xi_2),
\end{split}
\end{equation*}
where $\fraka_{++}(\xi_1,\xi_2)$ and $\fraka_{+-}(\xi_1,\xi_2)$ are the following polynomials
\begin{equation*}
\begin{split}
\fraka_{++}(\xi_1,\xi_2) &:=  \big(1-\xi_1^2-\xi_2^2 - 4\xi_1\xi_2+\xi_1^2\xi_2^2\big) + \frac{1}{12}\big(-6+4\xi_1\xi_2 + 3\xi_1^2 + 3\xi_2^2\big) \big(4 + (\xi_1+\xi_2)^2\big)\\
&\quad  + \frac{1}{60}\big(4 + (\xi_1+\xi_2)^2\big)\big(16 + (\xi_1+\xi_2)^2\big) + 4(\xi_1\xi_2-1) + \frac{2}{3}(4-\xi_1\xi_2)\big(4 + (\xi_1+\xi_2)^2\big)\\
&\quad - \frac{1}{10}\big(4 + (\xi_1+\xi_2)^2\big)\big(16 + (\xi_1+\xi_2)^2\big),
\end{split}
\end{equation*}
\begin{equation*}
\begin{split}
\fraka_{+-}(\xi_1,\xi_2) &:= 2(1-\xi_1^2-\xi_2^2-4\xi_1\xi_2+\xi_1^2 \xi_2^2) + \frac{1}{12}\big(-6+8\xi_1\xi_2 + 3\xi_1^2 + 3\xi_2^2\big)\big(4 + (\xi_1+\xi_2)^2\big)\\
&\quad + \frac{1}{60}\big(4 + (\xi_1+\xi_2)^2\big)\big(16 + (\xi_1+\xi_2)^2\big) + 8(\xi_1\xi_2-1)+ \frac{2}{3}(5-2\xi_1\xi_2)\big(4 + (\xi_1+\xi_2)^2\big) \\
&\quad - \frac{1}{10}\big(4 + (\xi_1+\xi_2)^2\big)\big(16 + (\xi_1+\xi_2)^2\big).
\end{split}
\end{equation*}
By patient direct computation, we have the simplified and factorized expressions
\begin{equation*}
\begin{split}
\fraka_{++}(\xi_1,\xi_2) &= \frac{1}{6}\big(2+3\xi_1^2 + \xi_1^4 - 2\xi_1 \xi_2 - \xi_1^3\xi_2 + 3\xi_2^2 + 2 \xi_1^2 \xi_2^2 - \xi_1 \xi_2^3 + \xi_2^4\big)\\
&= \frac{1}{6}\big(\xi_1^2 + \xi_2^2 + 2\big) \big(\xi_1^2 - \xi_1 \xi_2 + \xi_2^2+1\big),
\end{split}
\end{equation*}
and
\begin{equation*}
\begin{split}
\fraka_{+-}(\xi_1,\xi_2) &= \frac{1}{6}\big(\xi_1^2 + \xi_1^4 - 2\xi_1 \xi_2 - 3\xi_1^3\xi_2 + \xi_2^2 + 4 \xi_1^2 \xi_2^2 - 3\xi_1 \xi_2^3 + \xi_2^4\big)\\
&= \frac{1}{6}\big(\xi_1 - \xi_2\big)^2 \big(\xi_1^2 - \xi_1 \xi_2 + \xi_2^2+1\big).
\end{split}
\end{equation*}
Thus, we obtain 
\begin{equation*}
\begin{split}
\lambda_{++}(\xi_1,\xi_2) &= \frac{i}{6\pi} (\xi_1^2 + \xi_2^2 + 2) \frac{(\xi_1^2 - \xi_1 \xi_2 + \xi_2^2+1) (\xi_1+\xi_2)}{(|\xi_1|-i)^2 (|\xi_2|-i)^2} \kappa_{2,0}(\xi_1+\xi_2),\\
\lambda_{+-}(\xi_1,\xi_2) &= \frac{i}{6\pi} (\xi_1^2 - \xi_2^2) \frac{(\xi_1^2 - \xi_1 \xi_2 + \xi_2^2+1) (\xi_1-\xi_2)}{(|\xi_1|-i)^2 (|\xi_2|-i)^2} \kappa_{2,0}(\xi_1+\xi_2).
\end{split}
\end{equation*}
By rescaling these identities via \eqref{equ:proof_scaling_momentum} and using the fact that $\kappa_{2,0}(\ulomega^{-\frac12}(\xi_1+\xi_2)) \equiv \pi \kappa_{\ulomega}(\xi_1+\xi_2)$, we conclude \eqref{eqn: nu_p++} and \eqref{eqn: nu_p+-}.
\end{proof}

\subsection{Proof of Proposition~\ref{prop:modulation_parameters}}\label{subsec:Proof_modulation}
In this section, we prove the main modulation bootstrap estimate \eqref{equ:proof_mod_bootstrap}. From the modulation equations \eqref{equ:setup_modulation_equation} and Lemma~\ref{lem:prep_modulation} we obtain  for all times $0 \leq s \leq T$ that 
\begin{equation}\label{equ:proof_mod_equations}
\begin{split}
\begin{bmatrix}
	\dot{\gamma}+p^2-\omega-p\dot{\sigma}\\
	\dot{\omega}\\
	\dot{\sigma}-2p\\
	\dot{p}
\end{bmatrix} 
&=	\bbM(s)^{-1} \left(\begin{bmatrix}  \langle i \calQ_{\ulomega}(\ulPe \Vp),\sigma_2 Y_{1,\ulomega}\rangle \\
\langle i \calQ_{\ulomega}(\ulPe \Vp),\sigma_2 Y_{2,\ulomega}\rangle \\
\langle i \calQ_{\ulomega}(\ulPe \Vp),\sigma_2 Y_{3,\ulomega}\rangle \\
\langle i \calQ_{\ulomega}(\ulPe \Vp),\sigma_2 Y_{4,\ulomega}\rangle  \end{bmatrix} +  \begin{bmatrix} \langle i \calC(\ulPe \Vp),\sigma_2 Y_{1,\ulomega} \rangle\\
\langle i \calC(\ulPe \Vp),\sigma_2 Y_{2,\ulomega} \rangle\\
\langle i \calC(\ulPe \Vp),\sigma_2 Y_{3,\ulomega} \rangle\\
\langle i \calC(\ulPe \Vp),\sigma_2 Y_{4,\ulomega} \rangle \end{bmatrix} + \begin{bmatrix} \wtilcalE_{1} \\ \wtilcalE_{2} \\ \wtilcalE_{3} \\ \wtilcalE_{4}\end{bmatrix}\right)
\end{split}	
\end{equation}
with 
\begin{equation}\label{equ:proof_mod_bbM}
\bbM(s) = \begin{bmatrix}
0 & 2\omega^{-\frac12} & 0 & 0  \\
2\omega^{-\frac12} & 0 &0 & 0 \\
0 & 0 & 0 & - 4\omega^{\frac12}\\
0 & 0 & 4\omega^{\frac12} & 0				
\end{bmatrix}	+ \begin{bmatrix}
\langle U, \sigma_1Y_{1, \omega} \rangle &  \langle U,  \sigma_2 \partial_\omega Y_{1, \omega} \rangle &  \langle U,  -\sigma_2\py  Y_{1, \omega} \rangle &  \langle U,  y\sigma_1 Y_{1, \omega} \rangle \\
\langle U, \sigma_1 Y_{2, \omega} \rangle &  \langle U, \sigma_2 \partial_\omega Y_{2, \omega} \rangle &  \langle U,  -\sigma_2\py  Y_{2, \omega} \rangle &  \langle U,  y\sigma_1 Y_{2, \omega} \rangle\\
\langle U, \sigma_1 Y_{3, \omega} \rangle &  \langle U, \sigma_2 \partial_\omega Y_{3, \omega} \rangle&  \langle U,  -\sigma_2\py  Y_{3, \omega} \rangle &  \langle U,  y\sigma_1 Y_{3, \omega} \rangle\\
\langle U, \sigma_1 Y_{4, \omega} \rangle &  \langle U, \sigma_2 \partial_\omega Y_{4, \omega} \rangle&  \langle U,  -\sigma_2\py  Y_{4, \omega} \rangle &  \langle U,  y\sigma_1 Y_{4, \omega} \rangle\\										
\end{bmatrix}.
\end{equation}
The matrix $\bbM(s)$ is invertible  since the first matrix on the right-hand side of \eqref{equ:proof_mod_bbM} is invertible using the comparison estimate $\omega(s) \in [\frac12 \omega_0,2\omega_0]$, while the second matrix of \eqref{equ:proof_mod_bbM} has a small operator norm of size  $\calO(\eps)$ thanks to the stability bound \eqref{equ:setup_smallness_orbital}. Hence, we have uniformly the operator norm bound for the inverse matrix,
\begin{equation}\label{equ:proof_bbM_inv_opnorm}
	\sup_{0 \leq s \leq T} \| \bbM(s)^{-1}\| \lesssim_{\omega_0} 1.
\end{equation}

Since we need to pursue stronger decay estimates for $\dot{\omega}(s)$ and $\dot{p}(s)$, this bound is insufficient for our purposes. We instead decompose the inverse of $\bbM(s)$ into a time-independent leading order term plus decaying terms, and we begin to do so by expanding 
\begin{equation}\label{equ:proof_bbM_expand}
	\bbM(s) = \bbM_{\ulomega} + \bbA_\ulomega(s) + \bbB(s)
\end{equation}
with
\begin{align}
\bbM_{\ulomega} &:=	\begin{bmatrix}
	0 & 2\ulomega^{-\frac12} & 0 & 0  \\
	2\ulomega^{-\frac12} & 0 &0 & 0 \\
	0 & 0 & 0 & - 4\ulomega^{\frac12}\\
	0 & 0 & 4\ulomega^{\frac12} & 0				
\end{bmatrix}, \\
\bbA_\ulomega(s) &:= \begin{bmatrix}
\langle \ulPe \Vp, \sigma_1Y_{1, \ulomega} \rangle &  \langle \ulPe \Vp,  \sigma_2 \partial_\omega Y_{1, \ulomega} \rangle &  \langle \ulPe \Vp,  -\sigma_2\py  Y_{1, \ulomega} \rangle &  \langle \ulPe \Vp,  y\sigma_1 Y_{1, \ulomega} \rangle \\
\langle \ulPe \Vp, \sigma_1 Y_{2, \ulomega} \rangle &  \langle \ulPe \Vp, \sigma_2 \partial_\omega Y_{2, \ulomega} \rangle &  \langle \ulPe \Vp,  -\sigma_2\py  Y_{2, \ulomega} \rangle &  \langle \ulPe \Vp,  y\sigma_1 Y_{2, \ulomega} \rangle\\
\langle \ulPe \Vp, \sigma_1 Y_{3, \ulomega} \rangle &  \langle \ulPe \Vp, \sigma_2 \partial_\omega Y_{3, \ulomega} \rangle&  \langle \ulPe \Vp,  -\sigma_2\py  Y_{3, \ulomega} \rangle &  \langle \ulPe \Vp,  y\sigma_1 Y_{3, \ulomega} \rangle\\
\langle \ulPe \Vp, \sigma_1 Y_{4, \ulomega} \rangle &  \langle \ulPe \Vp, \sigma_2 \partial_\omega Y_{4, \ulomega} \rangle&  \langle \ulPe \Vp,  -\sigma_2\py  Y_{4, \ulomega} \rangle &  \langle \ulPe \Vp,  y\sigma_1 Y_{4, \ulomega} \rangle\\										
\end{bmatrix},	\label{equ:proof_bbA}
\end{align}
and $\bbB(s) := \bbM(s) - \bbM_{1}(\ulomega) - \bbA_\ulomega(s)$. From \eqref{equ:proof_bbM_expand} one finds that 
\begin{equation}\label{equ:proof_bbM_inv_expand}
\bbM(s)^{-1} = \bbM_{\ulomega}^{-1} - \bbM_{\ulomega}^{-1} \bbA_\ulomega(s) \bbM_{\ulomega}^{-1} + \bbD(s),
\end{equation}
where
\begin{equation}\label{equ:proof_bbM_ulomega_inv}
\bbM_{\ulomega}^{-1} = \begin{bmatrix}
0 & \frac12 \ulomega^{\frac12} & 0 & 0 \\
\frac12 \ulomega^{\frac12} & 0 & 0 & 0 \\
0 & 0 & 0 & \frac14 \ulomega^{-\frac12}\\
0 & 0 & -\frac14 \ulomega^{-\frac12} & 0 
\end{bmatrix},
\end{equation}
and
\begin{equation*}
\bbD(s) := \bbM(s)^{-1} \Big(-\bbB(s) \bbM_{\ulomega}^{-1} + \bbA_\ulomega(s) \bbM_{\ulomega}^{-1}\bbA_\ulomega(s)\bbM_{\ulomega}^{-1} + \bbB(s)\bbM_{\ulomega}^{-1}\bbA_\ulomega(s)\bbM_{\ulomega}^{-1}\Big).
\end{equation*}
Using \eqref{equ:consequences_ulPe_U_disp_decay}, we obtain the operator norm estimate 
\begin{equation}\label{equ:proof_bbA_opnorm}
\sup_{0\leq s \leq T} \js^{\frac12} \| \bbM_{\ulomega}^{-1} \bbA_\ulomega(s) \bbM_{\ulomega}^{-1} \| \lesssim_{\omega_0} \sup_{0\leq s \leq T} \js^{\frac12} \| \bbA_\ulomega(s) \| \lesssim_{\omega_0} \eps,
\end{equation}
and using \eqref{equ:prop_modulation_parameters_assumption1}, \eqref{equ:consequences_ulPe_U_disp_decay}, \eqref{equ:consequences_U_disp_decay}, as well as \eqref{equ:consequences_discrete_components_decay}, we infer that 
\begin{equation}\label{equ:proof_bbB_opnorm}
\sup_{0 \leq s \leq T} \js^{1-\delta} \|\bbB(s)\| \lesssim_{\omega_0} \eps.
\end{equation}
The two preceeding bounds combined with \eqref{equ:proof_bbM_inv_opnorm} also imply that 
\begin{equation}\label{equ:proof_bbD_opnorm}
\sup_{0 \leq s \leq T} \js^{1-\delta} \|\bbD(s)\| \lesssim_{\omega_0} \eps.
\end{equation}

After inserting \eqref{equ:proof_bbM_inv_expand} into \eqref{equ:proof_mod_equations}, we obtain for all $0 \leq s \leq T$
\begin{equation}\label{equ:proof_mod_equations_expanded}
\begin{split}
\begin{bmatrix}
	\dot{\gamma}+p^2-\omega-p\dot{\sigma}\\
	\dot{\omega}\\
	\dot{\sigma}-2p\\
	\dot{p}
\end{bmatrix} &= \bbM_{\ulomega}^{-1} \bfQ(s) +\bbM_{\ulomega}^{-1} \bfC(s)- \bbM_{\ulomega}^{-1} \bbA_\ulomega(s) \bbM_{\ulomega}^{-1} \bfQ(s)\\
&\quad  - \bbM_{\ulomega}^{-1} \bbA_\ulomega(s) \bbM_{\ulomega}^{-1} \bfC(s) + \bbD(s) \bfQ(s) + \bbD(s) \bfC(s) + \bbM(s)^{-1} \bfE(s)
\end{split}
\end{equation}
where we use the following notation for the vectors
\begin{equation}\label{equ:proof_vector_def}
\begin{split}
\bfQ(s) &:= \Big(\langle i \calQ_{\ulomega}(\ulPe \Vp(s)),\sigma_2 Y_{j,\ulomega}\rangle\Big)_{1\leq j \leq 4},\\
\bfC(s) &:= \Big(\langle i \calC(\ulPe \Vp(s)),\sigma_2 Y_{j,\ulomega}\rangle\Big)_{1\leq j \leq 4},\\
\bfE(s) &:= \Big(\wtilcalE_j(s)\Big)_{1\leq j \leq 4}.
\end{split}
\end{equation}

Based on the preceeding bounds on the operator norms, it is straightforward to determine the desired decay rates for the last four terms on the right-hand side of \eqref{equ:proof_mod_equations_expanded} that contributes to $\dot{\omega}(s)$ and $\dot{p}(s)$. Indeed, using the pointwise decay of $\ulPe \Vp(s)$ from \eqref{equ:consequences_ulPe_U_disp_decay} and the operator norm estimates \eqref{equ:proof_bbA_opnorm}, \eqref{equ:proof_bbD_opnorm}, we deduce the following acceptable bounds for all $0 \leq s \leq T$:
\begin{equation*}
\begin{split}
\big\| \bbM_{\ulomega}^{-1} \bbA_\ulomega(s) \bbM_{\ulomega}^{-1} \bfC(s)\big \| \lesssim \eps^4\js^{-2}, \qquad \big\| \bbD(s) \bfQ(s)\big \| \lesssim \eps^3\js^{-2+\delta}, \qquad \big\| \bbD(s) \bfC
(s)\big \| \lesssim \eps^4\js^{-\frac52+\delta}.
\end{split}
\end{equation*}
The bounds \eqref{equ:BS_modulation_error} and \eqref{equ:proof_bbM_inv_opnorm} also imply for all $0 \leq s \leq T$ that 
\begin{equation*}
\big\| \bbM(s)^{-1} \bfE(s) \big\| \lesssim \eps^3 \js^{-2+\delta}.
\end{equation*}
Integrating these bounds in time from $t \leq s \leq T$ leads to acceptable estimates for the contributions to \eqref{equ:proof_mod_bootstrap}. 

Therefore it remains to infer decay on the first three terms on the right-hand side of \eqref{equ:proof_mod_equations_expanded}, which we will carry out below using normal forms. The quadratic term $\bbM_{\ulomega}^{-1} \bfQ(s)$ is the most challenging term which additionally relies on the crucial null structures found in quadratic spectral distributions given in Lemma~\ref{lem:null_structure_modulation}. 

\medskip

\noindent \underline{Contribution of $\int_t^T \bbM_{\ulomega}^{-1} \bfQ(s) \,\ud s$ to $|\omega(t)-\omega(T)|$ and $|p(t)-p(T)|$}:
By inspection, the contribution %of the term $\int_t^T \bbM_{\ulomega}^{-1} \bfQ(s)\,\ud s$ 
to $\big(\omega(t)-\omega(T)\big)$ is given by 
\begin{equation}\label{equ:proof_modulation_omega}
- \frac{1}{2} \sqrt{\ulomega} \int_t^T \big\langle i \calQ_{\ulomega}(\ulPe \Vp(s)),\sigma_2 Y_{1,\ulomega} \big\rangle  \,\ud s = - \frac{i}{2} \sqrt{\ulomega} \int_t^T \big\langle  \big(\ve(s)^2 - \barve(s)^2\big),\phi_{\ulomega}^2 \big\rangle  \,\ud s,
\end{equation}
while the contribution to $\big(p(t)-p(T)\big)$ is given by
\begin{equation}\label{equ:proof_modulation_p}
\frac{1}{4\sqrt{\ulomega}} \int_t^T \big\langle i \calQ_{\ulomega}(\ulPe \Vp(s)),\sigma_2 Y_{3,\ulomega} \big\rangle  \,\ud s = \frac{1}{4\sqrt{\ulomega}}\int_t^T \big\langle  \big(\ve(s)^2 + 4 \ve \barve + \barve(s)^2\big),\phi_{\ulomega}\phi_{\ulomega}' \big\rangle  \,\ud s .
\end{equation}
We insert the representation formulas \eqref{equ:setup_v_e_repformula} for $\ve(s,y)$ and \eqref{equ:setup_barv_e_repformula} for $\barve(s,y)$ to obtain the following multilinear expressions on the distorted Fourier side for \eqref{equ:proof_modulation_omega}
\begin{equation}
\begin{split}
&\int_t^T \big\langle  \big(\ve(s)^2 - \barve(s)^2\big),\phi_{\ulomega}^2 \big\rangle\,\ud s\\
&= \int_t^T \iint e^{-is(\xi_1^2+\xi_2^2 + 2\ulomega)} \tilf_{+, \ulomega,\ulp}(s,\xi_1) \tilf_{+,\ulomega,\ulp}(s,\xi_2)  \nu_{++,\ulomega}(\xi_1,\xi_2)\,\ud \xi_1\,\ud \xi_2 \,\ud s\\
&\quad -2 \int_t^T \iint e^{-is(\xi_1^2-\xi_2^2)} \tilf_{+, \ulomega,\ulp}(s,\xi_1) \tilf_{-,\ulomega,\ulp}(s,\xi_2)  \nu_{+-,\ulomega}(\xi_1,\xi_2)\,\ud \xi_1\,\ud \xi_2 \,\ud s\\
&\quad +\int_t^T \iint e^{is(\xi_1^2+\xi_2^2 + 2\ulomega)} \tilf_{-, \ulomega,\ulp}(s,\xi_1) \tilf_{-,\ulomega,\ulp}(s,\xi_2)  \nu_{--,\ulomega}(\xi_1,\xi_2)\,\ud \xi_1\,\ud \xi_2 \,\ud s\\
&=: \calI_{++}(t;T)+\calI_{+-}(t;T)+\calI_{--}(t;T)
\end{split}
\end{equation}
with the following nonlinear spectral distributions
\begin{equation}\label{equ:nu-NSD}
\begin{split}
\nu_{++,\ulomega}(\xi_1, \xi_2) &:= \int_\bbR \bigl( \Psi_{1,\ulomega}(y,\xi_1) \Psi_{1,\ulomega}(y,\xi_2) - \Psi_{2,\ulomega}(y,\xi_1) \Psi_{2,\ulomega}(y,\xi_2) \bigr) \phi_\ulomega(y)^2 \, \ud y, \\
\nu_{+-,\ulomega}(\xi_1, \xi_2) &:= \int_\bbR \bigl( \Psi_{1,\ulomega}(y,\xi_1) \Psi_{2,\ulomega}(y,\xi_2) - \Psi_{2,\ulomega}(y,\xi_1) \Psi_{1,\ulomega}(y,\xi_2) \bigr) \phi_\ulomega(y)^2 \, \ud y,\\
\nu_{--,\ulomega}(\xi_1, \xi_2) &:= \int_\bbR \bigl( \Psi_{2,\ulomega}(y,\xi_1) \Psi_{2,\ulomega}(y,\xi_2) - \Psi_{1,\ulomega}(y,\xi_1) \Psi_{1,\ulomega}(y,\xi_2) \bigr) \phi_\ulomega(y)^2 \, \ud y.
\end{split}
\end{equation}
Similarly, for \eqref{equ:proof_modulation_p} we compute that 
\begin{equation}
\begin{split}
&\int_t^T \big\langle  \big(\ve(s)^2 +4 \ve(s)\barve(s)+ \barve(s)^2\big),\phi_{\ulomega}\phi_{\ulomega}' \big\rangle\,\ud s\\
&= \int_t^T \iint e^{-is(\xi_1^2+\xi_2^2 + 2\ulomega)} \tilf_{+, \ulomega,\ulp}(s,\xi_1) \tilf_{+,\ulomega,\ulp}(s,\xi_2)  \lambda_{++,\ulomega}(\xi_1,\xi_2)\,\ud \xi_1\,\ud \xi_2 \,\ud s\\
&\quad -2 \int_t^T \iint e^{-is(\xi_1^2-\xi_2^2)} \tilf_{+, \ulomega,\ulp}(s,\xi_1) \tilf_{-,\ulomega,\ulp}(s,\xi_2)  \lambda_{+-,\ulomega}(\xi_1,\xi_2)\,\ud \xi_1\,\ud \xi_2 \,\ud s\\
&\quad +\int_t^T \iint e^{is(\xi_1^2+\xi_2^2 + 2\ulomega)} \tilf_{-, \ulomega,\ulp}(s,\xi_1) \tilf_{-,\ulomega,\ulp}(s,\xi_2)  \lambda_{--,\ulomega}(\xi_1,\xi_2)\,\ud \xi_1\,\ud \xi_2 \,\ud s\\
&=: \calJ_{++}(t;T)+\calJ_{+-}(t;T)+\calJ_{--}(t;T)
\end{split}
\end{equation}
where 
\begin{equation}\label{equ:lambda-NSD}
\begin{aligned}
\lambda_{++,\ulomega}(\xi_1, \xi_2) &:= \int_\bbR \Bigl( \Psi_{1,\ulomega}(y,\xi_1) \Psi_{1,\ulomega}(y,\xi_2) + 2\Psi_{1,\ulomega}(y,\xi_1) \Psi_{2,\ulomega}(y,\xi_2) \\
&\qquad \qquad + 2\Psi_{2,\ulomega}(y,\xi_1) \Psi_{1,\ulomega}(y,\xi_2) +  \Psi_{2,\ulomega}(y,\xi_1) \Psi_{2,\ulomega}(y,\xi_2) \Bigr) \phi_\ulomega(y)\phi_\ulomega'(y) \, \ud y, \\
\lambda_{+-,\ulomega}(\xi_1, \xi_2) &:= \int_\bbR \Bigl( \Psi_{1,\ulomega}(y,\xi_1) \Psi_{2,\ulomega}(y,\xi_2) + 2\Psi_{1,\ulomega}(y,\xi_1) \Psi_{1,\ulomega}(y,\xi_2) \\
&\qquad \qquad + 2\Psi_{2,\ulomega}(y,\xi_1) \Psi_{2,\ulomega}(y,\xi_2) +  \Psi_{2,\ulomega}(y,\xi_1) \Psi_{1,\ulomega}(y,\xi_2) \Bigr) \phi_\ulomega(y)\phi_\ulomega'(y) \, \ud y,\\
\lambda_{--,\ulomega}(\xi_1, \xi_2) &:= \int_\bbR \Bigl( \Psi_{1,\ulomega}(y,\xi_1) \Psi_{1,\ulomega}(y,\xi_2) + 2\Psi_{1,\ulomega}(y,\xi_1) \Psi_{2,\ulomega}(y,\xi_2) \\
&\qquad \qquad + 2\Psi_{2,\ulomega}(y,\xi_1) \Psi_{1,\ulomega}(y,\xi_2) +  \Psi_{2,\ulomega}(y,\xi_1) \Psi_{2,\ulomega}(y,\xi_2) \Bigr) \phi_\ulomega(y)\phi_\ulomega'(y) \, \ud y.
\end{aligned}
\end{equation}	

Our aim is to prove for any $0 \leq t \leq T$ that 
\begin{equation*}
\big| \calI_{++}(t;T) \big| + \big|\calI_{+-}(t;T)\big| + \big|\calI_{--}(t;T) \big| + \big| \calJ_{++}(t;T) \big| + \big|\calJ_{+-}(t;T)\big| + \big|\calJ_{--}(t;T) \big|\lesssim \eps^2 \jt^{-1+\delta}.
\end{equation*}
These estimates will be achieved by exploiting the oscillations in the phases of the multilinear expressions. We note that the phase $e^{-is(\xi_1^2-\xi_2^2)}$ in $\calI_{+-}(t;T)$ and $\calJ_{+-}(t;T)$ contains the hyper-planes $\{\xi_1=\pm \xi_2\}$ as the set of time resonances, while the other terms do not have time resonances. However, thanks to the remarkable null structures in the distributions in $\nu_{+-,\ulomega}(\xi_1,\xi_2)$ and $\lambda_{+-,\ulomega}(\xi_1,\xi_2)$, we find that the expressions for $\frac{\nu_{+-,\ulomega}(\xi_1,\xi_2)}{\xi_1^2 - \xi_2^2}$ and $\frac{\lambda_{+-,\ulomega}(\xi_1,\xi_2)}{\xi_1^2 - \xi_2^2}$ remain smooth and bounded. See Lemma~\ref{lem:null_structure_modulation}. This key observation for $\nu_{+-,\ulomega}(\xi_1,\xi_2)$ was already recorded in \cite[Lemma~6.3]{LL24} but for it is new for $\lambda_{+-,\ulomega}(\xi_1,\xi_2)$. These null structures allows us to perform an integration by parts in time for  $\calI_{+-}(t;T)$ and $\calJ_{+-}(t;T)$ to infer decay.  Subsequently, we outline how to control the contributions of the remaining non-resonant terms $\calI_{--}(t;T)$, $\calI_{++}(t;T)$, $\calJ_{--}(t;T)$, and $\calJ_{++}(t;T)$. 

In view of Remark~\ref{remark: generic_null_expression}, the integrals $\calI_{+-}(t;T)$ and $\calJ_{+-}(t;T)$ are linear combinations of terms of the schematic form
\begin{equation*}
\begin{split}
\calI(t;T) := \int_t^T \iint e^{-is(\xi_1^2 - \xi_2^2)} (\xi_1^2-\xi_2^2) \fraka(\xi_1)\tilf_{+, \ulomega,\ulp}(s,\xi_1) \frakb(\xi_2) \tilf_{-, \ulomega,\ulp}(s,\xi_2) \kappa(\xi_1+\xi_2) \,\ud \xi_1 \,\ud \xi_2 \,\ud s,
\end{split}
\end{equation*}
where $\fraka(\eta)$ and $\frakb(\eta)$ are symbols of the form 
\begin{equation}\label{equ:proof_symbol_type}
\frac{1}{(|\eta|-i\sqrt{\ulomega})^2}, \quad \frac{\eta}{(|\eta|-i\sqrt{\ulomega})^2}, \quad \text{or} \quad \frac{\eta^2}{(|\eta|-i\sqrt{\ulomega})^2},
\end{equation}
and  $\kappa(\cdot)$ is the Fourier transform of some Schwartz function $q(y)$, that is,
\begin{equation}\label{equ:proof_FT_kappa}
\kappa(\xi_1+\xi_2) = \int_\bbR e^{iy(\xi_1+\xi_2)}q(y)\,\ud y.
\end{equation}
In what follows, we prove that $|\calI(t;T)| \lesssim \eps^2 \jt^{-1+\delta}$ which suffices for us to conclude the same bounds for $\calI_{+-}(t;T)$ and $\calJ_{+-}(t;T)$. 

First, we integrate by parts in time on $\calI(t;T)$ to get
\begin{equation*}
\begin{split}
\calI(t;T) &= i \left[ \iint e^{-is(\xi_1^2-\xi_2^2)}\fraka(\xi_1)\tilf_{+, \ulomega,\ulp}(s,\xi_1) \frakb(\xi_2) \tilf_{-, \ulomega,\ulp}(s,\xi_2) \kappa(\xi_1+\xi_2) \,\ud \xi_1 \,\ud \xi_2\right]_{s=t}^{s=T}\\
&\quad - i \int_t^T \iint e^{-is(\xi_1^2 - \xi_2^2)} \fraka(\xi_1)\ps \tilf_{+, \ulomega,\ulp}(s,\xi_1) \frakb(\xi_2) \tilf_{-, \ulomega,\ulp}(s,\xi_2) \kappa(\xi_1+\xi_2) \,\ud \xi_1 \,\ud \xi_2 \,\ud s\\
&\quad - i \int_t^T \iint e^{-is(\xi_1^2 - \xi_2^2)} \fraka(\xi_1)\tilf_{+, \ulomega,\ulp}(s,\xi_1) \frakb(\xi_2) \ps \tilf_{-, \ulomega,\ulp}(s,\xi_2) \kappa(\xi_1+\xi_2) \,\ud \xi_1 \,\ud \xi_2 \,\ud s\\
&=: \calI^1(t;T)+\calI^2(t;T)+\calI^3(t;T).
\end{split}
\end{equation*}
Then, we define the following free Schr\"odinger waves
\begin{align}
&v_+(s,y) := e^{-is\ulomega}\int_\bbR e^{iy \xi_1} e^{-is\xi_1^2} \fraka(\xi_1) \tilf_{+, \ulomega,\ulp}(s,\xi_1) \,\ud \xi_1, \\
&v_-(s,y) := e^{is\ulomega}\int_\bbR e^{iy \xi_2} e^{is\xi_2^2} \frakb(\xi_1) \tilf_{-, \ulomega,\ulp}(s,\xi_2) \,\ud \xi_2.
\end{align}
These free waves satisfy the bounds \eqref{equ:preparation_flat_Schrodinger_wave_bound1}--\eqref{equ:preparation_flat_Schrodinger_wave_bound5}. By reverting the Fourier transform \eqref{equ:proof_FT_kappa} in $\calI^1(t;T)$, we have
\begin{equation*}
\calI^1(t;T) = i \int_\bbR v_+(T,y)v_-(T,y)q(y)\,\ud y - i \int_\bbR v_+(t,y)v_-(t,y)q(y)\,\ud y.
\end{equation*}
By \eqref{equ:preparation_flat_Schrodinger_wave_bound1} we obtain the sufficient bound
\begin{equation*}
\left|\calI^1(t;T)\right| \lesssim \sup_{t\leq s \leq T}\big( \|q\|_{L_y^1} \| v_+(s)\|_{L_y^\infty} \| v_-(s)\|_{L_y^\infty} \big)\lesssim \eps^2 \jt^{-1}.
\end{equation*}

For the purpose of bounding  the two remaining terms in $\calI(t;T)$, we shall insert the right-hand side of the evoluation equations \eqref{equ:profile-dft-+} for $\ps\tilf_{+, \ulomega,\ulp}(s,\xi_1)$ into $\calI^2(t;T)$; respectively \eqref{equ:profile-dft-minus} for $\ps\tilf_{-, \ulomega,\ulp}(s,\xi_2)$ into $\calI^3(t;T)$. We provide detailed treatment for the term  $\calI^2(t;T)$ and leave the analogous details for the term $\calI^3(t;T)$ to the reader. Upon inserting \eqref{equ:profile-dft-+}, we have\footnote{in order to lighten the  notation, we have temporarily drop the subscripts $(\ulomega,\ulp)$ from $\tilf_{\pm, \ulomega,\ulp}$.} 
\begin{equation*}
\begin{split}
\calI^2(t;T) &= - \int_{t}^{T} \dot{\theta}_1(s) \iint e^{-is(\xi_1^2-\xi_2^2)} \fraka(\xi_1) \tilf_+(s,\xi_1) \frakb(\xi_2) \tilf_-(s,\xi_2) \kappa(\xi_1+\xi_2) \,\ud \xi_1 \,\ud \xi_2 \,\ud s\\
&\quad + \int_{t}^{T} \dot{\theta}_2(s) \iint e^{-is(\xi_1^2-\xi_2^2)}  \xi_1 \fraka(\xi_1)\tilf_+(s,\xi_1) \frakb(\xi_2) \tilf_-(s,\xi_2) \kappa(\xi_1+\xi_2) \,\ud \xi_1 \,\ud \xi_2 \,\ud s\\
&\quad - \int_{t}^{T} \dot{\theta}_1(s) \iint e^{is(\xi_2^2+\ulomega)} \fraka(\xi_1) \calL_{+, \ulomega}[\Vp(s)](\xi_1) \frakb(\xi_2) \tilf_-(s,\xi_2) \kappa(\xi_1+\xi_2) \,\ud \xi_1 \,\ud \xi_2 \,\ud s\\
&\quad -i \int_{t}^{T} \dot{\theta}_2(s) \iint e^{is(\xi_2^2+\ulomega)}   \fraka(\xi_1) \calK_{+, \ulomega}[\Vp(s)](\xi_1) \frakb(\xi_2) \tilf_-(s,\xi_2) \kappa(\xi_1+\xi_2) \,\ud \xi_1 \,\ud \xi_2 \,\ud s\\
&\quad - \int_{t}^{T}  \iint e^{is(\xi_2^2+\ulomega)}   \fraka(\xi_1) \wtilcalF_{+,\ulomega}[\calQ_{\ulomega}(\ulPe \Vp(s))](\xi_1) \frakb(\xi_2) \tilf_-(s,\xi_2) \kappa(\xi_1+\xi_2) \,\ud \xi_1 \,\ud \xi_2 \,\ud s\\
&\quad - \int_{t}^{T}  \iint e^{is(\xi_2^2+\ulomega)}   \fraka(\xi_1) \wtilcalF_{+,\ulomega}[\calC(\ulPe \Vp(s))](\xi_1) \frakb(\xi_2) \tilf_-(s,\xi_2) \kappa(\xi_1+\xi_2) \,\ud \xi_1 \,\ud \xi_2 \,\ud s\\
&\quad - \int_{t}^{T}  \iint e^{is(\xi_2^2+\ulomega)}   \fraka(\xi_1) \wtilcalF_{+,\ulomega}[\ulPe \calMod(s) + \calE_1(s) + \calE_2(s) + \calE_3(s)](\xi_1) \\
&\hspace{20em}\times \frakb(\xi_2) \tilf_-(s,\xi_2) \kappa(\xi_1+\xi_2) \,\ud \xi_1 \,\ud \xi_2 \,\ud s\\
&=: \calI^{2,1}(t;T)+\calI^{2,2}(t;T)+\calI^{2,3}(t;T)+\calI^{2,4}(t;T)+\calI^{2,5}(t;T)+\calI^{2,6}(t;T)+\calI^{2,7}(t;T).
\end{split}
\end{equation*}
We then proceed to estimate these integrals term by term. 

Using \eqref{equ:proof_FT_kappa}, we may rewrite 
\begin{equation*}
\begin{split}
\calI^{2,1}(t;T) = - \int_t^T \dot{\theta}_1(s) \int_\bbR q(y) v_+(s,y) v_-(s,y) \,\ud y \,\ud s.
\end{split}
\end{equation*}
By H\"older's inequality,  \eqref{equ:consequences_aux_bound_modulation2}, and \eqref{equ:preparation_flat_Schrodinger_wave_bound1}, we get 
\begin{equation*}
\big|\calI^{2,1}(t;T)\big| \lesssim \int_t^T |\dot{\theta}_1(s)| \|q \|_{L_y^1} \| v_+(s)\|_{L_y^\infty} \| v_-(s)\|_{L_y^\infty} \,\ud s \lesssim \int_t^T \eps \js^{-1+\delta} \cdot \eps^2 \js^{-1} \,\ud s \lesssim \eps^3 \jt^{-1+\delta}.
\end{equation*}

Noting that
\begin{equation*}
-i\py v_+(s,y) = \int_\bbR e^{-is(\xi_1^2+\ulomega)} e^{iy\xi_1}  \xi_1 \fraka(\xi_1) \tilf_{+}(s,\xi_1)\,\ud \xi_1 ,
\end{equation*}
we have 
\begin{equation*}
\begin{split}
\calI^{2,2}(t;T) = -i \int_t^T \dot{\theta}_2(s) \int_\bbR q(y) \py v_+(s,y) v_-(s,y) \,\ud y \,\ud s.
\end{split}
\end{equation*}
Hence \eqref{equ:consequences_aux_bound_modulation2}, \eqref{equ:preparation_flat_Schrodinger_wave_bound1}, and \eqref{equ:preparation_flat_Schrodinger_wave_bound2} imply that 
\begin{equation*}
\begin{split}
\big| \calI^{2,2}(t;T) \big| &\lesssim \int_t^T |\dot{\theta}_2(s)| \| \jy q(y) \|_{L_y^2} \| \jy^{-1} \py v_+(s,y) \|_{L_y^2} \| v_-(s,y) \|_{L_y^\infty} \,\ud y \,\ud s \\
&\lesssim \int_t^T \eps \js^{-1+\delta} \eps \js^{-1+\delta} \eps \js^{-\frac12}\,\ud s\lesssim \eps^3 \jt^{-\frac32 + \delta}.
\end{split}
\end{equation*}
Similarly, we can rewrite
\begin{equation*}
\begin{split}
\calI^{2,3}(t;T) &= - \int_t^T \dot{\theta}_1(s) \int_\bbR q(y) \left(\int_\bbR e^{iy\xi_1} \fraka(\xi_1) \calL_{+, \ulomega}[\Vp(s)](\xi_1)\,\ud \xi_1\right) v_-(s,y) \,\ud y \,\ud s,\\
\calI^{2,4}(t;T) &= -i \int_t^T \dot{\theta}_2(s) \int_\bbR q(y) \left(\int_\bbR e^{iy\xi_1} \fraka(\xi_1) \calK_{+, \ulomega}[\Vp(s)](\xi_1)\,\ud \xi_1\right) v_-(s,y) \,\ud y \,\ud s.
\end{split}
\end{equation*}
Using Plancherel's identity and the bounds \eqref{equ:consequences_aux_bound_modulation2}, \eqref{equ:consequences_calL_bounds}, \eqref{equ:consequences_calK_bounds}, \eqref{equ:preparation_flat_Schrodinger_wave_bound1}, we establish that
\begin{equation*}
\begin{split}
\big|\calI^{2,3}(t;T)\big| &\lesssim \int_t^T |\dot{\theta}_1(s)| \|q \|_{L_y^2} \big\| \calL_{+, \ulomega}[\Vp(s)]\big\|_{L_\xi^2} \| v_-(s)\|_{L_y^\infty} \,\ud s\\
&\lesssim \int_t^T \eps \js^{-1+\delta} \cdot \eps \js^{-\frac12} \cdot \eps \js^{-\frac12} \,\ud s \lesssim \eps^3 \jt^{-1+\delta},
\end{split}
\end{equation*}
and
\begin{equation*}
\begin{split}
\big|\calI^{2,4}(t;T)\big| &\lesssim \int_t^T |\dot{\theta}_2(s)| \|q \|_{L_y^2} \big\| \calK_{+, \ulomega}[\Vp(s)]\big\|_{L_\xi^2} \| v_-(s)\|_{L_y^\infty} \,\ud s\\
&\lesssim \int_t^T \eps \js^{-1+\delta} \cdot \eps \js^{-\frac12} \cdot \eps \js^{-\frac12} \,\ud s \lesssim \eps^3 \jt^{-1+\delta}.
\end{split}
\end{equation*}

For the term $\calI^{2,5}(t;T)$, we have 
\begin{equation*}
\calI^{2,5}(t;T) =  \int_t^T \int_\bbR q(y)  \left(\int_\bbR e^{iy\xi_1} \fraka(\xi_1) \wtilcalF_{+, \ulomega}[\calQ_{\ulomega}(\ulPe \Vp(s))](\xi_1) \,\ud \xi_1 \right)   v_-(s,y) \,\ud y \,\ud s.
\end{equation*}
Upon inserting the expressions for the quadratic nonlinearity and the distorted Fourier transform, we find that 
\begin{equation*}
\begin{split}
\wtilcalF_{+, \ulomega}[\calQ_{\ulomega}(\ulPe \Vp(s))](\xi_1) &= - \int_\bbR \phi_{\ulomega}(x) \ve(s,x) \big(\ve(s,x)+2  \barve(s,x)\big) \overline{\Psi_{1,\ulomega}(x,\xi_1)}\,\ud x\\
&\quad - \int_\bbR \phi_{\ulomega}(x) \barve(s,x) \big(\barve(s,x) + 2 \ve(s,x)\big) \overline{\Psi_{2,\ulomega}(x,\xi_1)}\,\ud x.
\end{split}
\end{equation*}
We will use the decomposition for $v_-(s,y)$ given by 
\begin{equation}\label{equ:proof_v-_decomp}
v_-(s,y) = \tilh_2(s)+R_{v_-}(s,y),
\end{equation}
where 
\begin{equation*}
\tilh_2(s) := e^{is\ulomega}\int_\bbR e^{is\xi^2} \chi_0(\xi) \frakb(\xi) \tilf_{-, \ulomega,\ulp}(s,\xi)\,\ud \xi.
\end{equation*}
Note that the above decomposition satisfies the estimates \eqref{equ:preparation_flat_Schrodinger_wave_bound3}--\eqref{equ:preparation_flat_Schrodinger_wave_bound5}. By inserting the leading order decomposition \eqref{equ:consequences_usube_leading_order_local_decay_decomp} and \eqref{equ:consequences_barusube_leading_order_local_decay_decomp} for $\ve(s,x)$ and $\barve(s,x)$ respectively, and \eqref{equ:proof_v-_decomp} for $v_-(s,y)$, we find that the term $\calI^{2,5}(t;T)$ is a sum of terms of the form 
\begin{equation}\label{equ:proof_I_1}
I_1(t;T) := C_{k_1,k_2,\ell} \int_t^T h_{j_1,\ulomega}(s) h_{j_2,\ulomega}(s) \tilh_2(s) \,\ud s, \quad j_1,j_2,k_1,k_2,\ell \in \{1,2\},
\end{equation}
with a constant $C_{k_1,k_2,\ell}$ given by 
\begin{equation}\label{equ:proof_C}
C_{k_1,k_2,\ell} = \int_\bbR q(y) \int_\bbR e^{iy\xi_1}\fraka(\xi_1) \int_\bbR \phi_{\ulomega}(x) \Phi_{k_1,\ulomega}(x) \Phi_{k_2,\ulomega}(x) \overline{\Psi_{\ell,\ulomega}(x,\xi_1)}\,\ud x \,\ud \xi_1 \,\ud y,
\end{equation}
and terms of the form
\begin{equation}\label{equ:proof_I_2}
I_2(t;T) := \int_t^T \int_\bbR q(y) \left(\int_\bbR e^{iy\xi_1} \fraka(\xi_1) \int_\bbR \phi_{\ulomega}(x) g_1(s,x) g_2(s,x) \overline{\Psi_{\ell,\ulomega}(x,\xi_1)}\,\ud x \,\ud \xi_1 \right) g_3(s,y)\,\ud y \,\ud s,
\end{equation}
with $\ell \in \{1,2\}$ and where the inputs $g_1(s,x), g_2(s,x)$ are given by $R_{\vp,\ulomega}(s,x)$, $R_{\barvp,\ulomega}(s,x)$, or $h_{j,\ulomega,\ulp}(s)\Phi_{k,\ulomega}(x)$, $ j,k \in\{1,2\}$, and the input $g_3(s,y)$ is given by $\tilh_2(s)$ or $R_{v_-}(s,y)$. 

Note that the constant  \eqref{equ:proof_C} is bounded thanks to the spatial localization of $q(y)$ and $\phi_\ulomega(x)$, and the tensorized sutrcture of $\Psi_{\ell,\ulomega}(x,\xi_1)$; see Lemma~\ref{lemma: PDO on m12}. Hence, it suffices to estimate the $s$-integral in \eqref{equ:proof_I_1} to estimate $I_1(t;T)$. Since $I_1(t;T)$ has the same structure as the term \eqref{equ:proof_mod_I_lead} below, we refer the reader to \eqref{equ:proof_h3} for the details. The conclusion is that 
\begin{equation*}
\left|I_1(t;T)\right| \lesssim \eps^3 \jt^{-1+\delta}.
\end{equation*}

For the term $I_2(t;T)$, we may assume that at least one of the inputs $g_j(s,\cdot)$, $1\leq j \leq 3$ is a remainder term among $R_{\vp,\ulomega}(s,x)$, $R_{\barvp,\ulomega}(s,x)$, or $R_{v_-}(s,y)$. For example, we provide details for the case where $g_1(s,x) = R_{\vp,\ulomega}(s,x)$, $g_2(s,x) = h_{1,\ulomega,\ulp}(s)\Phi_{1,\ulomega}(x)$, and $g_3(s,y) = \tilh_2(s)$. By applying Cauchy-Schwarz inequality, the Plancherel's identity, and then using the bounds \eqref{equ:consequences_h12_decay}, \eqref{equ:consequences_Ru_local_decay}, \eqref{equ:preparation_flat_Schrodinger_wave_bound3}, \eqref{equ:preparation_flat_Schrodinger_wave_bound5}, we have 
\begin{equation*}
\begin{split}
&\left |\int_t^T \int_\bbR q(y) \left(\int_\bbR e^{iy\xi_1} \fraka(\xi_1) \int_\bbR \phi_{\ulomega}(x) R_{\vp,\ulomega}(s,x) h_{1,\ulomega,\ulp}(s)\Phi_{1,\ulomega}(x) \overline{\Psi_{\ell,\ulomega}(x,\xi_1)}\,\ud x \,\ud \xi_1 \right) \tilh_2(s)\,\ud y \,\ud s\right |\\
&\lesssim \int_t^T \| \jy^2 q(y)\|_{L^2} \|\jx^2 \phi_{\ulomega}(x)\|_{L^2} \| \jx^{-2} R_{\vp,\ulomega}(s,x)\|_{L^\infty} |h_{1,\ulomega,\ulp}(s)| |\tilh_2(s)|\,\ud s\\
&\lesssim \int_t^T \eps \js^{-1+\delta} \cdot \eps^2 \js^{-1} \,\ud s \lesssim \eps^3 \jt^{-1+\delta}.
\end{split}
\end{equation*}
This is a sufficient bound and we conclude the discussion for the term $\calI^{2,5}(t;T)$.

For the next term, one has
\begin{equation*}
\calI^{2,6}(t;T) = \int_t^T \int_\bbR q(y)  \left(\int_\bbR e^{iy\xi_1} \fraka(\xi_1) \wtilcalF_{+, \ulomega}[\calC(\ulPe \Vp(s))](\xi_1) \,\ud \xi_1 \right)   v_-(s,y) \,\ud y \,\ud s.
\end{equation*}
We need to determine the type of symbol $\fraka(\cdot)$, as specified in \eqref{equ:proof_symbol_type}, to estimate this term.

\underline{Case 1:} $\fraka(\xi_1) = \frac{1}{(|\xi_1|-i\sqrt{\ulomega})^2}$ or $\fraka(\xi_1) = \frac{\xi_1}{(|\xi_1|-i\sqrt{\ulomega})^2}$. 

In this case, we  note that $\fraka(\cdot) \in H^1$. By examining the structure of the cubic term $\wtilcalF_{+, \ulomega}[\calC(\ulPe \Vp(s))](\xi_1)$ given in Section~\ref{subsec:cubic_spectral_distributions}, we find that the term 
\begin{equation*}
\int_\bbR e^{iy\xi_1} \fraka(\xi_1) \wtilcalF_{+, \ulomega}[\calC(\ulPe \Vp(s))](\xi_1) \,\ud \xi_1.
\end{equation*}
is a linear combination of three types of contributions: regular type, Hilbert type, and a Dirac delta type. Let $\frakn_0,\frakn_1,\frakn_2,\frakn_3 \in W^{1,\infty}(\bbR)$ be symbols of the form \eqref{eqn: symbol_frakb}. A regular contribution is of the form 
\begin{equation*}
I_{\mathrm{reg}}(s,y) := \int_\bbR e^{iy\xi_1} \fraka(\xi_1)  \frakn_0(\xi_1) \widehat{\calF}[\varphi(\cdot)w_{1}(s,\cdot)w_{2}(s,\cdot)w_{3}(s,\cdot)](\xi_1)\,\ud \xi_1 
\end{equation*}
with $\varphi\in \calS(\bbR)$ a Schwartz function and $w_{j}(s,y)$, $1 \leq j \leq 3$ are free waves given by 
\begin{equation*}
\int_\bbR e^{iy\eta} e^{\mp is(\eta^2+\ulomega)} \frakn_j(\eta) \tilf_{\pm, \ulomega,\ulp}(s,\eta)\,\ud \eta,
\end{equation*}
or complex conjugates thereof. A Hilbert type contribution is of the form 
\begin{equation*}
I_{\pvdots}(s,y) := \int_\bbR e^{iy\xi_1} \fraka(\xi_1) \frakn_0(\xi_1) \widehat{\calF}[\tanh(\sqrt{\ulomega}\cdot) v_1(s,\cdot) \overline{v_2(s,\cdot)} v_3(s,\cdot)](\xi_1)\,\ud \xi_1, 
\end{equation*}
while a Dirac delta type contribution is of the form
\begin{equation*}
I_{\delta_0}(s,y) := \int_\bbR e^{iy\xi_1} \fraka(\xi_1) \frakn_0(\xi_1) \widehat{\calF}[v_1(s,\cdot) \overline{v_2(s,\cdot)} v_3(s,\cdot)](\xi_1)\,\ud \xi_1, 
\end{equation*}
where 
\begin{equation*}
v_j(s,y) := \int_\bbR e^{iy\eta} e^{- is(\eta^2+\ulomega)} \frakn_j(\eta) \tilf_{+, \ulomega,\ulp}(s,\eta)\,\ud \eta, \quad 1 \leq j \leq 3.
\end{equation*}
We note that the bounds \eqref{equ:preparation_flat_Schrodinger_wave_bound1}--\eqref{equ:preparation_flat_Schrodinger_wave_bound5} also applies to the free Schr\"odinger waves $w_j(s,y)$, $1 \leq j \leq 3$, and $v_j(s,y)$, $1 \leq j \leq 3$. On the regular contribution, we use the fact that $\fraka(\cdot)\frakn_0(\cdot) \in H^1(\bbR)$ (hence $\widehat{\calF}^{-1}[\fraka(\cdot)\frakn_0(\cdot)](y) \in L_y^{2,1}\subset L_y^1$), the bound  \eqref{equ:preparation_flat_Schrodinger_wave_bound1} and Young's convolution inequality to obtain
\begin{equation*}
\begin{split}
&\left| \int_t^T \int_\bbR q(y)  I_{\mathrm{reg}}(s,y) v_-(s,y) \,\ud y \,\ud s \right|\\
&\lesssim \int_t^T \|q \|_{L_y^1} \left\|\widehat{\calF}^{-1}\left[\fraka(\cdot)\frakn_0(\cdot)\right] \ast \big(\varphi(y)w_1(s,y)w_2(s,y)w_3(s,y)\big) \right\|_{L_y^\infty} \|v_-(s)\|_{L_y^\infty} \,\ud s\\
&\lesssim  \int_t^T \|q \|_{L_y^1} \left\|\widehat{\calF}^{-1}\left[\fraka(\cdot)\frakn_0(\cdot)\right]\right\|_{L_y^1} \left\| \varphi(y)w_1(s,y)w_2(s,y)w_3(s,y) \right\|_{L_{y}^\infty} \| v_-(s)\|_{L_y^\infty }\,\ud s\\
&\lesssim \int_t^T \| \varphi \|_{L_y^\infty} \| w_1(s)\|_{L_y^\infty}\| w_2(s)\|_{L_y^\infty}\| w_3(s)\|_{L_y^\infty} \|v_-(s)\|_{L_y^\infty} \,\ud s\\
&\lesssim \int_t^T \eps^4 \js^{-2} \,\ud s \lesssim \eps^4 \jt^{-1}.
\end{split}
\end{equation*}
Arguing in the same manner, we conclude that 
\begin{equation*}
\begin{split}
\left| \int_t^T \int_\bbR q(y)  I_{\pvdots}(s,y) v_-(s,y) \,\ud y \,\ud s \right| &\lesssim \int_t^T \|\tanh(\sqrt{\ulomega}\cdot)\|_{L_y^\infty} \| v_1(s)\|_{L_y^\infty}\| v_2(s)\|_{L_y^\infty}\| v_3(s)\|_{L_y^\infty} \|v_-(s)\|_{L_y^\infty} \,\ud s\\
&\lesssim \eps^4 \jt^{-1},
\end{split}
\end{equation*}
and
\begin{equation*}
\left| \int_t^T \int_\bbR q(y)  I_{\delta_0}(s,y) v_-(s,y) \,\ud y \,\ud s \right| \lesssim \int_t^T \| v_1(s)\|_{L_y^\infty}\| v_2(s)\|_{L_y^\infty}\| v_3(s)\|_{L_y^\infty} \|v_-(s)\|_{L_y^\infty} \,\ud s\lesssim \eps^4 \jt^{-1}.
\end{equation*}

\underline{Case 2:} $\fraka(\xi_1) = \frac{\xi_1^2}{(|\xi_1|-i\sqrt{\ulomega})^2}$.

In this case, we turn one factor of $\xi_1$ into a spatial derivative via integration by parts and write
\begin{equation*}
\begin{split}
\calI^{2,6}(t;T) &= \int_t^T \int_\bbR q(y)  \left(\int_\bbR e^{iy\xi_1} \frac{\xi_1^2}{(|\xi_1|-i\sqrt{\ulomega})^2} \wtilcalF_{+, \ulomega}[\calC(\ulPe \Vp(s))](\xi_1) \,\ud \xi_1 \right)   v_-(s,y) \,\ud y \,\ud s\\
&= i \int_t^T \int_\bbR \py q(y)  \left(\int_\bbR e^{iy\xi_1} \frac{\xi_1}{(|\xi_1|-i\sqrt{\ulomega})^2} \wtilcalF_{+, \ulomega}[\calC(\ulPe \Vp(s))](\xi_1) \,\ud \xi_1 \right)   v_-(s,y) \,\ud y \,\ud s\\
&\quad + i\int_t^T \int_\bbR q(y)  \left(\int_\bbR e^{iy\xi_1} \frac{\xi_1}{(|\xi_1|-i\sqrt{\ulomega})^2} \wtilcalF_{+, \ulomega}[\calC(\ulPe \Vp(s))](\xi_1) \,\ud \xi_1 \right)   \py v_-(s,y) \,\ud y \,\ud s.
\end{split}
\end{equation*}
The first term on the right-hand side with the factor $\py q(y)$ can be estimated like in the previous case since $q(y)$ is a Schwartz function. For the other term, we use Plancherel's identity, the $L^2$-boundedness of distorted Fourier transform along with the bounds \eqref{equ:consequences_sobolev_bound_V}, \eqref{equ:consequences_U_disp_decay}, \eqref{equ:preparation_flat_Schrodinger_wave_bound2} to conclude that 
\begin{equation*}
\begin{split}
&\left| \int_t^T \int_\bbR q(y)  \left(\int_\bbR e^{iy\xi_1} \frac{\xi_1}{(|\xi_1|-i\sqrt{\ulomega})^2} \wtilcalF_{+, \ulomega}[\calC(\ulPe \Vp(s))](\xi_1) \,\ud \xi_1 \right)   \py v_-(s,y) \,\ud y \,\ud s\right|\\
&\lesssim \int_t^T \| \jy q \|_{L_y^2} \| \calC(\ulPe \Vp(s))\|_{L_y^2} \| \jy^{-1} \py v_-(s,y)\|_{L_y^2}\,\ud s\\
&\lesssim \int_t^T \| \Vp(s)\|_{L_y^2}  \| \Vp(s)\|_{L_y^\infty} \| \Vp(s)\|_{L_y^\infty} \| \jy^{-1} \py v_-(s,y)\|_{L_y^2}\,\ud s\\
&\lesssim \int_t^T \eps \cdot \eps^2 \js^{-1} \cdot \eps \js^{-1+\delta} \,\ud s \lesssim \eps^4 \jt^{-1+\delta}.
\end{split}
\end{equation*}

Thus, we finish the discussion for estimating the integral $\calI^{2,6}(t;T)$. Finally, we have 
\begin{equation*}
\begin{split}
\calI^{2,7}(t;T) = - \int_t^T \int_\bbR q(y)  \left(\int_\bbR e^{iy\xi_1} \fraka(\xi_1) \wtilcalF_{+, \ulomega}[\ulPe \calMod(s) + \calE_1(s) + \calE_2(s) + \calE_3(s)](\xi_1) \,\ud \xi_1 \right)\\
\times  v_-(s,y) \,\ud y \,\ud s.
\end{split}
\end{equation*}
Using Plancherel's identity, the $L^2$ boundedness of the distorted Fourier transform (by Proposition~\ref{prop:mapping_properties_dist_FT}) along with the bounds \eqref{equ:consequences_ulPe_Mod_bounds}, \eqref{equ:consequences_calE1_bounds}, \eqref{equ:consequences_calE2_bounds}, \eqref{equ:consequences_calE3_bounds}, \eqref{equ:preparation_flat_Schrodinger_wave_bound1}, we conclude that 
\begin{equation*}
\begin{split}
\big|\calI^{2,7}(t;T)\big| &\lesssim \int_t^T \| q \|_{L_y^2} \left(\| \ulPe \calMod(s) \|_{L_y^2} + \| \calE_1(s) \|_{L_y^2} +  \| \calE_2(s) \|_{L_y^2} +  \| \calE_3(s) \|_{L_y^2}\right) \| v_-(s)\|_{L_y^\infty} \,\ud s\\
&\lesssim \int_t^T \eps^2 \js^{-\frac32+ \delta} \cdot \eps \js^{-\frac12} \,\ud s \lesssim \eps^3 \jt^{-1+\delta}.
\end{split}
\end{equation*}
By combining the preceeding estimates, we obtain that
\begin{equation*}
|\calI^2(t;T)| \lesssim \eps^2 \jt^{-1+\delta}.
\end{equation*}
In conclusion, we have shown that 
\begin{equation*}
|\calI(t;T) | \lesssim \eps^2 \jt^{-1+\delta},
\end{equation*}
and hence
\begin{equation*}
|\calI_{+-}(t;T) | + |\calJ_{+-}(t;T) | \lesssim \eps^2 \jt^{-1+\delta}.
\end{equation*}

The contribution of the non-resonant terms $|\calI_{++}(t;T) |$, $|\calI_{--}(t;T) |$, $|\calJ_{++}(t;T) |$, $|\calJ_{--}(t;T) |$ can be treated in the same way since the spectral distributions in \eqref{eqn: nu++}, \eqref{eqn: nu--}, \eqref{eqn: nu_p++}, \eqref{eqn: nu_p--} contain the phase $(\xi_1^2+\xi_2^2 + 2 \ulomega)$. Thus, the same arguments will show that 
\begin{equation*}
|\calI_{++}(t;T)| + |\calI_{--}(t;T)| + |\calJ_{++}(t;T)| +|\calJ_{--}(t;T)|\lesssim \eps^2 \jt^{-1+\delta}.
\end{equation*}
However, we emphasize that the presence of the factor $(\xi_1+\xi_2+2\ulomega)$ in the spectral distributions is strictly not necessary since the phase does not vanish for any $\xi_1,\xi_2 \in \bbR$. That is, one can still perform an integration by parts in time to achieve the same estimates. 

\medskip \noindent \underline{Contribution of $\int_t^T \bbM_{\ulomega}^{-1} \bfC(s)\,\ud  s$ and $\int_t^T \bbM_{\ulomega}^{-1} \bbA_\ulomega(s) \bbM_{\ulomega}^{-1} \bfQ(s) \,\ud s$ to $|\omega(t)-\omega(T)|$ and $|p(t)-p(T)|$}:

We recall that the components in the eigenfunctions $\big( Y_{j,\ulomega}(y)\big)_{1 \leq j \leq 4}$ are all Schwartz functions. Upon inspecting \eqref{equ:proof_bbM_ulomega_inv}, \eqref{equ:proof_vector_def}, we find that every component of the vector $\int_t^T \bbM_{\ulomega}^{-1} \bfC(s)\,\ud s$ is a linear combination of the cubic expression (or its complex conjugate)
\begin{equation}\label{equ:proof_mod_I}
I(t;T) := \int_t^T \langle \ve(s) \barve(s) \ve(s), \varphi\rangle \,\ud s 
\end{equation}
for some Schwartz function $\varphi(y)$. 

 In order to infer decay for \eqref{equ:proof_mod_I}, we insert the leading order decomposition \eqref{equ:consequences_usube_leading_order_local_decay_decomp} for $\ve(s,y)$ and respectively  \eqref{equ:consequences_barusube_leading_order_local_decay_decomp} for $\barve(s,y)$. We find that the leading expression for $I(t;T)$ are sums of terms of the form
\begin{equation}\label{equ:proof_mod_I_lead}
I_{\mathrm{lead}}(t;T) := \int_t^T h_{j_1,\ulomega,\ulp}(s) h_{j_2,\ulomega,\ulp}(s) h_{j_3,\ulomega,\ulp}(s) \,\ud s \cdot \langle \Phi_{k_1,\ulomega}\Phi_{k_2,\ulomega}\Phi_{k_3,\ulomega},\varphi\rangle 
\end{equation}
with $j_1,j_2,j_3,k_1,k_2,k_3 \in\{1,2\}$, while the lower order terms for $I(t;T)$ are given by
\begin{equation}
I_{\mathrm{lower}}(t;T) = \int_t^T \langle f_1(s) f_2(s) f_3(s), \varphi\rangle \,\ud s
\end{equation}
where at least one of the inputs $f_1(s,y)$, $f_2(s,y)$, or $f_3(s,y)$ are given by $R_{\vp,\ulomega}(s,y)$ or $R_{\barvp,\ulomega}(s,y)$, while the other two inputs are arbitrary among $h_{j,\ulomega,\ulp}(s) \Phi_{k,\ulomega}(y)$, $j,k\in\{1,2\}$, and $R_{\vp,\ulomega}(s,y)$, $R_{\barvp,\ulomega}(s,y)$. Thanks to the spatial localization of $\varphi(y)$ and the estimates \eqref{equ:consequences_h12_decay}, \eqref{equ:consequences_Ru_local_decay}, we infer the sufficient bound for the lower order term,
\begin{equation*}
\begin{split}
\left |I_{\mathrm{lower}}(t;T)\right| 
&\lesssim \int_t^T \left( |h_{1,\ulomega,\ulp}(s)| + |h_{2,\ulomega,\ulp}(s)|  + \big\|\jy^{-2} R_{\vp,\ulomega}(s,y)\big\|_{L_y^\infty } +\big\|\jy^{-2} R_{\barvp,\ulomega}(s,y)\big\|_{L_y^\infty }\right)^2 \times\\
&\qquad \qquad \times \left( \big\|\jy^{-2} R_{\vp,\ulomega}(s,y)\big\|_{L_y^\infty } +\big\|\jy^{-2} R_{\barvp,\ulomega}(s,y)\big\|_{L_y^\infty } \right)\,\ud s\\
&\lesssim \int_t^T \eps^2 \js^{-1} \cdot \eps \js^{-1+\delta} \,\ud s \lesssim \eps^3 \jt^{-1+\delta}.
\end{split}
\end{equation*}

For the leading term \eqref{equ:proof_mod_I_lead}, we exploit the fact that there are no time resonances. Let us define the sign function $\frakl(1) = +1$ and $\frakl(2) = -1$, and the function $\Lambda(j_1,j_2,j_3) := \frakl(j_1)+\frakl(j_2)+\frakl(j_3)$. The observation is that the phase $\Lambda = \Lambda(j_1,j_2,j_3)$ only takes values in $\{-3,-1,1,3\}$ which allows us integrate by parts in time to obtain
\begin{equation}\label{equ:proof_h3}
\begin{split}
&\int_t^T h_{j_1,\ulomega,\ulp}(s) h_{j_2,\ulomega}(s) h_{j_3,\ulomega,\ulp}(s) \,\ud s \\
&= \int_t^T e^{-is\Lambda(j_1,j_2,j_3)\ulomega} \big(e^{i s\frakl(j_1)\ulomega}h_{j_1,\ulomega,\ulp}(s)\big) \big(e^{i s\frakl(j_2)\ulomega}h_{j_2,\ulomega,\ulp}(s)\big)\big(e^{i s\frakl(j_3)\ulomega}h_{j_3,\ulomega,\ulp}(s)\big)\,\ud s\\
&= \frac{i}{\Lambda \ulomega} \bigg(\left[e^{-is\Lambda\ulomega} \big(e^{i s\frakl(j_1)\ulomega}h_{j_1,\ulomega,\ulp}(s)\big) \big(e^{i s\frakl(j_2)\ulomega}h_{j_2,\ulomega,\ulp}(s)\big)\big(e^{i s\frakl(j_3)\ulomega}h_{j_3,\ulomega,\ulp}(s)\big) \right]_{s=t}^{s=T}\\
&\qquad \qquad - \int_t^T e^{-is\Lambda\ulomega} \ps \big(e^{i s\frakl(j_1)\ulomega}h_{j_1,\ulomega,\ulp}(s)\big) \big(e^{i s\frakl(j_2)\ulomega}h_{j_2,\ulomega,\ulp}(s)\big)\big(e^{i s\frakl(j_3)\ulomega}h_{j_3,\ulomega,\ulp}(s)\big)\,\ud s + \{\text{similar terms}\}\bigg).
\end{split}
\end{equation}
Thus, using the decay estimates \eqref{equ:consequences_h12_decay}, \eqref{equ:consequences_h12_phase_filtered_decay} for the ampltitudes we obtain 
\begin{equation*}
\begin{split}
\left| I_{\mathrm{lead}}(t;T)\right| &\lesssim \left | \int_t^T h_{j_1,\ulomega,\ulp}(s) h_{j_2,\ulomega,\ulp}(s) h_{j_3,\ulomega,\ulp}(s) \,\ud s\right |\\ &\lesssim \eps^3\jt^{-\frac32}  + \eps^3 \jap{T}^{-\frac32} + \int_t^T \eps \js^{-1+\delta} \cdot \eps^2 \js^{-1} \,\ud s \lesssim \eps^3 \jt^{-1+\delta}.
\end{split}
\end{equation*}

Finally we inspect the defintions \eqref{equ:proof_bbA}, \eqref{equ:proof_bbM_ulomega_inv}, \eqref{equ:proof_vector_def} and find that every component of the vector $\int_t^T \bbM_{\ulomega}^{-1} \bbA_\ulomega(s) \bbM_{\ulomega}^{-1} \bfQ(s) \,\ud s$ is a linear combination of the cubic expression
\begin{equation}\label{equ:proof_mod_II}
II(t;T) := \int_t^T \langle g_1(s,\cdot),q_1(\cdot)\rangle \langle \phi_\ulomega(\cdot) g_2(s,\cdot),g_3(s,\cdot),q_2(\cdot) \rangle\,\ud s
\end{equation} 
where $q_1(y), q_2(y)$ are Schwartz functions and where the inputs $g_1(s,y)$, $g_2(s,y)$, $g_3(s,y)$ are given by either $\ve(s,y)$ or $\barve(s,y)$. Upon inserting the leading order decomposition \eqref{equ:consequences_usube_leading_order_local_decay_decomp} for $\ve(s,y)$ and respectively  \eqref{equ:consequences_barusube_leading_order_local_decay_decomp} for $\barve(s,y)$, we find that estimating $II(t;T)$ can be handled analogously to $I(t;T)$. Therefore, we conclude that
\begin{equation*}
\big| I(t;T) \big|+\big| II(t;T)\big| \lesssim \eps^3 \jt^{-1+\delta}.
\end{equation*}

\medskip \noindent \underline{Conclusion}: By combining all of the preceeding estimates, we arrive at the bound \eqref{equ:proof_mod_bootstrap}, which closes the bootstrap argument for Proposition~\ref{prop:modulation_parameters}.

\section{Weighted Energy Estimates for the Profile}\label{sec:energy_estimates}
In this section we derive slowly growing weighted energy estimates for the distorted Fourier transform of the profile.
\subsection{Estimates for the bilinear form and remainder terms}
In this subsection, we collect some preparatory estimates that will be used in the derivation of the weighted energy estimates and in the derivation of the pointwise estimates. We start with decay estimates for the remainder terms in the renormalized part of the quadratic nonlinearity.
\begin{lemma}\label{lemma:prep_Q_r-ulomega}
Suppose the assumptions in the statement of Proposition~\ref{prop:profile_bounds} are in place. Then we have 
\begin{align}
\sup_{0 \leq t \leq T} \jt^{\frac32 - \delta} \| \jym \calQ_{\mathrm{r},\ulomega}\big((\ulPe \Vp)(t)\big)\|_{L_y^2} &\lesssim \eps^2,\label{equ:consequences_calQ_r}\\
\sup_{0 \leq t \leq T} \jt^{\frac32 - \delta} \| \jym \py \calQ_{\mathrm{r},\ulomega}\big((\ulPe \Vp)(t)\big)\|_{L_y^2} &\lesssim \eps^2.\label{equ:consequences_pycalQ_r}
\end{align}
\end{lemma}
\begin{proof}
By inspection, we find that every component in $\calQ_{\mathrm{r},\ulomega}\big((\ulPe \Vp)(t)\big)$ given in \eqref{eqn:def_renormalized_quadratic} is a linear combination of terms of the form $\phi_{\ulomega}(y)g_1(s,y)g_2(s,y)$ where $\phi_{\ulomega}(y) = \sqrt{2\ulomega}\sech(\sqrt{\ulomega}y)$ is a Schwartz function, with input functions $g_1(s,y)$ given by $h_{1,\ulomega,\ulp}\Phi_{1,\ulomega}(y)$, $h_{1,\ulomega,\ulp}\Phi_{2,\ulomega}(y)$, $h_{2,\ulomega,\ulp}\Phi_{1,\ulomega}(y)$, $h_{2,\ulomega,\ulp}\Phi_{2,\ulomega}(y)$, $R_{\vp,\ulomega}(t,y)$, or $R_{\barvp,\ulomega}(t,y)$, and $g_2(s,y)$ given by $R_{\vp,\ulomega}(t,y)$, or $R_{\barvp,\ulomega}(t,y)$. Hence, the asserted estimates \eqref{equ:consequences_calQ_r}, \eqref{equ:consequences_pycalQ_r} follows from spatial localization of $\phi_{\ulomega}$, the boundedness of $\Phi_{1,\ulomega},\Phi_{2,\ulomega}$ and its derivatives, and from the decay estimates \eqref{equ:consequences_h12_decay}, \eqref{equ:consequences_Ru_local_decay}, \eqref{equ:consequences_px_Ru_local_decay}.
\end{proof}

The following lemma derives the pointwise and weighted energy estimates for the bilinear form $\wtilB_{\ulomega,\ulp}(t,\xi)$ and the for remainder term $\wtilcalR_{\frakq,\ulomega,\ulp}$ in the renormalized evolution equation \eqref{equ:setup_evol_equ_renormalized_tilfplus}.
\begin{lemma}\label{lemma:prep_B-ulomega} Let $\wtilB_{\ulomega,\ulp}(t,\xi)$ and $\wtilcalR_{\frakq,\ulomega,\ulp}$ be given by \eqref{equ:setup_definition_wtilBulomega} and \eqref{equ:setup_definition_wtilcalR_qulomega} respectively. Suppose the assumptions in the statement of Proposition~\ref{prop:profile_bounds} are in place. Then we have 
\begin{align}
\sup_{0 \leq t \leq T} \jt \| \jxi^2\wtilB_{\ulomega,\ulp}(t,\xi)\|_{L_\xi^\infty} &\lesssim \eps^2,\label{equ:consequences_wtilcalB}\\
\sup_{0 \leq t \leq T} \| \jxi \pxi \wtilB_{\ulomega,\ulp}(t,\xi)\|_{L_\xi^2} &\lesssim \eps^2,\label{equ:consequences_pxi_wtilcalB}\\
\sup_{0 \leq t \leq T} \jt^{\frac32-\delta} \|\jxi^2 \wtilcalR_{\frakq,\ulomega,\ulp}(t,\xi)\|_{L_\xi^\infty} &\lesssim \eps^2,\label{equ:consequences_wtilcalR_q}\\
\sup_{0 \leq t \leq T} \jt^{\frac32-\delta} \|\jxi \pxi \wtilcalR_{\frakq,\ulomega,\ulp}(t,\xi)\|_{L_\xi^2} &\lesssim \eps^2.\label{equ:consequences_pxi_wtilcalR_q}
\end{align}
\end{lemma}
\begin{proof}
%We begin with the proof of \eqref{equ:consequences_wtilcalB} and \eqref{equ:consequences_pxi_wtilcalB} for $\wtilB_{\ulomega,\ulp}(t,\xi)$. 
From \eqref{equ:setup_definition_wtilBulomega} we have
\begin{equation*}
\wtilB_{\ulomega,\ulp}(t,\xi) = e^{it(\xi^2-\ulomega)} b_1(t)\frakq_{1,\ulomega}(\xi) + e^{it(\xi^2+\ulomega)} b_2(t)\frakq_{2,\ulomega}(\xi) + e^{it(\xi^2+3\ulomega)} b_3(t)\frakq_{3,\ulomega}(\xi),
\end{equation*}
where $\frakq_{1,\ulomega}(\xi),\frakq_{2,\ulomega}(\xi),\frakq_{3,\ulomega}(\xi)$ are defined in \eqref{equ:def_frakq_1}--\eqref{equ:def_frakq_3} respectively, and where we use the shortened notation 
\begin{equation}\label{equ:def_b1b2b3}
b_1(t) := \big(e^{it\ulomega}h_{1,\ulomega,\ulp}(t)\big)^2, \quad b_2(t) := \big(e^{it\ulomega}h_{1,\ulomega,\ulp}(t)\big)\big(e^{-it\ulomega}h_{2,\ulomega,\ulp}(t)\big),\quad b_3(t) := \big(e^{-it\ulomega}h_{2,\ulomega,\ulp}(t)\big)^2.
\end{equation}
By Remark~\eqref{remark:smoothness_of_q}, the functions  $\frakq_{1,\ulomega}(\xi),\frakq_{2,\ulomega}(\xi),\frakq_{3,\ulomega}(\xi)$ and its derivative are spatially localized, and by \eqref{equ:consequences_h12_decay} we infer that
\begin{equation*}
\sup_{0 \leq t \leq T} \jt \big(|b_1(t)|+|b_2(t)|+|b_3(t)|\big)\lesssim \eps^2.
\end{equation*}
Hence, the asserted bounds \eqref{equ:consequences_wtilcalB} and \eqref{equ:consequences_pxi_wtilcalB} follow from the preceeding estimates. In a similar manner, we rewrite \eqref{equ:setup_definition_wtilcalR_qulomega} as
\begin{equation*}
\wtilcalR_{\frakq,\ulomega,\ulp}(t,\xi) = r_1(t)\frakq_{1,\ulomega}(\xi) +  r_2(t)\frakq_{2,\ulomega}(\xi) + r_3(t)\frakq_{3,\ulomega}(\xi),
\end{equation*}
where
\begin{align}
r_1(t) &:= 2i e^{-2it\ulomega}\big(e^{it\ulomega}h_{1,\ulomega,\ulp}(t)\big)\pt\big(e^{it\ulomega}h_{1,\ulomega,\ulp}(t)\big),\label{equ:def_r1}\\
r_2(t) &:= i\pt\big(e^{it\ulomega}h_{1,\ulomega,\ulp}(t)\big)\big(e^{-it\ulomega}h_{2,\ulomega,\ulp}(t)\big)+i\big(e^{it\ulomega}h_{1,\ulomega,\ulp}(t)\big)\pt\big(e^{-it\ulomega}h_{2,\ulomega,\ulp}(t)\big),\label{equ:def_r2}\\
r_3(t) &:=2i e^{2it\ulomega}\big(e^{it\ulomega}h_{1,\ulomega,\ulp}(t)\big)\pt\big(e^{it\ulomega}h_{1,\ulomega,\ulp}(t)\big),\label{equ:def_r3}
\end{align}
and by \eqref{equ:consequences_h12_decay}, \eqref{equ:consequences_h12_phase_filtered_decay} we find that 
\begin{equation*}
\sup_{0 \leq t \leq T} \jt^{\frac32-\delta} \big(|r_1(t)|+|r_2(t)|+|r_3(t)|\big)\lesssim \eps^2.
\end{equation*}
The asserted bounds \eqref{equ:consequences_wtilcalR_q} and \eqref{equ:consequences_pxi_wtilcalR_q} also holds.
\end{proof}
The next corollary bounds the remainder term $\wtilcalR_{\ulomega,\ulp}(t,\xi)$ defined in  \eqref{equ:setup_definition_wtilcalRulomega}.
\begin{corollary}\label{corollary:prep_Q_r-ulomega} Let $\wtilcalR_{\ulomega,\ulp}(t,\xi)$ be given by \eqref{equ:setup_definition_wtilcalRulomega}.
Suppose the assumptions in the statement of Proposition~\ref{prop:profile_bounds} are in place. Then we have 
\begin{equation}\label{equ:consequences_wtilcalR}
\sup_{0 \leq t \leq T} \jt^{\frac32 - \delta} \| \wtilcalR_{\ulomega,\ulp}(t,\xi)\|_{H_\xi^1} \lesssim \eps^2.
\end{equation}
\end{corollary}
\begin{proof}The estimate \eqref{equ:consequences_wtilcalR} follows from the $L_y^{2,1}\rightarrow H_{\xi}^1$ boundedness of the distorted Fourier transform by Proposition~\ref{prop:mapping_properties_dist_FT}, and from the decay estimates \eqref{equ:consequences_aux_bound_modulation2}, \eqref{equ:consequences_calL_bounds}--\eqref{equ:consequences_calE3_bounds}, \eqref{equ:consequences_calQ_r}, \eqref{equ:consequences_pycalQ_r},  \eqref{equ:consequences_wtilcalR_q}, \eqref{equ:consequences_pxi_wtilcalR_q}.
\end{proof}

\subsection{Main weighted energy estimates}
The proposition in this section now goes into the derivation of the weighted energy estimates for the profile. 
\begin{proposition}\label{prop: weighted-energy-estimate}
Suppose the assumptions in the statement of Proposition~\ref{prop:profile_bounds} are in place. Then we have for all $0 \leq t \leq T$ that 
\begin{equation}\label{equ:weighted_estimate}
\big \| \pxi \tilf_{+, \ulomega,\ulp}(t,\xi)\big \|_{L_\xi^2}+\big \| \pxi \tilf_{-, \ulomega,\ulp}(t,\xi)\big \|_{L_\xi^2} \lesssim \eps + \eps^2 \jt^\delta .
\end{equation}
\end{proposition}
\begin{proof}
We recall from \eqref{equ:setup_components_relation} that 
\begin{equation*}
\tilf_{-,\ulomega,\ulp}(t,\xi) = - \frac{\big(|\xi|-i\sqrt{\ulomega}\big)^2}{\big(|\xi|+i\sqrt{\ulomega}\big)^2}\overline{\tilf_{+,\ulomega,\ulp}(t,-\xi)},
\end{equation*}
which implies via \eqref{equ:consequences_sobolev_bound_profile} that
\begin{equation*}
\big \| \pxi \tilf_{-, \ulomega,\ulp}(t,\xi)\big \|_{L_\xi^2} \lesssim \big \|  \tilf_{+, \ulomega,\ulp}(t,\xi)\big \|_{L_\xi^2}+\big \| \pxi \tilf_{+, \ulomega,\ulp}(t,\xi)\big \|_{L_\xi^2} \lesssim \eps + \big \| \pxi \tilf_{+, \ulomega,\ulp}(t,\xi)\big \|_{L_\xi^2}.
\end{equation*}
It therefore suffices to consider the bounds on $\pxi \tilf_{+, \ulomega,\ulp}(t,\xi)$ to obtain control over both $\tilf_{\pm, \ulomega,\ulp}(t,\xi)$. We filter out the phases to write 
\begin{equation*}
\pxi \tilf_{+, \ulomega,\ulp}(t,\xi) = i\theta_2(t) \tilf_{+, \ulomega,\ulp}(t,\xi) + e^{-i\theta_1(t)} e^{i\theta_2(t)\xi}\pxi \big(e^{i\theta_1(t)} e^{-i\theta_2(t)\xi} \tilf_{+, \ulomega,\ulp}(t,\xi) \big),
\end{equation*}
and use \eqref{equ:consequences_growth_bound_theta} to bound
\begin{equation*}
\begin{split}
\big \| \pxi \tilf_{+, \ulomega,\ulp}(t,\xi)\big \|_{L_\xi^2} &\lesssim |\theta_2(t)|\ \big \| \tilf_{+, \ulomega,\ulp}(t,\xi)\big \|_{L_\xi^2} + \big \| \pxi\big(e^{i\theta_1(t)} e^{-i\theta_2(t)\xi} \tilf_{+, \ulomega,\ulp}(t,\xi) \big)\big \|_{L_\xi^2}\\
&\lesssim \eps^2 \jt^{\delta} + \big \| \pxi\big(e^{i\theta_1(t)} e^{-i\theta_2(t)\xi} \tilf_{+, \ulomega,\ulp}(t,\xi) \big)\big \|_{L_\xi^2}.
\end{split}
\end{equation*}
By integrating the renormalized profile equation \eqref{equ:setup_evol_equ_renormalized_tilfplus} in time, we get 
\begin{align}
&\big \| \pxi\big(e^{i\theta_1(t)} e^{-i\theta_2(t)\xi} \tilf_{+, \ulomega,\ulp}(t,\xi) \big)\big \|_{L_\xi^2} \nonumber  \\
&\lesssim \big \| \pxi \tilf_{+, \ulomega,\ulp}(0,\xi)\big \|_{L_\xi^2} + \sup_{0\leq s \leq t} \Big\|\pxi \big(e^{i\theta_1(s)} e^{-i\theta_2(s)\xi} \wtilB_{\ulomega,\ulp}(s,\xi) \big) \Big\|_{L_\xi^2} \label{equ:proof_weighted_estimate_1}\\
&\quad + \left\|\int_0^t \pxi \Big(e^{i\theta_1(s)} e^{-i\theta_2(s)\xi}\big(\dot{\theta}_1(s)-\dot{\theta}_2(s)\xi\big)\wtilB_{\ulomega,\ulp}(s,\xi) \Big) \,\ud s \right\|_{L_\xi^2}\label{equ:proof_weighted_estimate_2}\\
&\quad + \left\|\int_0^t \pxi \Big(e^{i\theta_1(s)} e^{-i\theta_2(s)\xi} e^{is(\xi^2+\ulomega)} \wtilcalR_{\ulomega,\ulp}(s,\xi) \Big) \,\ud s \right\|_{L_\xi^2}\label{equ:proof_weighted_estimate_3}\\
&\quad + \left\|\int_0^t \pxi \Big(e^{i\theta_1(s)} e^{-i\theta_2(s)\xi} e^{is(\xi^2+\ulomega)} \wtilcalF_{+, \ulomega}\big[\calC\big((\ulPe \Vp)(s)\big)\big](\xi) \Big) \,\ud s \right\|_{L_\xi^2}.\label{equ:proof_weighted_estimate_4}
\end{align}

In what follows, we bound \eqref{equ:proof_weighted_estimate_1}--\eqref{equ:proof_weighted_estimate_4} term by term. 
Using the mapping properties of the distorted Fourier transform from Proposition~\ref{prop:mapping_properties_dist_FT} and the stability bound \eqref{equ:consequences_sobolev_bound_V}, we have
\begin{equation*}
\big \| \pxi \tilf_{+, \ulomega,\ulp}(0,\xi)\big \|_{L_\xi^2} = \big \| \pxi \wtilcalF_{+,\ulomega}[\Vp(0)](\xi)\big \|_{L_\xi^2} \lesssim_{\omega_0} \|U(0)\|_{L^{2,1}} \lesssim \eps.
\end{equation*}
Moreover, the bilinear estimates \eqref{equ:consequences_wtilcalB}--\eqref{equ:consequences_pxi_wtilcalB} and the decay estimates for the phases \eqref{equ:consequences_growth_bound_theta} imply for $0 \leq s \leq t$ that
\begin{equation*}
\begin{split}
\Big\|\pxi \big(e^{i\theta_1(s)} e^{-i\theta_2(s)\xi} \wtilB_{\ulomega,\ulp}(s,\xi) \big) \Big\|_{L_\xi^2} &\lesssim  |\theta_2(s)| \, \Big\|\wtilB_{\ulomega,\ulp}(s,\xi)  \Big\|_{L_\xi^2} +  \Big\|\pxi \wtilB_{\ulomega,\ulp}(s,\xi) \Big\|_{L_\xi^2} \\
&\lesssim \eps \js^{\delta} \eps^2 \js^{-1} + \eps^2 \lesssim \eps^2.
\end{split}
\end{equation*}
Hence we obtain acceptable bounds for \eqref{equ:proof_weighted_estimate_1}. Similarly, using again the bilinear estimates \eqref{equ:consequences_wtilcalB}--\eqref{equ:consequences_pxi_wtilcalB} and the bounds \eqref{equ:consequences_aux_bound_modulation2}, \eqref{equ:consequences_growth_bound_theta}, we estimate \eqref{equ:proof_weighted_estimate_2} by 
\begin{equation*}
\begin{split}
&\left\|\int_0^t \pxi \Big(e^{i\theta_1(s)} e^{-i\theta_2(s)\xi}\big(\dot{\theta}_1(s)-\dot{\theta}_2(s)\xi\big)\wtilB_{\ulomega,\ulp}(s,\xi) \Big) \,\ud s \right\|_{L_\xi^2}\\
&\lesssim  \int_0^t \big(1+|\theta_2(s)|\big) \big(|\dot{\theta}_1(s)|+|\dot{\theta}_2(s)|\big) \big\|\jxi \wtilB_{\ulomega,\ulp}(s,\xi) \big \|_{L_\xi^2} \,\ud s \\
&\quad  +  \int_0^t \big(|\dot{\theta}_1(s)|+|\dot{\theta}_2(s)|\big) \big\|\jxi \pxi \wtilB_{\ulomega,\ulp}(s,\xi) \big \|_{L_\xi^2} \,\ud s\\
&\lesssim \int_0^t \big(1+\eps \js^\delta\big) \cdot \eps \js^{-1+\delta} \cdot \eps^2 \js^{-1} \,\ud s + \int_0^t \eps\js^{-1+\delta} \cdot \eps^2 \,\ud s \lesssim \eps^3 \jt^\delta.
\end{split}
\end{equation*}
The estimates for \eqref{equ:proof_weighted_estimate_3} and \eqref{equ:proof_weighted_estimate_4} are more technical when the partial derivative hits the  phase $\pxi e^{is(\xi^2+\ulomega)} = (2is\xi) e^{is(\xi^2+\ulomega)}$ due to the growing factor of $2is\xi$. The details for estimating \eqref{equ:proof_weighted_estimate_3} and \eqref{equ:proof_weighted_estimate_4} will now  occupy the rest of the proof. 

\medskip \noindent \underline{Contribution of the remainder term \eqref{equ:proof_weighted_estimate_3}}: Let us differentiate in $\xi$ and write
\begin{equation}\label{equ:proof_weighted_estimate_3_expanded}
\begin{split}
\int_0^t \pxi \Big(e^{i\theta_1(s)} e^{-i\theta_2(s)\xi} e^{is(\xi^2+\ulomega)} \wtilcalR_{\ulomega,\ulp}(s,\xi) \Big) \,\ud s &= -i\int_0^t \theta_2(s) e^{i\theta_1(s)} e^{-i\theta_2(s)\xi} e^{is(\xi^2+\ulomega)} \wtilcalR_{\ulomega,\ulp}(s,\xi)  \,\ud s\\
&\quad + \int_0^t e^{i\theta_1(s)} e^{-i\theta_2(s)\xi} (2is \xi) e^{is(\xi^2+\ulomega)} \wtilcalR_{\ulomega,\ulp}(s,\xi)  \,\ud s \\
&\quad + \int_0^t e^{i\theta_1(s)} e^{-i\theta_2(s)\xi} e^{is(\xi^2+\ulomega)} \pxi \wtilcalR_{\ulomega,\ulp}(s,\xi) \,\ud s.
\end{split}
\end{equation}
By \eqref{equ:consequences_wtilcalR} and \eqref{equ:consequences_growth_bound_theta}, it is straightforward to deduce acceptable $L_\xi^2$-bounds for the first term on the right-hand side of \eqref{equ:proof_weighted_estimate_3_expanded}:
\begin{equation*}
\begin{split}
&\left \| \int_0^t \theta_2(s) e^{i\theta_1(s)} e^{-i\theta_2(s)\xi} e^{is(\xi^2+\ulomega)} \wtilcalR_{\ulomega,\ulp}(s,\xi)  \,\ud s\right \|_{L_\xi^2}\\
&\lesssim \int_0^t |\theta_2(s)| \, \left \| \wtilcalR_{\ulomega,\ulp}(s,\xi)\right \|_{L_\xi^2}  \,\ud s \lesssim \int_0^t \eps \js^{\delta} \eps^2 \js^{-\frac32 + \delta} \,\ud s \lesssim \eps^3 \jt^\delta,
\end{split}
\end{equation*}
as well as for the third term on the right-hand side of \eqref{equ:proof_weighted_estimate_3_expanded}:
\begin{equation*}
\left \| \int_0^t e^{i\theta_1(s)} e^{-i\theta_2(s)\xi} e^{is(\xi^2+\ulomega)} \pxi \wtilcalR_{\ulomega,\ulp}(s,\xi)  \,\ud s\right \|_{L_\xi^2} \lesssim \int_0^t \left \| \pxi \wtilcalR_{\ulomega,\ulp}(s,\xi)\right \|_{L_\xi^2}  \,\ud s \lesssim \int_0^t \eps^2 \js^{-\frac32 + \delta} \,\ud s \lesssim \eps^2 \jt^\delta.
\end{equation*}
For the weighted estimate of the second term in \eqref{equ:proof_weighted_estimate_3_expanded}, we need to ``undo" the (distorted) Fourier transform of every term in the definition of $\wtilcalR_{\ulomega,\ulp}(s,\xi)$. More precisely, we claim that $\wtilcalR_{\ulomega,\ulp}(s,\xi)$ is a linear combination of terms of the form
\begin{equation}\label{equ:proof_weighted_assp1} 
\int_\bbR e^{-iy\xi} \fraka(y)\frakb(\xi) H(s,y) \,\ud y
\end{equation}
where $\fraka(y)$ is smooth with bounded derivatives, $\frakb(\xi) \in W^{1,\infty}$, and $H(s,y)$ satisfies the local decay 
\begin{equation}\label{equ:proof_weighted_assp2}
\|\jym H(s,y)\|_{L_y^2} + \|\jym \py H(s,y)\|_{L_y^2}\lesssim \eps^2 \js^{-\frac32 + \delta}.
\end{equation}
From \eqref{equ:setup_definition_wtilcalRulomega} we recall that
\begin{equation*}
\begin{split}
\widetilde{\calR}_{\ulomega,\ulp}(s,\xi) &:= \wtilcalF_{+,\ulomega}\bigl[ \calQ_{\mathrm{r},\ulomega}\bigl((\ulPe \Vp)(s)\bigr) +\ulPe \calMod(s)+\calE_1(s)+\calE_2(s)+\calE_3(s)\bigr](\xi)\\
&\quad + \dot{\theta}_1(s) \calL_{+, \ulomega}[\Vp(s)](\xi)  + i \dot{\theta}_2(s) \calK_{+,\ulomega}[\Vp(s)](\xi)+\widetilde{\calR}_{\frakq,\ulomega,\ulp}(s,\xi).
\end{split}
\end{equation*}
The term $\wtilcalF_{+, \ulomega}[\calQ_{\mathrm{r},\ulomega}\bigl((\ulPe \Vp)(s)\bigr)+\cdots +\calE_3(s)](\xi)$ satisfy the claims \eqref{equ:proof_weighted_assp1} and \eqref{equ:proof_weighted_assp2} thanks to the mapping properties of the distorted Fourier transform by Lemma~\ref{lemma: PDO on m12}, and the bounds given by \eqref{equ:consequences_calQ_r}, \eqref{equ:consequences_pycalQ_r} of Lemma~\ref{lemma:prep_Q_r-ulomega} and by \eqref{equ:consequences_ulPe_Mod_bounds}--\eqref{equ:consequences_calE3_bounds} of Corollary~\ref{cor:consequences}. Moreover, as in the proof of \eqref{equ:consequences_calL_bounds} and \eqref{equ:consequences_calK_bounds}, we recall that 
\begin{equation*}
\calL_{+, \ulomega}[\Vp(t)](\xi) = \frac{2}{\sqrt{2\pi}} \int_\bbR e^{-iy\xi} \vp(s,y)\fraka_{\ulomega}(y)\frakb_{\ulomega}(\xi)\,\ud y,
\end{equation*}
and
\begin{equation*}
\calK_{+,\ulomega}[\Vp(t)](\xi) = \frac{1}{\sqrt{2\pi}}\int_\bbR e^{-iy\xi} \Big(\vp(s,y)\big(\fraka_{\ulomega}(y)\frakc_{\ulomega}(\xi)+\fraka_{\ulomega}'(y)\frakb_{\ulomega}(\xi)\big)+\barvp(s,y)\fraka_{\ulomega}'(y)\frakb_{\ulomega}(\xi) \Big)\,\ud y.
\end{equation*}
with $\fraka_{\ulomega}(y) \in \calS(\bbR)$, and $\frakb_{\ulomega}(\xi),\frakc_{\ulomega}(\xi) \in W^{1,\infty}(\bbR)$ defined in \eqref{equ:def_fraka-b-c}. Hence, by pairing the spatial components with $\dot{\theta}_1,\dot{\theta}_2$, we obtain from \eqref{equ:consequences_aux_bound_modulation2}, \eqref{equ:consequences_U_disp_decay},  \eqref{equ:consequences_pyU_disp_decay} and from the spatial localization of $\fraka_{\ulomega}(y)$ for all $0 \leq s \leq T$ that 
\begin{equation*}
\begin{split}
&\left\|\jym \big(\dot{\theta}_1(s) \Vp(s,y)\fraka_{\ulomega}(y)\big) \right\|_{L_y^2}+\left\|\jym \py\big(\dot{\theta}_1(s) \Vp(s,y)\fraka_{\ulomega}(y)\big) \right\|_{L_y^2}\\
&\quad+ \left\|\jym \big(\dot{\theta}_2(s) \Vp(s,y)\big(\fraka_{\ulomega}(y)+\fraka_{\ulomega}'(y)\big)\big) \right\|_{L_y^2}+\left\|\jym \py\big(\dot{\theta}_2(s) \Vp(s,y)\big(\fraka_{\ulomega}(y)+\fraka_{\ulomega}'(y)\big)\big) \right\|_{L_y^2}\\
&\lesssim \big(|\dot{\theta}_1(s)| + |\dot{\theta}_2(s)|\big)\, \big( \| \Vp(s,y) \|_{L_y^\infty} + \| \jy^{-3} \py \Vp(s,y) \|_{L_y^2}\big) \lesssim \eps^2 \js^{-\frac32+\delta}.
\end{split}
\end{equation*}
The claims \eqref{equ:proof_weighted_assp1} and \eqref{equ:proof_weighted_assp2} are now satisfied for the terms $\dot{\theta}_1(s) \calL_{+, \ulomega}[\Vp(s)](\xi)$ and  $i \dot{\theta}_2(s) \calK_{+,\ulomega}[\Vp(s)](\xi)$. Finally, we examine the structure of the remainder term $\widetilde{\calR}_{\frakq,\ulomega,\ulp}(s,\xi)$ defined in \eqref{equ:setup_definition_wtilcalR_qulomega}. We recall that in the proof of Lemma~\ref{lemma:prep_B-ulomega}, we have
\begin{equation*}
\widetilde{\calR}_{\frakq,\ulomega,\ulp}(s,\xi) = r_1(s)\frakq_{1,\ulomega}(\xi)+r_2(s)\frakq_{2,\ulomega}(\xi)+r_3(s)\frakq_{3,\ulomega}(\xi),
\end{equation*}
where the coefficients $r_1,r_2,r_3$ defined in \eqref{equ:def_r1}--\eqref{equ:def_r3} satisfy the decay estimate
\begin{equation*}
\sup_{0\leq s \leq T} \js^{\frac32-\delta}\big(|r_1(s)|+|r_2(s)|+|r_3(s)|\big) \lesssim \eps^2,
\end{equation*}
and where $\frakq_{1,\ulomega},\frakq_{2,\ulomega},\frakq_{3,\ulomega}$ were defined in \eqref{equ:def_frakq_1}--\eqref{equ:def_frakq_3}. In view of Lemma~\ref{lem:null_structure_radiation}, Remark~\ref{remark:smoothness_of_q}, and the mapping properties of the distorted Fourier transform by Lemma~\ref{lemma: PDO on m12}, the coefficients $\frakq_{1,\ulomega}(\xi),\frakq_{2,\ulomega}(\xi),\frakq_{3,\ulomega}(\xi)$ can be expressed as linear combination of terms of the form 
\begin{equation*}
\int_{\bbR} e^{-iy\xi} \varphi(y)\frakd(\xi)\,\ud y
\end{equation*}
for some $\varphi \in \calS(\bbR)$ and $\frakd \in W^{1,\infty}(\bbR)$. Therefore the term $\widetilde{\calR}_{\frakq,\ulomega,\ulp}(s,\xi)$ is of the generic form  \eqref{equ:proof_weighted_assp1} and satisfies \eqref{equ:proof_weighted_assp2}. Thus, we are in the position to estimate the second term in \eqref{equ:proof_weighted_estimate_3_expanded} where we may replace $\wtilcalR_{\ulomega,\ulp}(s,\xi)$ with the generic term \eqref{equ:proof_weighted_assp1}. Let us insert the cutoff function $\eta(\xi)$ specified in Proposition~\ref{prop:local-smoothing} into the generic expression \eqref{equ:proof_weighted_assp1} to obtain
\begin{equation}\label{equ:proof_weighted_local_smooth1}
\begin{split}
&\int_0^t e^{i\theta_1(s)} e^{-i\theta_2(s)\xi} (2is \xi) e^{is(\xi^2+\ulomega)} \wtilcalR_{\ulomega,\ulp}(s,\xi)  \,\ud s \\
&\simeq \int_0^t e^{is\xi^2} e^{-i\theta_2(s)\xi} \xi \left[\int_\bbR e^{-iy\xi} \fraka(y)\frakb(\xi) \big(s e^{is\ulomega} e^{i\theta_1(s)}H(s,y)\big)\,\ud y\right]\,\ud  s\\
&= \int_0^t e^{is\xi^2} e^{-i\theta_2(s)\xi} \big(\xi\eta(\xi)\big)  \left[\int_\bbR e^{-iy\xi} \fraka(y)\frakb(\xi) \big(s e^{is\ulomega} e^{i\theta_1(s)}H(s,y)\big)\,\ud y\right]\,\ud  s \\
&\quad + \int_0^t e^{is\xi^2} e^{-i\theta_2(s)\xi} \big(1 - \eta(\xi)\big)\xi \left[\int_\bbR e^{-iy\xi} \fraka(y)\frakb(\xi) \big(s e^{is\ulomega} e^{i\theta_1(s)}H(s,y)\big)\,\ud y\right]\,\ud  s\\
&= \int_0^t e^{is\xi^2} e^{-i\theta_2(s)\xi} \big(\xi\eta(\xi)\big) \left[\int_\bbR e^{-iy\xi} \fraka(y)\frakb(\xi) \big(s e^{is\ulomega} e^{i\theta_1(s)}H(s,y)\big)\,\ud y\right]\,\ud  s \\
&\quad -i \int_0^t e^{is\xi^2} e^{-i\theta_2(s)\xi} \big(1 - \eta(\xi)\big) \left[\int_\bbR e^{-iy\xi} \fraka'(y)\frakb(\xi) \big(s e^{is\ulomega} e^{i\theta_1(s)}H(s,y)\big)\,\ud y\right]\,\ud  s\\
&\quad -i \int_0^t e^{is\xi^2} e^{-i\theta_2(s)\xi} \big(1 - \eta(\xi)\big) \left[\int_\bbR e^{-iy\xi} \fraka(y)\frakb(\xi) \big(s e^{is\ulomega} e^{i\theta_1(s)} \py H(s,y)\big)\,\ud y\right]\,\ud  s,
\end{split}
\end{equation}
noting an integration by parts on the last equality. By invoking the smoothing estimates with a moving center \eqref{equ:estimate_smoothing1} and \eqref{equ:estimate_smoothing2} of Proposition~\ref{prop:local-smoothing} along with the bound \eqref{equ:proof_weighted_assp2}, we find that 
\begin{equation}\label{equ:proof_weighted_local_smooth2}
\begin{split}
&\left \|\int_0^t e^{i\theta_1(s)} e^{-i\theta_2(s)\xi} (2is \xi) e^{is(\xi^2+\ulomega)} \wtilcalR_{\ulomega,\ulp}(s,\xi)  \,\ud s \right \|_{L_\xi^2}\\
&\simeq \left\|\int_0^t e^{is\xi^2} e^{-i\theta_2(s)\xi} \xi \left[\int_\bbR e^{-iy\xi} \fraka(y)\frakb(\xi) \big(s e^{is\ulomega} e^{i\theta_1(s)}H(s,y)\big)\,\ud y\right]\,\ud  s\right \|_{L_\xi^2} \\
&\lesssim \left\|\int_0^t e^{is\xi^2} e^{-i\theta_2(s)\xi} \big(\xi\eta(\xi)\big) \left[\int_\bbR e^{-iy\xi} \fraka(y)\frakb(\xi) \big(s e^{is\ulomega} e^{i\theta_1(s)}H(s,y)\big)\,\ud y\right]\,\ud  s\right \|_{L_\xi^2} \\
&\quad + \left\|\int_0^t e^{is\xi^2} e^{-i\theta_2(s)\xi} \big(1-\eta(\xi)\big) \left[\int_\bbR e^{-iy\xi} \fraka'(y)\frakb(\xi) \big(s e^{is\ulomega} e^{i\theta_1(s)}H(s,y)\big)\,\ud y\right]\,\ud  s\right \|_{L_\xi^2}\\
&\quad + \left\|\int_0^t e^{is\xi^2} e^{-i\theta_2(s)\xi} \big(1-\eta(\xi)\big) \left[\int_\bbR e^{-iy\xi} \fraka(y)\frakb(\xi) \big(s e^{is\ulomega} e^{i\theta_1(s)} \py H(s,y)\big)\,\ud y\right]\,\ud  s\right \|_{L_\xi^2}\\
&\lesssim  \|s \jy^2  H(s,y)\|_{L_s^2([0,t];L_y^2)} + \|s \jy^2  \py H(s,y)\|_{L_s^2([0,t];L_y^2)}\lesssim \eps^2 \| s \cdot \js^{-\frac32 + \delta} \|_{L_s^2([0,t])} \lesssim \eps^2 \jt^{\delta}.
\end{split}
\end{equation}
This completes the discussion for \eqref{equ:proof_weighted_estimate_3}.

\medskip \noindent \underline{Contribution of the cubic term \eqref{equ:proof_weighted_estimate_4}}:
We have
\begin{align}
\text{\eqref{equ:proof_weighted_estimate_4}}& =  \int_0^t e^{i\theta_1(s)} (-i)\theta_2(s) e^{-i\theta_2(s)\xi} e^{is(\xi^2+\ulomega)} \wtilcalF_{+, \ulomega}\big[\calC\big((\ulPe \Vp)(s)\big)\big](\xi) \,\ud s \label{equ:proof_weighted_cubic1}\\
&\quad + \int_0^t e^{i\theta_1(s)}  e^{-i\theta_2(s)\xi} \pxi \Big( e^{is(\xi^2+\ulomega)} \wtilcalF_{+, \ulomega}\big[\calC\big((\ulPe \Vp)(s)\big)\big](\xi)  \Big)\,\ud s.\label{equ:proof_weighted_cubic2}
\end{align}
The $L_\xi^2$-norm of \eqref{equ:proof_weighted_cubic1} can be quickly estimated as follows: we use the growth bound \eqref{equ:consequences_growth_bound_theta} on $\theta_2$, the $L^2$-boundedness of the distorted Fourier transform (c.f. Proposition~\ref{prop:mapping_properties_dist_FT}), \eqref{equ:consequences_sobolev_bound_V}, and \eqref{equ:consequences_ulPe_U_disp_decay} to obtain the acceptable bound
\begin{equation*}
\begin{split}
&\left \| \int_0^t e^{i\theta_1(s)} (-i)\theta_2(s) e^{-i\theta_2(s)\xi} e^{is(\xi^2+\ulomega)} \wtilcalF_{+, \ulomega}\big[\calC\big((\ulPe \Vp)(s)\big)\big](\xi) \,\ud s \right \|_{L_\xi^2}\\
&\lesssim \int_0^t |\theta_2(s)| \big \| \wtilcalF_{+, \ulomega}\big[\calC\big((\ulPe \Vp)(s)\big)\big](\xi) \big \|_{L_\xi^2} \,\ud s \lesssim \int_0^t |\theta_2(s)| \big \| |\ve|^2 \ve \big \|_{L_y^2} \,\ud s\\
&\lesssim \int_0^t |\theta_2(s)| \| \ve(s) \|_{L_y^\infty}^2 \| \ve(s) \|_{L_y^2} \,\ud s \lesssim \int_0^t \eps \js^\delta \cdot \eps^2 \js^{-1} \cdot \eps \,\ud s \lesssim \eps^4 \jt^\delta.
\end{split}
\end{equation*}
For \eqref{equ:proof_weighted_cubic2}, we use the precise structure of the cubic nonlinearity described in Sect.~\ref{subsec:cubic_spectral_distributions}. We have the following decomposition
\begin{equation}\label{equ:cubic_decomp}
e^{is(\xi^2+\ulomega)} \wtilcalF_{+, \ulomega}\big[\calC\big((\ulPe \Vp)(s)\big)\big](\xi) = \calJ_{\delta_0}(s,\xi) + \calJ_{\pvdots}(s,\xi) + \calJ_{\mathrm{reg}}(s,\xi),
\end{equation}
where the terms on the right-hand side represent trilinear expressions involving a Dirac delta kernel, a Hilbert-type kernel, and a regular kernel respectively. In order to lighten the notation, we write $\tilf_{\pm} = \tilf_{\pm, \ulomega,\ulp}$. After some direct computations, the expressions for the Dirac delta and Hilbert-type cubic interactions are given respectively by 
\begin{equation}\label{equ:calJ_delta}
\begin{split}
\calJ_{\delta_0}(s,\xi) := \frac{-1}{2\pi} \iint e^{is\Phi(\xi,\xi_1,\xi_2,\xi-\xi_1+\xi_2)} \tilf_+(s,\xi_1)&\overline{\tilf_+(s,\xi_2)}\tilf_+(s,\xi-\xi_1+\xi_2) \times\\
&\times \frac{\frakp_1\Big(\frac{\xi}{\sqrt{\ulomega}},\frac{\xi_1}{\sqrt{\ulomega}},\frac{\xi_2}{\sqrt{\ulomega}},\frac{\xi-\xi_1+\xi_2}{\sqrt{\ulomega}}\Big)}{\frakp\Big(\frac{\xi}{\sqrt{\ulomega}},\frac{\xi_1}{\sqrt{\ulomega}},\frac{\xi_2}{\sqrt{\ulomega}},\frac{\xi-\xi_1+\xi_2}{\sqrt{\ulomega}}\Big)}\,\ud \xi_1\,\ud \xi_2,
\end{split}
\end{equation}
\begin{equation}\label{equ:calJ_pv}
\begin{split}
\calJ_{\pvdots}(s,\xi) := \frac{-1}{4\pi\sqrt{\ulomega}}\iiint e^{is \Phi(\xi,\xi_1,\xi_2,\xi-\xi_1+\xi_2-\xi_4)} \tilf_+(s,\xi_1)&\overline{\tilf_+(s,\xi_2)} \tilf_+(s,\xi-\xi_1+\xi_2-\xi_4) \times\\
\times \frac{\frakp_2\Big(\frac{\xi}{\sqrt{\ulomega}},\frac{\xi_1}{\sqrt{\ulomega}},\frac{\xi_2}{\sqrt{\ulomega}},\frac{\xi-\xi_1+\xi_2}{\sqrt{\ulomega}}\Big)}{\frakp\Big(\frac{\xi}{\sqrt{\ulomega}},\frac{\xi_1}{\sqrt{\ulomega}},\frac{\xi_2}{\sqrt{\ulomega}},\frac{\xi-\xi_1+\xi_2}{\sqrt{\ulomega}}\Big)} &\pvdots \cosech\Big(\frac{\pi\xi_4}{2 \sqrt{\ulomega}}\Big) \,\ud \xi_1\,\ud \xi_2\,\ud \xi_4,
\end{split}
\end{equation}
with $\frakp,\frakp_1,\frakp_2$ defined in \eqref{eqn: cubic-frakp}--\eqref{eqn: cubic-frakp2} from Lemma ~\ref{lemma:cubic_NSD}, and 
\begin{equation}
\Phi(\xi_0,\xi_1,\xi_2,\xi_3) := \xi_0^2 - \xi_1^2 + \xi_2^2 - \xi_3^2.
\end{equation}
The regular cubic interaction $\calJ_{\mathrm{reg}}(s,\xi)$ is a linear combination of terms of the schematic form 
\begin{equation}\label{equ:calJ_reg}
\begin{split}
\calJ_{\mathrm{reg}}^{\mathrm{schem}}(s,\xi) &:= e^{is(\xi^2+\ulomega)} \overline{\frakb_0(\xi)} \int_\bbR w_1(s,y) \overline{w_2(s,y)} w_3(s,y) \varphi(y) e^{-iy \xi} \,\ud y\\
&\quad \text{or} \quad  e^{is(\xi^2+\ulomega)} \overline{\frakb_0(\xi)} \int_\bbR \overline{w_1(s,y)} w_2(s,y) \overline{w_3(s,y)} \varphi(y) e^{-iy \xi} \,\ud y,
\end{split}
\end{equation}
where $\varphi(y)$ is some Schwartz function, and where  $w_{j}(s,y)$, $1\leq j \leq 3$ are free Schr\"odinger waves given by 
\begin{equation}\label{equ:proof_free_schr_waves}
w_{j}(s,y) = \int_\bbR e^{iy \eta} e^{-is(\eta^2+\ulomega)}\frakb_j(\eta) \tilf_+(s,\eta) \,\ud  \eta \quad \text{or} \quad \int_\bbR e^{iy \eta} e^{is(\eta^2+\ulomega)}\frakb_j(\eta) \tilf_-(s,\eta) \,\ud  \eta
\end{equation}
with symbols $\frakb_0,\ldots,\frakb_3 \in W^{1,\infty}(\bbR)$. By \eqref{equ:preparation_flat_Schrodinger_wave_bound1}--\eqref{equ:preparation_flat_Schrodinger_wave_bound3} the free waves $w_j(s,y)$ enjoy the following estimates 
\begin{align}
\big \| w_j(s,y) \big \|_{L_y^\infty} &\lesssim \eps \js^{-\frac12},\label{equ:proof-free-disp1}\\
\big \| w_j(s,y) \big \|_{L_y^2} &\lesssim \eps,\label{equ:proof-free-disp2}\\
\big \| \jy^{-1} \py w_j(s,y) \big \|_{L_y^2} &\lesssim \eps \js^{-1+\delta},\label{equ:proof-free-disp5}
\end{align}
for $1 \leq j \leq 3$. In the rest of the proof, we proceed by estimating each different type of cubic interaction.

\medskip \noindent \underline{Case 1: Cubic interactions with a Dirac kernel.} We observe that the phase in \eqref{equ:calJ_delta} can be factorized
\begin{equation*}
\Phi(\xi,\xi_1,\xi_2,\xi-\xi_1+\xi_2) = \xi^2 - \xi_1^2 + \xi_2^2 - (\xi-\xi_1+\xi_2)^2 = 2(\xi-\xi_1)(\xi_1-\xi_2).
\end{equation*}
By making the change of variables $\eta_1 = \xi - \xi_1$, $\eta_2 = \xi_1 - \xi_2$, we may rewrite \eqref{equ:calJ_delta} as
\begin{equation*}
\begin{split}
\calJ_{\delta_0}(s,\xi) = \iint e^{2 i s \eta_1 \eta_2} \tilf_+(s,\xi-\eta_1) &\overline{\tilf_+(s,\xi-\eta_1-\eta_2)} \tilf_+(s,\xi-\eta_2) \times\\
&\times \frac{\frakp_1\Big(\frac{\xi}{\sqrt{\ulomega}},\frac{\xi-\eta_1}{\sqrt{\ulomega}},\frac{\xi - \eta_1 - \eta_2 }{\sqrt{\ulomega}},\frac{\xi-\eta_2}{\sqrt{\ulomega}}\Big)}{\frakp\Big(\frac{\xi}{\sqrt{\ulomega}},\frac{\xi-\eta_1}{\sqrt{\ulomega}},\frac{\xi - \eta_1 - \eta_2  }{\sqrt{\ulomega}},\frac{\xi-\eta_2}{\sqrt{\ulomega}}\Big)}\,\ud \eta_1\,\ud \eta_2.
\end{split}
\end{equation*}
In view of tensorized structure of the symbol $\frakp_1/\frakp$ given by Item~(1) of Lemma~\ref{lemma:cubic_NSD}, the term $\calJ_{\delta_0}$ is a linear combination of terms of the schematic form 
\begin{equation*}
\calJ_{\delta_0}^{\mathrm{schem}}(s,\xi) := \overline{\frakb_0(\xi)}\iint e^{2is \eta_1 \eta_2} g_1(s,\xi-\eta_1) \overline{g_2(s,\xi-\eta_1-\eta_2)} g_3(s,\xi-\eta_2) \,\ud \eta_1\,\ud \eta_2
\end{equation*}
with inputs $g_j(s,\xi_j) := \frakb_j(\xi_j)\tilf_+(s,\xi_j)$ and  symbols $\frakb_j \in W^{1,\infty}(\bbR)$, $0 \leq j \leq 3$. We find that the partial derivative of the schematic term is given by
\begin{equation*}
\begin{split}
\pxi \calJ_{\delta_0}^{\mathrm{schem}}(s,\xi)  &= \overline{(\partial_\xi\frakb_0)(\xi)}\iint e^{2is \eta_1 \eta_2} g_1(s,\xi-\eta_1) \overline{g_2(s,\xi-\eta_1-\eta_2)} g_3(s,\xi-\eta_2) \,\ud \eta_1\,\ud \eta_2\\
&\quad + \overline{\frakb_0(\xi)}\iint e^{2is \eta_1 \eta_2} (\pxi g_1)(s,\xi-\eta_1) \overline{g_2(s,\xi-\eta_1-\eta_2)} g_3(s,\xi-\eta_2) \,\ud \eta_1\,\ud \eta_2\\
&\quad + \overline{\frakb_0(\xi)} \iint e^{2is \eta_1 \eta_2} g_1(s,\xi-\eta_1) \overline{(\pxi g_2)(s,\xi-\eta_1-\eta_2)} g_3(s,\xi-\eta_2) \,\ud \eta_1\,\ud \eta_2\\
&\quad + \overline{\frakb_0(\xi)} \iint e^{2is \eta_1 \eta_2} g_1(s,\xi-\eta_1) \overline{g_2(s,\xi-\eta_1-\eta_2)} g_3(s,\xi-\eta_2) \,\ud \eta_1\,\ud \eta_2\\
&=: \calI_{\delta_0}^1(s,\xi)  + \calI_{\delta_0}^2(s,\xi) + \calI_{\delta_0}^3(s,\xi) + \calI_{\delta_0}^4(s,\xi). 
\end{split}
\end{equation*}
In order to exploit the oscillatory nature of the inputs, we undo the change of variables back to $\xi_1$ and $\xi_2$ and express each of the preceeding terms as a convolution. For $1\leq j \leq3$, we define
\begin{equation*}
\wtilw_j(s,\xi_j) := e^{-is\xi_j^2}\frakb_j(\xi_j) \tilf_+(s,\xi_j), \quad \text{and} \quad \wtilw_j^\#(s,\xi_j) := e^{-is\xi_j^2}\pxi \big(\frakb_j(\xi_j) \tilf_+(s,\xi_j)\big),
\end{equation*}
and we find that
\begin{equation*}
\begin{split}
\calI_{\delta_0}^1(s,\xi) &= e^{is\xi^2}\overline{(\pxi\frakb_0)(\xi)} \Big(\wtilw_1(s,\cdot) \ast \overline{\wtilw_2(s,\cdot)} \ast \wtilw_3(s,\cdot)\Big)(\xi),\\
\calI_{\delta_0}^2(s,\xi) &=e^{is\xi^2}\overline{(\frakb_0)(\xi)} \Big(\wtilw_1^\#(s,\cdot) \ast \overline{\wtilw_2(s,\cdot)} \ast \wtilw_3(s,\cdot)\Big)(\xi),
\end{split}
\end{equation*}
with similar expressions for $\calI_{\delta_0}^3$ and $\calI_{\delta_0}^4$. Noting that $w_j = \widehat{\calF}^{-1}[\wtilw_j]$ with $w_j(s,y)$ given in \eqref{equ:proof_free_schr_waves}, we obtain from \eqref{equ:preparation_flat_Schrodinger_wave_bound1} and \eqref{equ:preparation_flat_Schrodinger_wave_bound2} respectively the bounds 
\begin{align*}
\big\|  \widehat{\calF}^{-1}[\wtilw_j(s,\cdot)](y)\big \|_{L_y^\infty} &= \big \| w_j(s,y) \big \|_{L_y^\infty} \lesssim \eps \js^{-\frac12},\\
\big\|  \widehat{\calF}^{-1}[\wtilw_j(s,\cdot)](y)\big \|_{L_y^2} &= \big \| w_j(s,y) \big \|_{L_y^2} \lesssim \eps,
\end{align*}
for $1 \leq j \leq 3$. On the other hand, the $L^2$-boundedness of the free Schr\"odinger group gives 
\begin{equation}\label{equ:proof-free-disp3}
\begin{split}
\big\|  \widehat{\calF}^{-1}[\wtilw_j^\#(s,\cdot)](y)\big \|_{L_y^2} &= \big\|  e^{is \py^2}\widehat{\calF}^{-1}[\pxi \big(\frakb_j(\xi_j) \tilf_+(s,\xi_j)\big)](y)\big \|_{L_y^2} \\
&\lesssim \big \| \pxi \big(\frakb_j(\xi_j) \tilf_+(s,\xi_j)\big) \|_{L_\xi^2} \lesssim \| \tilf_+(s,\cdot)\|_{H_\xi^1} \lesssim \eps \js^\delta,
\end{split}
\end{equation}
for $1 \leq j \leq 3$. Using the Plancherel's identity, the convolution property of the standard Fourier transform, \eqref{equ:proof-free-disp1}, \eqref{equ:proof-free-disp2}, and \eqref{equ:proof-free-disp3}, we obtain 
\begin{equation*}
\begin{split}
&\big \| \calI_{\delta_0}^1(s,\cdot) \big \|_{L_\xi^2} \lesssim \Big \| \big(\wtilw_1(s,\cdot) \ast \overline{\wtilw_2(s,\cdot)} \ast \wtilw_3(s,\cdot)\big)(\xi) \Big \|_{L_\xi^2}\lesssim \Big \| \widehat{\calF}^{-1}\big[\big(\wtilw_1(s,\cdot) \ast \overline{\wtilw_2(s,\cdot)} \ast \wtilw_3(s,\cdot)\big)\big](y) \Big \|_{L_y^2}\\
&\lesssim \Big \|\widehat{\calF}^{-1}[\wtilw_1(s,\cdot)]\cdot  \widehat{\calF}^{-1}[ \overline{\wtilw_2(s,\cdot)}] \cdot  \widehat{\calF}^{-1}[\wtilw_3(s,\cdot)]  \Big \|_{L_y^2}\lesssim \Big \|\widehat{\calF}^{-1}[\wtilw_1(s,\cdot)] \Big \|_{L_y^2}\Big \|\widehat{\calF}^{-1}[\wtilw_2(s,\cdot)] \Big \|_{L_y^\infty}\Big \|\widehat{\calF}^{-1}[\wtilw_3(s,\cdot)] \Big \|_{L_y^\infty}\\
&\lesssim \eps^3 \js^{-1},
\end{split}
\end{equation*}
and analogously
\begin{equation*}
\begin{split}
\big \| \calI_{\delta_0}^2(s,\cdot) \big \|_{L_\xi^2} \lesssim \Big \|\widehat{\calF}^{-1}[\wtilw_1^\#(s,\cdot)] \Big \|_{L_y^2}\Big \|\widehat{\calF}^{-1}[\wtilw_2(s,\cdot)] \Big \|_{L_y^\infty}\Big \|\widehat{\calF}^{-1}[\wtilw_3(s,\cdot)] \Big \|_{L_y^\infty} \lesssim \eps^3 \js^{-1+\delta}.
\end{split}
\end{equation*}
We estimate $\calI_{\delta_0}^3(s,\xi)$ and $\calI_{\delta_0}^4(s,\xi)$ identically to $\calI_{\delta_0^2}(s,\xi)$. Therefore, the contribution of cubic interations with a Dirac kernel to \eqref{equ:proof_weighted_cubic2} is given by 
\begin{equation*}
\begin{split}
&\left\| \int_0^t e^{i\theta_1(s)}  e^{-i\theta_2(s)\xi} \pxi \calJ_{\delta_0}(s,\xi) \,\ud s \right \|_{L_\xi^2} \lesssim \int_0^t \left \| \pxi \calJ_{\delta_0}^{\mathrm{schem}}(s,\xi) \right \|_{L_\xi^2} \,\ud s\\
&\lesssim \int_0^t \Big( \left \| \calI_{\delta_0}^1(s,\xi) \right \|_{L_\xi^2}+\left \| \calI_{\delta_0}^2(s,\xi) \right \|_{L_\xi^2}+\left \| \calI_{\delta_0}^3(s,\xi) \right \|_{L_\xi^2}+\left \| \calI_{\delta_0}^4(s,\xi) \right \|_{L_\xi^2} \Big) \,\ud s\\
&\lesssim \int_0^t \big(\eps^3 \js^{-1} + \eps^3 \js^{-1+\delta} \big)\,\ud s \lesssim \eps^3 \jt^\delta.
\end{split}
\end{equation*}
\medskip \noindent \underline{Case 2: Cubic interactions with a Hilbert-type kernel.} Considering the formula for the Fourier transform of $\tanh(\cdot)$ in \eqref{equ:preliminaries_FT_tanh} and the tensorized structure of the symbol $\frakp_2/\frakp$ in Item~(1) of Lemma~\ref{lemma:cubic_NSD}, we may write $\calJ_{\pvdots}(s,\xi)$ as a linear combination of terms of the schematic form 
\begin{equation*}
\calJ_{\pvdots}^{\mathrm{schem}}(s,\xi) := \overline{\fraka_0(\xi)} \iiint e^{is\Phi} g_1(s,\xi_1) \overline{g_2(s,\xi_2)} g_3(s,\xi-\xi_1+\xi_2-\xi_4)  \widehat{\calF}[\tanh(\sqrt{\ulomega}\cdot)](\xi_4) \,\ud \xi_1\,\ud \xi_2\,\ud \xi_4,
\end{equation*}
with inputs $g_j(s,\xi_j) = \fraka_j(\xi_j) \tilf_+(s,\xi_j)$ and symbols $\fraka_j \in W^{1,\infty}$ for $1 \leq j \leq 3$. In view of Item~(3) of Lemma~\ref{lemma:cubic_NSD}, at least one of the symbols satisfies the vanishing property $\fraka_j(0)=0$. By taking partial derivative in $\xi$, we get
\begin{equation*}
\begin{split}
&\pxi \calJ_{\pvdots}^{\mathrm{schem}}(s,\xi) \\
&= \overline{\fraka_0(\xi)} \iiint (is \pxi \Phi) e^{is\Phi} g_1(s,\xi_1) \overline{g_2(s,\xi_2)} g_3(s,\xi-\xi_1+\xi_2-\xi_4)  \widehat{\calF}[\tanh(\sqrt{\ulomega}\cdot)](\xi_4) \,\ud \xi_1\,\ud \xi_2\,\ud \xi_4\\
&\quad + \overline{\pxi \fraka_0(\xi)} \iiint e^{is\Phi} g_1(s,\xi_1) \overline{g_2(s,\xi_2)} g_3(s,\xi-\xi_1+\xi_2-\xi_4)  \widehat{\calF}[\tanh(\sqrt{\ulomega}\cdot)](\xi_4) \,\ud \xi_1\,\ud \xi_2\,\ud \xi_4\\
&\quad + \overline{\fraka_0(\xi)} \iiint e^{is\Phi} g_1(s,\xi_1) \overline{g_2(s,\xi_2)} (\pxi g_3)(s,\xi-\xi_1+\xi_2-\xi_4)  \widehat{\calF}[\tanh(\sqrt{\ulomega}\cdot)](\xi_4) \,\ud \xi_1\,\ud \xi_2\,\ud \xi_4\\
&=: \calI_{\pvdots}^1(s,\xi) + \calI_{\pvdots}^2(s,\xi)+\calI_{\pvdots}^3(s,\xi).
\end{split}
\end{equation*}
By direct computation, one has the following identity for the phase $\Phi \equiv  \xi^2 - \xi_1^2 + \xi_2^2 - (\xi-\xi_1+\xi_2-\xi_4)^2$,
\begin{equation*}
\pxi \Phi = - \partial_{\xi_1} \Phi - \partial_{\xi_2} \Phi + 2 \xi_4.
\end{equation*}
We insert this identity into the term $\calI_{\pvdots}^1(s,\xi)$ and integrate by parts in $\xi_1$ or $\xi_2$ using the identity $\partial_{\xi_j} e^{is\Phi} = (is \partial_{\xi_j}\Phi)e^{is\Phi}$ to obtain 
\begin{equation*}
\begin{split}
&\calI_{\pvdots}^1(s,\xi)\\
&=\overline{\fraka_0(\xi)} \iiint e^{is\Phi} (\pxi g_1)(s,\xi_1) \overline{g_2(s,\xi_2)} g_3(s,\xi-\xi_1+\xi_2-\xi_4)  \widehat{\calF}[\tanh(\sqrt{\ulomega}\cdot)](\xi_4) \,\ud \xi_1\,\ud \xi_2\,\ud \xi_4\\
&\quad + \overline{\fraka_0(\xi)} \iiint  e^{is\Phi} g_1(s,\xi_1) \overline{(\pxi g_2)(s,\xi_2)} g_3(s,\xi-\xi_1+\xi_2-\xi_4)  \widehat{\calF}[\tanh(\sqrt{\ulomega}\cdot)](\xi_4) \,\ud \xi_1\,\ud \xi_2\,\ud \xi_4\\
&\quad + 2 \overline{\fraka_0(\xi)} \iiint s \cdot e^{is\Phi} g_1(s,\xi_1) \overline{g_2(s,\xi_2)} g_3(s,\xi-\xi_1+\xi_2-\xi_4)  \hat{\varphi}(\xi_4) \,\ud \xi_1\,\ud \xi_2\,\ud \xi_4\\
&=: \calI_{\pvdots}^{1,1}(s,\xi) + \calI_{\pvdots}^{1,2}(s,\xi) +\calI_{\pvdots}^{1,3}(s,\xi),
\end{split}
\end{equation*}
with $\hat{\varphi}(\xi) := i \xi \widehat{\calF}[\tanh(\sqrt{\ulomega}\cdot)](\xi) = \sqrt{\ulomega} \widehat{\calF}[\sech^2(\sqrt{\ulomega}\cdot)](\xi)$ a Schwartz function. Let us use the same notation for the input functions 
\begin{equation*}
\wtilw_j(s,\xi_j) := e^{-is\xi_j^2}\fraka_j(\xi_j) \tilf_+(s,\xi_j), \quad \text{and} \quad \wtilw_j^\#(s,\xi_j) := e^{-is\xi_j^2}\pxi \big(\fraka_j(\xi_j) \tilf_+(s,\xi_j)\big),
\end{equation*}
for $1 \leq j \leq 3$. These input functions enjoy the same decay estimates \eqref{equ:proof-free-disp1}--\eqref{equ:proof-free-disp3} as in the previous case. We make a change of variables by setting $\xi_3 = \xi-\xi_1+\xi_2-\xi_4$ to obtain a convolution
\begin{equation*}
\begin{split}
\calI_{\pvdots}^{1,1}(s,\xi) = \overline{\fraka_0(\xi)} \iiint e^{is(\xi^2 - \xi_1^2 + \xi_2^2 - \xi_3^2)} (\pxi g_1)(s,\xi_1) \overline{g_2(s,\xi_2)} g_3(s,\xi_3) \times & \\
\times \widehat{\calF}[\tanh(\sqrt{\ulomega}\cdot)](\xi - \xi_1 + \xi_2 - \xi_3) &\,\ud \xi_1\,\ud \xi_2\,\ud \xi_3\\
= e^{is\xi^2 }\overline{\fraka_0(\xi)} \Big(\wtilw_1^\#(s,\cdot) \ast \overline{\wtilw_2(s,\cdot)} \ast \wtilw_3(s,\cdot) \ast \widehat{\calF}[\tanh(\sqrt{\ulomega}\cdot)] \Big)(\xi),
\end{split}
\end{equation*}
and we note that all other terms $\calI_{\pvdots}^2(s,\xi)$, $\calI_{\pvdots}^3(s,\xi)$, $\calI_{\pvdots}^{1,2}(s,\xi)$, $\calI_{\pvdots}^{1,3}(s,\xi)$ can also be rewritten as a convolution in a similar manner. Using the Plancherel's identity, the convolution property, \eqref{equ:proof-free-disp1}, \eqref{equ:proof-free-disp2}, and  \eqref{equ:proof-free-disp3}, we obtain that 
\begin{equation*}
\begin{split}
\left \| \calI_{\pvdots}^{1,1}(s,\xi) \right \|_{L_\xi^2} &\lesssim \left \| e^{is\xi^2 }\overline{\fraka_0(\xi)} \Big(\wtilw_1^\#(s,\cdot) \ast \overline{\wtilw_2(s,\cdot)} \ast \wtilw_3(s,\cdot) \ast \widehat{\calF}[\tanh(\sqrt{\ulomega}\cdot)] \Big)(\xi) \right \|_{L_\xi^2}\\
&\lesssim \left \| \wtilw_1^\#(s,\cdot) \ast \overline{\wtilw_2(s,\cdot)} \ast \wtilw_3(s,\cdot) \ast \widehat{\calF}[\tanh(\sqrt{\ulomega}\cdot)] \right \|_{L_\xi^2} \\
&= \left \| \widehat{\calF}^{-1}\big[ \wtilw_1^\#(s,\cdot) \ast \overline{\wtilw_2(s,\cdot)} \ast \wtilw_3(s,\cdot) \ast \widehat{\calF}[\tanh(\sqrt{\ulomega}\cdot)] \big](y) \right \|_{L_y^2}\\
&\lesssim  \Big \|\widehat{\calF}^{-1}[\wtilw_1^\#(s,\cdot)] \Big \|_{L_y^2}\Big \|\widehat{\calF}^{-1}[\wtilw_2(s,\cdot)] \Big \|_{L_y^\infty}\Big \|\widehat{\calF}^{-1}[\wtilw_3(s,\cdot)] \Big \|_{L_y^\infty} \Big \|\tanh(\sqrt{\ulomega}\cdot) \Big \|_{L_y^\infty} \\
&\lesssim \eps^3 \js^{-1+\delta},
\end{split}
\end{equation*}
and the following bounds hold analogously
\begin{equation*}
\left \| \calI_{\pvdots}^{1,2}(s,\xi) \right \|_{L_\xi^2} +\left \| \calI_{\pvdots}^{2}(s,\xi) \right \|_{L_\xi^2} +\left \| \calI_{\pvdots}^{3}(s,\xi) \right \|_{L_\xi^2} \lesssim \eps^3 \js^{-1+\delta}.
\end{equation*} 
For the remaining term $\calI_{\pvdots}^{1,3}(s,\xi)$, we use the vanishing property of the symbols $\fraka_j(\xi_j)$. More precisely, we assume either
\begin{equation}\label{equ:proof_pv_case1}
\fraka_0(\xi) = \frac{\xi}{(|\xi|-i\sqrt{\ulomega})^2},
\end{equation}
or 
\begin{equation}\label{equ:proof_pv_case2}
\fraka_k(\xi_k) = \frac{\xi_k}{(|\xi_k|-i\sqrt{\ulomega})^2}, \quad \text{for some $k \in \{1,2,3\}$}.
\end{equation}
In the former case \eqref{equ:proof_pv_case1}, the details proceed identically to the bound \eqref{equ:proof_cubic_reg_smoothing} considered in the regular cubic interactions case below. For the latter case \eqref{equ:proof_pv_case2}, we assume without lost of generality that $k = 1$ so from \eqref{equ:prep_flat_local_decay} we have the improved locay decay estimate 
\begin{equation}\label{equ:proof-free-disp4}
\left \| \jy^{-1} \widehat{\calF}^{-1} [ \wtilw_1(s,\cdot)](y) \right \|_{L_y^2} \lesssim \eps \js^{-1+\delta}.
\end{equation}
Using the Plancherel's identity, the spatial localization of $\varphi(y)$, \eqref{equ:proof-free-disp1}, and \eqref{equ:proof-free-disp4}, we obtain 
\begin{equation*}
\begin{split}
\left \| \calI_{\pvdots}^{1,3}(s,\xi) \right \|_{L_\xi^2} &\lesssim s \cdot  \Big \| \big(\wtilw_1(s,\cdot) \ast \overline{\wtilw_2(s,\cdot)} \ast \wtilw_3(s,\cdot) \ast \widehat{\varphi}(\cdot)\big)(\xi) \Big \|_{L_\xi^2}\\
&\lesssim s \cdot  \Big \| \widehat{\calF}^{-1}\big[ \big(\wtilw_1(s,\cdot) \ast \overline{\wtilw_2(s,\cdot)} \ast \wtilw_3(s,\cdot) \ast \widehat{\varphi}(\cdot)\big) \big] \Big \|_{L_y^2}\\
&\lesssim s \cdot \left \| \jy^{-1} \widehat{\calF}^{-1} [ \wtilw_1(s,\cdot)](y) \right \|_{L_y^2} \cdot \Big \|\widehat{\calF}^{-1}[\wtilw_2(s,\cdot)] \Big \|_{L_y^\infty}\Big \|\widehat{\calF}^{-1}[\wtilw_3(s,\cdot)] \Big \|_{L_y^\infty} \Big \| \jy \varphi(y)\Big \|_{L_y^\infty}\\
&\lesssim s \cdot \eps \js^{-1+\delta} \cdot \eps^2 \js^{-1} \lesssim \eps^3 \js^{-1+\delta}.
\end{split}
\end{equation*}
Thus, we gather all the preceeding estimates and conclude that the contribution of cubic interations with a Hilbert-type kernel to \eqref{equ:proof_weighted_cubic2} is given by 
\begin{equation*}
\begin{split}
&\left\| \int_0^t e^{i\theta_1(s)}  e^{-i\theta_2(s)\xi} \pxi \calJ_{\pvdots}(s,\xi) \,\ud s \right \|_{L_\xi^2} \lesssim \int_0^t \left \| \pxi \calJ_{\pvdots}^{\mathrm{schem}}(s,\xi) \right \|_{L_\xi^2} \,\ud s\\
&\lesssim \int_0^t \Big( \left \| \calI_{\pvdots}^1(s,\xi) \right \|_{L_\xi^2}+\left \| \calI_{\pvdots}^2(s,\xi) \right \|_{L_\xi^2}+\left \| \calI_{\pvdots}^3(s,\xi) \right \|_{L_\xi^2} \Big) \,\ud s\\
&\lesssim \int_0^t  \eps^3 \js^{-1+\delta} \,\ud s \lesssim \eps^3 \jt^\delta.
\end{split}
\end{equation*}

\medskip \noindent \underline{Case 3: Regular cubic interactions.} We consider one of the schematic terms from \eqref{equ:calJ_reg} for the cubic interactions with a regular kernel. Taking a partial derivative, we find that 
\begin{equation}\label{equ:proof_weighted_reg}
\begin{split}
&\int_0^t e^{i\theta_1(s)}  e^{-i\theta_2(s)\xi} \pxi \calJ_{\mathrm{reg}}^{\mathrm{schem}}(s,\xi) \,\ud s = \int_0^t e^{i\theta_1(s)}  e^{-i\theta_2(s)\xi} (2is\xi) e^{is(\xi^2+\ulomega)}\overline{\frakb_0(\xi)} \widehat{\calF}[G(s,\cdot)](\xi)\,\ud s \\
&\quad + \int_0^t e^{i\theta_1(s)}  e^{-i\theta_2(s)\xi}  e^{is(\xi^2+\ulomega)} \Big(\overline{(\pxi \frakb_0)(\xi)} \widehat{\calF}[G(s,\cdot)](\xi) +\overline{\frakb_0(\xi)} \pxi\widehat{\calF}[G(s,\cdot)](\xi) \Big)\,\ud s 
\end{split}
\end{equation}
with 
\begin{equation*}
G(s,y) = w_1(s,y) \overline{w_2(s,y)} w_3(s,y) \varphi(y) \quad \text{or} \quad G(s,y) = \overline{w_1(s,y)} w_2(s,y) \overline{w_3(s,y)} \varphi(y)
\end{equation*}
where $\varphi(y)$ is some fixed Schwartz function and the inputs $w_j(s,y)$ defined in \eqref{equ:proof_free_schr_waves}. Thanks to the estimates \eqref{equ:proof-free-disp1}, \eqref{equ:proof-free-disp5}, and the spatial localization of $\varphi(y)$, we infer that 
\begin{equation}\label{equ:proof_cubic_reg_input}
\big \| \jym G(s,y) \big \|_{L_y^2} + \big \| \jym \py G(s,y) \big \|_{L_y^2} \lesssim \eps^3 \js^{-\frac32 + \delta}.
\end{equation}
Hence, we are put back in the similar setting of \eqref{equ:proof_weighted_local_smooth1} and \eqref{equ:proof_weighted_local_smooth2} for the first term on the right hand side of \eqref{equ:proof_weighted_reg} which allow us to conclude that 
\begin{equation}\label{equ:proof_cubic_reg_smoothing}
\left \| \int_0^t e^{i\theta_1(s)}  e^{-i\theta_2(s)\xi} (2is\xi) e^{is(\xi^2+\ulomega)}\overline{\frakb_0(\xi)} \widehat{\calF}[G(s,\cdot)](\xi)\,\ud s \right \|_{L_\xi^2} \lesssim \eps^3 \jt^\delta,
\end{equation}
and we use \eqref{equ:proof_cubic_reg_input} to crudely bound the second term on the right-hand side of \eqref{equ:proof_weighted_reg} by
\begin{equation*}
\begin{split}
&\left \| \int_0^t e^{i\theta_1(s)}  e^{-i\theta_2(s)\xi}  e^{is(\xi^2+\ulomega)} \Big(\overline{(\pxi \frakb_0)(\xi)} \widehat{\calF}[G(s,\cdot)](\xi) +\overline{\frakb_0(\xi)} \pxi\widehat{\calF}[G(s,\cdot)](\xi) \Big)\,\ud s   \right \|_{L_\xi^2}\\
&\lesssim  \int_0^t  \big \| \frakb_0 \big \|_{W^{1,\infty}} \cdot \big \| \widehat{\calF}[G(s,\cdot)] \big \|_{H^1} \,\ud s \lesssim \int_0^t \big \| \jy G(s,y) \big \|_{L_y^2} \,\ud s \lesssim \int_0^t \eps^3 \js^{-1+\delta} \,\ud s \lesssim \eps^3 \jt^\delta.
\end{split}
\end{equation*}
This finishes the discussion of the weighted energy estimates for the regular cubic interaction terms and also the proof of the proposition.
\end{proof}

\section{Pointwise Estimates for the Profile}\label{sec:pointwise_profile}
In this section we establish uniform-in-time pointwise bounds for the profile. We refer the reader to \cite[Section~10]{LL24} for details on the nonlinear stationary phase estimates for the cubic terms.
 \begin{lemma} \label{lemma: effective-ODE-profile}
 Suppose the assumptions in the statement of Proposition~\ref{prop:profile_bounds} are in place. Assume $T \geq 1$. Then we have for all $\xi \in \bbR$ and for all $1 \leq t \leq T$ that
  \begin{equation}\label{eqn: effective-ODE-profile}
\begin{split}
 \partial_t \Big( e^{i\theta_1(t)}e^{-i\theta_2(t)\xi}&\big(\tilfplusulo(t,\xi)+\wtilB_{\ulomega,\ulp}(t,\xi)\big) \Big) \\
&= e^{i\theta_1(t)}e^{-i\theta_2(t)\xi}\Big(\frac{i}{2t} \vert \tilfplusulo(t,\xi)\vert^2 \tilfplusulo(t,\xi) + \wtilcalE_{\ulomega,\ulp}(t,\xi)\Big),
\end{split}
  \end{equation}
  where the remainder term  
  \begin{equation}\label{eqn: def-wtilcalE}
  \begin{split}
 \wtilcalE_{\ulomega,\ulp}(t,\xi) &:= \Big(-\frac{i}{2t}\vert \tilfplusulo(t,\xi)\vert^2 \tilfplusulo(t,\xi) -i e^{it(\xi^2+\ulomega)}\wtilcalF_{+,\ulomega}\big[\calC\big((\ulPe \Vp)(t)\big)\big](\xi) \Big) \\
 &\qquad -ie^{it(\xi^2+\ulomega)}\wtilcalR_{\ulomega,\ulp}(t,\xi) +i \big(\dot{\theta}_1(t) - \dot{\theta}_2(t)\xi\big)\wtilB_{\ulomega,\ulp}(t,\xi)
  \end{split}     
  \end{equation}
  satisfies
 \begin{equation}\label{eqn: pointwise_estimate_wtilcalE}
 \sup_{1 \leq t \leq T} \jt^{\frac{11}{10}-\delta} \big\Vert \wtilcalE_\ulomega(t,\xi) \big\Vert_{L_\xi^\infty} \lesssim \varepsilon^2.
 \end{equation}
 \end{lemma}
\begin{proof}
We note that \eqref{eqn: effective-ODE-profile} and \eqref{eqn: def-wtilcalE} are equivalent to the equations \eqref{equ:setup_evol_equ_renormalized_tilfplus} and \eqref{equ:setup_definition_wtilcalRulomega}. It remains to prove the estimate \eqref{eqn: pointwise_estimate_wtilcalE}. We note that the stationary phase analysis to infer the asymptotics for the cubic terms carries verbatim over from our previous work \cite{LL24}. More precisely, recall from \eqref{equ:cubic_decomp} that one has the following decomposition for the cubic terms
\begin{equation*}
e^{it(\xi^2+\ulomega)} \wtilcalF_{+, \ulomega}\big[\calC\big((\ulPe \Vp)(t)\big)\big](\xi) = \calJ_{\delta_0}(t,\xi) + \calJ_{\pvdots}(t,\xi) + \calJ_{\mathrm{reg}}(t,\xi),
\end{equation*}
where the right-hand side of the above equation represent trilinear expressions involving a Dirac delta kernel \eqref{equ:calJ_delta}, a Hilbert-type kernel \eqref{equ:calJ_pv}, and a regular kernel \eqref{equ:calJ_reg} respectively. We refer the reader to \cite[(10.21)--(10.23)]{LL24} for the proof of the following estimates 
\begin{align*}
\sup_{1 \leq t \leq T} \jt^{\frac65 - 3\delta} \left \| \calJ_{\delta_0}(t,\xi) + \frac{1}{2t}\vert \tilfplusulo(t,\xi)\vert^2 \tilfplusulo(t,\xi) \right \|_{L_\xi^\infty} &\lesssim \eps^3 ,\\
\sup_{1 \leq t \leq T} \jt^{\frac{11}{10} - \delta} \left \| \calJ_{\pvdots}(t,\xi)\right \|_{L_\xi^\infty} &\lesssim \eps^3 ,\\
\sup_{1 \leq t \leq T} \jt^{\frac32} \left \| \calJ_{\mathrm{reg}}(t,\xi) \right \|_{L_\xi^\infty} &\lesssim \eps^3 .
\end{align*}

Next, we estimate the contributions of the remaining terms in \eqref{eqn: def-wtilcalE}. By \eqref{equ:consequences_aux_bound_modulation2} and \eqref{equ:consequences_wtilcalB} we obtain for all $1 \leq t \leq T$ that 
\begin{equation*}
\left \|\big(\dot{\theta}_1(t) - \dot{\theta}_2(t)\xi\big)\wtilB_{\ulomega,\ulp}(t,\xi) \right \|_{L_\xi^\infty} \lesssim \big(|\theta_1(t)|+|\theta_2(t)|\big) \big \| \jxi \wtilB_{\ulomega,\ulp}(t,\xi) \big \|_{L_\xi^\infty} \lesssim \jt^{-2+\delta}\eps^3,
\end{equation*}
and using the bound \eqref{equ:consequences_wtilcalR} and Sobolev inequality, we conclude for all $1 \leq t \leq T$ that 
\begin{equation*}
\left\| e^{it(\xi^2+\ulomega)}\wtilcalR_{\ulomega,\ulp}(t,\xi) \right\|_{L_\xi^\infty} \lesssim \left\| \wtilcalR_{\ulomega,\ulp}(t,\xi) \right\|_{H_\xi^1} \lesssim \jt^{-\frac32+\delta}\eps^2.
\end{equation*}
Thus, the preceeding estimates imply the asserted bound \eqref{eqn: pointwise_estimate_wtilcalE}.
\end{proof}

The main uniform-in-time pointwise estimates for the profile are summarized in the following proposition.
 \begin{proposition}\label{prop:pointwise_estimate}
  Suppose the assumptions in the statement of Proposition~\ref{prop:profile_bounds} are in place. Assume $T \geq 1$. Then we have uniformly for all $0 \leq t \leq T$ that
  \begin{equation}\label{equ:pointwise_esimate}
   \bigl\| \tilf_{+, \ulomega,\ulp}(t,\xi) \bigr\|_{L^\infty_\xi} + \bigl\| \tilf_{-, \ulomega,\ulp}(t,\xi) \bigr\|_{L^\infty_\xi} \lesssim \eps + \varepsilon^2.
  \end{equation}
 Moreover, we obtain for arbitrary times $1 \leq t_1 \leq t_2 \leq T$ that
 \begin{align}
 \Big\Vert e^{-i\Lambda_+(t_2,\xi)} e^{i\theta_1(t_2)}e^{-i\theta_2(t_2)\xi}\tilf_{+,\ulomega,\ulp}(t_2,\xi) -e^{-i\Lambda_+(t_1,\xi)} e^{i\theta_1(t_1)}e^{-i\theta_2(t_1)\xi}\tilf_{+,\ulomega,\ulp}(t_1,\xi)\Big \|_{L_\xi^\infty} &\lesssim \varepsilon^2 t_1^{-\frac{1}{10}+\delta}, \label{eqn: Cauchy-in-time-estimate}\\
  \Big\Vert e^{i\Lambda_-(t_2,\xi)} e^{-i\theta_1(t_2)}e^{-i\theta_2(t_2)\xi}\tilf_{-,\ulomega,\ulp}(t_2,\xi) -e^{i\Lambda_-(t_1,\xi)} e^{-i\theta_1(t_1)}e^{-i\theta_2(t_1)\xi}\tilf_{-,\ulomega,\ulp}(t_1,\xi) \Big \|_{L_\xi^\infty} &\lesssim \varepsilon^2 t_1^{-\frac{1}{10}+\delta}, \label{eqn: Cauchy-in-time-estimate2}
 \end{align}
 where
 \begin{equation}\label{eqn: integrating-factor-plus}
     \Lambda_\pm(t,\xi) := \frac{1}{2}\int_1^t \frac{1}{s} \, \vert \tilf_{\pm,\ulomega,\ulp}(s,\xi)\vert^2 \,\ud s.
 \end{equation}
 \end{proposition} 
\begin{proof}
For short times $0\leq t \leq 1$, using the Sobolev's inequality with bounds \eqref{equ:consequences_sobolev_bound_profile} and \eqref{equ:weighted_estimate}, we obtain
\begin{equation*}
\begin{split}
\bigl\| \tilf_{+, \ulomega,\ulp}(t,\xi) \bigr\|_{L^\infty_\xi} + \bigl\| \tilf_{-, \ulomega,\ulp}(t,\xi) \bigr\|_{L^\infty_\xi} \lesssim \bigl\| \tilf_{+, \ulomega,\ulp}(t,\xi) \bigr\|_{H_\xi^1} + \bigl\| \tilf_{-, \ulomega,\ulp}(t,\xi) \bigr\|_{H_\xi^1} \lesssim   \eps + \varepsilon^2,
\end{split}
\end{equation*}
while for times $t \geq 1$, the pointwise estimate \eqref{equ:pointwise_esimate} follows from \eqref{eqn: Cauchy-in-time-estimate} and \eqref{eqn: Cauchy-in-time-estimate2} evaluated at times $t_2 = t$ and $t_1 = 1$, and by triangle inequality.  Let us show \eqref{eqn: Cauchy-in-time-estimate} using the previous lemma. From \eqref{eqn: effective-ODE-profile}, we multiply the integrating factor \eqref{eqn: integrating-factor-plus} to infer that 
 \begin{equation} \label{equ:effective_phased_ODE}
\begin{split}
& \partial_t \Big(e^{-i\Lambda_+(t,\xi)} e^{i\theta_1(t)}e^{-i\theta_2(t)\xi}\big(\tilfplusulo(t,\xi)+\wtilB_{\ulomega,\ulp}(t,\xi)\big) \Big) \\
&= -i \dot{\Lambda}_+(t,\xi)\Big( e^{-i\Lambda_+(t,\xi)} e^{i\theta_1(t)}e^{-i\theta_2(t)\xi}\big(\tilfplusulo(t,\xi)+\wtilB_{\ulomega,\ulp}(t,\xi)\big) \Big) \\
&\quad +  e^{-i\Lambda_+(t,\xi)}e^{i\theta_1(t)}e^{-i\theta_2(t)\xi}\Big(\frac{i}{2t} \vert \tilfplusulo(t,\xi)\vert^2 \tilfplusulo(t,\xi) + \wtilcalE_{\ulomega,\ulp}(t,\xi)\Big)\\
&= e^{-i\Lambda_+(t,\xi)}e^{i\theta_1(t)}e^{-i\theta_2(t)\xi}\Big(-\frac{i}{2t} \vert \tilfplusulo(t,\xi)\vert^2 \wtilB_{\ulomega,\ulp}(t,\xi) + \wtilcalE_{\ulomega,\ulp}(t,\xi)\Big).
\end{split}
  \end{equation}
By \eqref{equ:prop_profile_bounds_assumption2}, \eqref{equ:consequences_wtilcalB}, and \eqref{eqn: pointwise_estimate_wtilcalE}, we obtain the following estimate for the right-hand side of \eqref{equ:effective_phased_ODE} for all $1 \leq t \leq T$,
\begin{equation}\label{equ:RHS-ODE-estimate}
\begin{split}
&\left \|e^{-i\Lambda_+(t,\xi)}e^{i\theta_1(t)}e^{-i\theta_2(t)\xi}\Big(-\frac{i}{2t} \vert \tilfplusulo(t,\xi)\vert^2 \wtilB_{\ulomega,\ulp}(t,\xi) + \wtilcalE_{\ulomega,\ulp}(t,\xi)\Big) \right \|_{L_\xi^\infty}\\
&\lesssim t^{-1} \left \|\vert \tilfplusulo(t,\xi)\vert^2 \wtilB_{\ulomega,\ulp}(t,\xi) \right \|_{L_\xi^\infty} + \left \|\wtilcalE_{\ulomega,\ulp}(t,\xi) \right \|_{L_\xi^\infty} \\
&\lesssim t^{-2} \eps^4 + t^{-\frac{11}{10}+\delta}\eps^2.
\end{split}
\end{equation}
Using \eqref{equ:consequences_wtilcalB} and the preceeding estimate, we obtain the asserted bound  \eqref{eqn: Cauchy-in-time-estimate} by integrating \eqref{equ:effective_phased_ODE} from $t_1$ to $t_2$. It remains to prove \eqref{eqn: Cauchy-in-time-estimate2}. From \eqref{equ:setup_components_relation} we have 
\begin{equation*}
\big| \tilf_{-, \ulomega,\ulp}(t,\xi)\big|^2 = \big| \tilf_{+, \ulomega,\ulp}(t,-\xi)\big|^2 , \quad \forall \xi \in \bbR,
\end{equation*}
which then implies that $\Lambda_-(t,\xi) \equiv  \Lambda_+(t,-\xi)$.  Applying the relation \eqref{equ:setup_components_relation} to \eqref{equ:effective_phased_ODE} we infer that 
\begin{equation*}
\begin{split}
& \partial_t \Big(e^{i\Lambda_-(t,\xi)} e^{-i\theta_1(t)}e^{-i\theta_2(t)\xi}\big(\tilfplusulo(t,\xi)-\frac{(|\xi|-i\sqrt{\ulomega})^2}{(|\xi|+i\sqrt{\ulomega})^2}\overline{\wtilB_{\ulomega,\ulp}(t,-\xi)}\big) \Big) \\
&= -\frac{(|\xi|-i\sqrt{\ulomega})^2}{(|\xi|+i\sqrt{\ulomega})^2}e^{i\Lambda_-(t,\xi)}e^{-i\theta_1(t)}e^{-i\theta_2(t)\xi}\Big(\frac{i}{2t} \vert \tilfplusulo(t,-\xi)\vert^2 \overline{\wtilB_{\ulomega,\ulp}(t,-\xi)} + \overline{\wtilcalE_{\ulomega,\ulp}(t,-\xi)}\Big).
\end{split}
\end{equation*}
Thus, the asserted bound \eqref{eqn: Cauchy-in-time-estimate2} similarly follows by integrating the above equation from $t_1$ to $t_2$ and by the estimates \eqref{equ:consequences_wtilcalB}, \eqref{equ:RHS-ODE-estimate}.
\end{proof}

\bibliographystyle{amsplain}
\bibliography{references}

\end{document}